\documentclass{article}
\usepackage[a4paper, total={6in, 9in}]{geometry}
\usepackage{amsfonts}
\usepackage{amssymb}
\usepackage{amsmath}
\usepackage{amsthm}
\usepackage{mathrsfs}
\usepackage{enumerate}
\usepackage{cite}
\usepackage{color,soul}
\usepackage{setspace}
\usepackage{imakeidx}
\makeindex[name=n,title={Index of notation}]
\makeindex[title={Index of terminology}]

\usepackage[all,cmtip]{xy}
\usepackage{graphicx}

\usepackage{wrapfig}
\usepackage{lipsum}
\usepackage{multirow}

\newtheorem{theorem}{Theorem}[section]
\newtheorem{prop}[theorem]{Proposition}
\newtheorem{lemma}[theorem]{Lemma}
\newtheorem{conjecture}[theorem]{Conjecture}
\newtheorem{cor}[theorem]{Corollary}

\theoremstyle{definition}
\newtheorem{definition}[theorem]{Definition}
\newtheorem{remark}[theorem]{Remark}
\newtheorem{question}[theorem]{Question}

\newtheorem{hypothesis}[theorem]{Hypothesis}
\newtheorem{notation}[theorem]{Notation}
\newtheorem{example}[theorem]{Example}
\numberwithin{equation}{theorem}
\newcommand{\Z}{\mathbb{Z}}
\newcommand{\Q}{\mathbb{Q}}
\newcommand{\I}{\mathbb{I}}
\newcommand{\A}{\mathbb{A}}
\newcommand{\N}{\mathbb{N}}
\newcommand{\bbP}{\mathbb{P}}
\newcommand{\bbO}{\mathbb{O}}
\newcommand{\OX}{\mathcal{O}_X}
\newcommand{\FF}{\mathcal{F}}
\newcommand{\II}{\mathcal{I}}
\newcommand{\JJ}{\mathcal{J}}
\newcommand{\MM}{\mathcal{M}}
\newcommand{\HH}{\mathcal{H}}
\newcommand{\NN}{\mathcal{N}}
\newcommand{\OO}{\mathcal{O}}
\newcommand{\LL}{\mathcal{L}}
\newcommand{\GK}{\mathrm{GK}}
\newcommand{\GKdim}{\mathrm{GKdim}}
\newcommand{\gldim}{\mathrm{gl.dim}}
\newcommand{\Ext}{\mathrm{Ext}}

\newcommand{\Hom}{\mathrm{Hom}}
\newcommand{\End}{\mathrm{End}}

\newcommand{\coh}{\mathrm{coh}}
\newcommand{\gr}{\mathrm{gr}}
\newcommand{\Gr}{\mathrm{Gr}}
\newcommand{\Gl}{\mathrm{Gl}}
\newcommand{\rann}{\mathrm{r.ann}}
\newcommand{\lann}{\ell\mathrm{.ann}}
\newcommand{\qgr}{\mathrm{qgr}}
\newcommand{\Qgr}{\mathrm{Qgr}}
\newcommand{\Sg}{S_{(g)}}
\newcommand{\bfd}{\mathbf{d}}
\newcommand{\bfe}{\mathbf{e}}

\newcommand{\bfx}{\mathbf{x}}
\newcommand{\bfy}{\mathbf{y}}
\newcommand{\mbf}{\mathbf}

\newcommand{\ovl}{\overline}
\newcommand{\ehd}{\overset{\bullet}=}

\title{Maximal Orders in the Sklyanin Algebra}
\author{Dominic Hipwood}
\date{\today}
\begin{document}
\maketitle
\begin{abstract}

A major current goal of noncommutative geometry is the classification of noncommutative projective surfaces. The generic case is to understand algebras birational to the Sklyanin algebra. In this thesis we complete a considerable component of this problem.\par

Let $S$ denote the 3-dimensional Sklyanin algebra over an algebraically closed field, and assume that $S$ is not a finite module over its centre. In earlier work Rogalski, Sierra and Stafford classified the maximal orders inside the 3-Veronese $S^{(3)}$ of $S$. We complete and extend their work and classify all maximal orders inside $S$. As in Rogalski, Sierra and Stafford's work, these can be viewed as blowups at (possibly non-effective) divisors. A consequence of this classification is that maximal orders are automatically noetherian among other desirable properties. \par

This work both relies upon, and lends back to, the work of Rogalski, Sierra and Stafford. In particular, we also provide a converse for their classification of maximal orders in the higher Veronese subrings $S^{(3n)}$ of $S^{(3)}$. 

\end{abstract}
\tableofcontents

\section{Introduction}

Roughly, the goal of noncommutative projective geometry is to use techniques and intuition from commutative geometry to study noncommutative algebras. A major current goal of noncommutative geometry is the classification of noncommutative projective surfaces; in the language of algebras, a classification of connected graded domains of GK-dimension 3. Within this, a major target is a classification of algebras birational to the Sklyanin algebra. This problem provides the main motivation of our work. In this introduction we will first present our main results and then describe the role they play in the problem mentioned above.

\subsection{First definitions}\label{1st defs}

First, as always, we need to get some basic definitions and formalities out of the way. We note that the terms appearing below are given a proper introduction in section~\ref{background}.\par

Fix an algebraically closed field $\Bbbk$. Let $A=\bigoplus_{n\in\N} A_i$  be an $\N$-graded  $\Bbbk$-algebra and a domain. We call $A$ \textit{connected graded (cg)} if $A_0=\Bbbk$ and $\dim_\Bbbk A_i<\infty$ for all $i$. Almost all algebras considered will be cg domains of finite Gelfand Kirillov dimension (defined in Notation~\ref{GK dimension}). In particular we assume this for the next few definitions.\par
The algebra $A$ sits inside its \textit{graded quotient ring}, $Q_\gr(A)$, formed by inverting the Ore set of all nonzero homogenous elements. The division ring $D_\gr(A)=Q_\gr(A)_0$, will be called the \textit{noncommutative function field of $A$}. If $B$ is another cg domain with $Q_\gr(A)=Q_\gr(B)=Q$, then $A$ and $B$ are called \textit{equivalent orders} if there exist nonzero elements $x_1,x_2,y_1,y_2\in Q$ such that $x_1Ax_2\subseteq B$ and $y_1By_2\subseteq A$. We call $A$ a \textit{maximal order} if there is no equivalent order $B$ to $A$ such that $A\subsetneq B$. Let $R$ be another algebra such that $A\subseteq R\subseteq Q$. We call $A$ a \textit{a maximal $R$-order} if there exists no graded equivalent order $B$ to $A$ such that $A\subsetneq B\subseteq R$. Given $d\in\N$, the \textit{$d$-Veronese of $A$} is the subring $A^{(d)}=\bigoplus_{i\in\N}A_{di}$. It is (usually) given the grading $A^{(d)}_n=A_{dn}$.\par

A construction that plays a fundamental role in our work, and in general noncommutative geometry, is that of a twisted homogenous coordinate ring. Fix a projective scheme $X$. Let $\LL$ be an invertible sheaf on $X$ with global sections $H^0(X,\LL)$, and let $\sigma:X \rightarrow X $ be an automorphism of $X$. Write $\LL^\sigma$ for the pullback sheaf $\sigma^*\LL$. Set $\LL_0=\OX$ and $\LL_n=\LL\otimes \LL^\sigma \otimes \dots \otimes \LL^{\sigma^{n-1}}$ for $n\geq 1$. We define $B(X,\LL,\sigma)=\bigoplus_{n\in\N}H^0(X,\LL_n)$. There is a natural ring structure on $B(X,\LL,\sigma)$ and we call it a \textit{twisted homogenous coordinate ring}.\par

Our results regard certain subalgebras of the Sklyanin algebra, which we now define. Let $a,b,c\in\Bbbk$, then we set
$$S=\frac{\Bbbk\langle x,y,z \rangle}{(azy+byz+cx^2, axz+bzx+cy^2, ayx+bxy+cz^2)}.$$
Provided $a,b,c$ are general enough (see Definition~\ref{sklyanin}), then $S$ has a central element $g\in S_3$, unique up to scalar multiplication and such that $S/gS \cong B(E, \LL, \sigma)$. Here $E$ is a nonsingular elliptic curve, $\LL$ is an invertible sheaf of degree 3, and $\sigma$ is an automorphism of $E$. We call such $S$ a \textit{Sklyanin algebra}. \par

\subsection{The main results}

The results of this thesis are analogous to those of \cite{Ro, RSS} described in Section~\ref{Rogs blowup review} and Section~\ref{RSS review}, where the authors tackled similar problems inside the 3-Veronese subring $T=S^{(3)}$ of $S$. Here, as in \cite{Ro, RSS}, our results concern certain \textit{blowup subalgebras} and \textit{virtual blowup subalgebras} of $S$ at effective and so-called \textit{virtually effective divisors} on $E$. Our assumption throughout is Hypothesis~\ref{standing assumption intro}.

\begin{hypothesis}[Standing Assumption]\label{standing assumption intro}
Fix an algebraically closed field $\Bbbk$. Fix a 3 dimensional Sklyanin algebra $S$. Let $g\in S_3$ be such that $S/gS\cong B(E,\LL,\sigma)$, where $E$ is a smooth elliptic curve, $\LL$ an invertible sheaf on $E$ with $\deg\LL=3$ and $\sigma:E\to E$ an automorphism of $E$. Assume that $\sigma$ is of infinite order.
\end{hypothesis}

Before presenting our results we set some notation as standard.

\begin{notation}\label{ovlX notation}\index[n]{x@$\ovl{X}$}
Given a subset $X\subseteq S$, we write $\ovl{X}=(X+gS)/gS \subseteq S/gS$. Similarly for $x\in S$, $\ovl{x}$ will denote its image in $\ovl{S}=B(E,\LL,\sigma)$.
\end{notation}

\begin{definition}[Definition~\ref{S(p)} and Definition~\ref{S(p+q)}]\label{S(d) def} \index[n]{sd@$S(\bfd)$}\index{blowup of $S$ at $\bfd$}
Let $\bfd$ be an effective divisor on $E$ with $\deg\bfd\leq 2$. For $i=1,2,3$, put
\begin{itemize}
\item $V_1=\{x\in S_1\,|\; \ovl{x}\in H^0(E,\LL(-\bfd))\}$,
\item $V_2=\{x\in S_2\,|\; \ovl{x}\in H^0(E,\LL_2(-\bfd-\sigma^{-1}(\bfd)))\}$,
\item $V_3=\{x\in S_3\,|\; \ovl{x}\in H^0(E,\LL_3(-\bfd-\sigma^{-1}(\bfd)-\sigma^{-2}(\bfd)))\}$.
\end{itemize}
We define the \textit{blowup of $S$ at $\bfd$} to be the subalgebra of $S$ generated by $V_1$, $V_2$ and $V_3$:
\begin{equation}\label{S(d)=k<V1,V2,V3>} S(\bfd)=\Bbbk\langle V_1,V_2,V_3\rangle.\end{equation}
\end{definition}

When $\bfd=p$ is a single point, (\ref{S(d)=k<V1,V2,V3>}) is an overkill and in fact $S(p)=\Bbbk\langle V_1\rangle$ is generated in degree 1. The ring $S(p)$ was studied in by Rogalski, and proved to be the only degree 1 generated maximal order inside $S$. When $\bfd=p+q$ is two points Definition~\ref{S(d) def} is both new, and harder to understand than $S(p)$. The $S(\bfd)$ are the correct analogue of the blowup subalgebras $S^{(3)}(\bfe)$ of $S^{(3)}$ defined by Rogalski in \cite[Section~1]{Ro}. The results of Rogalski are reviewed in section~\ref{background} along with the other undefined terms below.

\begin{theorem}[Theorem~\ref{S(d) thm}]\label{S(d) thm intro}
Let $\bfd$ be an effective divisor on $E$ of degree $d\leq 2$. Set $R=S(\bfd)$. Then:
\begin{enumerate}[(1)]
\item $R\cap gS=gR$ with $R/gR\cong B(E,\LL(-\bfd),\sigma)$. The Hilbert series of $R$ is
$$h_{R}(t)=\frac{t^2+(1-d)t+1}{(t-1)^2(1-t^3)}.$$
\item The 3-Veronese $R^{(3)}$ is a blowup subalgebra of $S^{(3)}$. More specifically
    $$R^{(3)}=S^{(3)}(\bfd+\sigma^{-1}(\bfd)+\sigma^{-2}(\bfd)).$$
\item $R$ is strongly noetherian meaning that $R\otimes_\Bbbk Z$ remains noetherian for all commutative noetherian $\Bbbk$-algebras $Z$; $R$ satisfies $\chi$ on the left and right (Definition~\ref{chi}), has cohomological dimension 2, and possesses a balanced dualizing complex;
\item $R$ is Auslander-Gorenstein and Cohen-Macaulay;
\item $R$ is a maximal order in $Q_\gr(R)=Q_\gr(S)$.
\end{enumerate}
\end{theorem}

Theorem~\ref{S(d) thm intro}(3)(4) can be summarised as saying $S(\bfd)$ satisfies some of the most useful homological properties. We give more details about why these are desirable properties in Remark~\ref{desirable properties}.  Missing from Theorem~\ref{S(d) thm intro} is finite global dimension: Corollary~\ref{infinite global dim} shows that this is never obtained when $\deg\bfd\geq 1$. Although not explicit below, obtaining Theorem~\ref{S(d) thm intro} is absolutely essential for the rest of our main results.\par

For a complete classification of maximal $S$-orders we need to introduce virtual blowups at virtually effective divisors. We only define a virtually effective divisor here. For the purposes of the introduction one may take the conclusions of Proposition~\ref{vblowup exist intro} as the definition of a virtual blowup.

\begin{definition}\label{veff div intro}
A divisor $\bfx$ is called \textit{virtually effective} if for all $n\gg0$, the divisor
$\bfx+\sigma^{-1}(\bfx)+\dots+\sigma^{-(n-1)}(\bfx)$
is effective.
\end{definition}

\begin{prop}[Proposition~\ref{RSS 7.4(3)} and Definition~\ref{virtual blowup}]\label{vblowup exist intro}
Let $\mbf{x}$ be a virtually effective divisor of degree at most 2. Then there exists a virtual blowup $S(\bfx)$. In particular:
\begin{enumerate}[(1)]
\item $S(\bfx)$ is a maximal order in $Q_\gr(S)$ and uniquely defines a maximal $S$-order\\ $V=S(\bfx)\cap S$.
\item $S(\bfx)\cap gS=gS(\bfx)$ and so $S(\bfx)/gS(\bfx)\cong \ovl{S(\bfx)}$. Moreover, in high degrees $n\gg 0$,
    $$\ovl{S(\bfx)}_{\geq n} =B(E,\LL(-\bfx),\sigma)_{\geq n}.$$
\end{enumerate}
\end{prop}

Despite the current notation, it is unknown whether the algebra $S(\bfx)$ appearing in Proposition~\ref{vblowup exist intro} is unique for a fixed virtually effective divisor $\bfx$. Another unknown is when $S(\bfx)\subseteq S$ holds. We investigate these problems in Section~\ref{further vblowups}. \par

The classification of maximal $S$-orders then goes as follows.

\begin{theorem}[Theorem~\ref{RSS 8.11}]\label{main result 1}
Let $U$ be a connected graded maximal $S$-order such that $\ovl{U}\neq\Bbbk$. Then there exists a virtually effective divisor $\bfx$ with $0\leq\deg\bfx\leq 2$, and a virtual blowup $S(\bfx)$, such that $S(\bfx)$ is the unique maximal order containing $U=S(\bfx)\cap S$.
\end{theorem}

The assumption $\ovl{U}\neq \Bbbk$ also appears in Rogalski, Sierra and Stafford's classification of maximal $S^{(3)}$-orders. It is necessary as shown in Example~\ref{RSS 10.8}.\par

Out of proving the above theorems we also obtain many nice properties for maximal orders. The most striking of these is that we get that maximal ($S$-)orders are noetherian for free.

\begin{cor}[Corollary~\ref{RSS 8.11'} and Corollary~\ref{RSS 8.12}]\label{blowup properties}
Let $U$ be a cg maximal $S$-order such that $\ovl{U}\neq \Bbbk$. Equivalently, let $U=S(\bfx)\cap S$ for some virtual blowup $S(\bfx)$ at a virtually effective divisor with $0\leq\deg \bfx\leq 2$. Then
\begin{enumerate}[(1)]
\item $S(\bfx)$ and $U$ are strongly noetherian, and are finitely generated as $\Bbbk$-algebras.
 \item $S(\bfx)$ and $U$ satisfy the Artin-Zhang $\chi$ conditions, have finite cohomological dimension, and possess balanced dualizing complexes.
\end{enumerate}
\end{cor}

The properties Auslander-Gorenstein and Cohen-Macaulay properties are missing from Corollary~\ref{blowup properties}. It is shown in Corollary~\ref{S(p-p1+p2) bad homologically} that in general we cannot expect these homological properties. \par

The ultimate goal of this line of research is to understand connected graded algebras birational to the Sklyanin algebra. In other words, algebras $A$ satisfying $D_\gr(A)=D_\gr(S)$; or equivalently $Q_\gr(A)=Q_\gr(S)^{(d)}$ for some $d\geq 1$. Our results for higher Veroneses are as follows.

\begin{theorem}[Theorem~\ref{main thm converse}, Theorem~\ref{3 Veronese of virtual blowup} and Corollary~\ref{g-div max orders up n down generalised}]\label{vblowup Veroneses}
Let $U$ be a cg maximal $S$-order with $\ovl{U}\neq \Bbbk$. Equivalently let $U=S(\bfx)\cap S$ for a virtual blowup $S(\bfx)$ at a virtually effective divisor with $0\leq\deg \bfx\leq 2$. Then for any $d\geq 1$, $S(\bfx)^{(d)}$ is a maximal order in $Q_\gr(S)^{(d)}$ and $U^{(d)}=S(\bfx)\cap S^{(d)}$ is a maximal $S^{(d)}$-order.
\end{theorem}

and conversely . . .

\begin{theorem}[Theorem~\ref{RSS 8.11}]\label{RSS 8.11 intro}
Let $d\geq1$ be coprime to 3 and suppose that $U$ is a cg maximal $S^{(d)}$-order satisfying $\ovl{U}\neq \Bbbk$. Then there exists a virtually effective divisor $\bfx$ and virtual blowup $S(\bfx)$ such that $U=S(\bfx)\cap S^{(d)}$.
\end{theorem}

Retain the notation of Theorem~\ref{RSS 8.11 intro}. When $d=3e$ is divisible by 3, $U=(F\cap S)^{(e)}$
for a virtual blowup $F$ of $S^{(3)}$. This is a result of Rogalski, Sierra and Stafford. In contrast, when we prove Theorem~\ref{vblowup Veroneses}, we also prove the analogous statement for maximal $S^{(3)}$-orders. This result is an improvement on \cite{RSS}. Out of these results we are able to obtain the best answer yet to \cite[Question~9.4]{RSS}.

\begin{cor}[Corollary~\ref{g-div max orders up n down generalised}]\label{g-div max orders up n down generalised intro}
Let $U$ be a cg graded subalgebra of $S$ satisfying $D_\gr(U)=D_\gr(S)$ and such that $\ovl{U}\neq \Bbbk$. If $U$ is a maximal order then $U^{(d)}$ is a maximal order for all $d\geq 1$.
\end{cor}
\par
A final achievement of this thesis is the explicit construction of a virtual blowup - a first of its kind.

\begin{theorem}[Theorem~\ref{S(p-p1+p2)}]\label{S(p-p1+p2) intro}
Let $p\in E$ and $\mbf{x}=p-\sigma^{-1}(p)+\sigma^{-2}(p)$. We set
\begin{itemize}
\item $X_1=S(p+\sigma^{-2}(p))_1=\{ u\in S_1\,|\; \ovl{u}\in H^0(E,\LL(-\bfx))\};$
\item $X_2=S(p)_1S(\sigma^{-2}(p))_1\subseteq \{ u\in S_2\,|\; \ovl{u}\in H^0(E,\LL_2(-\bfx-\sigma^{-1}(\bfx)))\};$
\item $X_3=\{ u\in S_3\,|\; \ovl{u}\in H^0(E,\LL_3(-\bfx-\sigma^{-1}(\bfx)-\sigma^{-2}(\bfx)))\}.$
\end{itemize}
Put $U=\Bbbk\langle X_1, X_2, X_3\rangle$. Then $U=S(\bfx)$ is a virtual blowup of $S$ at the virtually effective divisor $\bfx$.
\end{theorem}
What's curious about $U$ from Theorem~\ref{S(p-p1+p2) intro} is the summand $X_2$. It would be more natural to take $U=\Bbbk\langle X_1, X'_2, X_3\rangle$ where
$$X'_2=\{ u\in S_2\,|\; \ovl{u}\in H^0(E,\LL_2(-\bfx-\sigma^{-1}(\bfx)))\}$$
- a definition more in keeping with the definition of $S(\bfd)$. We show in Example~\ref{original S(p-p1+p2)} that this second ring is far from being a maximal order.

\subsection{History and motivation}\label{history}

In 1987, Artin and Schelter started a project to classify the noncommutative analogues of polynomial rings in 3 variables \cite{ASc}. These are now called AS-regular algebras (Definition~\ref{AS regular}). The subject of noncommutative projective geometry was born when Artin, Tate and Van den Bergh completed this classification in \cite{ATV1, ATV2} (see Theorem~\ref{ATV thm}). Their results can be thought of as a classification of noncommutative projective planes. \par

More generally, one would like to classify all so-called noncommutative curves and surfaces. Let $A$ be a cg noetherian domain, then we can associate a noncommutative projective scheme $\qgr(A)$ to $A$ (Definition~\ref{noncommutative projective scheme}). It can be thought of as the noncommutative analogue of coherent sheaves, $\coh(X)$, over (the non-existent) $X=\mathrm{Proj}(A)$. The classification for noncommutative projective curves, when the GK-dimension is 2, was completed by Artin and Stafford in \cite{AS}. They show that in this case $\qgr(A)\sim\coh(X)$ for a genuine integral projective curve $X$ (Theorem~\ref{AS thm} and Corollary~\ref{nc curves are comm}). The question of noncommutative surfaces (when $\GKdim(A)=3$) is still very much open. It is this ultimate goal that motivates this work.\par

Let $A$ be a cg domain with $\GKdim(A)=3$. Then $Q_\gr(A)=D_\gr(A)[t,t^{-1};\alpha]$; a skew Laurent polynomial ring over the division ring $D_\gr(A)$ which is of transcendence degree 2. A programme for the classification is to first classify the possible birational classes (the possible $D_\gr(A)$'s), and then classify the algebras in each birational equivalence class. Artin conjectures in \cite{Ar} that we know all the possible division rings. They are:
\begin{enumerate}[(1)]
\item A division algebra which is finite dimensional over a central commutative subfield of transcendence degree 2.
\item A division ring of factions of a skew polynomial extensions of $\Bbbk(X)$, for a commutative curve $X$.
\item A noncommutative function field $D_{\gr}(S)$ of a 3 dimensional Sklyanin algebra $S$.
\end{enumerate}
Whilst this conjecture is still a long way off, significant work has been, and is being, done on the classification of algebras in each birational class. Algebras with $D_\gr(A)$ commutative (plus a geometric condition) have been successfully classified by Rogalski and Stafford and then Sierra in \cite{RoSt.naive.nc.blowups, RoSt.class.of.nc.surfaces} and \cite{Si} respectively. This is a significant subclass of (1) above.  We are interested in case (3) when $D_\gr(A)=D_\gr(S)$. More specifically, we look at subalgebras $A$ of $S$ with $D_\gr(A)=D_\gr(S)$. Where would be a good place to start looking for such algebras? Inside $S$ of course! How about a target to aim for? Maximal orders are the noncommutative analogue of integrally closed domains, or geometrically, of normal varieties. They are therefore a natural target for such a classification. \par
The first major results in this direction were given by Rogalski. In \cite{Ro} Rogalski successfully classifies the degree 1 generated maximal orders of the 3-Veronese ring $T=S^{(3)}$ of $S$. These are classified as so-called blowup subalgebras $T(\bfd)$ of $T$ at effective divisors on $E$ of degree at most $7$. This is extended to include all maximal orders and maximal $T$-orders by Rogalski, Sierra and Stafford in \cite{RSS, RSS2}. A detailed review of their work can be found in Section~\ref{Rogs blowup review} and Section~\ref{RSS review}. In this thesis we ask the question, why work with $S^{(3)}$? Surely it is $S$ we are interested in?!

\subsection{The proofs}

In this thesis two proof strategies are prominent throughout. The first follows the idea applied in \cite{Ro, RSS, RSS2}: prove what we can in the factor $S/gS$ and then lift up to $S$. Indeed $S/gS=B(E,\LL,\sigma)$ is a very well understood ring. Hence, if $U$ is a subalgebra of $S$, then $\ovl{U}$ is a subalgebra $B(E,\LL,\sigma)$ and can be well understood. We can then try to pull back information to $U$. This is most easily achieved when $U$ is \textit{$g$-divisible} (i.e. $U\cap gS=gU$). Despite there being a subtle difference between Hypothesis~\ref{standing assumption intro} and the assumptions of \cite{Ro, RSS, RSS2}, the techniques of Rogalski, Sierra and Stafford are often applicable in the situations we find ourselves in.\par

The second strategy is to try to utilise the results of \cite{Ro, RSS, RSS2} directly. Since what we want to prove here is essentially what Rogalski, Sierra and Stafford proved in the 3-Veronese $S^{(3)}$ of $S$, we can hope that their results will shortcut a lot of hard work. Given a subalgebra $U$ of $S$, clearly $U^{(3)}$ is a subalgebra of $S^{(3)}$; what does \cite{Ro, RSS, RSS2} say about $U^{(3)}$? What does this then say about $U$? Playing this game allows us to bypass much of the technicalities that Rogalski, Sierra and Stafford required.\par

Key to being able to utilise the strategies above is getting an analogue of Rogalski blowup subalgebras $T(\bfd)$. These are our rings $S(\bfd)$. Difficulties here arise in connection with the ring $S(p+q)$. Unlike $S(p)$, the ring $S(p+q)$ is not generated in degree 1. This means that the techniques of Rogalski from \cite{Ro} are not directly applicable. Rogalski, Sierra and Stafford tackled a similar problem in constructing $T(\bfd)$ when $\deg\bfd=8$ which resulted in the highly technical paper \cite{RSS2}. A great strength of the present work is that we are able to avoid such technicalities.

\subsection{Thesis layout and intermediate results}

As with any thesis, we have a background section. Here we develop the essential theory of noncommutative projective geometry that we require. Also included is a detailed review of the work of Rogalski, Sierra and Stafford due to its significance to us. \par

It is section~3 where the thesis properly begins. In the first half of section~3, we look at some of the first properties about subrings of $S$ and introduce the vital property \textit{$g$-divisibility} (i.e. $R\cap gS=gR$). Many of these results have analogues for subalgebras of $T=S^{(3)}$. The proof techniques used in that setting carry over without much problem to $S$. The second half of section~3 looks at the interaction between maximal orders and Veronese rings. We obtain the following result.

\begin{prop}[Proposition~\ref{g-div max orders up n down}]\label{g-div max orders up n down intro}
Let $d\geq 1$ and let $U$ be a $g$-divisible cg subalgebra of $S$ with $D_\gr(U)=D_\gr(S)$.
\begin{enumerate}[(1)]
\item  If $Q_\gr(U)=Q_\gr(S)$ and $d$ is coprime to 3, then $U$ is a maximal order if and only if $U^{(d)}$ is a maximal order;
\item If $Q_\gr(U)=Q_\gr(S^{(3)})$, then $U$ is a maximal order if and only if $U^{(d)}$ is a maximal order.
\end{enumerate}
\end{prop}

A limitation of this result is it does not pass between $S$ and $S^{(3)}$. It is not until later that we can bridge this gap.

\begin{theorem}[Theorem~\ref{3 Veronese of virtual blowup}]\label{3 Veronese of virtual blowup intro}
Let $F$ be a virtual blowup of $S$ at a $\sigma$-virtually effective divisor $\bfx$. Then $F^{(3)}$ is a virtual blowup of $T$ at the $\sigma^3$-virtually effective divisor $\bfy=\bfx+\bfx^\sigma+\bfx^{\sigma^2}$.
\end{theorem}

Proposition~\ref{g-div max orders up n down intro} and Theorem~\ref{3 Veronese of virtual blowup intro} are the crucial preliminary results for Theorem~\ref{vblowup Veroneses} and Corollary~\ref{g-div max orders up n down generalised intro}.

section~4 is dedicated to proving Theorem~\ref{S(d) thm intro}. In other words, we prove that the $S(\bfd)$ are the correct analogues of Rogalski's blowups $T(\bfd)$. We remind the reader that the ring $S(p)$ was studied by Rogalski in earlier work, and proved to satisfy Theorem~\ref{S(d) thm intro}. The definition of $S(p+q)$ is more complicated. It is then correspondingly more difficult to understand $S(p+q)$. The biggest challenge here is to prove $S(p+q)$ is $g$-divisible: at a certain point in the proof the Grassmannians even make an unexpected appearance. Once $g$-divisibility is established, the rest follows quite quickly.\par

In section~5 we head towards a classification of $g$-divisible maximal $S$-orders (Theorem~\ref{RSS 7.4} and Proposition~\ref{RSS 7.4(3)}). This is essentially Theorem~\ref{main result 1} with the assumption $g$-divisibility replacing $\ovl{U}\neq\Bbbk$. We are heavily indebted to \cite[Sections~6-7]{RSS} for the overall strategy employed in the section. The first major obstacle, and most crucial ingredient, is to relate a $g$-divisible maximal $S$-order to one of the blowup subalgebras $S(\bfd)$.

\begin{theorem}[Theorem~\ref{RSS 5.25} and Corollary~\ref{RSS 6.6}]\label{RSS 5.25 intro}
Let $U\subseteq S$ be a $g$-divisible maximal $S$-order with $Q_\gr(U)=Q_\gr(S)$. Then there exists an effective divisor $\bfd$ on $E$ with $\deg\bfd\leq 2$, and such that $U$ and $S(\bfd)$ are equivalent orders.\par
More precisely, there exists a $g$-divisible $(U,S(\bfd))$-bimodule $M$ that is finitely generated on both sides, and such that $U=\End_{S(\bfd)}(M)$.
\end{theorem}

This result is analogous to a result proved for orders in $T$ by Rogalski, Sierra and Stafford (see Theorem~\ref{real RSS 5.25} for more details). For their version, Rogalski, Sierra and Stafford require a detailed study of right ideals of their blowups $T(\bfd)$. On the other hand, we are able to utilise some of their results directly, and hence avoid many of their technicalities. Once we have Theorem~\ref{RSS 5.25 intro} many of Rogalski, Sierra and Stafford's ideas become available to us. We proceed to study in detail the endomorphism rings appearing in Theorem~\ref{RSS 5.25 intro}. It is this study that leads to the definition of a virtual blowup. With the language that we develop, and with Theorem~\ref{RSS 5.25 intro}, we are able to classify $g$-divisible maximal $S$-orders. We end section~5 with a further study of virtual blowups. The standout result is Theorem~\ref{3 Veronese of virtual blowup intro} above. In addition to this, we also look at the open problems mentioned after Proposition~\ref{vblowup exist intro}.

The job of section~6 is to take the classification of $g$-divisible maximal $S$-orders and turn it into Theorem~\ref{main result 1}. For this we need to replace the assumption of $g$-divisibility with the assumption $\ovl{U}\neq \Bbbk$. The idea here is to first prove $U$ (satisfying $\ovl{U}\neq \Bbbk$) is an equivalent order to the ring obtained by adding the element $g$, denoted $U\langle g\rangle$. Then one shows $U\langle g\rangle$ is equivalent to the unique smallest $g$-divisible overring, $\widehat{U\langle g\rangle}$. As a consequence, if $U$ is a maximal order then $U=\widehat{U\langle g\rangle}$, and therefore is $g$-divisible. In particular, $U$ then fits into the classification of $g$-divisible maximal $S$-orders from section~5. To end section~6, we apply a trick of Rogalski, Sierra and Stafford which allows us to disregard the assumption $\ovl{U}\neq\Bbbk$ too. For this one must allow the taking of Veronese subrings and the regrading of the algebra.

We finish this thesis with a section giving a few examples of the theory. The highlight is Theorem~\ref{S(p-p1+p2) intro} and its surrounding results. Let $U$ be as in Theorem~\ref{S(p-p1+p2) intro}. Central to understanding $U$ is realising it as an endomorphism ring, and studying its 3-Veronese.

\begin{prop}[Proposition~\ref{S(p-p1+p2) main prop} and Proposition~\ref{3 Veronese of H}]\label{S(p-p1+p2) endo intro}
Let $U$ be as in Theorem~\ref{S(p-p1+p2) intro}. Set $R=S(p)$ and $M=R+S(p)_1S_1R$. Then
$$U=\End_{R}(M) \; \text{ and } \;U^{(3)}=T(p+\sigma^{-2}(p)+\sigma^{-4}(p)).$$
\end{prop}

We in fact first prove that $\End_R(M^{**})$ above is the virtual blowup that we want. We then proceed to show $U=\End_R(M)=\End_R(M^{**})$. As mentioned after Theorem~\ref{S(p-p1+p2) intro}, the summand $X_2$ is intriguing. The more natural definition (in keeping with the definition of the $S(\bfd)$) would be $U'$ below. As we see, this definition leads to a ring which is far from being a maximal order.

\begin{example}[Example~\ref{original S(p-p1+p2)}]\label{original S(p-p1+p2) intro} \textit{Let $\bfx=p-\sigma^{-1}(p)+\sigma^{-2}(p)$.
Set
\begin{itemize}
\item $X'_1=\{ u\in S_1\,|\; \ovl{u}\in H^0(E,\LL(-\bfx))\};$
\item $X'_2= \{ u\in S_2\,|\; \ovl{u}\in H^0(E,\LL_2(-\bfx-\sigma^{-1}(\bfx)))\};$
\item $X'_3=\{ u\in S_3\,|\; \ovl{u}\in H^0(E,\LL_3(-\bfx-\sigma^{-1}(\bfx)-\sigma^{-2}(\bfx)))\};$
\end{itemize}
and put $U'=\Bbbk\langle X'_1,X'_2,X'_3\rangle$. Then $U'$ is neither left nor right noetherian; $U'$ is an equivalent order to $S$.}
\end{example}

Moving on, an observation regarding the definition $S(p+q)=\Bbbk\langle V_1,V_2,V_3\rangle$, is that we do not know whether the $V_3$ is always necessary. That is, can $S(p+q)$ ever be generated by its elements of degrees 1 and 2? We show in Example~\ref{S(p+p1)} that for specific choices of $p$ and $q$, $V_3$ is necessary. Our final example, Example~\ref{RSS 10.8}, is a modification of an example of Rogalski, Sierra and Stafford. It is an example of a maximal order $U$ satisfying $\ovl{U}= \Bbbk$.

At the end of this thesis there is an appendix. It consists of a result on Grassmannians which is required for Theorem~\ref{S(d) thm intro}. We also provide two indices, a notation index and a terminology index. These can be found at the back.

\section{Background and survey of the subject}\label{background}

\subsection{Connected graded domains}

Fix once and for good an algebraically closed field $\Bbbk$ of any characteristic.\index[n]{k@$\Bbbk$} We take the convention $0\in \N$ as it will be more convenient for our purposes. All algebras are associative and have a 1. \par
It is non-controversial to say that each algebraist has their favorite elementary algebraic structure to study. For us it will be connected graded domains.

\begin{definition}\label{cg domain}
Let $A=\bigoplus_{n\in\N} A_n$ be a $\N$-graded $\Bbbk$-algebra. The algebra $A$ is called \textit{connected graded (cg)}\index{connected graded (cg)} if $A_0=\Bbbk$ and $\dim_\Bbbk A_n<\infty$ for all $n\in\N$. If $A$ is also a domain we will call $A$ a \textit{cg domain}.
\end{definition}

The term ``connected" comes from geometry: if $A$ is a commutative connected graded domain, then $\mathrm{Proj}(A)$ is a connected topological space.\par
We now set out some basic notation.

\begin{notation}\label{graded ring notation}
Let $V=\bigoplus_{i\in\Z} V_i$ be a $\Z$-graded $\Bbbk$-vector space. For $k\in\Z$ we write
$$V_{\geq k}=\bigoplus_{i\geq k}V_i\;\text{ and } V_{\leq k}=\bigoplus_{i\leq k}V_i.$$ \index[n]{ak@$A_{\geq k}$, $A_{\leq k}$}
We also call $V_k$ the \textit{homogenous elements of $V$ of degree $k$}. \par
If $A$ is a cg $\Bbbk$-algebra we always consider $\Bbbk$ as a right (and left) $A$-module via
$$\Bbbk_A=A/A_{\geq 1}.\index[n]{ka@$\Bbbk_A$}$$
\end{notation}

\begin{notation}\label{generated in degree 1}
Given a $\Z$-graded ring $A=\bigoplus_{i\in\Z} A_i$, a graded subring $B\subseteq A$ and a subset $V\subseteq A$, we write $B\langle V\rangle$ to be the subalgebra of $A$ generated by $B$ and $V$. In particular $\Bbbk\langle V\rangle$\index[n]{kx@$\Bbbk\langle V\rangle$} is the subalgebra of $A$ generated by $V$.\par

If $A$ is generated as a $\Bbbk$-algebra by $A_{d_1},\dots,A_{d_n}$ for some $d_1,\dots,d_n\in \Z$ (i.e. $A=\Bbbk\langle A_{d_1},\dots,A_{d_n}\rangle$), then we say $A$ is \textit{generated in degrees $d_1$,\dots,$d_n$}.\index{generated in degrees $d_1$,\dots,$d_n$} A common assumption will be that $A$ is \textit{generated in degree 1}, i.e. $A=\Bbbk\langle A_1\rangle$.\index{generated in degree 1}
\end{notation}

We will often be concerned with behaviour in high degrees. Some important notation is therefore the following.

\begin{notation}\label{ehd}
Given two graded subspaces $V,W$ of a $\Z$-graded vector space, we will write $V\ehd W$\index[n]{=@$\ehd$} to mean $V_{\geq n}=W_{\geq n}$ for all $n\gg0$. If this occurs we say $V$ and $W$ are \textit{equal in high degrees}.\index{equal in high degrees}
\end{notation}

\begin{definition}
Let $A$ be a cg $\Bbbk$-algebra. A right $A$-module $M$ is a \textit{$\Z$-graded right $A$-module} if it has a $\Bbbk$-vector space decomposition $M=\bigoplus_{i\in\Z}$ such that $M_iA_n\subseteq M_{i+n}$ for all $i\in\Z$ and $n\in\N$.
\end{definition}

\begin{definition}
Let $A$ be a cg $\Bbbk$-algebra $A$. Given a $\Z$-graded right $A$-module $M$ and a $k\in\Z$, we can define the \textit{shift of $M$ by degree $k$} to be the new $\Z$-graded right $A$-module
$$M(k)=\bigoplus_{i\in\Z} M_{k+i}.$$
Thus the set of homogenous elements of $M(k)$ of degree $d$ is $M(k)_d=M_{k+d}$.
\end{definition}

\begin{notation}\label{Gr(A) and gr(A)}
Let $A$ be a cg $\Bbbk$-algebra. We define $\Gr(A)$\index[n]{Gra@$\Gr(A),\, \gr(A)$} to be the category of $\Z$-graded right $A$-modules. Given $M,N\in\Gr(A)$, $\Hom_{\Gr(A)}(M,N)$\index[n]{HomGr@$\Hom_{\Gr(A)}(M,N)$} consists of those homomorphisms of right $A$-modules $\theta:M\to N$ such that $\theta(M_i)\subseteq N_i$ for all $i\in\Z$. When $A$ is noetherian, we set $\gr(A)$ to be the subcategory of $\Gr(A)$ consisting of finitely generated modules.
\end{notation}

Typically the homomorphism set we want is $\Hom_A(M,N)\index[n]{homamn@$\Hom_A(M,N)$}$ - the set of all $A$-module homomorphisms. If $M$ from Notation~\ref{Gr(A) and gr(A)} is finitely generated, then the equation
$$\Hom_A(M,N)=\bigoplus_{j\in\Z}\Hom_{\Gr(A)}(M,N(j))$$
holds. In the later sections of this thesis, we will mostly be using $\Hom_A(M,N)$ in the form appearing in Notation~\ref{Hom inside Q}.\par

When one studies cg $\Bbbk$-algebras one of the first results one would come across is the graded version of Nakayama's Lemma. Below in Lemma~\ref{graded nakayama}, part (1) can be seen to be analogous its ungraded counterpart; whereas (2) is the form in which it will be regularly used. The proof is significantly easier than the ungraded case and is included for it demonstrates some basic techniques.

\begin{lemma}[Graded Nakayama's Lemma]\label{graded nakayama}\index{Graded Nakayama's Lemma}
Suppose $A$ is a finitely generated cg $\Bbbk$-algebra and let $M$ be a $\Z$-graded right $A$-module such that $M_n=0$ for all $n\ll 0$.
\begin{enumerate}[(1)]
\item If $MA_{\geq 1}=M$ then $M=0$.
\item A set of homogenous elements $\{x_i\}\subseteq M$ generates $M$ as an $A$-module if and only if the set $\{x_i+MA_{\geq 1}\}$ generates $M/MA_{\geq 1}$ as an $\Bbbk$-vector space.
\end{enumerate}
\end{lemma}
\begin{proof}
Suppose $M\neq 0$ and let $d$ be minimal such that $M_d\neq 0$. Then we have $MA_{\geq 1}=M_{\geq d}A_{\geq 1}\subseteq M_{\geq d+1}$, and hence $MA_{\geq 1}\neq M$ proving (1). For (2) apply (1) to the module $M/(\sum x_iA)$.
\end{proof}

\begin{notation}\label{GK dimension}
Given a cg $\Bbbk$-algebra $A$, the \textit{Gelfand-Kirillov (GK) dimension} of $A$ can be defined as, and will be denoted by,
\begin{equation}\label{GK(A)}\GK(A)=\limsup_{n\geq 0}\log_n(\dim_\Bbbk A_{\leq n} ),\index{Gelfand-Kirillov (GK) dimension}\index[n]{gk@$\GK(A)$, $\GKdim(A)$}\end{equation}
while $\GK(A)=0$ if and only if $A$ is finite dimensional as a $\Bbbk$-vector space. We will also use the notation $\GKdim(A)$ in sections where GK-dimension is less prominent. This definition is cg $\Bbbk$-algebra specific (see \cite[Lemma~6.1]{KL}); in general GK-dimension can be defined for an arbitrary $\Bbbk$-algebra. If $M$ is a finitely generated right (or left) $\Z$-graded $A$-module, then similarly
\begin{equation}\label{GK(M)}\GK_A(M)=\limsup_{n\geq 0}\log_n(\dim_\Bbbk M_{\leq n}).\index[n]{gk@$\GK_A(M)$}\end{equation}
If $M$ is not necessarily finitely generated, then equation (\ref{GK(M)}) does not necessarily hold and its GK-dimension will depend on the ring it is being considered over (hence the subscript $A$ in (\ref{GK(M)})). When $M$ is finitely generated, and the ring it is being considered over is clear, we will often drop the subscript in (\ref{GK(M)}). \end{notation}

In general it is possible for $\GK(A)\not\in\Z$ to be true although it will not be the case for us. Bergman's Gap Theorem \cite[Theorem~2.5]{KL} at least shows there are no algebras $A$ with $1<\GK(A)<2$. The standard book on GK-dimension (which will be referenced throughout) is \cite{KL}.
\par
Related to GK-dimension is the Hilbert series of a cg $\Bbbk$-algebra.

\begin{definition}\label{hilbert series}\index{Hilbert series, $h_A(t)$, $h_M(t)$}\index[n]{ha@$h_A(t)$, $h_M(t)$}
Let $A=\bigoplus_{n\in\N}A_n$ be a cg $\Bbbk$-algebra. The \textit{Hilbert series} of $A$ is the power series
$$h_A(t)=\sum_{n\in\N} (\dim_\Bbbk A_n) t^n.$$
Given a $\Z$-graded right or left $A$-module $M=\bigoplus_{n\in\Z}M_n$, we similarly have
$$h_M(t)=\sum_{n\in\Z} (\dim_\Bbbk M_n) t^n.$$
\end{definition}

A Hilbert series $h_A(t)$ lives inside $\Z[[t]]$; however, it is often more useful to consider it in the overring of Laurent series, $\Q((t))$. For example the Hilbert series of the polynomial ring in one variable is $h_{\Bbbk[z]}(t)=\frac{1}{1-t}$.   \par

\begin{remark}\label{Q always exists}
All cg algebras considered in the main body of work will be domains of finite GK-dimension (unless obviously not). In particular, by \cite[Corollary~8.1.12]{MR}, all our algebras are Ore domains and Definition~\ref{Qgr and Dgr} applies.
\end{remark}

\begin{definition}\label{Qgr and Dgr}\index{graded quotient ring} \index[n]{qgr@$Q_\gr(A)$} \index[n]{dgr@$D_\gr(A)$}\index{Goldie quotient ring}\index[n]{qa@$Q(A)$}\index{noncommutative function field}\index{graded division ring}
Let $A$ be a cg domain. Assume that $A$ is an Ore domain (see Remark~\ref{Q always exists}). Then the set $\mathcal{C}_\gr$ consisting of all nonzero homogenous elements of $A$ is an Ore set (see \cite[C.I.1.6]{NV}). We can then obtain the \textit{graded quotient ring $Q_\gr(A)$ of $A$} by localising at the set $\mathcal{C}_\gr$. By \cite[A.14.3]{NV}, $Q_\gr(A)$ is always of the form $Q_\gr(A)=D[t,t^{-1};\sigma]$; a skew Laurent polynomial ring over the division ring $D=Q_\gr(A)_0=D_\gr(A)$. Here $t$ can be any nonzero homogenous element of minimal positive degree and $\sigma: D\to D$ is the automorphism $a\mapsto tat^{-1}$. We call the division ring $D_\gr(A)$ \textit{the noncommutative function field of $A$}. A ring of the form $Q_\gr(A)$ for some cg domain $A$ will be called a \textit{graded division ring}. In particular an element of a graded division ring is invertible if and only if it is homogenous. \par
Similarly one can invert all nonzero elements of $A$ obtain the \textit{full Goldie quotient ring $Q(A)$ of $A$}.
\end{definition}

If $A$ is a commutative cg domain, then $Q_\gr(A)=\Bbbk(X)[t,t^{-1}]$, where necessarily $\Bbbk(X)$\index[n]{kx@$\Bbbk(X)$} is the function field of the projective variety $X=\mathrm{Proj}(A)$. For this reason we consider $D_\gr(A)$ as the noncommutative analogue of the function field. We will see more on the connection between noncommutative algebra and geometry throughout this thesis.\par

Our main results concern a classification of maximal orders which we now define.

\begin{definition}\label{max orders def}
Let $Q$ be a graded division ring. A graded subring $A$ of $Q$ is called \textit{an order in $Q$}\index{order} if $Q_\gr(A)=Q$. Given two orders $A$ and $B$ in $Q$, we call $A$ and $B$ \textit{equivalent orders}\index{equivalent orders} if there exists nonzero homogenous $a,b,c,d\in Q$ such that $aAb\subseteq B$ and $cBd\subseteq A$. The relation ``equivalent orders" forms an equivalence class on the orders in $Q$. We call $A$ a \textit{maximal order}\index{maximal order} if $A$ is maximal (with respect to inclusion) in its equivalence class. Given a graded subring $R$ of $Q$ with $Q_\gr(R)=Q$, an order $A$ satisfying $A\subseteq R$ is called a \textit{maximal $R$-order}\index{maximal r@maximal $R$-order} if it is maximal its equivalence class among graded rings contained in $R$.
\end{definition}

The reason why maximal orders are a good choice for a classification theorem might at first be unclear. However, for a commutative integral domain $R$ with field of fractions $F$, we have that $R$ is a maximal order in $F$ if and only if $R$ is integrally closed in $F$ \cite[5.1.3]{MR}. In other words maximal orders are the noncommutative analogue to integrally closed domains, or geometrically, the noncommutative analogue of normal varieties. Maximal order are therefore a natural class of rings to study. Definition~\ref{max orders def} is (obviously) graded specific; traditionally it is the ungraded version that has been the subject of more study. The ungraded definition is Lemma~\ref{graded max orders are max orders}(2) below.

\begin{lemma}\label{graded max orders are max orders}\cite[Lemma~9.1]{Ro.gncs}
Let $A$ be an $\N$-graded noetherian domain with graded quotient ring $Q_\gr(A)$ and full Goldie quotient ring $Q(A)$. Then the following are equivalent:
\begin{enumerate}[(1)]
\item The ring $A$ is a maximal order in $Q_\gr(A)$ in the sense of Definition~\ref{max orders def}.
\item For all (not necessarily graded) rings $B$ satisfying $A\subseteq B\subseteq Q(A)$: if $aBb\subseteq A$ for some nonzero $a,b\in Q(A)$, then $A=B$.
\end{enumerate}
\end{lemma}
The significance of Lemma~\ref{graded max orders are max orders} is that in the literature, rings satisfying Definition~\ref{max orders def} are rather uncommon; on the contrary, there is a wealth of results on standard ungraded maximal orders. With Lemma~\ref{graded max orders are max orders}, these results become available to us. As with Lemma~\ref{graded max orders are max orders}, it is often the case that many ring theoretic properties have graded analogues. Typically results carry over smoothly between the two settings without too much effort. See for example \cite{GSt} for a graded version of Goldie's Theorem.\par

We now review some homological properties that we will encounter. Perhaps the most important of these for us is the Cohen-Macaulay property.

\begin{definition}\label{homological defs}
Let $A$ be a ring and $M$ a right $A$-module. The \textit{grade of $M$}\index{grade of a module} is defined as
$$j(M_A)=\inf\{i\,|\;\Ext_A^i(M,A)\neq 0\},$$\index[n]{jm@$j(M)$}
where we allow $j(M_A)=\infty$ if no such $i$ exists. We say $M$ satisfies the \textit{Auslander condition}\index{Auslander condition} if for every $i\geq 0$ and all left submodules $N\subseteq \Ext^i(M,A)$ we have $j(_AN)\geq i$. The grade and the Auslander condition have analogous definition for left $A$-modules. \par
The ring $A$ is \textit{Auslander-Gorenstein}\index{Auslander-Gorenstein} if the modules $A_A$ and $_AA$ have the same finite injective dimension, and every finitely generated left and right $A$-module satisfies the Auslander condition. Suppose that in addition that $\GK(A)\in\Z$. Then we say $A$ is \textit{Cohen-Macaulay}\index{Cohen-Macaulay} if
$$\GKdim(M)+j(M)=\GKdim(A)$$
for all finitely generated left and right modules $M$.
\end{definition}

The Cohen-Macaulay property above was extensively studied in \cite{Lev}. Through \cite{Lev}, Levasseur gives us ways of obtaining the Cohen-Macaulay property (which turns out to be very useful) for certain rings. The Cohen-Macaulay property is closely related to maximal orders, as shown in \cite{St.CMmaxorders}, and more specifically for us \cite[Theorem~6.7]{Ro}.

Finally we look at Veronese subrings and submodules.

\begin{definition}
Let $A=\bigoplus_{i\in\N}A_i$ be an $\N$-graded ring and fix an integer $d\geq 1$. The \textit{$d$th Veronese subring of $A$}  is
$$A^{(d)}=\bigoplus_{i\in\N} A_{di}.\index[n]{ad@$A^{(d)}$, $M^{(d)}$}\index{Veronese subring/submodule}$$
with gradation given by $(A^{(d)})_n=A_{dn}$. For a $\Z$-graded $A$-module we can define the $A^{(d)}$-module $M^{(d)}=\bigoplus_{i\in\Z}M_{di}$, again with the grading $(M^{(d)})_n=M_{dn}$.
\end{definition}

We review some known results about properties passing between Veronese rings.

\begin{lemma}\label{AZ 5.10} \cite[Proposition~5.10]{AZ.ncps}\cite[Lemma~4.10]{AS}
Let $A$ be an $\N$-graded $\Bbbk$-algebra. If $A$ is noetherian then $A^{(d)}$ is noetherian for all $d\geq 1$. If in addition $A$ is a domain then the converse is also true. \qed
\end{lemma}

Given a cg $\Bbbk$-algebra $A$ and $d\geq 0$ there are naturally defined functors
\begin{equation}\label{calV}\mathcal{V}:\gr(A)\to\gr(A^{(d)});\;M\mapsto M^{(d)}\end{equation}
and
\begin{equation}\label{calT}\mathcal{T}:\gr(A^{(d)})\to \gr(A);\;N\mapsto N\otimes_{A^{(d)}} A.\end{equation}
Given $N\in\gr(A^{(d)})$ it is not hard to show $\mathcal{V}(\mathcal{T}(N))\cong N$. The converse is true if $A$ is generated in degree 1.
\begin{theorem}\label{VeA5}\cite[Theorem A-5]{Ve}
Let $A$ be a cg $\Bbbk$-algebra generated in degree 1, and fix $d\geq 1$. Then $\mathcal{V}$ and $\mathcal{T}$ as in (\ref{calV}) and (\ref{calT}) give an equivalence of categories
$$\gr(A)\sim \gr(A^{(d)}).$$
\end{theorem}

\subsection{Twisted homogenous coordinate rings}\label{HCR and TCR}

An important class of examples for us will be twisted homogenous coordinate rings. Not only do they motivate the definition of a noncommutative projective scheme, they play a fundamental role in noncommutative projective geometry as a whole. We start with the construction of a \textit{homogenous coordinate ring}.

\begin{notation}\label{HCR notation}\index[n]{ln@$\LL^{\otimes n}$}
Fix a projective scheme $X$. Let $\LL$ be an invertible sheaf on $X$, with global sections  $H^0(X,\LL)$\index[n]{hoxl@$H^0(X,\LL)$}. Write $\LL^{\otimes n}=\underset{n\text{ times}}{\underbrace{\LL\otimes_{\OX} \dots \otimes_{\OX} \LL}}$ for $n\geq 1$ and $\LL^{\otimes 0}=\OX$.
\end{notation}

\begin{definition}\label{HCR}
Fix Notation~\ref{HCR notation} and set
$$B(X,\LL)=\bigoplus_{n\in\N} H^0(X,\LL^{\otimes n}).$$
From the isomorphisms $\LL^{\otimes n}\otimes\LL^{\otimes m}\overset{\sim}{\to}\LL^{\otimes(n+m)}$ we have natural multiplication maps
$$H^0(X,\LL^{\otimes n})\otimes H^0(X,\LL^{\otimes m})\to H^0(X,\LL^{\otimes(n+m)}).$$
 These maps turn $B(X,\LL)$ into a commutative cg $\Bbbk$-algebra. We call $B(X,\LL)$ a \textit{homogenous coordinate ring}.
\end{definition}

Provided $X$ is a connected normal and closed, $B(X,\LL)$ is an integrally closed domain \cite[II~Exercise~ 5.14]{Ha}. A simple example is if $X=\bbP^1$, and $\LL=\OO_{\bbP^1}(1)$ - the \textit{twisting Serre sheaf} on $\bbP^1$, then $B(X,\LL)\cong\Bbbk[x_0,x_1]$.\par

The motivation for the definition of a noncommutative projective scheme comes from Serre's Theorem. To state Serre's Theorem we need a bit more notation. We remark that there is no commutativity assumption in Notation~\ref{Tors} and Notation~\ref{qgr(A)}

\begin{notation}\label{Tors} \index[n]{tors@$\mathrm{Tors}(A)$, $\mathrm{tors}(A)$}
Let $A$ be a cg $\Bbbk$-algebra. Recall $\Gr(A)$ and $\gr(A)$ from Notation~\ref{Gr(A) and gr(A)}. We call a $M\in \Gr(A)$ \textit{torsion} if for any $x\in M$, we have $xA_{\geq n}=0$ for all $n\gg0$. We write $\mathrm{Tors}(A)$ to be the subcategory of $\Gr(A)$ consisting of torsion modules. When $A$ is noetherian we put $\mathrm{tors}(A)=\mathrm{Tors}(A)\cap \gr(A)$.
\end{notation}

Assume that $A$ is a noetherian cg $\Bbbk$-algebra and that $M=x_1A+\dots+x_nA$ is a finitely generated torsion right $A$-module. Then, choosing $d\geq 0$ such that $x_iA_{\geq d}=0$ for each $i$, we see $M=x_1A_{\leq d}+\dots+x_nA_{\leq d}$ and hence $\dim_\Bbbk{M}<\infty$. Conversely, if $\dim_\Bbbk{M}<\infty$ then clearly $M$ is torsion. This shows $\mathrm{tors}(A)$ are exactly all the graded right $A$-modules such that $\dim_\Bbbk M<\infty$.\par
Another straightforward argument, with this time $A$ arbitrary, shows that given a short exact sequence $0\to K\to M\to N\to 0$ in $\Gr(A)$: $K,N\in\mathrm{Tors}(A)$ if and only if $M\in\mathrm{Tors}(A)$. In particular $\mathrm{Tors}(A)$ is a \textit{Serre subcategory} and the quotient category $\Gr(A)/\mathrm{Tors}(A)$ can be considered.

\begin{notation}\label{qgr(A)}\index[n]{qgr@$\Qgr(A)$, $\qgr(A)$}
Let $A$ be a cg $\Bbbk$-algebra. Let $\Gr(A)$ and $\mathrm{Tors}(A)$ be as in Notation~\ref{Tors}.  We set
$$\Qgr(A)=\Gr(A)/\mathrm{Tors}(A)$$
to be the quotient category and denote by $\pi$, the canonical functor $\Gr(A)\to\Qgr(A)$. Similarly, when $A$ is noetherian, we set
$$\qgr(A)=\gr(A)/\mathrm{tors}(A).$$
\end{notation}

Objects of $\Qgr(A)$ are easy to describe, they are all of the form $\pi M$ for some (non-unique) $M\in\Gr(A)$. The morphisms are more subtle and we only describe them for $\qgr(A)$ when $A$ is noetherian. Indeed in this case, given $M,N \in \gr (A)$
$$\Hom_{\qgr(A)}(\pi M,\pi N)=\lim_{n\rightarrow \infty} \Hom_{\gr(A)}(M_{\geq n},N) \index[n]{homq@$\Hom_{\qgr(A)}(\pi M,\pi N)$}$$
where the direct limit is taken over the maps induced from $M_{\geq n+1}\hookrightarrow M_{\geq n}$.

\begin{remark}\label{= in qgr iff ehd}
A key observation for us is that given $M,N\in\gr(A)$, $\pi M\cong \pi N$ in $\qgr(A)$ if and only if  for all $n \gg 0$, $M_{\geq n} \cong N_{\geq n}$ in $\gr(A)$.
\end{remark}

We will be returning to $\Qgr(A)$ and $\qgr(A)$ for $A$ noncommutative later in Section~\ref{nc geometry sec}. At this point we state Serre's Theorem.

\begin{theorem}[Serre's Theorem]\label{Serre's thm}\index{Serre's Theorem}
Let $X$ be a projective scheme and $\LL$ an invertible sheaf on $X$. Assume that $\LL$ is ample \cite[II~Definition~7.4]{Ha}. Then $B=B(X,\LL)$ is a commutative cg noetherian $\Bbbk$-algebra. Moreover there are equivalences of categories
$$\Qgr(B)\sim \mathrm{Qcoh}(X)\;\text{ and }\;\qgr(B)\sim\coh(X),$$
where $\mathrm{Qcoh}(X)$\index[n]{qcoh@$\mathrm{Qcoh}(X)$} and $\coh(X)$\index[n]{coh@$\coh(X)$} are denoting quasi-coherent sheaves and coherent sheaves on $X$ respectively.
\end{theorem}

We now head towards the noncommutative version of Serre's Theorem. The following notation will always be used when discussing twisted homogenous coordinate rings.

\begin{notation}\label{TCR notation}
Let $X$ and $\LL$ as in Notation~\ref{HCR notation}, and let $\sigma : X \rightarrow X$ be an automorphism of $X$. Write $\LL^\sigma$\index[n]{ls@$\LL^\sigma$} for the pullback sheaf $\sigma^*\LL$. Given an open $U\subseteq X$ and $f\in\LL^\sigma(U)=\LL(\sigma(U))$, write $f^\sigma=\sigma^*(f)=f \circ \sigma \in \LL(U)$. Finally set $\LL_0=\OX$ and $\LL_n=\LL\otimes \LL^\sigma \otimes \dots \otimes \LL^{\sigma^{n-1}}$\index[n]{ln@$\LL_n$} for $n\geq1$.
\end{notation}

\begin{definition}[Twisted homogenous coordinate ring]\label{twisted homogenous coordinate ring}\index{twisted homogenous coordinate ring}
Retain Notation~\ref{TCR notation} and define
$$B=B(X,\LL,\sigma)=\bigoplus_{n\in\N}H^0(X,\LL_n).\index[n]{bxls@$B(X,\LL,\sigma)$}$$
From the isomorphism $\LL_n \otimes \LL_m^{\sigma^n} \cong \LL_{n+m}$ we get an induced map of global sections
\begin{multline*}
 B_n\otimes B_m=H^0 (X,\LL_n)\otimes H^0 (X,\LL_m) \longrightarrow \\
  H^0 (X,\LL_n)\otimes H^0 (X,\LL_m^{\sigma^n})\longrightarrow H^0(X,\LL_{n+m})=B_{n+m}.
\end{multline*}

These maps make $B$ into a graded $\Bbbk$-algebra with $B_n=H^0(X,\LL_n)$. We call $B=B(X,\LL,\sigma)$ a \textit{twisted homogenous coordinate ring}. Details of this construction can be found in \cite{AV}.
\end{definition}

One of the first examples of a twisted homogenous coordinate ring is the following.

\begin{example}\cite[Example~3.4]{SV.survey}
Let $\OO_{\bbP_\Bbbk^1}(1)$ denote the Serre twisting sheaf on $\bbP_\Bbbk^1$, and let $\sigma:\mathbb{P}_\Bbbk^1 \rightarrow \mathbb{P}_\Bbbk^1$ be the automorphism $(a:b)\mapsto (a : qb)$ for some nonzero $q\in \Bbbk$. Then
$$B(\mathbb{P}_\Bbbk^1, \OO_{\bbP_\Bbbk^1}(1),\sigma)\cong \Bbbk_q[x,y],$$
where $\Bbbk_q[x,y]=\Bbbk\langle x,y\rangle/(yx-qxy)$.
\end{example}

Just as with $B(X,\LL)$ we need an ample hypotheses to ensure $B=B(X,\LL,\sigma)$ is nice in an appropriate sense.

\begin{definition}\label{sigma ample}\index{sigma@$\sigma$-ample}
Let $X$, $\LL$ and $\sigma$ be as above. Then $\LL$ is called \textit{$\sigma$-ample} if for any coherent sheaf $\FF$ and $n\gg0$, we have $H^i(X,\FF\otimes \LL_n)=0$ for all $i\geq1$.
\end{definition}

\begin{theorem}\label{AV 1.4}\cite[Theorem~1.4]{AV}.
Retain Notation~\ref{TCR notation}. Suppose that $\LL$ is $\sigma$-ample. Then $B(X,\LL,\sigma)$ is a finitely generated noetherian cg $\Bbbk$-algebra. Moreover, if $X$ is irreducible then $B(X,\LL,\sigma)$ is a domain.
\end{theorem}

When $\sigma$ is the identity on $X$, the definition of $\sigma$-ampleness coincides with the traditional definition of ampleness \cite[Proposition III.5.3]{Ha}. The concept of $\sigma$-ampleness as in Definition~\ref{sigma ample} was originally called right $\sigma$-ampleness, with $H^i(X,\LL_n\otimes \FF^{\sigma^n})$ replacing $H^i(X,\FF\otimes \LL_n)$ for left $\sigma$-ampleness. Keeler proved that these two notions are equivalent along with a characterisation of $\sigma$-ampleness in \cite{Ke}. A useful fact for us will be that
\begin{equation}\text{an ample invertible sheaf on a curve is automatically $\sigma$-ample.}\end{equation}

\begin{theorem}[Noncommutative Serre's Theorem]\label{nc serre}\index{Noncommutative Serre's Theorem} \cite[Theorem 1.3]{AV}
Retain Notation~\ref{TCR notation}.  Assume that $\LL$ is $\sigma$-ample and let $B=B(X,\LL,\sigma)$. By Theorem~\ref{AV 1.4}, $B$ is noetherian. Moreover there are equivalences of categories
$$\Qgr(B)\sim \mathrm{Qcoh}(X)\;\text{ and }\;\qgr(B)\sim\coh(X).$$
\end{theorem}

Let $B=B(X,\LL,\sigma)$ as in Theorem~\ref{nc serre}. The functor $\coh(X)\to\qgr(B)$ in Theorem~\ref{nc serre} is given by the formula
\begin{equation}\label{nc serre funtor}\FF\mapsto \pi\left(\bigoplus_{n\in\N}H^0(X,\FF\otimes \LL_n)\right). \end{equation}
Now necessarily, $Q_\gr(B)=\Bbbk(X)[t,t^{-1};\sigma]$, hence given a finitely generated graded right $B$-submodule $M$ of $Q_\gr(B)$, one can identify $M_n\cong M_nt^{-n}\subseteq \Bbbk(X)$ for each $n\in\Z$. When $X$ is a smooth curve, Theorem~\ref{nc serre} along with (\ref{nc serre funtor}) then gives us a classification (up to a finite dimensional vector space) of finitely generated modules inside $Q_\gr(B)$. It is the next statement that is most useful for us. We recall Notation~\ref{ehd} and Remark~\ref{= in qgr iff ehd} for the statement.

\begin{cor}\label{nc serre2}
Retain the hypotheses of Theorem~\ref{nc serre}. Additionally assume that $X$ is a smooth curve. If $M\subseteq Q_\gr(B)=\Bbbk(X)[t,t^{-1};\sigma]$ is a finitely generated graded right $B$-module, then
$$M\ehd \bigoplus_{n\in\N} H^0(X,\OO_X(\bfd)\otimes\LL_n)$$
for some divisor $\bfd$ on $X$.
\end{cor}

When studying homogenous and twisted homogenous coordinate rings, the appearance of the $H^0(X,\LL)$ tells us that the Riemann-Roch theorem will be important to us. To end this section we state the Riemann-Roch Theorem in the form that we will most regularly use and make a few remarks on its significance for twisted homogenous coordinate rings.

\begin{theorem}[Riemann-Roch Theorem]\index{Riemann-Roch} Let $X$ be a smooth irreducible curve of genus $g$ and let $\LL$ be an invertible sheaf on $X$. Then
$$\dim_\Bbbk H^0(X,\LL)-\dim_\Bbbk H^1(X,\LL)=\deg\LL+1-g.$$
\end{theorem}

Let $B=B(X,\LL,\sigma)$ be a twisted homogenous coordinate where $X$ is a smooth and irreducible curve. Assume that $X$ is an integral scheme. Then by \cite[Proposiiton~II6.15]{Ha} $\LL=\OX(\bfd)$ \index[n]{oxd@$\OX(\bfd)$} for some divisor $\bfd$ on $X$. It then follows that $\LL_n=\OX(\bfd+\sigma^{-1}(\bfd)+\dots+\sigma^{-(n-1)}(\bfd))$, and hence
\begin{equation}\label{deg Ln=ndeg L}\deg\LL_n=n\deg\bfd=n\deg\LL.\end{equation}
Applying the Riemann-Roch theorem we have
\begin{equation}\label{RR to B}\dim_\Bbbk B_n=\dim_\Bbbk H^0(X,\LL_n)=\deg\LL_n+1-g+\dim_\Bbbk H^1(X,\LL_n).\end{equation}
Moreover, by Serre duality and \cite[Example IV1.3.4]{Ha}, if $\deg\LL_n> 2g-2$ then $\dim_\Bbbk H^1(X,\LL_n)=0$. Thus by (\ref{deg Ln=ndeg L}) and (\ref{RR to B}), if $\deg\LL>0$, we have that for $n\gg 0$
\begin{equation}\label{RR to B2}\dim_\Bbbk B_n=\dim_\Bbbk H^0(X,\LL_n)=\deg\LL_n+1-g.\end{equation}
This is easily calculable. For our work, we will be interested in the case when $X=E$ is a (necessarily smooth) elliptic curve. In this case the genus of $E$ is 1 and so (provided $\deg\LL>0$) $\deg\LL_n> 2g-2$ always holds and (\ref{RR to B2}) reduces to the following corollary.

\begin{cor}\label{RR to E}
Retain Notation~\ref{TCR notation}. Suppose that $X=E$ is a smooth elliptic curve and that $\LL$ is ample. Set $B=B(E,\LL,\sigma)$, then
$$\dim_\Bbbk B_n=\dim_\Bbbk H^0(E,\LL_n)=n\deg\LL\;\text{ for all } n\geq 1.$$
\end{cor}

\subsection{Noncommutative projective curves and surfaces}\label{nc geometry sec}

Roughly, the goal of noncommutative geometry is to use techniques from commutative geometry to study noncommutative algebras. Now, vital to commutative geometry are prime ideals which we cannot expect noncommutative algebras to have in general. Serre Theorem provides us with an alternative: the categories $\Qgr(A)$ and $\qgr(A)$ can still be defined in the noncommutative case.

\begin{definition}\label{noncommutative projective scheme}\index{noncommutative projective scheme}
Let $A$ be a cg $\Bbbk$-algebra. Recall
$$\Qgr(A)=\Gr(A)/\mathrm{Tors}(A) \;\text{ and (for $A$ noetherian) }\; \qgr(A)=\gr(A)/\mathrm{tors}(A),$$
and the canonical functor $\pi:\Gr(A)\to\Qgr(A)$ from Notation~\ref{qgr(A)}. We call the pair ($\Qgr(A)$,$\pi A$) the \textit{noncommutative projective scheme} associated to the graded algebra $A$.
\end{definition}

For arbitrary noetherian $A$, $\Qgr(A)$ and $\qgr(A)$ were first studied by Artin and Zhang in \cite{AZ.ncps}. For example, it was shown that there exists is a right adjoint to $\pi$, $\omega:\Qgr(A)\to \Gr(A)$; and that the category $\Qgr(A)$ is a Grothendieck category. The latter means in particular that there are enough injective objects. A key property that arose in their work is the $\chi$ property (Definition~\ref{chi} below). Indeed, with $\chi$, Artin and Zhang are able to provide a noncommutative analogue of Serre's Finiteness Theorem \cite[Theorem~7.4]{AZ.ncps}. Again strong evidence that we have a good noncommutative analogue.

\begin{definition}\label{chi}\index{Artin Zhang $\chi$ conditions}\index{chi@ $\chi$, $\chi_i$ conditions}
Let $A$ be noetherian cg $\Bbbk$-algebra and write $\Bbbk=\Bbbk_A$ for the right $A$-module $A/A_{\geq 1}$. Fix $i\geq1$, then $A$ is said to \textit{satisfy $\chi_i$ on the right}  if $\dim_\Bbbk\Ext^j_A(\Bbbk, M)<\infty$ for all finitely generated right $A$-modules $M$, and all $1\leq j\leq i$. We say $A$ \textit{satisfies $\chi$} if $A$ satisfies $\chi_i$ for all $i\geq 1$. We can define $\chi_i$ and $\chi$ analogously on the left. The conditions $\chi$ and $\chi_i$ are often referred to as the \textit{Artin-Zhang $\chi$-conditions}.
\end{definition}

The property $\chi$ holds for most ``natural" noetherian rings. An example of a noetherian ring that does not satisfy $\chi$ is given in \cite{StZh}. \par

Now we claim to have a good definition of a noncommutative projective scheme we can start trying to classify noncommutative projective curves and noncommutative projective surfaces. There is no one definition for what noncommutative projective curves and surfaces should be - we take the next definition.

\begin{definition}\label{nc curves}\index{noncommutative projective curve/surface}
Let $A$ be a cg noetherian $\Bbbk$-algebra. Suppose $\GKdim(A)=1+d$. Then we call $\qgr(A)$ a \textit{$d$-dimensional noncommutative projective scheme}.
\end{definition}

We will be restricting to the case where $A$ is a domain, and hence the noncommutative analogue of integral projective schemes. By work of Small and Warfield, any finitely generated domain with GK dimension 1 is commutative \cite{SmWa}. By Bergman's Gap Theorem there are no algebras $A$ with $1<\GKdim(A)<2$. Hence the first main case is when $\GKdim(A)=2$, or in other words noncommutative projective integral curves. As a future warning the ``projective" and ``integral" will often be dropped from  noncommutative projective integral curves/surfaces. \par
Artin and Stafford proved very strong results on the noncommutative curve case.

\begin{theorem}\label{AS thm}\cite[Theorem 0.1(i), 0.2]{AS}
Let $A$ be a cg finitely generated domain with $\GKdim(A)=2$. Then:
\begin{enumerate}[(1)]
\item $Q_\gr(A)=\Bbbk(X)[t,t^{-1};\sigma]$ for some integral projective commutative curve $X$, and an automorphism $\sigma:\Bbbk(X)\to\Bbbk(X)$ induced from an automorphism $\sigma:X\to X$.
\item If $A$ is generated in degree 1, then
$$A\ehd B(X,\LL,\sigma),$$
for some invertible sheaf $\LL$ on $X$. That is, in high degrees $A$ is a twisted homogenous coordinate ring.
\end{enumerate}
\end{theorem}

If $A$ above is not necessarily generated in degree one Artin and Stafford still have strong results on its structure. This will depend on the order of the automorphism $\sigma$ appearing. We will be interested in the case when $|\sigma|=\infty$. Applying the Noncommutative Serre's Theorem (Theorem~\ref{nc serre}) we get the fun-to-say statement that ``noncommutative curves are commutative."

\begin{cor}\label{nc curves are comm}
Let $A$ be a cg domain GK-dimension 2 generated in degree 1 and let $X$ be given by Theorem~\ref{AS thm}. Then we have an equivalence of categories
$$\qgr(A)\sim \mathrm{coh}(X).$$
\end{cor}

In the reverse direction to Theorem~\ref{AS thm} Artin and Stafford also prove:

\begin{theorem}\label{AS thm2}\cite[Theorem 0.1(ii)]{AS}
Let $X$ be an integral projective commutative curve and $\sigma:X\to X$ an automorphism. If $A$ is a cg subalgebra of $\Bbbk(X)[t,t^{-1};\sigma]$, then $\GKdim(A)\leq 2$.
\end{theorem}

In \cite{AS}, Artin and Stafford conjecture that there is no cg domain with GK-dimension strictly between 2 and 3. In other words, there is nothing between noncommutative curves and noncommutative surfaces. This conjecture was proved to be true by Smoktunowicz in \cite[Theorem~1]{Smok.gap}. Hence the next step is to classify noncommutative surfaces.\par

Such a complete classification as Theorem~\ref{AS thm} does not yet exist for noncommutative integral projective surfaces; it is in fact the subject of much research today. indeed our work lies within this realm. A strategy for this has been suggested by Artin in \cite{Ar} - it is to model such a classification on the commutative version. In short, one first classifies commutative projective surfaces into birational classes, and then within each birational class, one can obtain all other surfaces by a process of blowing up and down a various points.

\begin{definition}\label{birational}\index{birational}
Let $A$ and $B$ be two cg domains. Write $Q_\gr(A)=D_\gr(A)[s,s^{-1}; \sigma]$ and $Q_\gr(B)=D_\gr(B)[t,t^{-1};\tau]$ as in Notation~\ref{Qgr and Dgr}. Then $A$ and $B$ are said to be \textit{birational} if $D_\gr(A)\cong D_\gr(B)$; that is, if $A$ and $B$ have isomorphic  noncommutative function fields. \par
\end{definition}

In \cite{Ar}, Artin conjectures that for noncommutative integral projective surfaces, we at least know all the birational classes. The definition of a skew polynomial extension ring\index{skew polynomial ring}, appearing in (2) below, can be found in \cite[section~1]{GW}. The definition of the Sklyanin algebra appearing in (3) is saved for the next section.

\begin{conjecture}[Artin's Conjecture]\label{Artin conj}\index{Artin's Conjecture}
Let $D=D_\gr(A)$ for a connected graded domain $A$ with $\GK(A)=3$. Then one of the following holds.
\begin{enumerate}[(1)]
\item $D$ is finite dimensional over a central commutative subfield of transcendence degree 3.
\item $D$ is a division ring of fractions of a skew polynomial extension of $\Bbbk(X)$, for a commutative curve $X$.
\item $D\cong D_{\gr}(S)$ for a three dimensional Sklyanin algebra $S$.
\end{enumerate}
\end{conjecture}

Despite this conjecture still being a long way off, much work has been (and is being) done on classifying algebras in each of the birational classes of Conjecture~\ref{Artin conj}. When $D_\gr(A)$ is commutative, or more specifically $Q_\gr(A)=\Bbbk(X)[t,t^{-1};\sigma]$ for a integral projective surface $X$ and automorphism $\sigma$ induced from an automorphism of $X$, we say $A$ is \textit{birationally geometric}.\index{birationally geometric} Birationally geometric surfaces - a significant subclass of (1) - have been successfully classified by Rogalski and Stafford and then Sierra in \cite{RoSt.naive.nc.blowups, RoSt.class.of.nc.surfaces} and \cite{Si} respectively. These were classified in terms of so called \textit{na\"{i}ve blowups}\index{na\"{i}ve blowups}, which are certain subalgebras of twisted homogenous coordinate rings. Work by Chan and Ingalls in \cite{ChIn} looks at the case for rings finitely generated over their centres which also has significant overlap with case (1).\par

Connected graded algebras with noncommutative fields falling into class (2) of Artin's Conjecture has yet to be studied. Our work concerns class (3), or in other words, algebras that are birational to the Sklyanin algebra. We meet the Sklyanin algebra in the next section.

\subsection{AS-regular algebras and ATV}

In \cite{ASc} Artin and Schelter begun a project to classify the noncommutative analogues of polynomial rings, now called AS-regular algebras.

\begin{definition}\label{AS regular}\index{Artin-Schelter (AS) regular}
Let $A$ be a finitely generated connected graded $\Bbbk$-algebra, and identify $\Bbbk=A/A_{\geq 1}$ as a right and left $A$-module. We call $A$ \textit{Artin-Schelter regular}, or \textit{AS-regular}, \textit{of dimension $d$}  if
\begin{enumerate}[(1)]
\item $\gldim(A)=d<\infty$;
\item $\GKdim(A)<\infty$;
\item As left $A$-modules, and for some degree shift $\ell$,
$$\Ext_A^i(\Bbbk_A , A_A)\cong \begin{cases} 0 & \text{ if } i\neq d\\
                                              _A\Bbbk(\ell) &\text{ if } i=d . \end{cases}$$
\end{enumerate}
Property (3) is called the \textit{Artin-Schelter Gorenstein}\index{Artin-Schelter (AS) Gorenstein} or \textit{AS-Gorenstein} condition.
\end{definition}

Artin and Schelter successfully classified the AS-regular algebras of dimension 2. There are only two:
the quantum plane and the Jordan plane, respectively
 $$ \Bbbk_q[x,y]=\Bbbk\langle x,y\rangle/(yx-qxy) \; \text{ and }\;\Bbbk_J[x,y]=\Bbbk\langle x,y\rangle/(yx-xy-x^2).$$
In studying AS-regular algebras of dimension 3, they ran into an algebra that they did not understand - the Sklyanin algebra.

\begin{definition}[The Sklyanin algebra]\label{sklyanin}\index{Sklyanin algebra}
Let $a,b,c\in \Bbbk$. Suppose that $a,b,c$ satisfy
\begin{equation}\label{sklyanin abc conditions}
abc\neq 0 \;\text{ and } \left( \frac{a^3+b^3+c^3}{3abc}\right)^3\neq 1.
\end{equation}
Then the algebra
$$S=\Bbbk\langle x_0,x_1,x_2\rangle /(ax_ix_{i+1}+bx_{i+1}x_i+cx_{i+2}^2\,|\; i\in \Z/3\Z )\index[n]{s@$S$}\index{Sklyanin algebra}$$
is called a \textit{3-dimensional Sklyanin algebra}.
\end{definition}

We will drop the ``3-dimensional" from above as we will discuss no other Sklyanin algebra. We do however warn that a 4-dimensional Sklyanin algebra exists and is a subject of current research. \par

Retain notation from Definition~\ref{sklyanin}. An important fact for us, which Artin and Schelter showed via a computer search in \cite{ASc}, is that $S$ contains a central element $g$ of degree 3 given by
\begin{equation}\label{g=}
g=c(a^3-c^3)x_0^3+a(b^3-c^3)x_0x_1x_2+b(c^3-a^3)x_1x_0x_2+c(c^3-b^3)x_2^3.\index[n]{g@$g$}
\end{equation}

The classification of 3 dimensional AS-regular algebras was completed by Artin, Tate and Van den Bergh in \cite{ATV1, ATV2}. Key to their paper, and a basis for noncommutative projective geometry, is relating the algebras to an underlying geometry via point modules.

\begin{definition}\label{point module}\index{point modules}
Let $A$ a finitely generated cg $\Bbbk$-algebra. Suppose in addition that $A$ is generated in degree 1. A graded right $A$-module $M$ is a \textit{point module} if:
\begin{enumerate}[({1}a)]
\item $M=xA$ for some $x\in M_0$,
\item $\dim_\Bbbk M_i=1$ for all $i\geq 0$.
\end{enumerate}
\index{rpoint@$R$-point modules}More generally, if $R$ is a commutative $\Bbbk$-algebra an \textit{$R$-point module for $A$} is a graded $A_R=R\otimes_\Bbbk A$-module $M$ such that
\begin{enumerate}[(2a)]
\item $M=M_0A_R$ where $M_0\cong R$,
\item for each $i$, $M_i$ is a locally free $R$-module of rank 1.
\end{enumerate}
\end{definition}

The point modules of a cg algebra $A$ are always in one to one correspondence with an inverse limit, $X$, of closed subsets of $\prod_{i=0}^m \bbP^n$ for some $n$ and $m=1,2,\dots$. Thus

\begin{equation}\label{P^n^infty} X \subseteq \prod_{i=0}^\infty \bbP^n \text{ for some } n. \end{equation}

In fact what Artin, Tate and Van den Bergh show is that the point modules are \textit{parameterised} by $X$ in a formal way. We use \cite[section I.3.1]{Rog.notes} as a reference for this formalisation.\par

Definition~\ref{point module}(2) gives rise to a functor $P$ from the category of commutative $\Bbbk$-algebras to the category of sets. This functor takes $R$ to the set of isomorphism classes of $R$-point modules $P(R)$. Given a ring homomorphism $R\to R'$, the corresponding function $P(R)\to P(R')$ is defined by tensoring up: $M\mapsto R'\otimes_R M$. On the other hand, given a $\Bbbk$-scheme $X$, there is a natural functor $h_X$ from commutative $\Bbbk$-algebras to sets: $R\mapsto\Hom_{\Bbbk-\mathrm{schemes}}(\mathrm{Spec}R,X)$. We say $X$ \textit{parameterises}\index{point modules!parameterised} the point modules of $A$ if the functors $P$ and $h_X$ are naturally isomorphic.
\par

One issue above is we do not know whether the product of $\bbP^n$'s in (\ref{P^n^infty}) can be made finite. This is remedied by the next definition.

\begin{definition}\label{str noeth}\index{strongly noetherian} A noetherian $\Bbbk$-algebra $A$ is called \textit{strongly noetherian} if for all commutative noetherian $\Bbbk$-algebras $R$, $A\otimes_\Bbbk R$ remains noetherian.
\end{definition}

\begin{theorem}\label{AZ E4.12}\cite[Corollary~E4.12]{AZ.hs}. Let $A$ be a cg strongly noetherian $\Bbbk$-algebra. Then the point modules of $A$ are parameterised by a closed subscheme of $\prod_{i=1}^m\bbP^n$ for some $n,m\geq 1$.
\end{theorem}

In this thesis we will often obtain the strongly noetherian property for the rings that we are interested in. We immediately then have the nice corollary given by Theorem~\ref{AZ E4.12}. A warning is that not every noetherian ring is strongly noetherian. The first such examples were the na\"{i}ve blowups appearing in the classification of birational geometric surfaces \cite[Theorem~1.1]{RoSt.naive.nc.blowups}. Indeed for point modules of na\"{i}ve blowups the product in (\ref{P^n^infty}) cannot be made finite. \par
We now give this parametrisation for the point modules of the Sklyanin algebra.

\begin{example}\label{E from Sklyanin}
Let $S$ denote the Sklyanin algebra from Definition~\ref{sklyanin}. Let $E$\index[n]{e@$E$} be the closed subset of $\bbP^2$ defined by the equation
\begin{equation}\label{E's eq} (a^3+b^3+c^3)xyz-abc(x^3+y^3+z^3)=0.\end{equation}
Then conditions (\ref{sklyanin abc conditions}) imply that $E$ is a smooth elliptic curve, and the closed points of $E$ parameterise the point modules of $S$.\par

In more detail, for $S$, the set $X$ in (\ref{P^n^infty}) can be taken as  the subset of $\bbP^2\times\bbP^2$, $X=\{(p,\sigma(p))\,|\; p\in E \}$ for an automorphism $\sigma:E\to E$\index[n]{sigma@$\sigma$}. In particular $X\cong E$.  On the open set of $E$ for which the formula (\ref{sigma}) below is valid, $\sigma$ can be given by
\begin{equation}\label{sigma}(x_0:x_1:x_2)\mapsto (a^2xz-bcy^2:b^2yz-acx^2:c^2xy-abz^2).\end{equation}
A detailed argument for the above can be found in \cite[Example~I.3.6]{Rog.notes}.
\end{example}

Artin, Tate and Van den Bergh's classification of AS-regular algebras on dimension 3 is in terms of this geometric data. We give a simple version of this classification when $A$ is assumed to be generated in degree 1.

\begin{theorem}\label{ATV thm} \cite{ATV1, ATV2} Let $A$ be an AS regular algebra of dimension 3. Assume that $A$ is generated in degree 1.
\begin{enumerate}[(1)]
\item  $A$ is a noetherian domain and has the Hilbert series of a weighted polynomial ring in 3 variables.
\item Now (restricting to the case of interest to us) assume that the polynomial ring in (1) is unweighted. Let $X$ be the closed subset of $\prod_{i=0}^\infty \bbP^n$ parameterising the point modules of $A$. Then either:
\begin{enumerate}[(a)]
\item $X=\bbP^2$ and $A=B(X,\LL,\sigma)$ for some invertible sheaf $\LL$ on $X$ and automorphism $\sigma:X\to X$;
\item $X$ is a cubic curve in $\bbP^2$ and there is a natural surjection $A\to B(X,\LL,\sigma)$ for some invertible sheaf $\LL$ on $X$ and automorphism $\sigma:X\to X$.
\end{enumerate}
The AS-regular algebras of dimension 3 are then classified from this data.
\end{enumerate}
\end{theorem}

The proof of Theorem~\ref{ATV thm}(1) goes via part (2). Artin, Tate and Van den Bergh first prove $B(X,\LL,\sigma)$ is a noetherian domain. Then, when in case (b) of Theorem~\ref{ATV thm}(2), they study the ring homomorphism $A\to B(X,\LL,\sigma)$ to obtain Theorem~\ref{ATV thm}(1).\par

For the Sklyanin algebra Theorem~\ref{ATV thm} means the following.

\begin{example}\label{ATV applied to S}
Let $S$ be a Sklyanin algebra as in Definition~\ref{sklyanin}. Then $S$ is AS-regular of dimension 3. By Example~\ref{E from Sklyanin} the point modules of $S$ are parameterised by a (necessarily smooth) elliptic curve $E$ and so $S$ falls into case (2b) of Theorem~\ref{ATV thm}. More specifically, there exists a central homogenous element $g\in S$ of degree 3, such that
\begin{equation}\label{S/gS=B} S/gS\cong B(E,\LL,\sigma).\index[n]{els@$(E,\LL,\sigma)$} \end{equation}
Here $E$ (given by (\ref{E's eq})) is embedded as a degree 3 divisor $E\hookrightarrow\bbP^2$, $\LL$ is the pull back sheaf of the twisting Serre sheaf $\OO(1)$ of $\bbP^2$, and $\sigma$ is given in (\ref{sigma}). The element $g$ is given in (\ref{g=}) and is unique up to scalar.
\end{example}

\subsection{The generic Sklyanin algebra}\label{generic sklyanin}

We switch our attention to the Sklyanin algebra appearing in Definition~\ref{sklyanin}, Example~\ref{E from Sklyanin} and Example~\ref{ATV applied to S}. The Sklyanin algebra is named after E. K. Sklyanin who was first to construct a (in fact 4-dimensional) Sklyanin algebra in \cite{Skl1, Skl2}.
$$\text{The notation appearing in Example~\ref{ATV applied to S} is from this point reserved.}$$
\par
The use of geometry in proving Example~\ref{ATV applied to S} is absolutely crucial; it was not until recently in \cite{IySh} where a purely algebraic and combinatorial approach yielded similar results. A notably absentee of \cite{IySh} is proving that $S$ is a domain. On the contrary one can quite easily prove $S$ is a domain via (\ref{S/gS=B}) and the fact that $B(E,\LL,\sigma)$ is a domain. Indeed a key strategy in understanding $S$ is going via (\ref{S/gS=B}). We have already seen in Section~\ref{HCR and TCR} that twisted homogenous coordinate rings are well understood. Let $B=B(E,\LL,\sigma)$ as in (\ref{S/gS=B}). We gather some more properties on $B$ in the next theorem. The result we use is \cite[Lemma~2.2]{Ro} where the author gives relevant references for the original proofs.

\begin{theorem}\label{B properties}
Let $E$ be a nonsingular elliptic curve, $\sigma: E\to E$ an automorphism and $\HH$ a $\sigma$-ample invertible sheaf with $\deg\HH\geq 1$. Set $B=B(E,\HH,\sigma)$.
\begin{enumerate}[(1)]
\item $B$ is a strongly noetherian domain with $\GKdim(B)=2$.
\item $B$ satisfies $\chi$ on both sides.
\item $B$ is Auslander-Gorenstein and Cohen-Macaulay.
\item If $\deg\HH\geq 2$, then $B$ is generated in degree 1. Otherwise $B$ is generated in degrees $1$ and $2$.
\end{enumerate}
\end{theorem}

Using (\ref{S/gS=B}), much of Theorem~\ref{B properties} can be lifted to $S$. For referencing purposes we include a proof of Theorem~\ref{sklyanin thm} below.

\begin{theorem}\label{sklyanin thm}
Retain the notation of Example~\ref{ATV applied to S}. Then
\begin{enumerate}[(1)]
\item $S$ is AS-regular of dimension 3 with $\GKdim(S)=3$, and has Hilbert series given by $h_S(t)=1/(1-t^3)$.
\item $S$ is a strongly noetherian domain and satisfies $\chi$ on both sides.
\item $S$ is Auslander-Gorenstein and Cohen-Macaulay.
\item $S$ is a maximal order in $Q_\gr(S)$.
\end{enumerate}
\end{theorem}
\begin{proof}
(1) is Example~\ref{ATV applied to S} above. The first part of (2) follows from Theorem~\ref{B properties}(1), (\ref{S/gS=B}) and \cite[Proposition~4.9(1)]{ASZ}. The second part of (2) is by Theorem~\ref{B properties}(2), (\ref{S/gS=B}) and \cite[Theorem~8.8]{AZ.ncps}. (3) comes from \cite[Theorem~5.10]{Lev}, Theorem~\ref{B properties}(3) and (\ref{S/gS=B}). Finally (4) follows from \cite[Theorem~2.10]{St.CMmaxorders}.
\end{proof}

Our standing assumption (Hypothesis~\ref{standing assumption intro}) assumes that $\sigma$ has infinite order. With this assumption being the generic case, such a Sklyanin algebra is often referred to as the \textit{the generic Sklyanin algebra}\index{generic Sklyanin algebra}. When $\sigma$ is of finite order $S$ has very different properties.

\begin{lemma}\label{ATV2 result}\cite[Theorem II]{ATV2}
Retain the notation of Example~\ref{ATV applied to S}. The following are equivalent:
\begin{enumerate}[(1)]
\item The automorphism $\sigma$ has finite order;
\item $S$ satisfies a polynomial identity;
\item $S$ is finitely generated over its centre.
\end{enumerate}
\end{lemma}

For two different approaches to the Sklyanin algebra with the order of $\sigma$ finite, see work by Walton, Wang and Yakimov in \cite{WWY} and work by De Laet and Le Bruyn in \cite{DeLLeB}.\par
With the additional assumption that $|\sigma|=\infty$ one can say more about $B(E,\LL,\sigma)$. First we make an observation which will be regularly used.

\begin{remark}\label{sigma is translation}
Let $\sigma: E\to E$ be a automorphism of $E$ of infinite order. Then $\sigma$ is a translation in the group law on $E$ by a point of infinite order (see \cite[pages 248-249]{AS}). In particular $\sigma$ will have no fixed points.
\end{remark}

The next result is a special case A
of \cite[Theorem~5.11]{AS} given in \cite{RSS}.

\begin{theorem}\label{RSS 3.1}\cite[Theorem~3.1]{RSS}
Let $A$ be a cg $\Bbbk$-algebra with $Q_\gr(A)=\Bbbk(E)[t,t^{-1};\sigma]$, for a smooth elliptic curve $E$ and automorphism $\sigma:E\to E$ of infinite order. By regrading we may assume $\deg t=1$. Then there exists an ideal sheaf $\JJ$ and a $\sigma$-ample invertible sheaf $\HH$ on $E$ such that
$$A\ehd \bigoplus_{n\geq 0} H^0(E,\JJ\HH_n).$$
\end{theorem}

Retain the notation of Theorem~\ref{RSS 3.1}. Then $J=\bigoplus_{n\geq 0} H^0(E,\JJ\HH_n)$ is a right ideal of $B=B(E,\HH,\sigma)$, and thus a rephrasing of Theorem~\ref{RSS 3.1} is that $A$ is equal in high degrees to the ring $\Bbbk+J$ for some right ideal $J$ of $B$.

Another advantage of working with $|\sigma|=\infty$ is that the ideal structure of $B(E,\HH,\sigma)$ is very rigid as demonstrated in \cite{RSS.aiwesin}. Our wording is slightly different to \cite[Corollary~2.10]{RSS.aiwesin} and so we include a short proof.

\begin{prop}\label{B just infinite}\cite[Corollary~2.10]{RSS.aiwesin}
Retain the hypotheses of Theorem~\ref{RSS 3.1}. Then
\begin{enumerate}[(1)]
\item Every subalgebra of $A$ is both finitely generated and noetherian.
\item $A$ is \textit{just infinite}\index{just infinite} in the sense that for every nonzero ideal $I$ of $A$, $\dim_\Bbbk A/I<\infty$.
\end{enumerate}
In particular (1) and (2) hold for the twisted homogenous coordinate ring $B(E,\HH,\sigma)$.
\end{prop}
\begin{proof} (1) By Theorem~\ref{RSS 3.1}, $A$ is (at least in high degrees) a subring of a twisted homogenous $B=B(E,\HH,\sigma)$. By \cite[Theorem~2.9]{RSS.aiwesin} every subalgebra of $B$ is both finitely generated and noetherian and hence the same is true for $A$. (2) then follows from (1) as shown in the second paragraph of the discussion in \cite[Section~3]{RSS.aiwesin}.
\end{proof}

\subsection{Rogalski's blowup subalgebras of $S^{(3)}$}\label{Rogs blowup review}

As mentioned in the introduction and at the end of Section~\ref{nc geometry sec} a current and realistic goal is the classification of algebras birational to the (generic) Sklyanin algebra. A starting place for this has been the classification of the maximal orders inside the Sklyanin algebra. It was Rogalski in \cite{Ro} who first look at this problem. Rogalski successfully classified the degree one generated maximal orders contain inside the 3-Veronese $S^{(3)}$ of $S$. These were classified as so-called \textit{blowup subalgebras} of $S^{(3)}$. In later work, \cite{RSS, RSS2} Rogalski, along with Sierra and Stafford extended this to include all maximal orders inside $S^{(3)}$. The reason for working inside $S^{(3)}$ as opposed to $S$, is that now the central element $g$ from Example~\ref{ATV applied to S} is in degree 1. Our work aims to generalise the results of Rogalski, Sierra and Stafford and for that reason we briefly summarize their work here. For consistency we use notation of the papers \cite{RSS, RSS2}.

\begin{notation}\label{3 Veronese notation bg}\index[n]{emt@$(E,\MM,\tau)$}\index[n]{tau@$\tau$}
Let $S$ be a Sklyanin algebra and retain the notation of Example~\ref{ATV applied to S}.  Assume that $|\sigma|=\infty$. Set $T=S^{(3)}$\index[n]{t@$T$}, $\MM=\LL_3$ and $\tau=\sigma^3$; whence
\begin{equation*} T/gT=B(E,\MM,\tau).\end{equation*}
\end{notation}

We give a few observations regarding Notation~\ref{3 Veronese notation bg}. Firstly (and the whole point of considering $T$) we indeed have $g\in T_1$. Second, since $\deg\LL=3$, $\MM=\LL\otimes\LL^\sigma\otimes\LL^{\sigma^2}$ has degree $9$, and by the Riemann-Roch (Corollary~\ref{RR to E}) $B(E,\MM\,\tau)_1=H^0(E,\MM)$ is a 9 dimensional $\Bbbk$-vector space. It then follows $\dim_\Bbbk T_1=10$. Finally, $\tau$ is also a translation by a point of $E$ of infinite order, and one $\sigma$-orbit is partitioned into 3 distinct $\tau$-orbits.

We are now ready to define Rogalski's blowup subalgebras of $T$. We recall that for $X\subseteq T$, we are writing $\ovl{X}=(X+gT)/gT$.

\begin{definition}\label{T(d) def}\cite[Section~1]{Ro}
Let $\bfd$ be an effective divisor on $E$ with $\deg\bfd\leq 7$. Set
$$T(\bfd)_1=\{ x\in T_1\,|\; \ovl{x}\in H^0(E,\MM(-\bfd))\}\;\text{ and } \; T(\bfd)=\Bbbk\langle T(\bfd)_1\rangle .\index[n]{td@$T(\bfd)$}$$
The subalgebra $T(\bfd)$ of $T$ is called the \textit{blowup of $T$ at $\bfd$}.\index{blowup of $T$ at $\bfd$}
\end{definition}

Let $\bfd$ be an effective divisor on $E$ with $\deg\bfd=d\leq 7$. One immediately sees $g\in T(\bfd)_1$. Moreover
$$\ovl{T(\bfd)_1}=(T(\bfd)_1+gT)/gT=H^0(E,\MM(-\bfd))=B(E,\MM(-\bfd),\tau)_1,$$
and it follows by Theorem~\ref{B properties}(4) that $\ovl{T(\bfd)}=(T(\bfd)+gT)/gT \cong B(E,\MM(-\bfd),\tau)$ - straight away we have a good understanding of $\ovl{T(\bfd)}$. From the Riemann-Roch Theorem we know $\dim_\Bbbk \ovl{T(\bfd)_1}=\deg\MM(-\bfd)=9-d$, and then $\dim_\Bbbk T(\bfd)_1=10-d$.

\begin{remark}\label{VdBblowups}
The use of the term ``blowup" is due to the link with Van den Bergh's categorical blowups of $\mathrm{Qgr}(T)$ from \cite{VdB.blowup}. In unpublished work by Rogalski it is shown that $\mathrm{Qgr}(T(\bfd))$ is such a blowup. In \cite{VdB.blowup}, Van den Bergh proves that $\mathrm{Qgr}(T(\bfd))$ has many properties in common with commutative blowups. One of the advantages of the ring theoretic blowups of Definition~\ref{T(d) def} over Van den Bergh's blowups in the ease at which they are defined; a disadvantage is that we are missing a blowup. If one studies \cite{VdB.blowup}, one sees that it should be possible to blowup at 8 points (instead of the 7 above). The natural definition $T(\bfd)=\Bbbk\langle V\rangle$ where $V=\{ x\in T_1\,|\; \ovl{x}\in H^0(E,\LL(-\bfd))\}$ for a degree 8 effective divisor $\bfd$, gives $T(\bfd)\cong \Bbbk[g,x]$: a commutative polynomial ring. Using \cite{VdB.blowup} as a guide, Rogalski, Sierra and Stafford found the correct definition for the 8 point blowup subalgebra of $T$ given in \cite{RSS2}. We will not be needing the 8 point blowup and so we omit it from this work. \par
\end{remark}

The significance of Definition~\ref{T(d) def} is shown by the following theorems.

\begin{theorem}\label{T(d) properties}\cite[Theorem 1.1]{Ro}
Let $\bfd$ be an effective divisor with $\deg\bfd=d\leq 7$, and set $R=T(\bfd)$. Then
\begin{enumerate}[(1)]
\item $R\cap Tg=Rg$ and $R/Rg\cong B(E,\MM(-\bfd),\tau)$. It has Hilbert series given by $h_R(t)=\frac{t^2+(7-d)t+1}{(1-t)^3}$.
\item $R$ is strongly noetherian and satisfies the Artin-Zhang $\chi$ conditions. It has cohomological dimension 2 and (by \cite[Proposition~2.4]{RSS2}) has a balanced dualizing complex.
\item $R$  Auslander-Gorenstein and Cohen-Macaulay.
\item $R$ is a maximal order in $Q_\gr(R)=Q_\gr(T)$.
\end{enumerate}
\end{theorem}

\begin{remark}\label{desirable properties}\index{cohomological dimension}\index{balanced dualizing complex}
We should explain why the properties appearing in Theorem~\ref{T(d) properties}(2)(3) are desirable. For strongly noetherian and the Artin-Zhang $\chi$ conditions this can be found in and around their respective definitions (Definition~\ref{str noeth} and Definition~\ref{chi}). Cohomological dimension and balanced dualizing complexes will not feature in this thesis, with the exception of results in the same vein as Theorem~\ref{T(d) properties}. The technical definitions are therefore omitted. We note that a definition of cohomological dimension can be found in \cite[Section~2]{Ro}. For a definition of a balanced dualizing complex the reader is referred to the original paper by Yekutieli \cite{Ye}. The strength of balanced dualizing complexes is that they allow one to use significant homological techniques from algebraic geometry. This is demonstrated in a series of papers by Yekutieli, Zhang and Van den Bergh among others. Having finite cohomological dimension is important for the dualizing complexes. The properties Auslander-Gorenstein and Cohen-Macaulay are prominent in \cite{Ro, RSS, RSS2} as well as this thesis. In particular the proof of Theorem~\ref{T(d) properties}(4) (presented in \cite[Theorem~6.7]{Ro}) uses the Cohen-Macaulay property.
\end{remark}

For Theorem~\ref{T(d) properties} the key is obtaining $T(\bfd)\cap Tg=gT(\bfd)$. This property will be important for us.
\begin{definition}\label{bg g div def}\index{gdivisible@$g$-divisible}\index{gdivisiblehull@$g$-divisible hull}
Let $R$ be a subalgebra of $T$ containing $g$ and $M$ a right $R$-module with $M\subseteq T$. The module $M$ is called \textit{$g$-divisible} if $M\cap Tg=gM$. The ring $R$ is \textit{$g$-divisible} if it is $g$-divisible as a right (equivalently left) module over itself. The right $R$-module $\widehat{M}=\{ x\in T\,|\; xg^n\in M\,\text{ for some }n\geq0\}$ is the unique smallest $g$-divisible right $R$-module containing $M$. It is called the \textit{$g$-divisible hull} of $M$.
\end{definition}

Once $T(\bfd)$ is proven to be $g$-divisible we have $$T(\bfd)/gT(\bfd)\cong\ovl{T(\bfd)}= B(E,\MM(-\bfd),\tau).$$
Theorem~\ref{T(d) properties}(2)(3) can then be lifted freely from the corresponding properties of $B(E,\MM(-\bfd),\tau)$. In general, $g$-divisible rings will be easier to understand. The proof of Theorem~\ref{T(d) properties}(1) is something we generalise later. For this reason, we include a detailed sketch of the proof. Recall that a map $f:E\to \N$ is called \textit{lower semi-continuous}\index{lower semi-continuous} if for each $\ell\in\N$, the set $\{ p \in E\,|\; f(p)\leq \ell\}$ is closed in the Zariski topology \cite[Remark~12.7.1]{Ha}.

\begin{proof}[Sketch proof of Theorem~\ref{T(d) properties}(1)]
As remarked above, the key step is to prove that $T(\bfd)$ is $g$-divisible. The proof is by induction on $\deg\bfd$ with the base case $\bfd=0$ being trivial. So suppose that $\bfd=\bfd'+p$ for some $p\in E$ and that $T(\bfe)$ is $g$-divisible for all effective divisors $\bfe$ with $\deg\bfe<\deg\bfd$.\par
First we add the additional hypothesis that $p$ and $\bfd'$ are supported on different $\tau$-orbits. With this assumption it is fairly easy to prove that
\begin{equation}\label{T(d)=T(d')cap T(p)} T(\bfd)=T(\bfd')\cap T(p).\end{equation}
The $g$-divisibility of $T(\bfd)$ follows from this by induction.\par
To prove the general statement one uses a semi-continuous argument. More specifically one constructs a lower semi-continuous map $E \to \N$ as follows. Let $d=\dim_\Bbbk T(\bfd')_1$, identify $\bbP^d$ with the Grassmannian of co-dimension 1 subspaces of the vector space $T(\bfd')_1$ and now allow $p\in E$ to vary. The natural map $\phi: E \to  \bbP^d$, $p\mapsto T(\bfd'+p)_1$ can be shown to be a morphism of algebraic varieties. On the other hand, for fixed $n\geq 1$, one can define a function $\mu: \bbP^d \to \N$, sending a codimension 1 subspace $Y$ of $T(\bfd')_1$ to $\dim_\Bbbk Y^n$. It can be argued that $\mu$ is lower semi-continuous. It then follows that composition $\mu\circ \phi: E \to \N$ is also a lower semi-continuous function; furthermore, since all $T(\bfe)$'s are generated in degree 1, $\mu\circ \phi$ is exactly the map $p\mapsto T(\bfd'+p)_n$.  \par
Finally, from the lower semi-continuity of the map $p\mapsto T(\bfd'+p)_n$, together with the fact that $T(\bfd'+p)$ is $g$-divisible when $p$ and $\bfd'$ are on distinct $\tau$-orbits, it is proved that for general $p$, $T(\bfd'+p)$ has the correct Hilbert series. This forces $T(\bfd'+p)$ to be $g$-divisible.
\end{proof}

Once $g$-divisibility is obtained the properties of $T(\bfd)$ in Theorem~\ref{T(d) properties} can be lifted from $\ovl{T(\bfd)}$ via the next result. Proposition~\ref{RSS2 2.4} is a really a result of Rogalski, who proved it with the additional assumption that $\deg\HH\geq 2$. It is noted in \cite{RSS2} that this assumption was unnecessary.

\begin{prop}\label{RSS2 2.4}\cite[Proposition~2.4]{RSS2}
Let $R$ be a cg $\Bbbk$-algebra which is a domain. Suppose that there exists a homogenous central element $g\in R_d$ for some $d\geq 1$ such that $R/gR\cong B(E,\HH,\rho)$ for some elliptic curve $E$, invertible sheaf $\HH$ with $\deg\HH\geq1$, and an infinite order automorphism $\rho$. \par
Then $R$ is strongly noetherian, Auslander-Gorenstein, Cohen-Macaulay and a maximal order. Also $R$ satisfies the Artin-Zhang $\chi$ conditions, has cohomological dimension 2 and possesses a balanced dualizing complex.
\end{prop}

The important thing is of course to understand subalgebras of $T$. This is provided by the following result.

\begin{theorem}\label{Ro thm}\cite[Theorem~1.2]{Ro}
Let $A=\Bbbk\langle V\rangle$ for some $V\subseteq T_1$ such that $Q_\gr(A)=Q_\gr(T)$. Then there exists a unique effective divisor $\bfd$ on $E$ with $\deg\bfd\leq 7$ such that $T(\bfd)$ is the unique maximal order containing, and an equivalent order to, $A$.\par
In particular, if $A$ is a maximal order generated in degree 1, then $A=T(\bfd)$ for some effective divisor $\bfd$ on $E$ with $\deg\bfd\leq 7$.
\end{theorem}

Theorem~\ref{Ro thm} is where the majority of effort in \cite{Ro} is spent. A crucial ingredient is the next proposition. We recall Notation~\ref{GK dimension} for GK-dimension.

\begin{prop}\label{Ro min sporadic} \cite[Theorem~10.4]{Ro}
Let $\bfd$ be an effective divisor with degree $d\leq 7$. Put $R=T(\bfd)$. Then $R$ has a graded ideal $K$ of $R$ such that $\GK(R/K)\leq 1$ and is minimal in the following sense. Given any graded ideal $I$ of $R$ with $\GK(R/I)\leq 1$, there exits an $n\geq 0$ such that $K_{\geq n} \subseteq I$.
\end{prop}

Ideals like $K$ in Proposition~\ref{Ro min sporadic} are important for us.

\begin{definition}\label{min sporadic bg}\cite[Definition~6.5]{Ro}\index{minimal sporadic ideal}
An ideal $K$ in a cg algebra $R$ satisfying $\GK(R/K)\leq 1$ is called a \textit{minimal sporadic ideal} if for all other ideals $I$ of $R$ satisfying $\GK(R/I)\leq 1$ we have that $K_{\geq n}\subseteq I$ for some $n\geq 0$.
\end{definition}

Recall Notation~\ref{qgr(A)} and Remark~\ref{= in qgr iff ehd}. Another way of phrasing Definition~\ref{min sporadic bg} is that $\pi K$ is minimal in $\qgr(R)$ among objects of the form $\pi I$ where $I$ an ideal of $R$ satisfying $\GK(R/I)\leq 1$.

One difficulty with $T(\bfd)$ when $\deg\bfd>0$ is that unlike $S$ and $T$, there may exist ideals $I$ such that $\GK(T(\bfd)/I)=1$ (see \cite[Section~11]{Ro}). Typically these ideals can lead to strange behaviours. The existence of  minimal sporadic ideals as in Proposition~\ref{Ro min sporadic} helps control precisely this. \par

Also included in \cite{Ro} is a classification of the degree 1 generated maximal orders inside $S$ \cite[Theorem~12.1]{Ro}. There are only two of these $S$ and $S(p)$ from Definition~\ref{S(d) def}. We save the statement for later (Proposition~\ref{Ro 12.2}) as best fits our work.

\subsection{Rogalski, Sierra and Stafford's classification of maximal orders in $S^{(3)}$}\label{RSS review}

The next direct work on classifying algebras birational to the Sklyanin algebra was conducted by Rogalski, Sierra and Stafford in twin papers \cite{RSS, RSS2} where they generalised the earlier work of Rogalski. The paper \cite{RSS2} is dedicated to finding an alternative definition to Definition~\ref{T(d) def} which includes $T(\bfd)$ for $\deg\bfd=8$. The definition (which is \cite[Definition~5.1]{RSS2}) is significantly harder to write down compared with Definition~\ref{T(d) def}. As already mentioned Definition~\ref{T(d) def} is sufficient for us and so a brave reader is referred to \cite{RSS2} for details. We will concentrate on \cite{RSS} and ask the reader to trust that when $\deg\bfd=8$, $T(\bfd)$ exists and satisfies the conclusions of Theorem~\ref{T(d) properties}. \par

In short \cite{RSS} successfully classifies all maximal $T$-orders. These are classified in terms of so-called \textit{virtual blowups} at \textit{virtually effective divisors}. We give the definition of a virtually effective divisor now but postpone the definition of a virtual blowup until Definition~\ref{vblowup def bg}. Throughout this subsection we continue to use Notation~\ref{3 Veronese notation bg} for $T=S^{(3)}$.

\begin{definition}\label{blowup of T def bg}\cite[Definitions~7.1]{RSS}.
A divisor $\bfx$ on $E$ is called \textit{$\tau$-virtually effective} if $\bfx+\tau^{-1}(\bfx)+\dots+\tau^{-(n-1)}(\bfx)$ is effective for all $n\gg0$.
\end{definition}

The main results of \cite{RSS} are given in the next two theorems.

\begin{theorem}\label{RSS 1.4}\cite[Theorem~1.5, Corollary~1.4]{RSS}
Let $\bfx$ be a $\tau$-virtually effective divisor. Then a virtual blowup $F$ of $T$ at $\bfx$ exists. Moreover:
\begin{enumerate}[(1)]
\item $F$ is a maximal order in $Q_\gr(T)=Q_\gr(F)$ and uniquely defines a maximal $T$-order $V=F\cap T$. Moreover $F$ and $V$ are closely related, in the sense that: there is an ideal $K$ of $F$ contained in $V$ with $\GK(F/K)\leq 1$.
\item $F$ and $U$ are strongly noetherian, satisfy the Artin-Zhang $\chi$ conditions and have balanced dualizing complexes.
\end{enumerate}
\end{theorem}

\begin{theorem}\label{RSS 1.2}\cite[Theorem~1.2]{RSS}
Let $U$ be a cg maximal $T$-order such that $\ovl{U}\neq\Bbbk$. Then there is a unique virtually effective divisor $\bfx$ with $\deg\bfx\leq 8$ and a unique virtual blowup $F$ at $\bfx$ such that $F\cap T=U$. In particular $U$ is noetherian and Theorem~\ref{RSS 1.4} holds for the pair $(F,U)$.
\end{theorem}

We give an outline of how these results were obtained. The first significant, and perhaps most crucial step is Theorem~\ref{real RSS 5.25}. Here the authors link a general $g$-divisible subalgebra to a blowup $T(\bfd)$.

\begin{theorem}\label{real RSS 5.25}\cite[Theorem~5.26]{RSS}.
Let $U$ be a $g$-divisible graded subalgebra of $T$ with $Q_\gr(U)=Q_\gr(T)$. Then there exists an effective divisor $\bfd$ on $E$ with $\deg\bfd\leq 8$ called a \emph{normalised divisor for $U$} (Definition~\ref{geo data, norm div}), and such that $U$ and $T(\bfd)$ are equivalent orders.
\end{theorem}

Suppose that $U$ is a $g$-divisible graded subalgebra of $T$ with $Q_\gr(U)=Q_\gr(T)$. Applying Theorem~\ref{RSS 3.1} to $\ovl{U}$, one knows its structure in high degrees. It is then reasonably straight forward to obtain a similar result to Theorem~\ref{real RSS 5.25} for $\ovl{U}$ inside $Q_\gr(\ovl{U})=Q_\gr(\ovl{T})$ \cite[Proposition~5.7]{RSS}. More precisely, out of the geometric data of $\ovl{U}$ one can construct a divisor $\bfd$ such that $\ovl{U}$ and $B(E,\MM(-\bfd),\tau)=\ovl{T(\bfd)}$ are equivalent orders, and that $\ovl{U}\subseteq \ovl{T_{\leq k}}\ovl{T(\bfd)}$ for some $k\geq 0$. The hard work in proving Theorem~\ref{real RSS 5.25} lies in lifting this up to $T$.

\begin{prop}\label{RSS 5.20 bg}\cite[Proposition~5.19]{RSS}\cite[Theorem~5.3(6)]{RSS2}
Suppose that $U$ is a cg $g$-divisible subalgebra of $T$ with $Q_\gr(U)=Q_\gr(T)$. Let $\bfd$ be a normalised divisor for $U$. Then there exists a finitely generated right $T(\bfd)$-module $M$ such that $UT(\bfd)\subseteq M\subseteq T$. If $\deg\bfd\leq 7$, then one can take $M=T_{\leq k}T(\bfd)$ for some $k\geq 0$.
\end{prop}

When we encounter a similar problem in Proposition~\ref{RSS 5.25}, we are able to use a few Veronese tricks to utilise Proposition~\ref{RSS 5.20 bg} directly. This bypasses much of the technicalities that Rogalski, Sierra and Stafford required. Perhaps more important than the statement of Theorem~\ref{real RSS 5.25} was how the equivalence was obtained.

\begin{cor}\label{real RSS 5.25 2}\cite[Theorem~5.26]{RSS}
Let $U$ and $\bfd$ be as in Theorem~\ref{RSS 3.1}. The $(U,T(\bfd))$ bimodule $M=\widehat{UT(\bfd)}$ is a finitely generated $g$-divisible $T(\bfd)$-module. If $W=\End_{T(\bfd)}(M)$, then $U\subseteq W\subseteq T$, $_WM$ is finitely generated and $W$, $U$, and $T(\bfd)$ are all equivalent orders.
\end{cor}

Let $U$ be a $g$-divisible maximal $T$-order, then by Corollary \ref{real RSS 5.25 2} we know that $U=\End_{T(\bfd)}(M)$ for some $T(\bfd)$ and $g$-divisible $M$. Rogalski, Sierra and Stafford go on to study these endomorphism rings appearing. By \cite[Theorem~2.7]{Coz} one knows that $\End_{T(\bfd)}(M^{*})=\End_{T(\bfd)}(M^{**})=F$ is the unique maximal order containing $U$. In \cite[Proposition~6.4]{RSS} Rogalski, Sierra and Stafford investigate the connection between $U$ and $F$, and the ideal in Theorem~\ref{RSS 1.4}(1) appears. A key result is then to understand the factor $F/gF$ in $T/gT$.

\begin{prop}\label{real RSS 6.7}\cite[Theorem~6.7, Proposition~7.3(1)]{RSS}
Let $\bfd$ be an effective divisor on $E$ of $\deg\bfd\leq 8$ and $M$ a finitely generated $g$-divisible right $T(\bfd)$-module satisfying $T(\bfd)\subset M\subset T$. Set $U=\End_{T(\bfd)}(M)$ and $F=\End_{T(\bfd)}(M^{**})$. Then there exists an effective divisor $\mbf{y}$ on $E$ such that (using Notation~\ref{ehd})
$$\ovl{U}\ehd \ovl{F}\ehd B(E,\MM(-\bfx),\tau)\;\;\text{ where }\bfx=\bfd-\mbf{y}+\tau^{-1}(\mbf{y}).$$
Moreover, $F$ and $F\cap T$ satisfy Theorem~\ref{RSS 1.4}(1). The divisor $\bfx$ is virtually effective.
\end{prop}

It is Proposition~\ref{real RSS 6.7} that gives rise to the definitions of a virtual blowup and virtually effective divisor.

\begin{definition}\label{vblowup def bg}
Retain the notation of Proposition~\ref{real RSS 6.7}. The algebra $F$ is called a \textit{virtual blowup at~$\bfx$}.
\end{definition}

We note that Definition~\ref{vblowup def bg} looks different to the definition of a virtual blowup in \cite[Definition~6.9]{RSS}. It is implicit in \cite[Proposition~6.4 and Corollary~6.6]{RSS} that these definitions are equivalent. \par

With this terminology Rogalski, Sierra and Stafford then are able to classify the $g$-divisible maximal orders \cite[Theorem~7.4]{RSS}, which is a version of Theorem~\ref{RSS 1.2} with $g$-divisibility replacing $\ovl{U}\neq\Bbbk$). The final task is to weaken the assumption of $g$-divisibility to obtain Theorem~\ref{RSS 1.4}. For this one needs to show that any subalgebra $A$ of $T$ satisfying $\ovl{A}\neq\Bbbk$ is contained in, and an equivalent order to, a $g$-divisible ring. This is done via two steps. First one shows $A$ is equivalent to the algebra obtained by adding $g$, $A\langle g\rangle=A+gA+g^2A+\cdots$\index[n]{ag@$A\langle g\rangle$}. Next one proves the algebra $A\langle g\rangle$ is an equivalent order with its $g$-divisible hull, $\widehat{A\langle g\rangle}$. This idea first appeared in \cite[Section~7]{Ro} and is where minimal sporadic ideals appear.\par

After the success of \cite{Ro, RSS, RSS2} and developing a good notion of noncommutative ring theoretic blowing up, Rogalski, Sierra and Stafford have continued the project to understand algebras birational to the Sklyanin algebra. Following ideas of the commutative classification of projective surfaces, one needs a notion of blowing down. Hidden in the above, but prominent in \cite{Ro,RSS2}, is the \index{exceptional line module}\textit{exceptional line module associated to a blowup at a point} \cite[Definition~9.2]{Ro}. We won't define it but the analogy with commutative blowups is clear: blowup at a point and get a line. In \cite{RSS.blowdown} Rogalski, Sierra and Stafford go in the reverse direction. Given a line module $L$ (with a noncommutative analogue of self intersection equal to -1) over an elliptic algebra $R$ (a generalisation of the $T(\bfd)$'s) then one can \index{blowdown} ``blowdown $L$ to a point $p$". In practice this means one obtains a overring $\widetilde{R}\supseteq R$ such that if we blowup $\widetilde{R}$ at $p$ we have $\widetilde{R}(p)\cong R$ with exceptional line module associated to this blowup equal to $L$. \par
Having good notions of blowing up and down, and following ideas from commutative geometry, leads to the following question. At what point is it no longer possible to blowdown? Or in other words, what are the minimal models? At various seminars Rogalski, Sierra and Stafford have announced results in this direction.\par
Our research takes a different direction to \cite{RSS.blowdown}. In this thesis we say ``hold on, why did we look at $T$? Surely it is $S$ we are interested in?" Indeed we generalise the work in \cite{Ro, RSS, RSS2} to hold in the entire Sklyanin algebra.

\section{Preliminary results}\label{prelims}

It is finally time to do some actual mathematics. We fix once and for all the standing assumption of this thesis.

\begin{hypothesis}[Standing Assumption]\label{standing assumption 2}
Fix an algebraically closed field $\Bbbk$. Let $S$ be a 3-dimensional Sklyanin algebra. Suppose that $g\in S_3$ is central and such that $S/gS\cong B(E,\LL,\sigma)$, where $E$ is a smooth elliptic curve, $\LL$ an invertible sheaf on $E$ with $\deg\LL=3$ and $\sigma:E\to E$ an automorphism of $E$. Assume that $\sigma$ has infinite order.
\end{hypothesis}

\begin{remark}
Retain notation from Hypothesis~\ref{standing assumption 2}. The results of Section~\ref{generic sklyanin}, Section~\ref{Rogs blowup review} and Section~\ref{RSS review} will always be applicable to $S$, $T=S^{(3)}$ and $B(E,\LL,\sigma)$.
\end{remark}

Our first proper section is split distinctly into two halves. Section~\ref{g div rings sec} and Section~\ref{endo over g divs} develop a basic theory of $g$-divisible subalgebras of $S$ (Definition~\ref{g div def}). Much of this follows in a similar manner to $g$-divisible subalgebras of $S^{(3)}$. Indeed this part of the section follows \cite[Section~2]{RSS} closely. In the latter part, Section~\ref{orders sec} and Section~\ref{Veronese sec} we turn our attention to the interaction between maximal orders and Veronese rings. In particular we are able to prove (in both $S$ and $S^{(3)}$) that given a $g$-divisible subalgebra $U$, $U$ is a maximal order in $Q_\gr(U)=Q_\gr(S)$ (or $Q_\gr(S^{(3)})) $ if and only if a higher Veronese subring $U^{(d)}$ is a maximal order (Proposition~\ref{g-div max orders up n down}). The case when $Q_\gr(U)=Q_\gr(S^{(3)})$ is something that was missing from \cite{RSS}. A limitation of this result is that it does not pass from $S$ to $S^{(3)}$. More precisely, when $Q_\gr(U)=Q_\gr(S)$, we can only take $d$ to be coprime to $3$.

\subsection{Basics}\label{basics}

Before we start looking at subalgebras of the Sklyanin algebra we state and prove a few simple results about general connected graded domains for which we lack good references.

\begin{lemma}\label{choose x of appropriate deg}
Let $A$ be a cg domain with $Q=Q_\gr(A)$ and such that $Q_1\neq 0$. Then $A_n\neq 0$ for all $n\gg 0$.
\end{lemma}
\begin{proof}
Since $Q_1\neq 0$ there exist nonzero homogenous $x,y\in A$ with $\deg x=n+1$ and $\deg y=n$ for some $n\geq 0$ so that $xy^{-1}\in Q_1$. First note that because $A$ is a domain $y^n, y^{n-1}x, \dots, yx^{n-1},x^n$ are nonzero elements of $A_{n^2}, A_{n^2+1},\dots, A_{n^2+n}$ respectively. Then for $m\geq n^2$, write $m=n^2+kn+r$ with $k\geq 0$ and $0\leq r\leq n-1$ then $0\neq y^kA_{n^2+r} \subseteq A_m$; proving $A_m \neq 0$ for all $m\geq n^2$.
\end{proof}

\begin{remark}\label{Q_1 neq 0}
The assumption $Q_\gr(A)_1\neq 0$ in Lemma~\ref{choose x of appropriate deg} does not cause any problems. Given any cg domain $A$, it is clear from Definition~\ref{Qgr and Dgr} that we can always regrade such that $Q_\gr(A)_1\neq 0$ holds.
\end{remark}

\begin{lemma}\label{Ab ehd (Ageqn)B}
Let $A$ be a cg domain with graded quotient ring $Q=Q_\gr(A)$. Suppose $M,X\subseteq Q$ are finitely generated graded right $A$-modules such that $M\ehd X$, and that $N,Y\subseteq Q$ are finitely generated graded left $A$-modules such that $N\ehd Y$. Then $MN\ehd XY$.
\end{lemma}
\begin{proof}
Let $k,\ell \geq 0$ be such that $M_{\geq k}=X_{\geq k}$ and $N_{\geq\ell}=Y_{\geq \ell}$. By repeating the argument for $X$ and $Y$, it is enough to show $MN\ehd M_{\geq k}N_{\geq\ell}$. Write
$$M=x_1A+\dots + x_nA\;\text{ and }\;N=Ay_1+\dots +Ay_m\;\text{ with $x_i,y_j$ homogenous},$$
 then $MN=\sum_{i,j}x_iAy_j$. Let $d=\max_{i,j}\{\deg x_i,\deg y_j\}$. We claim that the inclusion $(MN)_{\geq k+2d}\subseteq M_{\geq k}N$ holds. Take a homogenous $z\in (MN)_{\geq k+2d }$, then
$$z=\sum_{i,j}x_ia_{ij}y_j=\sum_j\left(\sum_i x_ia_{ij}\right)y_j, \; \text{ for some homogenous } a_{ij}\in A.$$
Since $\deg(z)\geq k+2d$, we have $k+2d\leq \deg (x_ia_{ij}y_j)\leq 2d+\deg (a_{ij})$ for each $i,j$. This forces $\deg(a_{ij})\geq k$. Therefore, for each $j$, $\sum_i x_ia_{ij}\in M_{\geq k}$. Thus $z\in M_{\geq k}N$, proving the claim. Since we have the obvious inclusion $M_{\geq_k}N\subseteq MN$, the claim shows that $MN\ehd M_{\geq k}N$. By replacing $M$ by $M_{\geq k}$ a symmetric argument with $N$ shows $MN\ehd M_{\geq k}N_{\geq\ell}$ as required.
\end{proof}

Recall Remark~\ref{Q always exists} which observes that our rings are all Ore domains. We give a couple of consequences of the Ore condition \cite[1.13]{MR}.

\begin{lemma}\label{Ore argument}
Let $A$ be a $\N$-graded domain with graded quotient ring $Q=Q_\gr(A)$.
\begin{enumerate}[(1)]
\item For any homogenous $x_1,\dots,x_n\in Q$, we can choose homogenous elements $a_1,\dots,a_n,c$ of $A$ with $c\neq 0$, and such that  $x_i=a_ic^{-1}$ for each $i$.
\item Let $M$ be a finitely generated graded right $A$-submodule of $Q$. Then there exists a nonzero $c\in A$ such that $cM\subseteq A$.
\end{enumerate}
\end{lemma}
\begin{proof}
Since $Q_\gr(A)$ exists $A$ must be an Ore domain. (1) is then a simple application of the Ore condition (see \cite[Lemma~4.21]{GW}). For (2) write $M=x_1A+\dots +x_nA$ with each $x_i$ homogenous. By the left hand version of (1), we can pick a nonzero $c,a_i\in A$, with $c\neq 0$ such that $x_i=c^{-1}a_i$. In which case
\begin{equation*} cM=c(c^{-1}a_1A+\dots +c^{-1}a_iA)=a_1A+\dots +a_nA\subseteq A.  \qedhere \end{equation*}
\end{proof}

Another standard application of the Ore condition is the following.

\begin{lemma}\label{Hom in Qgr}
Let $A$ be a graded domain with graded quotient ring $Q=Q_\gr(A)$. Let $M,N$ be nonzero right $A$-submodules of $Q$. Then we can identify
\begin{equation*}\pushQED{\qed}  \Hom_R(M,N)=\{ q\in Q\,|\; qM\subseteq N\}. \qedhere \end{equation*}
\end{lemma}

We make the identification in Lemma~\ref{Hom in Qgr} a standard throughout.

\begin{notation}\label{Hom inside Q}\index[n]{homa@$\Hom_A(M,N)$}\index[n]{homa@$\Hom_A(M,N)$}
Let $A$ be a cg domain with graded quotient ring $Q=Q_\gr(A)$. Let $M,N$ be finitely generated graded nonzero right $A$-submodules of $Q$. Via Lemma~\ref{Hom in Qgr} we will identify $\Hom_A(M,N)=\{q\in Q \, |\; qM\subseteq N\}.$ In particular \index[n]{end@$\End_A(M)$}\index[n]{M@$M^*$}
$$\End_A(M)=\{q\in Q \, |\; qM\subseteq M\} \;\text{ and }\; M^*=\Hom_A(M,A)=\{q\in Q \, |\; qM\subseteq A\}.$$
Moreover $\Hom_A(M,N)$ is a graded subspace of $Q$. We also make the analogous identifications on the left.
\end{notation}

\begin{lemma}\label{End is cg}
Let $R$ be a cg domain with $Q_\gr(R)=Q$, and $M\subseteq Q$ a finitely generated nonzero graded right $R$-module. Then $U=\End_R(M)$ is also a connected graded domain satisfying $Q_\gr(U)=Q$.
\end{lemma}
\begin{proof}
By Lemma~\ref{Hom in Qgr} we can identify $U=\{q\in Q \, |\; qM\subseteq M\}$. It is then clear $U$ is an $\N$-graded domain and $Q_\gr(U) \subseteq Q$. Let $M=x_1R+\dots+x_nR$, with $x_i$ homogenous, say $\deg x_i=d_i$. Then, for $j\in \Z$, $M_j=\sum x_iR_{j-d_i}$, and so $\dim_\Bbbk M_j<\infty$. Now fix $k\in \N$ and $j\in \Z$ such that $M_j\neq 0$. We have $U_kM_j\subseteq M_{k+j}$ which forces $\dim_\Bbbk U_k<\infty$. Finally, $U_0$ is an artinian domain, hence division ring. It is finite dimensional over a central algebraically closed field $\Bbbk$; this ensures $U_0=\Bbbk$. Thus $U$ is connected graded.\par
It remains to show $Q_\gr(U)=Q$. By Lemma~\ref{Ore argument}(2) there exists a nonzero homogenous $c\in R$ such that $cM\subseteq R$. Let $U'=\End_R(cM)$, then one can check
\begin{equation}\label{U'=cUc^-1} U'=cUc^{-1}\;\text{ and }\;U=c^{-1}U'c.\end{equation}
 Clearly also $Q_\gr(U')\subseteq Q$. Moreover, $cM\subseteq U'$ because $(cM)(cM)\subseteq cMR=cM$. Then given $s^{-1}r\in Q$ with $r,s\in R$ and $y\in cM$, we have $s^{-1}r=(ys)^{-1}(yr)\in Q_\gr(U')$, showing $Q_\gr(U')=Q$. Using (\ref{U'=cUc^-1}), if $uv^{-1}\in Q_\gr(U')$ with $u,v\in U$, then
$$Q_\gr(U)\ni (c^{-1}uc)(c^{-1}vc)^{-1}=c^{-1}(uv^{-1})c.$$
It follows $Q_\gr(U)\supseteq c^{-1}Q_\gr(U')c=c^{-1}Qc=Q$.
\end{proof}

Lastly we present a standard lemma on the Cohen-Macaulay property. The notation relegating to GK-dimension can be found in Notation~\ref{GK dimension}; it is applicable thought this thesis.

\begin{lemma}\label{CM lemma}
Let $A$ be a cg noetherian domain with $Q_\gr(A)=Q$. Suppose that $A$ is Auslander-Gorenstein and Cohen-Macaulay of GK-dimension at least 2.
\begin{enumerate}[(1)]
\item There does not exist a right nor left $A$-module $M$ such that $A\subsetneq M\subseteq Q$ and $\dim_\Bbbk M/A <\infty$.
\item  If $B$ is another cg subalgebra of $Q$ such that $B\ehd A$, then $B\subseteq A$.
\item  Let $M,N\subseteq Q$ be finitely generated right $A$-modules. If  $M\ehd N$, then $M^*=N^*$.
\end{enumerate}
\end{lemma}
\begin{proof}(1). Assume for a contradiction that such $M$ exists, and set $N=M/A$. Consider the short exact sequence
$$0\longrightarrow A\longrightarrow M\longrightarrow N\longrightarrow 0.$$
Since $A$ is a domain, this sequence cannot be split. It is well known (see \cite[Proposition~7.24]{Rot}) that this implies $\Ext_A^1(N,A)\neq 0$, and hence $j(N)\leq 1$. On the other hand, because $A$ is Cohen-Macaulay, we have $\GK(N)+j(N)=\GK(A)$. Now $\GK(N)=0$ because $\dim_\Bbbk N<\infty$, and $\GK(A)\geq 2$ by assumption; thus $j(N)\geq 2$. This contradicts $j(N)\leq 1$. \par
(2). Assume that $B\not\subseteq A$. Let $n\gg0$ and set $I=A_{\geq n}=B_{\geq n}$. Consider $I$ as a right $A$-module. Then as in Notation~\ref{Hom inside Q},
$$I^{*}=\Hom_A(I,A)=\{x\in Q\,|\; xI\subseteq A\}.$$
Clearly both $I^*\supseteq A$ and $I^*\supseteq B$; in particular, $I^*\supsetneq A$ because $B\not\subseteq A$. Choose homogenous $y\in I^*\setminus A$, then we have a natural surjection
$A/I\to (yA+A)/A$ given by $a+I\mapsto ya+A$. It is well defined because $y\in I^*$. Hence $M=(yA+A)/A$ is finite dimensional, which contradicts (1).\par

(3). Assume that $M^*\neq N^*$. Without loss of generality say $M^*\not\subseteq N^*$ (otherwise relabel), then $M^*N\not\subseteq A$. Pick a homogenous $x\in M^*N\setminus A$, and consider the right $A$-module $X=xA+A\supsetneq A$. Because $M\ehd N$, for $n\gg 0$ we have
$$xA_n\subseteq M^*NA_n=M^*MA_n\subseteq M^*M\subseteq A.$$
Hence $\dim_\Bbbk X/A <\infty$, again contradicting (1).
\end{proof}

\subsection{$g$-divisible rings}\label{g div rings sec}

What is clear from the work Rogalski, Sierra and Stafford is that the property of $g$-divisibility will play an important role: a $g$-divisible subring is significantly easier to describe as properties pass a lot more smoothly between a ring and its image in $\ovl{S}=B(E,\LL,\sigma)$. Here, and in Section~\ref{endo over g divs}, we will mimic large chunks of \cite[Section~2]{RSS}. In particular, the proofs of Lemma~\ref{ovlR is order} and Lemma~\ref{RSS 2.15} here come from analogous results \cite[Lemma~2.10(1)]{RSS} and \cite[Lemma~2.15]{RSS}. These proofs demonstrate many useful techniques for $g$-divisible rings and are included for the reader's convenience. \par

We work in the slightly bigger ring obtained by inverting elements outside $gS$. 

\begin{lemma}
The set $\mathcal{C}$ consisting of the homogenous elements in $S\setminus gS$ is an Ore set.
\end{lemma}
\begin{proof}
Take homogeneous $a\in S$ and $x\in S\setminus gS$. Write $a=a'g^n$ with $n$ maximal, and hence $a'\notin gS$. Since $S$ is an Ore domain (see Remark~\ref{Q always exists}), there exists nonzero homogenous $b,y\in S$ such that $a'y=xb$. Moreover, because $a',x\notin gS$ and $gS$ is a completely prime ideal, $b\in gS$ if and only if $y\in gS$. Cancelling $g$'s if necessary we may assume $b,y\notin gS$. But then $ay=x(bg^n)$, showing $\mathcal{C}$ is right Ore. Symmetrically $\mathcal{C}$ is left Ore.
\end{proof}

\begin{notation}\label{S_(g)}\index[n]{sg@$\Sg$} Let $S_{(g)}$ denote the homogenous localisation of $S$ at the completely prime ideal $gS$. That is, $S_{(g)}=S\mathcal{C}^{-1}$ where $\mathcal{C}$ is the set consisting of the homogenous elements in $S\setminus gS$. It is clear that $Q_\gr(\Sg)=Q_\gr(S)$ and \begin{equation}\label{ovlS_(g)}\Sg/g\Sg=Q_\gr(S/gS)=\Bbbk(E)[t,t^{-1};\sigma].\end{equation}
For a subset $X\subseteq \Sg$, we extended Notation~\ref{ovlX notation} to $\Sg$: $$\ovl{X}=(X+g\Sg)/g\Sg\subseteq \Bbbk(E)[t,t^{-1};\sigma].\index[n]{x@$\ovl{X}$}$$
 As $S\cap g\Sg=gS$ one has $S/gS\hookrightarrow \Sg/g\Sg$; in particular the notation $\ovl{X}$ is unambiguous.
\end{notation}

Before we start with $g$-divisibility we prove a straight forward fact about $S_{(g)}$.

\begin{lemma}\label{ideals of S_(g)}
The only homogenous right or left ideals of $S_{(g)}$ are those of the form $g^nS_{(g)}$ for some $n\geq0$.
\end{lemma}
\begin{proof}
First note from (\ref{ovlS_(g)}) that $S_{(g)}/gS_{(g)}$ is a graded division ring; it therefore has no non-trivial proper homogenous right ideals. Thus, if $I\subsetneq S_{(g)}$ is a nonzero proper homogenous right ideal, then $I\subseteq gS_{(g)}$. Choose $n$ maximal such that $I\subseteq g^nS_{(g)}$ (and so $I\nsubseteq g^{n+1}S_{(g)}$). Write $I=g^nJ$ for some $J\subseteq S$. A quick check shows $J$ is a nonzero right ideal of $S_{(g)}$. If $J\neq S_{(g)}$, then $J\subseteq gS_{(g)}$, which would contradict the choice of $n$. Thus $J=S_{(g)}$ and $I=g^nS_{(g)}$. A symmetric argument proves the statement for left ideals.
\end{proof}

\begin{definition}\label{g-torsion} \index{gtorsion@$g$-torsion}\index{gtorsionfree@$g$-torsionfree}\index[n]{torsgm@$\mathrm{tors}_g(M)$}
Let $R\subseteq S_{(g)}$ be a cg subalgebra containing $g$, and let $M$ be a graded right $R$-module. Put
$$\mathrm{tors}_g(M)=\{a\in M|\, ag^n =0 \, \text{ for some } \, n\geq 0\}.$$
We call $\mathrm{tors}_g(M)$ the $g$-\textit{torsion} submodule of $M$. We say that $M$ is $g$-\textit{torsion} if $\mathrm{tors}_g(M)=M$, and $M$ is $g$-\textit{torsionfree} if $\mathrm{tors}_g(M)=0$.
\end{definition}

\begin{lemma}\label{g-div}
Let $R$ be a cg subalgebra of $\Sg$ containing $g$. Let $M\subseteq S_{(g)}$ be a graded right $R$-module and set $\widehat{M}=\{a\in S_{(g)}|\, ag^n\in M \text{ for some } n\geq 0\}$. Then $\widehat{M}$ is also a graded right $R$-module. Moreover the following are equivalent:
\begin{enumerate}[(1)]
\item $M\cap Sg= gM$;
\item $M=\widehat{M}$;
\item $S_{(g)}/M $ is a $g$-torsionfree $R$-module.
\end{enumerate}
\end{lemma}
\begin{proof} If $xg^n,yg^m \in M$ then $(x+y)g^{n+m}\in M$, and because $g$ is central, $(xa)g^n\in M$ for all $a\in R$. So $\widehat{M}$ is a right $R$-module, clearly graded. Suppose $M\cap Sg=Mg$. If $x\in \widehat{M}$ then for some $n\geq 1$, $xg^n\in M\cap Sg=Mg$, and so $xg^{n-1}\in M$. By induction we get $x\in M$. The converse is even more obvious. (2) and (3) basically have the same definition.
\end{proof}

\begin{definition}\label{g div def}\index{gdivisible@$g$-divisible}  \index{gdivisiblehull@$g$-divisible hull} \index[n]{mr@$\widehat{M}$, $\widehat{R}$}
Let $R\subseteq S_{(g)}$ be a cg subalgebra containing $g$, and let $M\subseteq S_{(g)}$ be a graded right $R$-module. We say $M$ is \textit{$g$-divisible} if any, and hence all, of the equivalent conditions of Lemma~\ref{g-div} hold. The graded right $R$-module
$$\widehat{M}=\{a\in S_{(g)}|\, ag^n\in M \text{ for some } n\geq 0\} $$
is called the \textit{$g$-divisible hull} of $M$. The ring $R$ is \textit{$g$-divisible} if it is $g$-divisible as a right (equivalently left) module over itself.
\end{definition}

As a general strategy we will be proving many properties in the factor rings $\Sg/g\Sg$ and $S/gS$, and be lifting these up to $\Sg$ and $S$. The significance of the next lemma is therefore clear.

\begin{lemma}\label{ovlR is order}
Let $R$ be a $g$-divisible cg subalgebra of $\Sg$ with $Q_\gr(R)=Q_\gr(S)$. Then $Q_\gr(\ovl{R})=Q_\gr(\ovl{S})$.
\end{lemma}
\begin{proof}
Let $S_1=x_0\Bbbk +x_1\Bbbk + x_2\Bbbk$. Since $Q_\gr(R)=Q_\gr(S)$, there exist nonzero homogenous $r_i,z_i\in R$ such that $z_i^{-1}r_i=x_i$. By Lemma~\ref{Ore argument}(1) we may assume $z_0=z_1=z_2=z$, say with $\deg(z)=n$. Hence we see $zS_1\subseteq R_{n+1}$, and it follows $\ovl{z}\ovl{S_1}\subseteq \ovl{R}_{n+1}$. If $\ovl{z}\neq 0$ we're done, as then $Q_\gr(\ovl{R})$ would contain the generating set of $Q_\gr(\ovl{S})$. So suppose $\ovl{z}=0$. Write $z=g^ky$ for some $k\geq 1$ and $y\in \Sg\setminus g\Sg$; as $R$ is $g$-divisible $y\in R$.  Again by $g$-divisibility   $g^kyS_1\subseteq R\cap g^k\Sg=g^kR$, and so $yS_1\subseteq R$. But this means $\ovl{y}\ovl{S_1}\subseteq \ovl{R}$ with $\ovl{y}\neq 0$, and again $Q_\gr(\ovl{R})$ contains the generating set of $Q_\gr(\ovl{S})$.
\end{proof}

As demonstrated in the next result, $g$-divisible rings are ``nice".

\begin{prop}\label{RSS 2.9}
Let $R$ be a $g$-divisible cg subalgebra of $S_{(g)}$ with $Q_\gr(R)=Q_\gr(S)$. Then $R$ is strongly noetherian (in particular noetherian) and finitely generated as a $\Bbbk$-algebra.
\end{prop}
\begin{proof}
Since the generators of $R_{\geq 1}$ as a right (or left) ideal generate $R$ as a $\Bbbk$-algebra, it is enough to only prove $R$ is strongly noetherian. Now,
$$R/gR\cong \frac{R+g\Sg}{g\Sg} =\ovl{R}\subseteq \ovl{\Sg}=\Bbbk(E)[t,t^{-1};\sigma],$$
and moreover, $Q_\gr(\ovl{R})=Q_\gr(\ovl{S})=\Bbbk(E)[t,t^{-1};\sigma]$ by Lemma~\ref{ovlR is order}. By Theorem~\ref{RSS 3.1} there exists $m\in\N$ such that
$$\ovl{R}_{\geq m}=\bigoplus_{n\geq m} H^0(E,\II\HH_n)\subseteq B(E,\HH,\sigma),$$
for some ideal sheaf $\II$ and $\sigma$-ample invertible sheaf $\HH$. In particular the ring $\Bbbk+\ovl{R}_{\geq m}$ is a subring of the twisted homogenous coordinate ring $B(E,\HH,\sigma)$. In which case $\Bbbk+\ovl{R}_{\geq m}$ is noetherian by Proposition~\ref{B just infinite}. Clearly then $\ovl{R}$ is noetherian also. Since we have $\GK(\ovl{R})=2$, $\ovl{R}$ is strongly noetherian by \cite[Theorem~4.24]{ASZ}. Finally $R$ is then strongly noetherian by \cite[Theorem~ 4.9(1)]{ASZ}.
\end{proof}

Next we provide a useful lemma about the structure of ideals in a $g$-divisible ring. Recall an algebra $A$ is called \textit{just infinite} if for every nonzero ideal $I$ of $A$ one has $\dim_\Bbbk A/I <\infty$.

\begin{lemma}\label{RSS 2.15}
Let $R$ be a $g$-divisible cg subalgebra of $\Sg$ with $Q_\gr(R)=Q_\gr(S)$. Then:
\begin{enumerate}[(1)]
\item If $J$ is a nonzero $g$-divisible ideal of $R$, then $\GK(R/J)\leq 1$.
\item Conversely, if $J$ is an ideal of $R$ such that $\GK(R/J)\leq 1$, then $J\not\subseteq gR$ and $\dim_\Bbbk \widehat{J}/J \leq \infty$.
\item If $K$ is any ideal of $R$, then $K=g^nI$ for some $n\geq 1$ and ideal $I$ satisfying $\GK(R/I)\leq 1$.
\item Suppose $X,Y$ are graded subspaces of $\Sg$ such that $X,Y\not\subseteq g\Sg$. Assume that $I=XY$ is an ideal of $R$. Then $\GK(R/I)\leq 1$.
\end{enumerate}
\end{lemma}
\begin{proof}
(1). By Lemma~\ref{ovlR is order}, $\ovl{R}\subseteq \Bbbk(E)[t,t^{-1};\sigma]=Q_\gr(\ovl{R})$. By Proposition~\ref{B just infinite} $\ovl{R}$ is just infinite. Because $J$ is $g$-divisible, $J\not\subseteq gR$ and therefore $\ovl{J}$ is a nonzero ideal of $\ovl{R}$. Hence $\dim_\Bbbk \ovl{R}/\ovl{J}<\infty$. Set $\widetilde{R}=R/J$. We have $\widetilde{R}/g\widetilde{R}\cong R/(J+gR)$, and since $R$ is $g$-divisible, we also have $\ovl{R}/\ovl{J}\cong R/(J+gR)$. Hence $\dim_\Bbbk \widetilde{R}/g\widetilde{R} <\infty$. It then follows that $\widetilde{R}_{m}=g\widetilde{R}_{m-3}$ for all $m\gg 0$. Since $J$ is $g$-divisible, the element $g$ remains regular in $\widetilde{R}$ by Lemma~\ref{g-div}(3), in particular we get that $\dim_\Bbbk \widetilde{R}_m=\dim_\Bbbk \widetilde{R}_{m-3}$ for all $m\gg0$. This clearly implies $\GK(\widetilde{R})\leq 1$. \par

(2). First note that if $J\subseteq gR$ then, $R/J\twoheadrightarrow R/gR=\ovl{R}$. By Theorem~\ref{RSS 3.1}, $\GK(\ovl{R})=2$, and so $\GK(R/J)\geq 2$ by \cite[Proposition~5.1]{KL}. Thus $J\not\subseteq gR$. Therefore, since $\ovl{R}$ is just infinite (Proposition~\ref{B just infinite}), $\dim_\Bbbk \ovl{R}/\ovl{J}<\infty$. Since $R$ is noetherian by Proposition~\ref{RSS 2.9}, $\widehat{J}$ is a finitely generated as a right $R$-module, and so there exists $n\geq 0$ such that $\widehat{J}g^n\subseteq J$.
Let $J'$ is the largest right ideal such that $J\subseteq J'\subseteq \widehat{J}$ and $\dim_\Bbbk J'/J<\infty$ (it exists by \cite[Proposition~5.1(f)]{KL} and the noetherianess of $R$). Then \cite[Proposition~5.6]{KL} implies $\dim_\Bbbk RJ'/J<\infty$, and so $J'=RJ'$. This shows $J'$ is an ideal of $R$. Without loss of generality replace $J$ with $J'$. Suppose we still have $J\neq \widehat{J}$, say $x\in \widehat{J}\setminus J$. Replacing $x$ with $xg^i$ where $i\geq 0$ is maximal such that $xg^i\notin \widehat{J}$, we may assume $xg\in J$. In which case $x(gR+J)\subseteq J$ and so left multiplication by $x$ defines a surjection $R/(gR+J)\to (xR+J)/J$. By maximality of $J'=J$, $\dim_\Bbbk (xR+J)/J= \infty$. However, $R/(gR+J)\cong\ovl{R}/\ovl{J}$ is finite dimensional; hence we have a contradiction. Thus $\widehat{J}=J$. \par

(3). Write $K=g^nI$ with $I\subseteq R$ and $n$ maximal. It is easily checked $I$ is an ideal of $R$. Then since $n$ is maximal, $I\not\subseteq gR$. By part (1), $\GK(R/\widehat{I})\leq 1$, and by part (2), $\dim_\Bbbk \widehat{I}/I <\infty$. Hence $\GK(R/I)\leq 1$ follows.\par
(4). Since $g\Sg$ is completely prime, $I=XY\not\subseteq g\Sg$ and so $I\not\subseteq gR$. Now apply part (3).
\end{proof}

\subsection{Modules, duals and endomorphism rings over $g$-divisible rings}\label{endo over g divs}

For now we continue following \cite[Section~2]{RSS} and investigate the first properties of $g$-divisible subrings rings of $\Sg$. Below the proof Lemma~\ref{RSS 2.13} is copied from the proof of \cite[Lemma~2.13]{RSS}. The proof of Lemma~\ref{RSS 2.12} is also essentially that of \cite[Lemma~2.12]{RSS}, however on this occasion there was a small hole which needed mending.
\begin{lemma}\label{RSS 2.12}
Let $R$ be a $g$-divisible cg subalgebra of $\Sg$ with $Q_\gr(R)=Q_\gr(S)$. Suppose $M,N\subseteq \Sg$ are finitely generated nonzero right $R$-modules.
\begin{enumerate}[(1)]
\item If $M\not\subseteq g\Sg$, then $\Hom_R(M,N)\subseteq \Sg$.
\item If $N$ is $g$-divisible and $M\not\subseteq g\Sg$ (in particular this holds if $M$ is $g$-divisible), then $\Hom_R(M,N)$ is also $g$-divisible.
\item If $M$ is $g$-divisible, then $U=\End_R(M)\subseteq \Sg$ is $g$-divisible and noetherian, and $M$ is finitely generated left $U$-module. Moreover $\ovl{U}\subseteq \End_{\ovl{R}}(\ovl{M})$.
\end{enumerate}
\end{lemma}
\begin{proof}
(1). Since $M\not\subseteq g\Sg$, $M\Sg=\Sg$ by Lemma~\ref{ideals of S_(g)}, hence using Notation~\ref{Hom inside Q},
$$\Hom_R(M,N)\subseteq \Hom_{\Sg}(M\Sg, N\Sg)= \Hom_{\Sg}(\Sg, N\Sg)=N\Sg\subseteq \Sg.$$
\par
(2). By (1), $P=\Hom_R(M,N)\subseteq \Sg$. Take $x\in P\cap g\Sg$, say $x=gs$ for some $s\in\Sg$. Then $sgM=xM\subseteq N\cap g\Sg =gN$. Hence $sM\subseteq N$ and $s\in P$. \par
(3). As $M$ is $g$-divisible, clearly $M\not\subseteq g\Sg$, and so $U\subseteq\Sg$ by part (1). By part (2) and Lemma~\ref{End is cg}, $U$ is $g$-divisible and connected graded, hence it is noetherian by Proposition~\ref{RSS 2.9}.
Since $M_R$ is finitely generated and $Q_\gr(R)=Q_\gr(S)$, there exists a nonzero homogenous $c\in R$ such that $cM\subseteq R$ by Lemma~\ref{Ore argument}(2). So then $(Mc)M\subseteq MR\subseteq M$. Therefore, as a left $U$-module $M\cong Mc\subseteq U$ is finitely generated. Finally, as $UM\subseteq M$ implies $\ovl{U}\,\ovl{M}\subseteq \ovl{M}$, we get $\ovl{U}\subseteq \End_{\ovl{R}}(\ovl{M})$.
\end{proof}

In the proof of \cite[Lemma~2.12(3)]{RSS} (the $S^{(3)}$-version of Lemma~\ref{RSS 2.12}) there was a small hole. Their proof required $\End_R(M)\subseteq S^{(3)}$; however it was only guaranteed that $U\subseteq (\Sg)^{(3)}$. This can be repaired with an identical proof to that of Lemma~\ref{RSS 2.12}(3) using Lemma~\ref{End is cg} and a $S^{(3)}$-version of Lemma~\ref{RSS 2.9}.

\begin{lemma}\label{RSS 2.13}
Let $R$ be a cg subalgebra of $\Sg$ with $Q_\gr(R)=Q_\gr(S)$, and let $M$ be a right $R$-submodule of $\Sg$ such that $M\not\subseteq g\Sg$.
\begin{enumerate}[(1)]
\item For every $x\in \Sg\setminus g\Sg$, we have $\widehat{xM}=x\widehat{M}$.
\item If $R$ is $g$-divisible and $M$ is finitely generated $R$-module, then $\widehat{M}$ is also a finitely generated $R$-module.
\item If $R$ is $g$-divisible, then $M^*=\widehat{M^*}\subseteq \Sg$ and so $M^*\not\subseteq g\Sg$. Similarly for the double dual $M^{**}=\widehat{(M^{**})}\subseteq \Sg$. Moreover $(\widehat{M})^*=M^*$ and $(\widehat{M})^{**}=M^{**}$.
\end{enumerate}
\end{lemma}
\begin{proof}
(1). Let $r\in\widehat{M}$, say $rg^n\in M$, and so $xrg^n\in xM$. Because $xr\in\Sg$, $xr\in\widehat{xM}$. Conversely, let $r\in\Sg$ with $rg^n\in xM$. Then $rg^n=g^nr\in g^n\Sg\cap x\Sg$. Since $g\Sg$ is a completely prime ideal and $x\notin g\Sg$, we have $g^n\Sg\cap x\Sg=g^nx\Sg$. Thus $r=xs$ for some $s\in\Sg$ and $xM\ni rg^n=xsg^n$. Therefore $sg^n\in M$, showing $s\in\widehat{M}$ and $r\in x\widehat{M}$. \par

(2). By Lemma~\ref{Ore argument}(2), there exists a  nonzero homogenous  $x\in\Sg$ such that $xM\subseteq R$. If $x=gy$ for some $y\in\Sg$, then $g(yM)\subseteq R$, and so $yM\subseteq R$ by $g$-divisibility of $R$. So we may assume $x\in\Sg\setminus g\Sg$. Also, again by $g$-divisibility of $R$, $\widehat{xM}\subseteq \widehat{R}=R$. Since $R$ is noetherian by Proposition~\ref{RSS 2.9}, $\widehat{xM}$ is a finitely generated right ideal of $R$. Using (1), we then have $\widehat{M}\cong x\widehat{M}=\widehat{xM}$ is a finitely generated right $R$-module.\par

(3). By Lemma~\ref{RSS 2.12}(2), $M^*=\Hom_R(M,R)\subseteq \Sg$ and is $g$-divisible. Clearly then $M^*\not\subseteq g\Sg$, and so symmetrically we can prove $M^{**}=\widehat{M^{**}}\subseteq S_{(g)}.$  Now as $M\subseteq\widehat{M}$, certainly $(\widehat{M})^*\subseteq M^*$. Conversely, if $a\in M^*$ and $x\in\widehat{M}$, say $xg^n\in M$, then $a(xg^n)=(ax)g^n\in R=\widehat{R}$ and so $ax\in R$. Hence $a\in (\widehat{M})^*$, proving $(\widehat{M})^*=M^*$. Taking the second dual gives $(\widehat{M})^{**}=M^{**}$.
\end{proof}

\begin{lemma}\label{hat end commute}
Let $R$ be a $g$-divisible cg subalgebra of $\Sg$. Let $M\subseteq \Sg$ be a finitely generated right $R$-module. Then $\End_R(\widehat{M})=\widehat{\End_R(M)}.$
\end{lemma}
\begin{proof}
By Lemma~\ref{RSS 2.13}(2) $\widehat{M}$ is finitely generated, say $\widehat{M}=x_1R+\dots+x_nR$. Each $x_i\in \widehat{M}$, so there exists $k_i\geq 0$ such that $g^{k_i}x_i\in M$. Hence, if $k=\max_i\{k_i\}$, then $g^k\widehat{M}\subseteq M$. Set $U=\End_R(M)$ and $V=\End_R(\widehat{M})$. By Lemma~\ref{RSS 2.12}(3), $V$ is $g$-divisible, and hence to prove $\widehat{U}\subseteq V$, it is enough to prove $U\subseteq V$. Let $u\in U$. Then
$$(ug^k)\widehat{M}=u(g^k\widehat{M})\subseteq uM\subseteq M\subseteq \widehat{M},$$
so $ug^k\in V$. Since $V$ is $g$-divisible $u\in V$. Therefore $U\subseteq V$, and then $\widehat{U}\subseteq V$. Now take $v\in V$. Then
$$(vg^k)\widehat{M}=g^k(v\widehat{M})\subseteq g^k\widehat{M}\subseteq M,$$
so certainly $vg^kM\subseteq M$. Therefore $vg^k\in U$ and $v\in\widehat{U}$.
\end{proof}

Here we note some special properties of modules of GK-dimension less than or equal to 1.

\begin{lemma}\label{RSS 2.14}
Let $R$ be a cg $g$-divisible subalgebra of $\Sg$ and $N\subseteq M$ be finitely generated right $R$-submodules of $\Sg$. Suppose that $M$ is $g$-divisible.
\begin{enumerate}[(1)]
\item If $\ovl{N}\ehd\ovl{M}$, then $\GK(M/N)\leq 1$.
\item If $N$ is also $g$-divisible, then $\ovl{N}\ehd\ovl{M}$ if and only if $\GK(M/N)\leq 1$.
\item If (2) is satisfied, then $M/N$ is finitely generated as a $\Bbbk[g]$-module and has an eventually 3-periodic Hilbert series.
\end{enumerate}
\end{lemma}
\begin{proof}
Set $X=M/N$ and consider $X/gX$. We have
$$X/gX =\frac{M/N}{g(M/N)}=\frac{M/N}{(gM+N)/N}\cong\frac{M}{gM+N}\cong\frac{M/gM}{(gM+N)/gM};$$
whilst, because $N\subseteq M$ and $M$ is $g$-divisible
$$ \ovl{M}/\ovl{N}=\frac{(M+Sg)/Sg}{(N+Sg)/Sg}=\frac{M/(M\cap Sg)}{(N+M\cap Sg)/M\cap Sg}=\frac{M/gM}{(gM+N)/gM}.$$
Thus we have
\begin{equation}\label{X/gX=ovlM/ovlN} X/gX\cong \ovl{M}/\ovl{N}.\end{equation}
\par
(1). If $\ovl{N}\ehd\ovl{M}$, then $\dim_\Bbbk X/gX<\infty$ by (\ref{X/gX=ovlM/ovlN}). It follows that the linear map $X_{n}\to X_{n+3}$: $x\mapsto gx$ is surjective for all $n\gg 0$. This forces $\dim_\Bbbk X_n\geq \dim_\Bbbk X_{n+3}$ for all $n\gg 0$ and thus $\GK(X)\leq 1$.\par
(2). If $N$ is $g$-divisible, then $X$ is $g$-torsionfree. It follows that if $\GK(X)\leq 1$, then $\GK(X/gX)\leq 1-1=0$ by \cite[Proposition~5.1(e)]{KL}. Hence $\ovl{N}\ehd\ovl{M}$ by (\ref{X/gX=ovlM/ovlN}). \par
(3). By (\ref{X/gX=ovlM/ovlN}), $\dim_\Bbbk X/Xg<\infty$. Hence $X_ng=X_{n+3}$ for all $n\geq n_0$ say. As $X$ is $g$-torsionfree, this shows $\dim_\Bbbk X_n=\dim_\Bbbk X_{n+3}$ for $n\geq n_0$. To see $X$ is finitely generated as a $\Bbbk[g]$-module notice that $X_{\leq (n_0+2)}\Bbbk[g]=X$.
\end{proof}

\subsection{Orders in $S$}\label{orders sec}

With the ultimate aim to classify maximal orders in $S$, we need a few preliminary results to do with orders inside $S$. Without a good reference for maximal orders as defined in Definition~\ref{max orders def} we need to prove some standard lemmas. We start with a very easy but useful lemma.\par

\begin{lemma}\label{WLOG a,b homogenous}
Let $A$ and $B$ be two graded domains in the graded quotient ring $Q=Q_\gr(A)=Q_\gr(B)$. Suppose $xAy\subseteq B$ for some nonzero $x,y\in Q$. Then there exists nonzero homogenous $a,b\in A$ also satisfying $aAb\subseteq B$.
\end{lemma}
\begin{proof}
Write $x=\sum x_i$ and $y=\sum y_{i}$ with $x_i,y_i$ homogenous of degree $i$. Choose $k,\ell$ minimal such that $x_k\neq 0\neq y_\ell$. Let $a\in A$ be homogenous, say $\deg a=n$, then
$$B\ni xay = x_kay_\ell +\text{terms of strictly higher degree.}$$
In particular, we must have $x_kay_\ell\in B_{k+n+\ell}\subseteq B$. Thus $x_kAy_\ell\subseteq B$ and we may assume $x$ and $y$ are homogenous. Since $x,y\in Q=Q_\gr(A)$ and are homogenous, there exists homogenous $a,b,c,d\in A$ such that $x=ac^{-1}$ and $y=d^{-1}b$. In which case, $aAb=ac^{-1}(cAd)d^{-1}b\subseteq ac^{-1}Ad^{-1}b=xAy\subseteq B.$
\end{proof}

Typically Lemma~\ref{WLOG a,b homogenous} will be used without reference. Next we develop a few standard techniques to conclude when two rings are in fact equivalent orders.

\begin{lemma}\label{hence equiv orders}
Let $A,B$ be two cg domains with $Q_\gr(A)=Q_\gr(B)=Q$.
\begin{enumerate}[(1)]
\item If $A\subseteq B$ and both $B_A$ and $_AB$ are finitely generated, then $A$ and $B$ share a common ideal.\index{common ideal} More precisely, there exists a nonzero ideal of $A$ which is also an ideal of $B$.
\item If $A\subseteq B$ and $B_A$ is finitely generated then $A$ and $B$ are equivalent orders.
\item If there exists a nonzero graded $(A,B)$-bimodule $M\subseteq Q$ that is finitely generated on both sides, then $A$ and $B$ are equivalent orders.
\end{enumerate}
\end{lemma}
\begin{proof}
(1). Set $I=\rann_A B/A$. By a left handed version of Lemma~\ref{Ore argument}(2) there exists a nonzero $c\in A$ such that $Bc\subseteq A$. In which case, $c\in I$, showing $I\neq 0$. Hence $I$ is a nonzero right ideal of $A$, and since $BI\subseteq A\subseteq B$, it is a left ideal of $B$. Similarly $J=\lann_A B/A$ is a nonzero ideal of $A$ that is a right ideal of $B$. Thus $IJ$ is an ideal of both $A$ and $B$. It is nonzero since $A$ is a domain. \par
(2). This follows immediately from Lemma~\ref{Ore argument}(2). \par
(3). Since $_AM$ is finitely generated there exists a nonzero homogenous $c\in A$ such that $Mc\subseteq A$. Given any nonzero $y\in M$, $yB\subseteq M$ and so $yBc\subseteq A$. A symmetric argument finds nonzero $b,z\in Q$ such that $bAz\subseteq B$.
\end{proof}

We now turn to the first significant result on orders inside $S$. It is the $S$-analogue of \cite[Proposition~2.16]{RSS}. Indeed for Proposition~\ref{RSS 2.16} we had no need to alter the proof given in \cite{RSS}.

\begin{lemma}\label{RSS 2.16=>}
Suppose that $U$ and $R$ are two $g$-divisible subalgebras of $\Sg$ with $Q_\gr(U)=Q_\gr(R)=Q_\gr(S)$. If $U$ and $R$ are equivalent orders, then $\ovl{U}$ and $\ovl{R}$ are equivalent orders in $Q_\gr(\ovl{U})=Q_\gr(\ovl{R})=Q_\gr(\ovl{S})$.
\end{lemma}
\begin{proof}
By Lemma~\ref{ovlR is order}, $Q_\gr(\ovl{U})=Q_\gr(\ovl{R})=Q_\gr(\ovl{S})$. Say we have $aRb\subseteq U$ with $a,b\in A$ nonzero and homogenous. Write $a=a'g^n$ with $n\geq 0$ and $a'\in U\setminus gU$. Then $g^na'Rb\subseteq U$, and hence $a'Rb\subseteq U$ because $U$ is $g$-divisible. Replacing $a$ by $a'$, we may assume $a\notin gU$. Similarly we can assume $b\notin gU$. Hence $\ovl{a},\ovl{b}\neq 0$ and $\ovl{a}\ovl{R}\ovl{b}\subseteq \ovl{U}$.
\end{proof}

\begin{prop}\label{RSS 2.16}
Let $U$ and $R$ two be $g$-divisible subalgebras of $\Sg$ such that $Q_\gr(U)=Q_\gr(R)=Q_\gr(S)$. Assume that $U\subseteq R$. Then $U$ and $R$ are equivalent orders if and only if $\ovl{U}$ and $\ovl{R}$ are equivalent orders.
\end{prop}
\begin{proof} ($\Rightarrow$). This follows immediately from Lemma~\ref{RSS 2.16=>}.\par
($\Leftarrow$). Suppose that $a\ovl{R}b\subseteq \ovl{U}$ with $a,b\in \ovl{U}$ nonzero and homogenous. We set $C=\ovl{U}+\ovl{R}b\ovl{U}$. It is easily seen that $C$ is a ring, satisfying $aC\subseteq \ovl{U}$ and $\ovl{R}b\subseteq C$. By Proposition~\ref{B just infinite}, $\ovl{R}$ and all subalgebras of $\ovl{R}$, in particular $\ovl{U}$ and $C$, are noetherian.  Thus $aC\subseteq\ovl{U}$ and $\ovl{R}b\subseteq C$ ensure   that $C_{\ovl{U}}$ and $_C\ovl{R}$ are finitely generated. \par

Now, since  $C_{\ovl{U}}$ is finitely generated, we can choose a finite dimensional subspace $X\subseteq R$ such that $\ovl{XU}=C$. There is no harm to assume $1\in X$. Put $M=\widehat{XU}$ and $V=\End_U(M)$. By Lemma~\ref{End is cg} $Q_\gr(V)=Q_\gr(U)=Q_\gr(S)$. Since $1\in M$, $M\not\subseteq g\Sg$, and so Lemma~\ref{RSS 2.12}(1) implies $V=\{x\in Q_\gr(U)\,|\, xM\subseteq M\}\subseteq \Sg$. By Lemma~\ref{RSS 2.12}(3), $V=\widehat{V}$ and $_VM$ is finitely generated. By Lemma~\ref{RSS 2.13}, $M_U$ is finitely generated. It follows from Lemma~\ref{hence equiv orders}(3) that $U$ and $V$ are equivalent orders.\par

Since $U\subset R$, $XU\subseteq R$, and hence $M=\widehat{XU}\subseteq \widehat{R}=R$ as $R$ is $g$-divisible. Because $1\in M$, $MR=R$. Hence $VR=VMR=MR=R$ and $V\subseteq R$. Since $_VM$ and $_C\ovl{R}$ are finitely generated, we can choose finite dimensional subspaces $Y,Z\subseteq R$ with $VY=M$ and $C\ovl{Z}=\ovl{R}$. We then have
$$\ovl{R}\supseteq \ovl{VYZ}\supseteq\ovl{MZ}\supseteq\ovl{XUZ}=C\ovl{Z}=\ovl{R},$$
forcing everything above to be an equality. In particular $\ovl{R}=\ovl{VYZ}$, proving that $_{\ovl{V}}{\ovl{R}}$ is finitely generated. By the Graded Nakayama's Lemma (Lemma~\ref{graded nakayama}) $\ovl{R}/\ovl{V_{\geq 1}R}$ is a finite dimensional vector space over $\ovl{V}/\ovl{V_{\geq 1}}=\Bbbk$. But since $g\in V_{\geq 1}\subseteq R_{\geq 1}$; $V_{\geq 1}\cap gV=gV$ and $V_{\geq 1}R\cap gR=gR$. It follows,
$$\Bbbk\cong \ovl{V}/\ovl{V_{\geq 1}}=\frac{V/gV}{V_{\geq 1}/(V_{\geq 1}\cap gV)} \cong V/V_{\geq 1},$$
and
$$\ovl{R}/\ovl{V_{\geq 1}R}=\frac{R/gR}{V_{\geq 1}R/(V_{\geq 1}R\cap gR)}\cong R/V_{\geq1}R.$$
Thus $R/V_{\geq 1}R$ is a finite dimensional $V/V_{\geq 1}$-vector space. The Graded Nakayama Lemma then applies again to show $_VR$ is finitely generated. By Lemma~\ref{hence equiv orders}(2), $R$ and $V$ are equivalent orders. Since $V$ and $U$ are equivalent orders from earlier, we get that $U$ and $R$ are equivalent orders.
\end{proof}

We finish this section by further studying orders in $S$ and in $S^{(3)}$. We make an extension of  Notation~\ref{3 Veronese notation bg}.

\begin{notation}\label{3 Veronese notation2}\index[n]{emt@$(E,\MM,\tau)$} \index[n]{t@$T$} \index[n]{tau@$\tau$} \index[n]{tg@$T_{(g)}$}
Let $T=S^{(3)}$, $T_{(g)}=(\Sg)^{(3)}$, $\MM=\LL_3$ and $\tau=\sigma^3$. Whence we have
\begin{equation*} T/gT=B(E,\MM,\tau)\; \text{ and } \;T_{(g)}/gT_{(g)}=\Bbbk(E)[t,t^{-1};\tau].\end{equation*}
\end{notation}

In \cite[Notation~2.5]{RSS} the authors defined $T_{(g)}$ as the ring obtained from $T$ by inverting all homogenous elements of $T\setminus gT$. Clearly for this definition $T_{(g)}\subseteq (\Sg)^{(3)}$.  Conversely, if $s=ax^{-1}\in (\Sg)^{(3)}$ for $a,x\in S$ homogenous and $x\not\in gS$, then $\deg a-\deg x$ is divisible by 3. Say $\deg(a)\equiv \deg(x)\equiv \alpha \mod 3$ for $0\leq\alpha\leq 2$. Choosing homogenous $z\in S_{3-\alpha}\setminus gS$, we have $s=(az)(xz)^{-1}$, where $az,xz\in T$. As $gS$ is completely prime, $xz\notin gT$. This shows $(\Sg)^{(3)}$ indeed equals $T_{(g)}$ as defined in \cite[Notation~2.5]{RSS}. A similar proof shows
\begin{equation}\label{Dgr(S)=Dgr(T)} D_\gr(S)=D_\gr(T)\;\text{ and }Q_\gr(S)^{(3)}=Q_\gr(T).\end{equation}
\par

Recall that for a cg domain $A$ we can write $Q_\gr(A)=D_\gr(A)[t,t^{-1};\alpha]$ as in Definition~\ref{Qgr and Dgr}. We warn that below, in Lemma~\ref{Qgr(S) or Qgr(T)} and Lemma~\ref{max Sg-order is max order}, we give $T$, $T_{(g)}$ and $Q_\gr(T)$ the grading induced from $Q_\gr(S)$; so for example $T_n=T\cap S_n$. This is in contrast with Lemma~\ref{U'<g>} and Proposition~\ref{g-div max orders up n down} later where $T$ is regraded to have $T_n=S_{3n}$.

\begin{lemma}\label{Qgr(S) or Qgr(T)}
Let $U\subseteq Q_\gr(S)$ be a subalgebra containing $g$ and such that $D_\gr(U)=D_\gr(S)$. Then either $Q_\gr(U)=Q_\gr(S)$ or $Q_\gr(U)=Q_\gr(T)$.
\end{lemma}
\begin{proof}
Since $g\in U$, certainly $D_\gr(S)[g,g^{-1}]\subseteq Q_\gr(U)\subseteq Q_\gr(S)$. Furthermore, $D_\gr(S)[g,g^{-1}]=Q_\gr(S)^{(3)}=Q_\gr(T)$ by (\ref{Dgr(S)=Dgr(T)}). Thus $Q_\gr(T)\subseteq Q_\gr(U)\subseteq Q_\gr(S)$. If $Q_\gr(U)\supsetneq Q_\gr(T)$, then we can pick homogenous $x\in Q_\gr(U)$, with $\deg(x)$ coprime to 3. In which case there exists $k,\ell\in \N$ such that $x^kg^{-\ell}\in Q_\gr(U)_1$, which forces $Q_\gr(U)=Q_\gr(S)$.
\end{proof}

The next proposition shows maximal $\Sg$-orders are in fact maximal orders. Recall from Notation~\ref{generated in degree 1} that for a cg subalgebra $A$ of $Q_\gr(S)$, the subalgebra of $Q_\gr(S)$ generated by $A$ and $g$ is denoted by $A\langle g\rangle=A+Ag+Ag^2+\dots$\index[n]{ag@$A\langle g\rangle$}.

\begin{lemma}\label{max Sg-order is max order}
Let $U$ be a cg subalgebra of $\Sg$ with $D_\gr(U)=D_\gr(S)$. If $U\subseteq A$ are equivalent orders for some cg subalgebra $A$ of $Q_\gr(S)$, then $A\subseteq \Sg$.
\end{lemma}
\begin{proof}
For a contradiction assume that $A\not\subseteq \Sg$. Say that $xAy\subseteq U$ for nonzero $x,y\in Q_\gr(S)$. By Lemma~\ref{WLOG a,b homogenous} we may assume that $x,y\in  U$ (in particular $x,y\in\Sg$) and are homogenous. Set $C=A\langle g\rangle$ and $V=U\langle g\rangle$; clearly $V\subseteq C$ whilst $C\not\subseteq \Sg$. By Lemma~\ref{Qgr(S) or Qgr(T)}, there are two cases:
\begin{enumerate}[(1)]
\item $Q_\gr(C)=Q_\gr(V)=Q_\gr(S)$;
\item $Q_\gr(C)=Q_\gr(V)=Q_\gr(T)$.
\end{enumerate}\par
(1). Take $c\in C\setminus \Sg$. Then certainly $c\in Q_\gr(S)$. Now note that $Q_\gr(S)$ equals $\Sg$ localised at the Ore set $\{1,g^{-1},g^{-2},\dots \}$. So, cancelling $g$'s if necessary, we can write $c=sg^{-n}$ for some $s\in\Sg\setminus g\Sg$ and $n\geq 1$. Now write $x=x_1g^k$ and $y=y_1g^\ell$ where $x_1,y_1\in \Sg\setminus g\Sg$ and $k,\ell\geq 0$. Choose an integer $m$ such that $nm>k+\ell$ and consider $xc^my$. On the one hand $c^m\in C$, and therefore $xc^my\in xCy\subseteq V\subseteq \Sg$, while on the other $xc^my=x_1s^my_1g^{k+\ell-nm}$. Now $x_1,y_1\notin g\Sg$, while also $s^m\notin g\Sg$ as $g\Sg$ is a completely prime ideal of $\Sg$, Thus all of $x_1^{-1},y_1^{-1},s^{-m}\in \Sg$. But then.
$$ g^{k+\ell-nm}=y_1^{-1}s^{-m}x_1^{-1}xc^my \in\Sg.$$
Since $k+\ell-nm<0$, this is our desired contradiction. \par
(2) follows in the same way with $T$ in place of $S$.
\end{proof}

The importance of Lemma~\ref{max Sg-order is max order} is that it allows us to concentrate on rings inside $\Sg$ where we are more comfortable.

\subsection{Veronese rings of $g$-divisible algebras}\label{Veronese sec}

To help prove many results in $S$, we will want to utilise the results of \cite{Ro} and \cite{RSS} as much as possible. A clear strategy is therefore to pass information between Veronese rings. We develop some basic machinery for this here.

\begin{lemma}\label{noeth up n down}
Let $A$ be a cg domain. Fix $d\geq 1$ and write $A'=A^{(d)}$.
\begin{enumerate}[(1)]
\item Then $A$ is (strongly) noetherian if and only if $A'$ is (strongly) noetherian.
\item Suppose that $A$ is noetherian and let $M\subseteq Q_\gr(A)$ be a finitely generated right $A$-module. Then as a right $A'$-module
$$M=\bigoplus_{k=0}^{d-1}M^{(k\;\mathrm{mod}\, d)}\index[n]{Mkmodd@$M^{(k\;\mathrm{mod}\, d)}$} $$
where $M^{(k\;\mathrm{mod}\, d)}=\bigoplus_{i}M_{di+k}$ are finitely generated right $A'$-modules. In particular $M^{(d)}$, $M$ and $A$ are all finitely generated as right $A'$-modules.
\end{enumerate}
\end{lemma}
\begin{proof} (1). The noetherian part is Lemma~\ref{AZ 5.10}. The strongly noetherian statement is \cite[Proposition~4.9(2)(3)]{ASZ}.\par
(2). Write $Q=Q_\gr(A)$. Fix $k=0,1,2,\dots,d-1$ and write $M'=M^{(k\;\mathrm{mod}\,d)}$. If $M=0$ there is nothing to prove so assume $M\neq 0$. By Lemma~\ref{choose x of appropriate deg} both $A'\neq 0$ and $M'\neq 0$. Since $M_A$ is finitely generated there exists a homogenous $x\in A$ such that $xM\subseteq A$ by Lemma~\ref{Ore argument}(2). Using Lemma~\ref{choose x of appropriate deg} we can multiply $x$ on the left with an element of $A$ with appropriate degree so that we may assume $x\in A'$. In which case $xM'\subseteq A\cap Q'=A'$, in particular $M'$ embeds into $A'$ as a right ideal. Since $A$ is noetherian, $A'$ is also noetherian by part (1). Hence $M'_{A'}$ is finitely generated. The module $M$ is then a direct sum of finitely generated modules, hence itself is finitely generated. To see $A_{A'}$ is finitely generated just take $M=A$ in the above.
\end{proof}

\begin{lemma}\label{equiv orders go up}
Let $A$ and $B$ be two cg domains with $Q_\gr(A)=Q_\gr(B)=Q$. Assume that $Q_1\neq 0$ and let $d\geq 1$. If $A$ and $B$ are equivalent orders, then $A^{(d)}$ and $B^{(d)}$ are also equivalent orders with $Q_\gr(A^{(d)})=Q_\gr(B^{(d)})=Q^{(d)}$
\end{lemma}
\begin{proof}
The fact that $Q_\gr(A^{(d)})=Q_\gr(B^{(d)})=Q^{(d)}$ follows as in (\ref{Dgr(S)=Dgr(T)}). Assume that $aAb\subseteq B$, with $a,b\in A$ nonzero and homogenous. Say $\deg(a)=\alpha$ and $\deg(b)=\beta$. By Lemma~\ref{choose x of appropriate deg} we can find nonzero homogenous elements $u,v\in B$ of degrees $k$ and $\ell$ respectively satisfying $k\equiv-\alpha\mod d$ and $\ell\equiv-\beta\mod d$. Then $ua,bv\in Q^{(d)}$, and still satisfy $uaAbv\subseteq B$. It then follows that $uaA^{(d)}bv\subseteq B^{(d)}$. In a symmetric fashion we can also find nonzero $a',b'\in Q^{(d)}$ such that $a'B^{(d)}b'\subseteq A^{(d)}$.
\end{proof}

Since $g\in S_3$, the property of $g$-divisibility also makes sense in the 3-Veronese $T_{(g)}=(\Sg)^{(3)}$. A trivial but useful result is the following.

\begin{lemma}\label{g div up}
Let $R$ be a cg subalgebra of $\Sg$ containing $g$ and let $M$ be a graded right $R$-submodule of $\Sg$. If $M$ is $g$-divisible then $M^{(3)}$ is also $g$-divisible. \qed
\end{lemma}

The next lemma is key to utilising the results of \cite{Ro} and \cite{RSS} fully. It allows us to build appropriate ideals in $S$ out of ideals from $S^{(3)}$ and vice versa. \par
Before Lemma~\ref{sporadics up n down} we recall Notation~\ref{GK dimension} for GK-dimension.

\begin{lemma}\label{sporadics up n down}
Let $A$ be a cg noetherian domain of finite GK-dimension. Fix $d\in \N$ and let $A'=A^{(d)}$.
\begin{enumerate}[(1)]
\item  For every homogenous ideal $I$ of $A$, the homogenous ideal $I'=I^{(d)}$ of $A'$ satisfies $\GK(A'/I')= \GK(A/I)$.
\item Let $J$ be a homogenous ideal of $A'$.
\begin{enumerate}[(a)]
\item The ideal $I=AJA$ of $A$ satisfies $\GK(A/I)\leq \GK(A'/J)$. Furthermore, $J\subseteq I^{(d)}$ holds.
\item  The ideal $K=\mathrm{r.ann}_A(A/JA)$ of $A$ also satisfies $\GK(A/K)\leq\GK(A'/J)$. In this case $K^{(d)}\subseteq J$.
\end{enumerate}
\end{enumerate}
\end{lemma}
\begin{proof}
In this proof $A'$ is graded as a subset of $A$, i.e. $A'_n=A_n\cap S'$. We also remark that by Lemma~\ref{noeth up n down} we have that $A'$ is noetherian, and that $A$ is finitely generated on both sides as a right and left $A'$-module. Finally, all modules below will always be finitely generated; in particular the GK-dimension of a module is independent of the ring it is being considered over by (\ref{GK(M)}). \par 

(1). One can easily check $A'/I' \cong (A/I)^{(d)}$. Since $A$ is finitely generated on both sides as an $A'$-module, we have that $A/I$ is finitely generated on both sides as an $A'/I'$-module. Thus by \cite[Proposition~5.5]{KL}, $\GK(A'/I')=\GK(A/I)$.\par

(2a). There is an isomorphism of right $A'$-modules $A/JA \cong A'/J \otimes_{A'}A$. Therefore, since $A$ is finitely generated as an $A'$-module on both sides, we can apply \cite[Proposition~5.6]{KL} and deduce $\GK(A/JA)\leq \GK(A'/J)$ as right $A'$-modules.  Clearly $JA\subseteq AJA=I$, and so there is a surjection $A/JA\to A/I$. Hence by \cite[Proposition~5.1(b)]{KL} we have
\begin{equation}\label{GK(A/J)}\GK(A'/J)\geq \GK(A/JA)\geq \GK(A/I)\end{equation}
as right $A'$-modules. It is obvious that $J\subseteq I^{(d)}$ holds. \par

(2b). To deal with the ideal $K=\mathrm{r.ann}_A(A/JA)$ we first need
\begin{equation}\label{GK(A/JA)=GK(A'/J)}\GK(A/JA)=\GK(A'/J).\end{equation}
We already have $\GK(A/JA)\leq\GK(A'/J)$ from (\ref{GK(A/J)}). For the reverse inequality note that because $(JA)^{(d)}=J$; $A'/J \cong (A/JA)^{(d)}\hookrightarrow A/JA$ as right $A'$-modules. So $\GK(A'/J)\leq \GK(A/JA)$ follows from \cite[Proposition~5.1(b)]{KL}. Hence (\ref{GK(A/JA)=GK(A'/J)}) holds.\par

Now $_{A'}A$ is finitely generated, therefore so is the left $A'$-module $A/JA$, say that
$A/JA=A'\ovl{x_1}+\dots+ A'\ovl{x_n}$ where $\ovl{x_i}=x_i+JA$ for some homogenous $x_i\in A$.
It follows that $K=\cap_{i=1}^n K_i$ where $K_i=\rann_A(\ovl{x_i})$ are right ideals of $A$. Fix $i$, and say $\deg(x_i)\equiv \alpha_i\mod d$. By Lemma~\ref{choose x of appropriate deg} we can choose a nonzero $a_i\in A$ such that $\deg(a_i)\equiv -\alpha_i\mod d$. Then $x_ia_iJ\subseteq A'J=J\subseteq JA$ which shows $K_i\neq0$. Because $K$ is an intersection of finitely many nonzero right ideals in a noetherian domain, $K$ must also be nonzero (see \cite[Exercise~4N]{GW}). For every $i$, clearly $(\ovl{x_i}A)_{A}\cong [A/K_i]_{A}$ and $\ovl{x_i}A\subseteq A/JA$, thus
$$\GK_{A}(A/K_i)\leq \GK(A/JA)=\GK(A'/J)$$
by \cite[Proposition~ 5.1(b)]{KL} and (\ref{GK(A/JA)=GK(A'/J)}). But $A/K$ embeds into $A/K_1 \oplus \dots \oplus A/K_n$ in the obvious way, forcing $\GK(A/K)\leq \max_i\{\GK(A/K_i)\}\leq \GK(A'/J).$
Here we are using \cite[Proposition~5.1(a)(b)]{KL}. For the last line,
\begin{equation*}K^{(d)}= K\cap A'\subseteq JA\cap A'=J.\qedhere\end{equation*}
\end{proof}

\begin{remark}\label{sporadics up n down applies}
Let $R$ be a cg $g$-divisible subalgebra of $S_{(g)}$ with $Q_\gr(R)=Q_\gr(S)$. By Proposition~\ref{RSS 2.9}, $R$ is noetherian, and so Lemma~\ref{sporadics up n down} will apply in this situation.
\end{remark}


Our major result of this section partially answers \cite[Question~9.4]{RSS}: we will prove that a $g$-divisible subalgebra $U$ is a maximal order if and only its $d$th Veronese subring $U^{(d)}$ is a maximal order. Our results will be applicable in both $S$ and its 3-Veronese $S^{(3)}=T$.

\begin{lemma}\label{U'<g>}
Let $d\geq 1$ and $U$ be a $g$-divisible cg subalgebra of $\Sg$ such that $D_\gr(U)=D_\gr(S)$.
\begin{enumerate}[(1)]
\item Suppose that $Q_\gr(U)=Q_\gr(S)$ and $d$ is coprime to 3. Set $U'=U^{(d)}$, then $\widehat{U'\langle g\rangle}=U$.
\item Suppose that $Q_\gr(U)=Q_\gr(T)$. Set $U'=U^{(d)}$, then $\widehat{U'\langle g\rangle}=U$.
\end{enumerate}
\end{lemma}

\begin{proof}
(1). Put $V=\widehat{U'\langle g\rangle}$. Since $U$ is $g$-divisible, $g\in U$, and so $V\subseteq U$. Conversely let $u\in U$. Since $d$ is coprime to $3=\deg(g)$, there exists an integer $n\geq 0$ such that $g^nu\in U\cap S^{(d)}=U'$. In which case $u\in \widehat{U'\langle g\rangle}=V$.\par
(2). This is the same as (1) but with the phrase ``$d$ is coprime to $3$" replaced with the phrase ``$\deg(g)=1$".
\end{proof}

\begin{prop}\label{g-div max orders up n down}
Let $d\geq 1$ and let $U$ be a $g$-divisible cg subalgebra of $\Sg$ with $D_\gr(U)=D_\gr(S)$.
\begin{enumerate}[(1)]
\item  Suppose that $Q_\gr(U)=Q_\gr(S)$ and $d$ is coprime to 3, and let $U'=U^{(d)}$. Then:
\begin{enumerate}[(a)]
\item $U$ is a maximal order if and only if $U'$ is a maximal order;
\item $U$ is a maximal $S$-order if and only if $U'$ is a maximal $S^{(d)}$-order.
\end{enumerate}
\item Suppose that $Q_\gr(U)=Q_\gr(T)$, and let $U'=U^{(d)}$. Then:
\begin{enumerate}[(a)]
\item $U$ is a maximal order if and only if $U'$ is a maximal order;
\item $U$ is a maximal $T$-order if and only if $U'$ is a maximal $T^{(d)}$-order.
\end{enumerate}
\end{enumerate}
\end{prop}

\begin{proof}
($1a$) ($\Rightarrow$) Suppose that $U$ is a maximal order. Assume that $U'\subseteq A$ for some graded subalgebra $A$ of $Q_\gr(S)$ such that $xAy\subseteq U'$ for some nonzero homogenous $x,y\in Q_\gr(S)^{(d)}$. By Lemma~\ref{WLOG a,b homogenous} we may assume $x,y\in A$,
\begin{equation}\label{A subset Sg}
\text{while by Lemma~\ref{max Sg-order is max order} } A\subseteq \Sg.
\end{equation}
Set $C=\widehat{A\langle g\rangle}$. Since $xAy\subseteq U'\subseteq U$ and $g\in U$, $x(Ag^i)y=(xAy)g^i\subseteq U$ for all $i\geq 0$. Thus $xA\langle g\rangle y \subseteq U$. Now let $c\in C$, then $cg^n \in A\langle g\rangle$ for some $n\geq 0$, and hence $x(cg^n)y=(xcy)g^n\in U$. Since $x,y\in A\subseteq \Sg$, $xcy\in\Sg$. Hence we get $xcy\in U$ because $U$ is $g$-divisible. This shows $xCy\subseteq U$. On the other hand, because $A\supseteq U'$, $C=\widehat{A\langle g\rangle}\supseteq \widehat{U'\langle g\rangle}=U$ by Lemma~\ref{U'<g>} . By assumption $U$ is a maximal order and hence $C=U$. It then follows
$$A\subseteq C\cap (\Sg)^{(d)}=U\cap (\Sg)^{(d)}=U',$$
and so $A=U'$.\par

($\Leftarrow$) Suppose now that $U'$ is a maximal order, and assume that $U\subseteq A$ are equivalent orders, for some graded subalgebra $A$ of $Q_\gr(S)$. By Lemma~\ref{max Sg-order is max order}, we know that in fact $A\subseteq \Sg$. By Lemma~\ref{equiv orders go up}, $U'\subseteq A'=A^{(d)}$ are equivalent orders, and so by hypothesis $U'=A'$. By Lemma~\ref{U'<g>}, $\widehat{A'\langle g\rangle}=\widehat{U'\langle g\rangle}= U$, and so it is enough to prove $A\subseteq \widehat{A'\langle g\rangle}$. Let $a\in A$. Note that since $U\subseteq A$, $g\in A$. Because $d$ and $3=\deg(g)$ are coprime there exists $n\geq 0$ such that $ag^n\in A\cap S^{(d)}=A'$. Thus $a\in \widehat{A'\langle g\rangle}$. \par

($1b$). This is proved in the same way as (1a). Here one replaces uses of Lemma~\ref{max Sg-order is max order} (for example (\ref{A subset Sg})) with ``$A\subseteq S$ by assumption". Then replace all instances of $\Sg$ with $S$.

(2). This follows in the same as part (1) with ``$d$ is coprime to $3$" replaced with ``$\deg(g)=1$".
\end{proof}

\section{The noncommutative blowups $S(\bfd)$}\label{The rings S(d)}

This chapter is dedicated to finding what subalgebras of $S$ should be analogous to the blowup subalgebras $T(\bfd)$ as defined in Definition~\ref{T(d) def}. In other words, what are the blowup subalgebras of $S$? These are the rings $S(\bfd)$ which were defined in Definition~\ref{S(d) def}. We will prove that these rings are very closely related to the $T(\bfd)$ and satisfy similar abstract properties. The work in this chapter is absolutely crucial to be able to apply techniques from \cite{Ro, RSS, RSS2} and obtain our main classification of maximal $S$-orders. \par

Our noncommutative blowup subalgebras, the $S(\bfd)$'s, can be defined for an effective divisor on $E$ of degree at most 2. We give a brief explanation of why this should be true, for this reason was the main factor that lead to Definition~\ref{S(d) def}. In short it is that we would expect the 3-Veroneses of our blowups to coincide with those blowups of $T$ from Definition~\ref{T(d) def}. More precisely, if $S(\bfd)$ is one of our blowups we would expect
\begin{equation}\label{S(d) to T([d]3)} S(\bfd)^{(3)}=T(\bfd+\sigma^{-1}(\bfd)+\sigma^{-2}(\bfd)) \end{equation}
to hold. With this as a clue, we realise we should be ``blowing up" at effective divisors $\bfd$ with $\deg\bfd\leq 2$. This being because blowups of $T=S^{(3)}$ are at effective divisors of degree at most 8. So we must have $\deg(\bfd+\sigma^{-1}(\bfd)+\sigma^{-2}(\bfd))\leq 8$ for (\ref{S(d) to T([d]3)}) to make sense.\par

When $\deg\bfd=1$, the ring $S(\bfd)$ was first studied in \cite[Section~12]{Ro}, where the author proves it to be the only degree 1 generated maximal order in $S$. We recall the definition here.

\begin{definition}\label{S(p)}\index[n]{sp@$S(p)$}
Let $p\in E$. We define the \textit{blowup of $S$ at $p$}\index{blowup of $S$ at $p$} as the ring
$$S(p)=\Bbbk \langle V\rangle \; \text{ where } V=\{x\in S_1\,|\; \ovl{x}\in H^0(E,\LL(-p))\}.$$
\end{definition}

Rogalski proves strong results on $S(p)$.

\begin{prop}\label{Ro 12.2} \cite[Theorem~12.2]{Ro}
Let $p\in E$ and $R=S(p)$. Then:
\begin{enumerate}[(1)]
\item $R^{(3)}=S^{(3)}(p+p^\sigma+p^{\sigma^2})$, $R$ is $g$-divisible with
 $R/gR\cong \ovl{R}=B(E,\LL(-p),\sigma),$
 and has Hilbert series $h_{R}(t)=\frac{t^2+1}{(1-t)^2(1-t^3)}$.
\item $R$ is strongly noetherian. The ring $R$ satisfies $\chi$ on the left and right, has cohomological dimension 2 and (by \cite[Proposition~2.4]{RSS2}) $R$ has a balanced dualizing complex.
\item $R$ is Auslander-Gorenstein and Cohen-Macaulay.
\item $R$ is a maximal order in $Q_\gr(R)=Q_\gr(S)$.\qed
\end{enumerate}
\end{prop}

Proposition~\ref{Ro 12.2}(1) is enough to show that Definition~\ref{S(p)} coincides the definition of $S(p)$ given in Definition~\ref{S(d) def}. We prefer Definition~\ref{S(p)} as it emphasises that $S(p)$ is generated in degree 1.

\subsection{The two point blowup $S(p+q)$}

When $\bfd=p+q$ for some $p,q\in E$, understanding $S(\bfd)$ becomes much harder. The main hinderance is that $S(p+q)$ is no longer generated in a single degree. We introduce some notation which will help us along the way.

\begin{notation}\label{p^sigma^j}
Let $\rho:E\to E$ be an automorphism.  Given $p\in E$, or more generally a divisor $\bfx$ of $E$ we will write $p^{\rho^j}=\rho^{-j}(p)$\index[n]{psj@$p^{\sigma^j}$} and $\bfx^{\rho^j}=\rho^{-j}(\bfx)$ \index[n]{xsj@$\bfx^{\sigma^j}$} for $j\in\Z$.
\end{notation}

Notation~\ref{p^sigma^j} will be typically applied with $\rho=\sigma$ or $\rho=\sigma^3=\tau$. In contrast Notation~\ref{[d]_n} is reserved for the automorphism $\sigma$ unless explicitly said otherwise.

\begin{notation}\label{[d]_n}
For $n\geq 1$, we put $[\bfx]_n=\bfx+\bfx^\sigma+\dots+\bfx^{\sigma^{n-1}}$, while we define $[\bfx]_0=0$. \index[n]{xn@$[\bfx]_n$}
\end{notation}

Notation~\ref{p^sigma^j} and Notation~\ref{[d]_n}, along with Notation~\ref{ovlX notation} from the introduction, will be applicable for the rest of this thesis. Using Notation~\ref{p^sigma^j} and Notation~\ref{[d]_n} we recall the definition of $S(p+q)$.

\begin{definition}\label{S(p+q)}
Let $p,q\in E$. Put $$V_i=\{x\in S_i\,|\; \ovl{x}\in H^0(E,\LL_i(-[p+q]_i))\}\;\text{ for } i=1,2,3.$$
We define the \textit{blowup of $S$ at $p+q$}\index{blowup of $S$ at $p+q$} to be the ring $S(p+q)=\Bbbk\langle V_1, V_2, V_3\rangle.$\index[n]{spq@$S(p+q)$}
\end{definition}

Let $p,q\in E$ and $S(p+q)=\Bbbk\langle V_1,V_2,V_3\rangle$ as in Definition~\ref{S(p+q)}. It is clear that $\ovl{V_1}=H^0(E,\LL(-p-q))$, and because $gS\subseteq S_{\geq 3}$, that $\dim_\Bbbk V_1=\dim_\Bbbk \ovl{V_1}$. Hence by the Riemann-Roch Theorem (specifically Corollary~\ref{RR to E}), $\dim_\Bbbk V_1=1$. In which case $\Bbbk\langle V_1\rangle \cong \Bbbk[t]$. In particular one sees $S(p+q)$ cannot be generated in degree~1. In Example~\ref{S(p+p1)} we show that, at least for some choices of $p$ and $q$, $S(p+q)$ is not generated in degrees 1 and 2. For general $p$ and $q$, it is unknown if $S(p+q)$ can be generated in degrees 1 and 2. We include $V_3$ in Definition~\ref{S(p+q)} to be sure that $g\in S(p+q)$. It is implicit in our proof of Theorem~\ref{S(p+q) g div} that $S(p+q)=\Bbbk \langle V_1,V_2,g\rangle$.  \par

A lemma we will often use is the following simple geometric result which comes from the definition of a Cartier divisor.

\begin{lemma}\label{on divisors}
Let $\mbf{a},\mathbf{b}$ be divisors on $E$. Consider the invertible sheaves $\OO_E(\mbf{a})$ and $\OO_E(\mbf{b})$ corresponding to $\mbf{a}$ and $\mbf{b}$ via \cite[Proposition~6.13]{Ha}. Then,
\begin{equation*}\pushQED{\qed}
\OO_E(\mbf{a})\subseteq \OO_E(\mbf{b}) \Leftrightarrow \mbf{a}\leq\mbf{b}.\qedhere
\end{equation*}
\end{lemma}

Let $p,q\in E$, and $V_1$, $V_2$ and $V_3$ be as in Definition~\ref{S(p+q)}. Set $\NN=\LL(-p-q)$. One has $$\NN_2=\NN\otimes \NN^\sigma=\LL(-p-q)\otimes \LL^\sigma(-p^\sigma-q^\sigma)=\LL_2(-p-q-p^\sigma-q^\sigma)=\LL_2(-[p+q]_2).$$
Similarly $\NN_3=\LL_3(-[p+q]_3)$, and in general
\begin{equation}\label{NN_n=LL_n(-[d]_n)} \NN_n=\LL_n(-[p+q]_n)\;\text{ for all }n\geq 1.\end{equation}
 It is then clear from Definition~\ref{S(p+q)} that
\begin{equation}\label{ovlVi=}
\ovl{V_i}=H^0(E,\NN_i)=B(E,\NN,\sigma)_i\;\text{ for } i=1,2,3.
\end{equation}
\par An immediate consequence of Definition~\ref{S(p+q)} and the above is that the image of $S(p+q)$ inside $S/gS=B(E,\LL,\sigma)$ is easy to understand.

\begin{lemma}\label{S(p+q) bar}
Let $p,q\in E$ and $R=S(p+q)$. Then $\ovl{R}=B(E,\NN,\sigma)\subseteq \ovl{S}$ where $\NN=\LL(-p-q)$. The  Hilbert series of $\ovl{R}$ is given by
$$h_{\ovl{R}}(t)=\frac{t^2-t+1}{(t-1)^2}.$$
\end{lemma}
\begin{proof}
Set $B=B(E,\NN,\sigma)$. By Lemma~\ref{on divisors} and (\ref{NN_n=LL_n(-[d]_n)}) we have that $\NN_n\subseteq \LL_n$ for all $n\geq 0$. Hence $H^0(E,\NN_n)\subseteq H^0(E,\LL_n)$ for all $n\geq 1$, which shows that $B\subseteq B(E,\LL,\sigma)=\ovl{S}$. By definition, $\ovl{R}$ is the subring of $\ovl{S}$ generated by $\ovl{V_1}$, $\ovl{V_2}$ and $\ovl{V_3}$; while from (\ref{ovlVi=}) we have that $\ovl{V_1}= B_1$, $\ovl{V_2}=B_2$ and $\ovl{V_3}=B_3$. But Theorem~\ref{B properties}(4) says that $B$ is generated in degrees 1 and 2. Thus $\ovl{R}=B$. \par
For the Hilbert series, we note that $\ovl{R}_n=H^0(E,\NN_n)$ for $n\geq 0$. So by the Riemann-Roch Theorem (specifically Corollary~\ref{RR to E}), $\dim_\Bbbk\ovl{R}_n=\deg\NN_n=n$ for $n\geq 1$, while $\dim_\Bbbk \ovl{R}_0=1$.
\end{proof}

As sort of a disclaimer, we warn that future uses of Lemma~\ref{on divisors} and Lemma~\ref{S(p+q) bar} may come without reference. The next result we state is the key result which allows us to understand $S(p+q)$. The majority of this chapter will be dedicated to proving it.

\begin{theorem}\label{S(p+q) g div}
Let $p,q\in E$. Then
\begin{enumerate}[(1)]
\item $S(p+q)$ is $g$-divisible with
\begin{equation}\label{S(p+q)/gS(p+q)}
S(p+q)/gS(p+q)\cong \ovl{S(p+q)}=B(E,\LL(-p-q),\sigma).
\end{equation}
\item The Hilbert series of $S(p+q)$ is given by
\begin{equation}\label{S(p+q) hilbert series} h_{S(p+q)}(t)=\frac{t^2-t+1}{(t-1)^2(1-t^3)}.\end{equation}
\end{enumerate}
\end{theorem}

The statement and proof of Theorem~\ref{S(p+q) g div} should be compared with those of Theorem~\ref{T(d) properties} (in particular the sketch proof provided there). Here we are made to work harder to compensate for $S(p+q)$ not being generated in degree one.

\begin{lemma}\label{1 iff 2 of S(p+q) g div}
Parts (1) and (2) of Theorem~\ref{S(p+q) g div} are equivalent.
\end{lemma}
\begin{proof}
Fix $p,q\in E$ and let $R=S(p+q)$. Using $Rg\subseteq R\cap Sg$ and that $R$ is a domain, we have
\begin{equation}\label{dim Rn geq}\dim_\Bbbk R_n=\dim_\Bbbk \ovl{R_n}+\dim_\Bbbk (R_n\cap Sg)\geq \dim_\Bbbk\ovl{R_n}+\dim_\Bbbk gR_{n-3}=\dim_\Bbbk\ovl{R_n}+\dim_\Bbbk R_{n-3}.\end{equation}
We know the Hilbert series of $\ovl{R}$ from Lemma~\ref{S(p+q) bar}. This and (\ref{dim Rn geq}) gives
\begin{equation}\label{hilbert series ineq} h_R(t)\geq \frac{h_{\ovl{R}}(t)}{1-t^3}=\frac{t^2-t+1}{(t-1)^2(t^3-1)}.\end{equation}
Now, this is an equality if and only if there is equality in (\ref{dim Rn geq}). Since $Rg\subseteq R\cap Sg$, this is true if and only if $Rg=R\cap Sg$; in other words, if and only if $R$ is $g$-divisible. The isomorphism (\ref{S(p+q)/gS(p+q)}) is a consequence of $R$ being $g$-divisible and Lemma~\ref{S(p+q) bar}.
\end{proof}

\begin{lemma}\label{S(p+q)_i=V_i}
Let $p,q\in E$ and let $V_1$, $V_2$ and $V_3$ be as in Definition~\ref{S(p+q)}. Then
\begin{enumerate}[(1)]
\item $\dim_\Bbbk V_1=1$, $\dim_\Bbbk V_2=2$ and $\dim_\Bbbk V_3=4$;
\item $S(p+q)_2=V_2$ and $S(p+q)_3=V_3$.
\end{enumerate}
\end{lemma}
\begin{proof}
(1). From the definition we have $\ovl{V_i}=H^0(E,\LL_i(-[p+q]_i)$ for $i=1,2,3$. Hence by the Riemann-Roch (specifically Corollary~\ref{RR to E}), $\dim_\Bbbk \ovl{V_i}=i$. Since $gS\subseteq S_{\geq 3}$, we then have, for $i=1,2$, that $\dim_\Bbbk V_i=\dim_\Bbbk\ovl{V_i}=i$. Again from the definition $V_3\cap gS=S\cap gS=g\Bbbk$. Thus
$$\dim_\Bbbk V_3=\dim_\Bbbk\ovl{V_3}+\dim_\Bbbk(V_3\cap gS)=3+1=4.$$\par
(2). Here we must show $V_1^2\subseteq V_2$ and $V_1^3+V_1V_2+V_2V_1\subseteq V_3$. By Lemma~\ref{S(p+q) bar} we have that $\ovl{V_1^2}\subseteq \ovl{V_2}$. Since $gS\subseteq S_{\geq 3}$, this implies $V_1^2\subseteq V_2$. Set $W=V_1^3+V_1V_2+V_2V_1$. Again by Lemma~\ref{S(p+q) bar}, $\ovl{W}\subseteq H^0(E,\LL_3(-[p+q]_3))$. Then
\begin{equation*}W\subseteq  \{x\in S_3\,|\;\ovl{x}\in H^0(E,\LL(-[p+q]_3))\}=V_3.\qedhere\end{equation*}
\end{proof}

As mentioned previously, we are adapting Rogalski's proof of Theorem~\ref{T(d) properties} to prove Theorem~\ref{S(p+q) g div}.  The main step where Rogalski uses generation in degree 1 (which is something we do not have) is the semi-continuity argument he uses. We are able to come up with our own semi-continuity argument. For this we require the introduction of the Grassmannians.\par

We write $\Gr(m,V)$\index[n]{grmv@$\Gr(m,V)$}\index{Grassmannian} for the Grassmannian of $m$-dimensional subspaces of a fixed $n$-dimensional $\Bbbk$-vector space $V$. More details on the Grassmannian can be found in the appendix. The result on Grassmannians we need is Lemma~\ref{Grassmannian} below. We acknowledge that surely there exists a suitable reference somewhere for it despite the author being unable to find such. Since at first sight it did not appear obvious to the author why such a statement should be true, a (rather long-winded) proof is included in an appendix.\par

\begin{lemma}\label{Grassmannian}
Let $R_2$ and $R_3$ be a $4$-dimensional and $7$-dimensional $\Bbbk$-vector space respectively, and set
$$\Omega_2=\{(W_1,W_2)\in \Gr(3,R_2)^2\,|\; \dim_\Bbbk (W_1\cap W_2)=2\}$$
and
\begin{equation}\label{Omega}
\Omega_3=\{(W_1,W_2,W_3)\in \Gr(6,R_3)^3\,| \dim_\Bbbk (W_1\cap W_2\cap W_3)=4\}.
\end{equation}
Then $\Omega_2$ and $\Omega_3$ are Zariski open subsets of $ \Gr(3,R_2)^2$ and $\Gr(6,R_3)^3$ respectively; and the maps $\psi_2:\Omega_2\to \Gr(2,R_2)$ and  $\psi_3:\Omega_3\to \Gr(4,R_3)$, given by
\begin{equation}\label{psi}\psi_2:(W_1,W_2)\longmapsto W_1\cap W_2 \;\text{ and }\; \psi_3:(W_1,W_2,W_3)\longmapsto W_1\cap W_2\cap W_3\end{equation}
are morphisms of varieties.
\end{lemma}

\begin{proof} See Lemma~\ref{grass appendix}. \end{proof}

Another concept we require to prove Theorem~\ref{S(p+q) g div} is the geometric quotient of a variety $X$ by a group $G$. Let $G$ be an algebraic group and suppose that $G$ acts upon a variety $X$. If this action $G\times X\to X$ is a morphism of algebraic varieties, then $X$ is called a \textit{$G$-variety} \index{gvar@$G$-variety} \cite[Definition~21.4.1]{TY}. Let $X$ be a $G$-variety; one may then define the \textit{geometric quotient of $X$ by $G$} \cite[Definition~25.3.1]{TY}. The reader is referred to \cite{TY} for the exact definition. Suppose that we have a surjective morphism of varieties $\pi:X\to Y$ such that the fibers of $\pi$ are precisely the $G$-orbits of $X$. Then it is the case that $Y$ is the geometric quotient of $X$ by $G$ provided $Y$ is a normal variety \cite[Proposition~25.34]{TY}.\par

Finally, at a certain point in the proof of Theorem~\ref{S(p+q) g div} we will need the field we are working over to be uncountable. Clearly if $\mathbb{K}\supseteq \Bbbk$ is a field extension, then $\mathbb{K}$ is trivially flat as a $\Bbbk$-module. The following lemma is then immediate.

\begin{lemma}\label{extend field}
Let $\mathbb{K}\supseteq \Bbbk$ be a field extension. Let $V$ a finite dimensional $\Bbbk$-vector space and denote $V_\mathbb{K}=V\otimes_\Bbbk\mathbb{K}$. Then $\dim_\Bbbk V=\dim_\mathbb{K}V_\mathbb{K}$. In particular, for a cg $\Bbbk$-algebra $A$ and $A_\mathbb{K}=A\otimes_\Bbbk\mathbb{K}$, we have the equality of Hilbert series
\begin{equation*} \pushQED{\qed}
\sum_{i\in\N} \dim_\Bbbk(A_i)t^i =\sum_{i\in\N}\dim_\mathbb{K}(A_i\otimes_\Bbbk \mathbb{K})t^i. \qedhere
\end{equation*}
\end{lemma}


We are now ready to prove $S(p+q)$ is a $g$-divisible ring.

\begin{proof}[Proof of Theorem~\ref{S(p+q) g div}]
By Lemma~\ref{1 iff 2 of S(p+q) g div}, we can prove both (1) and (2) in tandem. We break the proof up into 3 steps.\\\par

\textit{Step 1}. For now, fix $p,q\in E$ and let $R=S(p+q)$. For Step 1 we prove that $R$ is $g$-divisible with the additional assumption that $p$ and $q$ are in distinct $\sigma$-orbits. We need to assume that $\Bbbk$ is uncountable to be sure such a situation exists; by Lemma~\ref{extend field} we lose no generality in doing so.\par
Put $R'=S(p)\cap S(q)$. Since $S$ is a domain and, $S(p)$ and $S(q)$ are already $g$-divisible from Proposition~\ref{Ro 12.2},
\begin{equation*}
R'\cap gS=(S(p)\cap gS)\cap (S(q)\cap gS)=gS(p)\cap gS(q)=g(S(p)\cap S(q))=gR'.
\end{equation*}
That is, $R'$ is $g$-divisible. We claim that in fact $R'=R$. \par
To prove this fix $i=1,2,3$. Clearly $[p]_i\leq [p+q]_i$, thus $\LL_i(-[p+q]_i)\subseteq \LL_i(-[p]_i)$ by Lemma~\ref{on divisors}. Taking global sections gives $H^0(E,\LL_i(-[p+q]_i))\subseteq H^0(E,\LL_i(-[p]_i))$, and hence
$$R_i=\{x\in S_i\,|\; \ovl{x}\in H^0(E,\LL_i(-[p+q]_i))\}\subseteq
\{x\in S_i\,|\; \ovl{x}\in H^0(E,\LL(-p)) \}=S(p)_i.$$
Since $R$ is generated by $R_1,R_2,R_3$ we have $R\subseteq S(p)$. Symmetrically we get $R\subseteq S(q)$ also, which forces $R\subseteq R'.$  \par
We now show $\ovl{R'}=\ovl{R}$. Fix $n\in\N$. By Lemma~\ref{S(p+q) bar},  $\ovl{R_n}=H^0(E,\LL_n(-[p+q]_n)$, and so
\begin{equation}\label{ovlR subset ovlR'}
H^0(E,\LL_n(-[p+q]_n)\subseteq\ovl{R'_n}\subseteq\ovl{S(p)_n}\cap \ovl{S(q)_n}=H^0(E,\LL_n(-[p]_n))\cap H^0(E,\LL_n(-[q]_n)).
\end{equation}
Now, the right hand side of (\ref{ovlR subset ovlR'}) consists of global sections of $\LL_n$ which vanish at both of the effective divisors $[p]_n=p+p^\sigma+\dots +p^{\sigma^{n-1}}$ and $[q]_n=q+q^\sigma+\dots +q^{\sigma^{n-1}}$. But we are assuming $p$ and $q$ are in distinct $\sigma$-orbits, hence $[p]_n\cap[q]_n=0$. Therefore, this is the same as saying it consists of global sections of $\LL_n$ which vanish at the effective divisor $[p+q]_n=p+q+(p+q)^\sigma+\dots +(p+q)^{\sigma^{n-1}}$; in other words the right hand side equals $H^0(E,\LL_n(-[p+q]_n)$ - the left hand side. Thus (\ref{ovlR subset ovlR'}) implies
\begin{equation}\label{ovlR=ovlR'} \ovl{R'_n}=H^0(E,\LL_n(-[p+q]_n)=\ovl{R_n}\;\text{ for all }\;n\geq 0.\end{equation}\par

We complete Step~1 with induction. For $n=0,1,2$, clearly $R_n\cap Sg=R'_n\cap gS=0$ because $gS\subseteq S_{\geq 3}$. Thus $R_n=R'_n$ follows from (\ref{ovlR=ovlR'}). Now fix $n\geq 3$, and assume $R_{n-3}=R'_{n-3}$. Then, since $R'$ is $g$-divisible and $R\subseteq R'$, we have that
$$R_n\cap gS\supseteq gR_{n-3}=gR'_{n-3}=R'_{n}\cap gS\supseteq R_n\cap gS.$$
This certainly implies $\dim_\Bbbk(R_n\cap gS)=\dim_\Bbbk(R'_n\cap gS)$.  We know $\dim_\Bbbk \ovl{R_n}=\dim_\Bbbk\ovl{R'_n}$ from (\ref{ovlR=ovlR'}), and so we get
$$\dim_\Bbbk R_n=\dim_\Bbbk \ovl{R_n}+\dim_\Bbbk(R_n\cap gS)=\dim_\Bbbk \ovl{R'_n}+\dim_\Bbbk(R'_n\cap gS)=\dim_\Bbbk R'_n.$$
Since $R\subseteq R'$ already, this forces $R_n=R'_n$. This completes the proof under the additional assumption that $p$ and $q$ are in distinct $\sigma$-orbits.\\\par

\textit{Step 2}. Now fix $p\in E$ and allow $q\in E$ to vary. To get rid of the additional assumption of Step 1 the strategy is to prove that the map
$$\nu_n: E\to \N, \;\;q\mapsto \dim_\Bbbk S(p+q)_n$$
is lower semi-continuous for each $n\geq 0$. In other words, for every $\ell\in\N$, the set $\{q\in E\,|\; \nu_n(q)\leq \ell\}$ is Zariski-closed in $E$. Fix $n\geq 0$. We prove $\nu_n$ is lower semi-continuous in two further sub-steps.

\begin{enumerate}[(2a)]
\item By Lemma~\ref{S(p+q)_i=V_i} there is a map $\theta:E\to \Gr(1,S(p)_1)\times \Gr(2,S(p)_2)\times \Gr(4,S(p)_3)$,
$$\theta: q\mapsto (S(p+q)_1,S(p+q)_2,S(p+q)_3).$$
We claim $\theta$ is a morphism of varieties.
\item For each $n\in\N$ consider $\mu_n: \Gr(1,S(p)_1)\times \Gr(2,S(p)_2)\times \Gr(4,S(p)_3)\to \N$,
$$\mu_n: (Y_1,Y_2,Y_3)\mapsto\dim_\Bbbk(\Bbbk\langle Y_1,Y_2,Y_3\rangle_n).$$
We claim $\mu_n$ lower semi-continuous for every $n\in\N$.
\end{enumerate}

Assume for the moment that these claims are true. Since clearly $\nu_n=\mu_n\circ \theta$,
$$\{q\in E\,|\; \nu_n(q)\leq \ell\}=\{q\in E\,|\; \mu_n(\theta(q))\leq \ell\}=\theta^{-1}(C),$$
where $C=\{(Y_1,Y_2,Y_3)\,|\;\mu_n(Y_1,Y_2,Y_3)\leq \ell\}$. If (2b) holds, then $C$ is closed. If also (2a) holds, then $\theta^{-1}(C)$ is closed because then $\theta$ is continuous with respect to the Zariski topology. This would prove $\nu_n$ is lower-semi-continuous completing Step 2.  \\\par

\textit{Proof of 2a}. Below when we say ``morphism" we mean ``morphism of varieties".  Write
$$\theta=(\theta_1,\theta_2,\theta_3),\text{ where }\theta_i: q\mapsto S(p+q)_i.$$
To contain notation we only prove $\theta_3:E\to\Gr(4,S(p)_3)$ is a morphism. The proof for $\theta_1, \theta_2$ are similar.  We construct $\theta_3$ as a composite of other morphisms. First we have
$$\gamma: E\to E^3,\; q\mapsto (q,q^\sigma,q^{\sigma^2}).$$
Since $\sigma:E\to E$ is an automorphism, $\gamma$ is certainly a morphism. Let $\NN=\LL_3(-[p]_3)$ and $W=H^0(E,\NN)=\ovl{S(p)_3}$; by the Riemann-Roch Theorem (Corollary~\ref{RR to E}) $\dim_\Bbbk W=6$. Next, we have the map $\ovl{\phi}:E\to \Gr(5,W)$ determined by a basis of sections for $\NN$ (see \cite[Theorem~II.7.1]{Ha}). This is the map $\ovl{\phi}:q\mapsto H^0(E,\NN(-q))$. We remark that for the use of \cite[Theorem~II.7.1]{Ha} above, we are identifying $\Gr(5,W)$ with lines through the dual vector space, i.e. $\bbP(W^*)$. We then extend $\ovl{\phi}$ to obtain a morphism $\phi:E\to\Gr(6,S(p)_3)$. Explicitly we are post-composing $\ovl{\phi}$ with the morphism
$$\Gr(5,W)\to \Gr(6,S(p)_3): \;W'\mapsto \{x\in S(p)_3\,|\; \ovl{x}\in W'\}.$$
We then set
$$\Phi=(\phi,\phi,\phi):E\times E\times E\to \Gr(6,S(p)_3)\times \Gr(6,S(p)_3)\times \Gr(6,S(p)_3).$$
Now we use Lemma~\ref{Grassmannian} with $R_3=S(p)_3$. Recall $\Omega_3=\Omega$ and $\psi_3=\psi$ from (\ref{Omega}) and (\ref{psi}). By Lemma~\ref{Grassmannian} we have the morphism
$$\psi:\Omega\to \Gr(4,S(p)_3).$$
We claim $\theta_3=\psi\circ\Phi\circ \gamma$. Given $q\in E$, we have
$$\Phi(\gamma(q))=(X_{0},X_{1},X_{2}),$$
where $X_{i}=\{x\in S(p)_3\, |\; \ovl{x}\in H^0(E,\NN_i(-q^{\sigma^{i}}))\}$. By our main assumption, $\sigma$ has infinite order, and so there are no points of finite order. In particular $p$, $p^\sigma$ and $p^{\sigma^2}$ are all pairwise distinct. Now the argument used to prove (\ref{ovlR=ovlR'}) can then be repeated to show
\begin{equation}\label{X0capX1capX2 1}
\ovl{X_{0}}\cap\ovl{X_{1}}\cap \ovl{X_{2}}=H^0(E,\LL_3(-[p+q]_3))=\ovl{S(p+q)_3}.
\end{equation}
Certainly $\ovl{X_{0}\cap X_{1}\cap X_{2}}\subseteq \ovl{X_{0}}\cap\ovl{X_{1}}\cap \ovl{X_{2}}$. If $\ovl{y}\in \ovl{X_{0}}\cap\ovl{X_{1}}\cap \ovl{X_{2}}$ for some $y\in S_3$, then there exists $x_i\in X_i$ and $\lambda_i\in S_0=\Bbbk$, such that $y=x_i+g\lambda_i$ for $i=1,2,3$. But by the definition of each $X_i$, $g\lambda_i\in X_i$, hence $y\in X_{0}\cap X_{1}\cap X_{2}$. This shows $\ovl{X_{0}\cap X_{1}\cap X_{2}}= \ovl{X_{0}}\cap\ovl{X_{1}}\cap \ovl{X_{2}}$. It then follows from (\ref{X0capX1capX2 1}) that
$$X_0\cap X_1\cap X_2\subseteq \{x\in S_3\,|\;\ovl{x}\in H^0(E,\LL_3(-[p+q]_3))\}=V_3.$$
Since $gS_0=g\Bbbk\subseteq X_0\cap X_1\cap X_2$, we in fact have $X_0\cap X_1\cap X_2=V_3$. By Lemma~\ref{S(p+q)_i=V_i}, $V_3=S(p+q)_3$ and $\dim_\Bbbk S(p+q)_3=4$. In particular, $\Phi(\gamma(q))\in \Omega$ and $\psi(\Phi(\gamma(q)))$ is defined with
$$\psi(\Phi(\gamma(q)))=S(p+q)_3=\theta_3(q).$$
Thus $\theta_3$, as a composition of morphisms, is itself a morphism of varieties. Similarly one can prove $\theta_1$ and $\theta_2$ are morphisms, and hence $\theta$ is also a morphism of varieties. \\\par

\textit{Proof of 2b}. Fix bases for $S(p)_1$, $S(p)_2$, and $S(p)_3$. By Proposition~\ref{Ro 12.2}(1), $S(p)_1$, $S(p)_2$ and $S(p)_3$ are respectively 2, 4 and 7 dimensional $\Bbbk$-vector spaces. We can identify a vector of $S(p)_1\times S(p)_2\times S(p)_3$ with its coordinates in $\A^2\times\A^4\times \A^7$. Let $\mathcal{U}$ be the collection of $(u,(v_1,v_2),(w_1,\dots,w_4))\in \A^2\times(\A^4)^2\times (\A^7)^4$ such that $u\neq 0$ and both $(v_1,v_2)$ and $(w_1,\dots,w_4)$ are linearly independent collections of vectors in $S(p)_2$ and $S(p)_3$ respectively. Linear independence is an open condition, and so $\mathcal{U}$ is open subset. To ease notation put $\mathcal{H}=\Gr(1,S(p)_1)\times \Gr(2,S(p)_2)\times \Gr(4,S(p)_3)$. We have a natural surjection $\pi: \mathcal{U}\to \mathcal{H}$,
$$(u,(v_1,v_2),(w_1,\dots,w_4))\mapsto (\mathrm{span}\{u\},\mathrm{span}\{v_1,v_2\},\mathrm{span}\{w_1,\dots,w_4\}).$$
We will show this is nothing but the geometric quotient (in the sense of \cite[Definition~25.3.1]{TY}) of $\mathcal{U}$ by the natural action of $G=\Bbbk^\times\times \Gl_2(\Bbbk)\times \Gl_4(\Bbbk)$
given by left multiplication. To this end, we remark that $\mathcal{U}$ is obviously a irreducible $G$-variety (see \cite[Definition~21.4.1]{TY}) and moreover, given $\mathfrak{h}\in \mathcal{H}$, $\pi^{-1}(\mathfrak{h})$ is a $G$-orbit. By \cite[Proposition~25.3.5]{TY}, it is the case that $\mathcal{H}$ is the desired quotient provided $\mathcal{H}$ is a normal variety. But this is obviously true as $\mathcal{H}$ is a product of Grassmannians, and hence has an open affine cover $\{U_i\times V_j\times W_k\}_{i,j,k}$ with $U_i\cong \A^3$, $V_j\cong (\A^2)^2$ and $W_k\cong (\A^{3})^4$ (see \cite[Lemma~11.15]{Hassett}).\par

Fix $n\in \N$. We now prove $\mu=\mu_n$ is lower semi-continuous. First we lift the map $\mu$ to a map $\mu':\mathcal{U}\to \N$. Given $\mathfrak{u}=(u,(v_1,v_2),(w_1,\dots,w_4))\in \mathcal{U}$ set $$\begin{array}{l}
Y_1(\mathfrak{u})=\mathrm{span}\{u\}\subseteq S(p)_1,\\
Y_2(\mathfrak{u})=\mathrm{span}\{v_1,v_2\}\subseteq S(p)_2, \\
Y_3(\mathfrak{u})=\mathrm{span}\{w_1,w_2,w_3, w_4\}\subseteq S(p)_3,
\end{array}
$$
and define $\mu'(\mathfrak{u})=\dim_\Bbbk(\Bbbk\langle Y_1(\mathfrak{u}), Y_2(\mathfrak{u}), Y_3(\mathfrak{u}) \rangle_n).$
We have the commuting diagram
\[
\xymatrix{
\mathcal{U} \ar[d]_{\pi} \ar[dr]^{\mu'} \\
\mathcal{H} \ar[r]^{\mu} & \N. }
\]
To prove $\mu$ is lower semi-continuous we first show that $\mu'$ is lower semi-continuous.  \par
Let $x_1,\dots, x_m$ be the spanning set of $S(p)_n$ consisting of all degree $n$ products of the original fixed bases of $S(p)_1$, $S(p)_2$ and $S(p)_3$. Take $\mathfrak{u}\in \mathcal{U}$ and write $Y_i=Y_i(\mathfrak{u})$ for $i=1,2,3$. Then $\Bbbk\langle Y_1,Y_2,Y_3\rangle_n$ is spanned by elements of the form $y_i=\sum f_{ij}x_j$, where $f_{ij}$ are polynomials in the coordinates of $\mathfrak{u}$. Take $\ell\in \N$, then $\dim_\Bbbk (\Bbbk\langle Y_1,Y_2,Y_3\rangle_n)\leq \ell$ if and only if no collection of $\ell+1$ of the $y_i$ are linearly independent. This in turn is equivalent to the determinants of all $(\ell+1)\times(\ell+1)$ matrix minors of the matrix $\left(f_{ij}\right)_{i,j}$ vanishing. Thus $\{\mathfrak{u}\in \mathcal{U}\,|\; \mu'(\mathfrak{u})\leq \ell \}$ is closed, proving $\mu'$ is lower semi-continuous. \par
For $\ell\in \N$, set $Z=\{ (Y_1,Y_2,Y_3)\in \mathcal{H}\,|\; \mu(Y_1,Y_2,Y_3)\leq \ell \}$, we must show $Z$ is a closed subset of $\mathcal{H}$. We know $\pi^{-1}(Z)=\{\mathfrak{u}\in \mathcal{U}\,|\; \mu'(\mathfrak{u})\leq \ell \}$ is a closed subset of $\mathcal{U}$ - we have just proved it. Because $\mathcal{H}$ is the geometric quotient of $\mathcal{U}$ by $G$, it indeed follows that $Z=\pi(\pi^{-1}(Z))$ is closed (see \cite[Lemma~25.3.2]{TY}). Thus $\mu$ is lower semi-continuous as claimed.\par

As an aside, before moving on to Step~3 we remark that the lower semi-continuity of $\mu$ suggests that for generic $(Y_1,Y_2,Y_3)\in \mathcal{H}$, $\Bbbk\langle Y_1,Y_2,Y_3\rangle\ehd S(p)$. We however are only concerned about $\mu$ on the image of $\theta$. In Step~3 we will see this is exactly what we require. \\\par

\textit{Step 3}. Finally we complete the proof. By Lemma~\ref{extend field} we may assume without loss of generality that $\Bbbk$ is uncountable. In which case $E$ has uncountably many points. By Step~1, $S(p+q)$ has the correct Hilbert series (\ref{S(p+q) hilbert series}) when $p$ and $q$ are in different $\sigma$-orbits, say $\dim_\Bbbk S(p+q)_n=\alpha_n$ for $p,q$ in different $\sigma$-orbits. By Step~2, $E_n=\{q\in E\,|\; \dim_\Bbbk S(p+q)_n\leq \alpha_n\}$ is a closed subset of $E$. If $E_n\subsetneq E$, then because $E_n$ is closed, $E_n$ would be finite. But $E_n$ contains the uncountable set $\{q\in E\,|\; q\neq p^{\sigma^i}\,\text{ for all }\, i\in\Z\}$ by Step~1. Hence $E_n=E$. In other words, $\dim_\Bbbk S(p+q)_n\leq \alpha_n$ for all $n\geq 0$ and $p,q\in E$. Thus
$$h_{S(p+q)}(t)\leq \sum_{n\geq 0}\alpha_n=\frac{t^2-t+1}{(t-1)^2(t^3-1)},$$
for any $p,q\in E$. The reverse inequality is (\ref{hilbert series ineq}). By Lemma~\ref{1 iff 2 of S(p+q) g div}, this completes the proof.
\end{proof}

Now we know that the $S(p+q)$'s are $g$-divisible we will quickly obtain a version of Proposition~\ref{Ro 12.2} for $S(p+q)$. Before stating the main theorem of this section we prove the equality (\ref{S(d) to T([d]3)}). We must therefore recall Notation~\ref{3 Veronese notation bg} and Definition~\ref{T(d) def} for the $T(\bfd)$. In the following the Veronese subring $T=S^{(3)}$ of $S$ is given the grading induced from $S$, that is $T_n=T\cap S_n$.

\begin{cor}\label{S(p+q) 3 Veronese}
Let $p,q\in E$. Then $S(p+q)^{(3)}=T([p+q]_3)$.
\end{cor}
\begin{proof}
Set $R=S(p+q)$, $R'=T([p+q]_3)$ and $\NN=\LL(-p-q)$. Write $R=\Bbbk\langle V_1,V_2,V_3\rangle$ where $V_i=\{ x\in  S_i\,|\; \ovl{x}\in H^0(E,\NN_i)\}$ as in Definition~\ref{S(p+q)}. By Definition~\ref{T(d) def},  $R'=\Bbbk\langle V_3\rangle$, and so we clearly have $R'\subseteq R^{(3)}$. Since we already have this inclusion, to prove equality it is enough to prove that we have an equality of Hilbert series. By Lemma~\ref{S(p+q) bar} and \cite[Theorem~1.1(1)]{Ro}, $\ovl{R_n}=H^0(E,\NN_n)=\ovl{R'_n}$ for each $n\geq 0$ divisible by 3. Thus, for $\dim_\Bbbk R_n=\dim_\Bbbk R'_n$ to hold, it suffices to prove $R_n\cap Sg= R'_n\cap Sg$. But $R_n\cap Sg=R_{n-3}g$  and  $R'_n\cap Sg=R'_{n-3}g$ by Theorem~\ref{S(p+q) g div} and \cite[Theorem~5.2]{Ro}. By induction it is therefore enough to prove $\dim_\Bbbk R_0=\dim_\Bbbk R'_0$, which is trivial.
\end{proof}

\begin{cor}\label{Qgr(S(p+q))}
Let $p,q\in E$ and $R=S(p+q)$. Then $Q_\gr(R)=Q_\gr(S)$.
\end{cor}
\begin{proof}
Trivially $Q_\gr(R)\subseteq Q_\gr(S)$. By Corollary~\ref{S(p+q) 3 Veronese}, $R^{(3)}=T([p+q]_3)$, which has $Q_\gr(T([p+q]_3))=Q_\gr(T)$ by \cite[Theorem~5.4(2)]{Ro}. Thus $D_\gr(R)= D_\gr(T)=D_\gr(S)$. By Lemma~\ref{Qgr(S) or Qgr(T)} either $Q_\gr(R)=Q_\gr(T)$ or $Q_\gr(R)=Q_\gr(S)$. Since $R_1\neq 0$, it is must be that $Q_\gr(R)=Q_\gr(S)$.
\end{proof}

We now present our main theorem of this chapter. For future referencing purposes we incorporate Proposition~\ref{Ro 12.2} into Theorem~\ref{S(d) thm}.

\begin{theorem}\label{S(d) thm}
Let $\bfd$ be an effective divisor on $E$ with degree $\deg\bfd=d\leq 2$. Set $R=S(\bfd)$. Then:
\begin{enumerate}[(1)]
\item $R$ is $g$-divisible with $R/gR\cong \ovl{R}=B(E,\LL(-\bfd),\sigma)$ and $R^{(3)}=S^{(3)}(\bfd+\bfd^\sigma+\bfd^{\sigma^2})$. The Hilbert series of $R$ is given by
$$h_{R}(t)=\frac{t^2+(1-d)t+1}{(t-1)^2(1-t^3)}.$$
\item $R$ is Auslander-Gorenstein and Cohen-Macaulay.
\item $R$ satisfies $\chi$ on the left and right, has cohomological dimension 2 and possesses a balanced dualizing complex.
\item $R$ is a maximal order in $Q_\gr(R)=Q_\gr(S)$.
\end{enumerate}
\end{theorem}
\begin{proof}
(1) follows from Proposition~\ref{Ro 12.2}, Corollary~\ref{S(p+q) 3 Veronese},  Corollary~\ref{Qgr(S(p+q))} and Theorem~\ref{S(p+q) g div}. The rest follows from part (1) and Proposition~\ref{RSS2 2.4}.
\end{proof}

\begin{remark}\label{why blowup}
Unsurprisingly, it is the connection between our subalgebras, $S(\bfd)$ of $S$, and the blowup subalgebras, $T(\bfe)$ of $T$, that motivates the name ``blowup subalgebra of $S$". In particular, $S(\bfd)^{(3)}=T([\bfd]_3)$, and $S(\bfd)$ and $T(\bfe)$ satisfy very similar properties (compare Theorem~\ref{S(d) thm} with Theorem~\ref{T(d) properties}). Beyond this, we will see in the coming chapters that the $S(\bfd)$ play a similar role to that which the $T(\bfe)$ played in \cite{RSS}. Reasons for the $T(\bfe)$ being called ``blowup subalgebras of $T$" can be found in Remark~\ref{VdBblowups}.
\end{remark}

As in \cite[Theorem~5.4(1)]{Ro}, by examining Hilbert series of our blowups $S(\bfd)$ we see can they have infinite global dimension.

\begin{cor}\label{infinite global dim}
Let $\bfd$ be an effective divisor with $1\leq\deg\bfd\leq 2$ and let $R=S(\bfd)$. Then $R$ has infinite global dimension.
\end{cor}
\begin{proof}
By Theorem~\ref{S(d) thm}(1) either
\begin{equation}\label{h_R(t)} h_R(t)=\frac{t^2+1}{(t-1)^2(1-t^3)}\;\;\text{ or }\;\; h_R(t)=\frac{t^2-t+1}{(t-1)^2(1-t^3)}.\end{equation}
A standard argument, which examines a graded resolution of $\Bbbk$ considered as a right $R$-module (see \cite[Chapter I, (2.E)]{Rog.notes}) shows that if $R$ has finite global dimension, then $h_R(t)=\frac{1}{f(t)}$ for some polynomial $f(t)$. It can be checked that neither  of the Hilbert series in (\ref{h_R(t)}) can be written so, and hence $R$ must have infinite global dimension.
\end{proof}

\section{Classifying $g$-divisible maximal orders}\label{g-divisible max orders}

In this chapter we aim to obtain a classification of $g$-divisible maximal $S$-orders: Theorem~\ref{RSS 7.4} and Proposition~\ref{RSS 7.4(3)}. We prove that any $g$-divisible maximal $S$-order is obtain from a virtual blowup in the sense of Definition~\ref{virtual blowup}. Here we follow parts of \cite[Sections 5, 6, 7]{RSS} closely. A philosophy for this chapter is wherever Rogalski, Sierra and Stafford had their effective blowups $T(\bfd)$, we (with a bit of work) are able to replace them with our rings $S(\bfd)$ from Chapter~\ref{The rings S(d)}. A significant reason why this approach works is Theorem~\ref{S(d) thm} and Theorem~\ref{RSS 5.25} below.


\subsection{Endomorphism rings over the $S(\bfd)$}

Our first aim, and key to rest of this chapter, is to come up with a version of Rogalski, Sierra and Stafford's \cite[Theorem~2.24]{RSS} - a result we presented earlier in Theorem~\ref{real RSS 5.25} and Corollary~\ref{real RSS 5.25 2}. In other words, we need to prove that a $g$-divisible maximal order $U$ is an equivalent order to some $S(\bfd)$ (see Theorem~\ref{RSS 5.25} and Corollary~\ref{RSS 6.6}).  To obtain their result, Rogalski, Sierra and Stafford require a detailed and technical study of right ideals of $T$ and $T(\bfd)$ given in \cite{RSS2} and \cite[5.13-5.23]{RSS}. A great strength of our work is that we are able to bypass these technicalities. Before this, we require a few results from \cite[Section~5]{RSS}. The first is a consequence of~Theorem~\ref{RSS 3.1}.

\begin{lemma}\label{RSS 5.3}\cite[Corollary~5.3]{RSS} Let $E$ be a smooth elliptic curve, $\HH$ an invertible sheaf on $E$ with $\deg\HH\geq 2$, and $\rho:E\to E$ an automorphism of infinite order. Set $B=B(E,\HH,\rho)$. Suppose that $C$ is a cg subalgebra of $B$ with $Q_\gr(C)=Q_\gr(B)
$. Then there exist divisors $\bfx$, $\bfy$ on $E$ with $0\leq\deg\bfx<\deg\HH$, and $k\geq 0$ such that
\begin{equation*}\pushQED{\qed}
 C_n=H^0(E,\HH(-\bfy-[\bfx]_n)) \;\text{ for all } n\geq k. \qedhere\end{equation*}
\end{lemma}

The next result we need is \cite[Proposition~5.7]{RSS}. We in fact also require the construction of the divisor $\bfd$ appearing in this result. We include this construction with the reader being referred to \cite{RSS} for a detailed proof.

\begin{prop}\label{RSS 5.7}\cite[Proposition~5.7]{RSS} Retain the hypothesis for $C \subseteq B(E,\HH,\rho)$ from Lemma~\ref{RSS 5.3}. There exists an effective divisor $\bfd$ on $E$ supported at points on distinct $\rho$-orbits, with $\deg\bfd<\deg\HH$, and such that $C$ and $B(E,\HH(-\bfd),\rho)$ are equivalent orders. Moreover, there exists a $k\geq 0$, such that for all $n\geq k$
\begin{equation}\label{RSS 5.7 eq}
C_n\subseteq H^0(E,\HH_n(-\bfd^{\rho^k}-\bfd^{\rho^{k+1}}-\dots-\bfd^{\rho^{n-1}})).
\end{equation}
\end{prop}
\begin{proof}[Construction of $\bfd$]
For this construction we use the notation $[\bfx]_n=\bfx+\bfx^\rho\dots+\bfx^{\rho^{n-1}}$ as well as the more compact notation $p_j=p^{\rho^j}=\rho^{-j}(p)$ for $p\in E$ and $j\in\Z$.\par
By Lemma~\ref{RSS 5.3} there exist divisors $\bfx,\bfy$ with $0\leq \deg\bfx<\deg\HH$, and $k\geq 1$ such that
\begin{equation*}\label{geo data} C_n=H^0(E,\HH_n(-\bfy-[\bfx]_n)) \; \text{ for all } n\geq k.\end{equation*}
Fix an $\rho$-orbit $\mathbb{O}$ on $E$, and (enlarging $k$ if necessary) pick $p\in E$ such that on $\mathbb{O}$ we have
\begin{equation}\label{mathbbO}\index[n]{xo@$\bfx|_{\bbO}$}
\bfx|_{\mathbb{O}}=\sum_{i=0}^{k} a_ip_i \;\;\;\;\text{ and }\;\;\;\; \bfy|_{\mathbb{O}}=\sum_{i=0}^{k-1}b_ip_i
\end{equation}
for some $a_i,b_i\in \Z.$ For this $p$, we set
\begin{equation}\label{e_p=} e_p=\sum_{i=0}^k a_i.\index[n]{ep@$e_p$}\end{equation}
In \cite{RSS} an argument is proved showing that $e_p\geq 0$. We define the effective divisor
\begin{equation}\label{bfd=} \bfd=\sum_p e_p p,\end{equation}
where the sum is taken over one closed point $p$ for each $\rho$-orbit chosen as above. It is shown in \cite{RSS} that, for this $\bfd$ and $k$, (\ref{RSS 5.7 eq}) holds and that $C$ and $B(E,\HH(-\bfd),\rho)$ are equivalent orders.
\end{proof}

Rogalski, Sierra and Stafford gave names to some for the terms appearing in Lemma~\ref{RSS 5.3} and Proposition~\ref{RSS 5.7}.

\begin{definition}\label{geo data, norm div}\cite[Definition~5.12]{RSS}
Retain the hypothesis for $C \subseteq B(E,\HH,\rho)$ from Lemma~\ref{RSS 5.3}.
\begin{enumerate}[(1)]
\item Let $\bfx$, $\bfy$ and $k$ be given by Lemma~\ref{RSS 5.3}. We call the data $(E,\HH,\rho,\bfx,\bfy,k)$ \textit{geometric data for $C$}.
\item We call the divisor $\bfd$ constructed in Proposition~\ref{RSS 5.7} a \textit{$\rho$-normalised divisor for $C$}. When no confusion can occur, we simply call $\bfd$ a normalised divisor.
\end{enumerate}
\end{definition}

\begin{remark}\label{norm div not unique}
Let $C$ be as in Definition~\ref{geo data, norm div}. Suppose that $\bfd$ is a normalised divisor for $C$ constructed out of geometric data $(E,\HH,\rho,\bfx,\bfy,k)$ for $C$. We make some observations about Definition~\ref{geo data, norm div} that follow from the construction of $\bfd$.
\begin{itemize}
\item The normalised divisor $\bfd$ depends on the geometric data $(E,\HH,\rho,\bfx,\bfy,k)$ and not $C$ itself.
\item For every $\ell\geq k$, the data $(E,\HH,\rho,\bfx,\bfy,\ell)$ is also geometric data for $C$. Moreover, the normalised divisor $\bfd$ is also a normalised divisor corresponding to the data $(E,\HH,\rho,\bfx,\bfy,\ell)$.
\item For every $n\geq 0$, the divisor $\bfd^{\rho^{-n}}=\rho^n(\bfd)$ is also a normalised divisor for $C$.
\end{itemize}
\end{remark}

We now put these terms to use. Proposition~\ref{RSS 5.25 lemma}(2) is the analogue of Proposition~\ref{RSS 5.20 bg}. We in fact utilise Proposition~\ref{RSS 5.20 bg} directly in the proof. For this reason we must recall the notation for $T=S^{(3)}$ (Notation~\ref{3 Veronese notation bg}). The grading of $T$ we use here is $T_n=T\cap S_n$.

\begin{prop}\label{RSS 5.25 lemma}
Let $U$ be a $g$-divisible cg subalgebra of $S$ with $Q_\gr(U)=Q_\gr(S)$. Suppose that $\bfd$ is a $\sigma$-normalised divisor for $\ovl{U}$ corresponding to geometric data $(E,\LL,\sigma,\bfx,\bfy,k)$ of $\ovl{U}$, as constructed in Proposition~\ref{RSS 5.7}. Then
\begin{enumerate}[(1)]
\item $\ovl{U}$ and $B(E,\LL(-\bfd),\sigma)$ are equivalent orders and
\begin{equation*}\ovl{U}_n\subseteq H^0(E,\LL_n(-\bfd^{\sigma^k}-\bfd^{\sigma^{k+1}}-\dots-\bfd^{\sigma^{n-1}}))\;\text{ for all }n\geq k;\end{equation*}
\item $US(\bfd)\subseteq S_{\leq k}S(\bfd)$.
\end{enumerate}
\end{prop}
\begin{proof}
(1). Since $Q_\gr(U)=Q_\gr(S)$, $Q_\gr(\ovl{U})=Q_\gr(\ovl{S})$ by Lemma~\ref{ovlR is order}. Moreover, as $\deg\LL=3$, we can take $B(E,\HH,\rho)=B(E,\LL,\sigma)=\ovl{S}$ in Lemma~\ref{RSS 5.3} and Proposition~\ref{RSS 5.7}. Part (1) is therefore exactly Proposition~\ref{RSS 5.7}.

(2). Retain the notation of the construction of the $\sigma$-normalised divisor $\bfd$ from Proposition~\ref{RSS 5.7} with $(\ovl{U},\LL,\sigma)$ in place of $(C,\HH,\rho)$. In particular we use $p_j=p^{\sigma^j}$ again. Set $V=U^{(3)}$. The key here is to understand $V$ so that we can apply Proposition~\ref{RSS 5.20 bg}.\par
We claim that $[\bfd]_3=\bfd+\bfd^\sigma+\bfd^{\sigma^2}$ is a $\tau$-normalised divisor for $\ovl{V}$. Since $(E,\LL,\sigma,\bfx,\bfy,k)$ is geometric data for $\ovl{U}$ we have
\begin{equation*} \ovl{U}_n=H^0(E,\LL(-\bfy-[\bfx]_n)) \;\text{ for all } n\geq k. \end{equation*}
By enlarging $k$ if necessary, we may assume $k$ is divisible by 3. It then follows that
$$\ovl{V}_{3n}=H^0(E,\LL_{3n}(-\bfy-[\bfx]_{3n}))\;\text{ for } n\geq \frac{k}{3}.$$
We see $(E,\LL_3,\tau,\bfy,[\bfx]_3,k)$ is geometric data for $\ovl{V}$ ($\frac{k}{3}$ also works but there is no harm in taking the bigger $k$). Now, the $\sigma$-orbit $\mathbb{O}$ from (\ref{mathbbO}) is the disjoint union of three $\tau$-orbits, $\mathbb{O}=\mathbb{O}_0\sqcup\mathbb{O}_1 \sqcup\mathbb{O}_2$. Here $\mathbb{O}_\ell=\{ p_{\ell+3j}\,|\; j\in\Z\}$, where we are writing $\mathbb{O}=\{ p_j\,|\; j\in\Z\}$. Since $\bfx|_{\mathbb{O}}=\sum a_ip_i$,
\begin{multline*} [\bfx]_3|_{\mathbb{O}}=\sum_{i=0}^{k}a_i(p_i+p_{i+1}+p_{i+2})
\\
=a_0p_0+(a_0+a_1)p_1+\left( \sum_{i=2}^{k} (a_{i-2}+a_{i-1}+a_{i})p_{i}\right)+(a_{k}+a_{k-1})p_{k+1}+a_kp_{k+2},
\end{multline*}
 and therefore:
$$\begin{array}{l}
\,[\bfx]_3 |_{\mathbb{O}_0}=a_0p_0+(a_1+a_2+a_3)p_3+(a_4+a_5+a_6)p_6+\dots+ (a_{k-2}+a_{k-1}+a_k)p_k; \\
\,[\bfx]_3 |_{\mathbb{O}_1}=(a_0+a_1)p_1+(a_2+a_3+a_4)p_4+\dots+ (a_{k-4}+a_{k-3}+a_{k-2})p_{k-2}+(a_{k-1}+a_k)p_{k+1};\\
\,[\bfx]_3 |_{\mathbb{O}_2}=(a_0+a_1+a_2)p_2+\dots+ (a_{k-3}+a_{k-2}+a_{k-1})p_{k-1}+a_kp_{k+2}.
\end{array}$$
Put $e_p=\sum a_i$ as in (\ref{e_p=}). It follows that a $\tau$-normalised divisor for $\ovl{V}$ can be given by
$$\bfe=\sum_p e_pp_0+e_pp_1+e_pp_2$$
where the sum ranges over the same points used to define $\bfd$ in (\ref{bfd=}). Note that indeed $p_0$, $p_1$ and $p_2$ are on different $\tau$-orbits. On the other hand, by considering different $\sigma$-orbits separately, it is not hard to see
$$[\bfd]_3=\bfd+\bfd^\sigma+\bfd^{\sigma^2}=\bfe.$$
Hence $[\bfd]_3$ is a $\tau$-normalised divisor for $\ovl{V}$ as claimed.\par

By Proposition~\ref{RSS 5.20 bg} and Theorem~\ref{S(d) thm}(1) we have
\begin{equation}\label{apply RSS 5.20} V\subseteq T_{\leq k}T([\bfd]_3)=T_{\leq k}(S(\bfd)^{(3)}).\end{equation}
Since $U\subseteq S$ is $g$-divisible, it is noetherian by Proposition~\ref{RSS 2.9}, and therefore $U$ is a finitely generated as a right $V$-module by Lemma~\ref{noeth up n down}. By enlarging $k$ if necessary, we may assume $U_V$ is generated in degrees less than $k$. In which case, using (\ref{apply RSS 5.20}), we have
$$U=U_{\leq k}V\subseteq S_{\leq k}V\subseteq S_{\leq k} T_{\leq k}(S(\bfd)^{(3)})\subseteq S_{\leq 2k}(S(\bfd)^{(3)})$$
Hence,
$$US(\bfd)\subseteq S_{\leq 2k}(S(\bfd)^{(3)})S(\bfd)=S_{\leq 2k}S(\bfd).$$
Since it is harmless to replace $k$ by $2k$, this proves (2).
\end{proof}

With Proposition~\ref{RSS 5.25 lemma} we are able to prove the crucial result that links a general $g$-divisible subalgebra of $S$ to one of the $S(\bfd)$. Theorem~\ref{RSS 5.25} is analogous to Theorem~\ref{real RSS 5.25}. With Proposition~\ref{RSS 5.25 lemma} at hand we can model our proof on that of Theorem~\ref{real RSS 5.25} given in \cite{RSS}.

\begin{theorem}\label{RSS 5.25}
Let $U\subseteq S$ be a $g$-divisible cg subalgebra of $S$ with $Q_\gr(U)=Q_\gr(S)$. Then there exists an effective divisor $\bfd$ on $E$ with $\deg\bfd\leq 2$, supported on points with distinct $\sigma$-orbits, and such that $U$ and $S(\bfd)$ are equivalent orders. \par
In more detail, for this $\bfd$, the $(U,S(\bfd))$-bimodule $M=\widehat{US(\bfd)}$ is a finitely generated $g$-divisible right $S(\bfd)$-module with $S(\bfd)\subseteq M\subseteq S$. If $W=\End_{S(\bfd)}(M)$, then $U\subseteq W\subseteq S$, $_WM$ is finitely generated and $W$, $U$, and $S(\bfd)$ are equivalent orders. The divisor $\bfd$ is any $\sigma$-normalised divisor constructed for $\ovl{U}$.
\end{theorem}

\begin{proof}
Choose $k\geq 0$ and an effective divisor $\bfd$ on $E$, supported on distinct orbits, satisfying the conclusion of Proposition~\ref{RSS 5.25 lemma}. In particular we have $US(\bfd)\subseteq S_{\leq k}S(\bfd)$, and it follows that
$$\widehat{US(\bfd)}\subseteq \widehat{S_{\leq k}S(\bfd)}.$$
\par
Write $R=S(\bfd)$ and $M=\widehat{UR}$. Now $R$ is noetherian and $g$-divisible by Theorem~\ref{S(d) thm}(1), and $S_{\leq k}R$ is clearly a finitely generated right $R$-module. Thus $N=\widehat{S_{\leq k}R}$ is a finitely generated by Lemma~\ref{RSS 2.13}(2). In particular, $N$ is a noetherian $R$-module, and hence the submodule $M_R$ is finitely generated too. Put $W=\End_{R}(M)$; as $1\in M\subseteq S$, $W\subseteq S$. By Lemma~\ref{RSS 2.12}(3), $W$ is $g$-divisible and $_WM$ is finitely generated. By Lemma~\ref{hence equiv orders}(3), $W$ and $R$ are then equivalent orders. It is easily checked that $UM\subseteq M$, and thus $U\subseteq W$. \par
Now consider the $(\ovl{W},\ovl{R})$-bimodule $\ovl{M}$. This is nonzero as $M\supseteq R$ and is finitely generated on both sides because $_WM_R$ is finitely generated on both sides. Hence $\ovl{W}$ and $\ovl{R}$ are equivalent orders again by Lemma~\ref{hence equiv orders}(3). By Proposition~\ref{RSS 5.25 lemma}, $\ovl{R}$ and $\ovl{U}$ are equivalent orders, and therefore $\ovl{W}$ and $\ovl{U}$ are likewise. Since $U\subseteq W\subseteq S$ we can now apply Proposition~\ref{RSS 2.16} to conclude $U$ and $W$, and hence also $U$ and $R$, are equivalent orders.
\end{proof}

\begin{remark}\label{RSS 6.2}
Let $A$ be a noetherian domain with full quotient ring $Q$. If $N$ is a finitely generated $A$-submodule of $Q$, then \cite[Theorem~2.7]{Coz} says that $X=\End_A(N^*)$ is the unique maximal order containing and equivalent to $\End_A(N)$. Moreover, it is easy to see $\End_A(N^*)=\End_A(N^{**})$, and so we can replace $\End_A(N^*)$ with $\End_A(N^{**})$ in the above statement. By Lemma~\ref{graded max orders are max orders} this observation also holds in the graded setting.
\end{remark}

The above observation gave Rogalski, Sierra and Stafford (and hence us) a clear strategy to get from their $T(\bfd)$'s to other maximal orders. \par

As a standard, all duals $M^*$ will be taken as $R$-modules, where $R=S(\bfd)$. The definition of a maximal order and a maximal $S$-order can be found in Definition~\ref{max orders def}.

\begin{cor}\label{RSS 6.6}
Let $U\subseteq S$ be a $g$-divisible connected graded maximal $S$-order.
\begin{enumerate}[(1)]
\item There exists an effective divisor $\bfd$ on $E$ with $\deg\bfd\leq 2$, and a $g$-divisible $(U,S(\bfd))$-bimodule $M$, such that $U=\End_{S(\bfd)}(M)$. The bimodule $M$ is finitely generated on both sides and satisfies $S(\bfd)\subseteq M\subseteq S$.
\item Set $F=\End_{S(\bfd)}(M^{**})$. Then $F$ is the unique maximal order containing $U$ and $U=F\cap S$.
\end{enumerate}
\end{cor}
\begin{proof}
(1). By Theorem~\ref{RSS 5.25}, there is an effective divisor $\bfd$ with $\deg\bfd\leq 2$ and such that
$ U\subseteq W=\End_{S(\bfd)}(M)\subseteq S$, where $M=\widehat{US(\bfd)}$ is finitely generated on both sides as a $(W,S(\bfd))$-bimodule. Clearly we have $S(\bfd)\subseteq M\subseteq S$ also. By Theorem~\ref{RSS 5.25} again, $U$ and $W$ are equivalent orders. Since $U$ is a maximal $S$-order and $W\subseteq S$, $U=W$. \par
(2) By Remark~\ref{RSS 6.2}, $F$ is the unique maximal order equivalent to and containing $\End_{S(\bfd)}(M)=U$. Let $V=F\cap S$. Since $U\subseteq S$, we have $U\subseteq V\subseteq F$. If $x,y\in Q_\gr(S)$ are nonzero and such that $xFy\subseteq U$, then clearly $xVy\subseteq U$ also. Thus $U$ and $V$ are also equivalent orders. But $U$ is a maximal $S$-order; therefore $U=V$.
\end{proof}

The next two results investigate the situation $\End_{S(\bfd)}(M)\subseteq\End_{S(\bfd)}(M^{**})$ arising in Corollary~\ref{RSS 6.6}. These results are our analogue of \cite[Proposition~6.4]{RSS}. The proof given in \cite{RSS} is in fact sufficiently general to work in our case as well. For completeness we include the proofs. \par
Recall GK-dimension is defined in Notation~\ref{GK dimension}. In particular, we recall that we are writing $\GK(A)$ for the GK-dimension a $\Bbbk$-algebra $A$ and $\GK_A(M)$ for the GK-dimension of a right (or left) $A$-module $M$.

\begin{prop}\label{RSS 6.4}
Let $\bfd$ be an effective divisor with $\deg \bfd\leq 2$ and let $R=S(\bfd)$. Let $M\subseteq \Sg$ be a $g$-divisible finitely generated graded right $R$-module such that $R\subseteq M\subseteq S$. Put $W=\End_R(M)$, $F=\End_R(M^{**})$ and $V=F\cap S$. Then:
\begin{enumerate}[(1)]
\item $F$, $V$ and $W$ are $g$-divisible algebras with $Q_\gr(W)=Q_\gr(V)=Q_\gr(F)=Q_\gr(S)$.
\item $F$ is the unique maximal order containing and equivalent to $W$, while $V$ is the unique maximal $S$-order containing and equivalent to $W$.
\item $R=\End_W(M)=\End_F(M^{**})$.
\end{enumerate}
\end{prop}
\begin{proof}
Since $Q_\gr(R)=Q_\gr(S)$,  also $Q_\gr(W)=Q_\gr(V)=Q_\gr(F)=Q_\gr(S)$ by Lemma~\ref{End is cg}. By Lemma~\ref{RSS 2.12}(3), $W$ is $g$-divisible and noetherian, whilst $_WM$ is finitely generated. Thus, by a left handed version of \cite[Lemma~6.3]{RSS}, $\End_W(M)$ is a finitely generated right $R$-module. Moreover, since $MR\subseteq M$, $R\subseteq \End_W(M)$. So $R$ and $\End_W(M)$ are equivalent orders by Lemma~\ref{hence equiv orders}(2). Because $R$ is a maximal order by Theorem~\ref{S(d) thm}, $R=\End_W(M)$. \par
Clearly $M^{**}$ embeds into $R$ as a right ideal via $x\mapsto \alpha x$ for any nonzero $\alpha\in M^*$; in particular $M^{**}$ is a finitely generated right $R$-module. By Lemma~\ref{RSS 2.13}(3), $M^{**}$ is $g$-divisible with $M^{**}\subseteq \Sg$. Arguing as in the previous paragraph, with $M^{**}$ and $F$ in place of $M$ and $W$, shows $F$ is $g$-divisible, $_FM^{**}$ is finitely generated and $R=\End_F(M^{**})$. By Remark~\ref{RSS 6.2}, $F$ is the unique maximal order containing and equivalent to $W$. Because $1\in M\subseteq S$, $W\subseteq S$; and hence $W\subseteq V$.\par

Let $X$ be an equivalent order to $V$ satisfying $V\subseteq X\subseteq S$. If $\widetilde{X}$ is a maximal order containing $X$, then it is a maximal order containing $W$. Therefore $\widetilde{X}=F$ by the uniqueness of $F$. It then follows $X=V$ and $V$ is a maximal $S$-order. Similarly, if $V'$ is another maximal $S$-order containing and equivalent to $W$, and $F'$ a maximal order containing $V'$, then $F'=F$ and $V'=V$ by uniqueness of $F$ again. So $V$ is unique among maximal $S$-orders containing and equivalent to $W$.
\end{proof}

From here the Cohen-Macaulay property, defined in Definition~\ref{homological defs} and satisfied by the $S(\bfd)$'s by Theorem~\ref{S(d) thm}(2), becomes increasingly important. It will often be used in tandem with \cite[Lemma~4.11]{RSS}. We state the lemma here considering the frequency which it is used. For Lemma~\ref{RSS 4.11}(2) we also need concept of a pure module. Let $A$ be a ring of finite GK-dimension and suppose that $M$ is a right $A$-module with $\GK_A(M)=\beta$. Then $M$ is called \textit{$\beta$-pure}\index{pure module} if $\GK(N)=\beta$ for all nonzero submodules $N$ of $M$.

\begin{lemma}\label{RSS 4.11} \cite[Lemma~4.11]{RSS}
Let $A$ be a cg noetherian domain with $Q=Q_\gr(A)$ and $\GK(A)=\alpha<\infty$. Assume that $A$ is Auslander-Gorenstein and Cohen-Macaulay.
\begin{enumerate}[(1)]
\item Let $M\subseteq Q$ be a finitely generated right $A$-module. Then $M^{**}$ is the unique largest $A$-submodule of $Q$ containing $M$ and satisfying $\GK(M^{**}/M)\leq \alpha -2$. In particular, there is no $A$-module $A\subsetneq M\subseteq Q$ such that $\GK(M/A)\leq \alpha-2$.
\item If $I=I^{**}\subsetneq A$ is a proper reflexive right ideal of $A$, then $A/I$ is $(\alpha-1)$-pure.\qed
\end{enumerate}
\end{lemma}

In Corollary~\ref{RSS 6.6} there is the possibility that $F\not\subseteq S$. The next lemma helps us to control this situation.

\begin{lemma}\label{RSS 6.4(3)}
Retain the hypothesis of Proposition~\ref{RSS 6.4}. There exists an ideal $K$ of $F$, contained in $W$ (and hence $V$) such that $\GK(F/K)\leq 1$.
\end{lemma}
\begin{proof}
By Theorem~\ref{S(d) thm}, $R$ is Auslander-Gorenstein and Cohen-Macaulay. Therefore Lemma~\ref{RSS 4.11}(1) applies and shows that $\GK_R(M^{**}/M)\leq \GK(R)-2=1.$ Since $M$ is $g$-divisible, $M^{**}/M$ is $g$-torsionfree. By Lemma~\ref{RSS 2.14}(3) it is then a finitely generated right $\Bbbk[g]$-module. From here, the proof goes exactly as in \cite[Proposition~6.4(3)]{RSS}. We choose to omit the rest.
\end{proof}
We give a definition for pairs of algebras satisfying the conclusions of Proposition~\ref{RSS 6.4} and Lemma~\ref{RSS 6.4(3)}.

\begin{definition}\label{max order pair}\index{maximal order pair}
\begin{enumerate}[(1)]
\item A pair $(V,F)$ of connected graded subalgebras of $\Sg$ are called \textit{a maximal order pair of $S$} if:
\begin{enumerate}[(a)]
\item $V$ and $F$ are $g$-divisible with $V\subseteq F$ and $V=F\cap S$;
\item $F$ is a maximal order in $Q_\gr(F)=Q_\gr(S)$ while $V$ is a maximal $S$-order;
\item there exists an ideal $K$ of $F$, contained in $V$, and such that $\GK(F/K)\leq 1$.
\end{enumerate}
\item \cite[Definition~6.5]{RSS} Recall $T$ and $T_{(g)}$ from Notation~\ref{3 Veronese notation2}. Interchanging $S$ and $\Sg$ with $T$ and $T_{(g)}$ in (1), we get the definition of a maximal order pair of $T$.
\end{enumerate}
\end{definition}

The ``of $S$" from ``maximal order pair of $S$" is there to distinguish it from Definition~\ref{max order pair}(2). When no ambiguity can arise it will often be dropped. We will see later in Theorem~\ref{3 Veronese of virtual blowup} that if $(V,F)$ is a maximal order pair of $S$, then $(V^{(3)},F^{(3)})$ is a maximal order pair of $T$.\par
Definition~\ref{max order pair} is hard to work with. We will often be using an equivalent formulation given by the following consequence of Corollary~\ref{RSS 6.6} and Proposition~\ref{RSS 6.4}. We will be needing also a $T$-version of Lemma~\ref{max order pair equiv def} which is implicit in \cite{RSS}.

\begin{lemma}\label{max order pair equiv def}
\begin{enumerate}[(1)]
\item Let $F$ be a cg subalgebra of $\Sg$ and $V$ a cg subalgebra of $S$. Then the following are equivalent:
\begin{enumerate}[(a)]
\item $(F,V)$ is a maximal order pair of $S$.
\item There exist an effective divisor $\bfd$ with $\deg\bfd\leq 2$ and a finitely generated $g$-divisible right $S(\bfd)$-module $M$ with $S(\bfd)\subseteq M\subseteq S$ and such that
    $$F=\End_{S(\bfd)}(M^{**})\;\text{ and }V=F\cap S.$$

\end{enumerate}
\item Recall $T$ and $T_{(g)}$ from Notation~\ref{3 Veronese notation2}. Let $F$ be a cg subalgebra of $T_{(g)}$ and $V$ a cg subalgebra of $T$. Then the following are equivalent:
\begin{enumerate}[(a)]
\item $(F,V)$ is a maximal order pair of $T$

\item There exist an effective divisor $\bfd$ with $\deg\bfd\leq 8$ and a finitely generated $g$-divisible right $T(\bfd)$-module $M$ with $T(\bfd)\subseteq M\subseteq T$ and such that
    $$F=\End_{T(\bfd)}(M^{**})\;\text{ and }V=F\cap T.$$
\end{enumerate}
\end{enumerate}
\end{lemma}

\begin{proof} (1). Suppose that $(F,V)$ is a maximal order pair of $S$. Then $V$ is $g$-divisible maximal $S$-order. Now apply Corollary~\ref{RSS 6.6} to $V$. The reverse implication follows from Proposition~\ref{RSS 6.4}(2) and Lemma~\ref{RSS 6.4(3)}.

(2). This is essentially \cite[Corollary~6.6(1)]{RSS} and \cite[Proposition~6.4]{RSS}.
\end{proof}


Let $(U,F)$ be a maximal order pair. We now aim for an analogue of \cite[Proposition~6.7]{RSS}: describing the images of $U$ and $F$ in $\Sg/g\Sg =\Bbbk(E)[t,t^{-1};\sigma]$. We first need a few results. These are our analogues of \cite[Lemma~7.9]{RSS2} and \cite[Proposition~6.12]{RSS}.\par

\begin{lemma}\label{RSS2 7.9}
Let $R=S(\bfd)$ where $\bfd$ is an effective divisor with $\deg\bfd \leq 2$. Let $M$ be a $g$-torsionfree right $R$-module. Then as left $\ovl{R}$-modules
\begin{equation*}\pushQED{\qed}
\Ext_R^i(M,\ovl{R})\cong \Ext_{\ovl{R}}^i(M/Mg,\ovl{R})\;\text{ for every }i\geq 0.
\end{equation*}
\end{lemma}
\begin{proof}
The proof of \cite[Lemma~7.9]{RSS2} uses \cite[Proposition~VI.4.1.3, page 118]{CE} and \cite[Lemma~7.3]{RSS2}. The former is an abstract statement which only requires a ring homomorphism $T(\bfd)\to\ovl{T(\bfd)}$; the latter is a statement that also applies to $S(\bfd)$. In particular the proof presented there also proves this lemma.
\end{proof}

For an $S$-version of \cite[Proposition~6.12]{RSS} we are again in a position where the proof given by Rogalski, Sierra and Stafford can be recycled for our purposes.

\begin{lemma}\label{RSS 6.12}
Let $R=S(\bfd)$ where $\bfd$ is an effective divisor with $\deg\bfd \leq 2$.
\begin{enumerate}[(1)]
\item Let $I$ be a proper, $g$-divisible left ideal of $R$ such that $R/I$ is $2$-pure. Then $I^*/R$ is a $g$-torsionfree, $2$-pure right $R$-module. Moreover $I^*\subseteq \Sg$ and $\ovl{I^*}\overset{\bullet}= (\ovl{I})^*$.
\item If $M$ is a finitely generated $g$-divisible right $R$-module with $R\subseteq M\subseteq S$, then $\ovl{M^*}\overset{\bullet}= (\ovl{M})^*$ and $\ovl{M^{**}}\ehd(\ovl{M})^{**}$
\end{enumerate}
\end{lemma}

\begin{proof}
(1). First we note that both $R=S(\bfd)$ and $\ovl{R}=B(E,\LL(-\bfd),\sigma)$ are Auslander-Gorenstein and Cohen-Macaulay by Theorem~\ref{S(d) thm}(1)(2) and Theorem~\ref{B properties}(3). By Lemma~\ref{RSS 4.11}(1) we have that $I^*/R$ is $2$-pure. By Lemma~\ref{RSS 2.12}, $I^*\subseteq \Sg$. Moreover because $R$ is $g$-divisible, both $\Sg/R$ and $I^{*}/R$ are $g$-torsionfree right $R$-modules.\par
It is now left to prove $\ovl{I^*}\overset{\bullet}= (\ovl{I})^*$. The proof of the analogous part of \cite[Proposition~6.12(1)]{RSS} relies upon $T(\bfd)$ and $\ovl{T(\bfd)}$ being Auslander-Gorenstein and Cohen-Macaulay and \cite[Lemma~7.9]{RSS2}. Replacing \cite[Lemma~7.9]{RSS2} with Lemma~\ref{RSS2 7.9} and using the fact that $R$ and $\ovl{R}$ are Auslander-Gorenstein and Cohen-Macaulay one may follow that proof to obtain (1). The proof is quite technical and not particularly enlightening. For these reasons we omit the details with the interested reader being refer to \cite{RSS}.\par

(2). Since $R$ is Cohen-Macaulay, $\GK(M^{**}/M)\leq 1$ by Lemma~\ref{RSS 4.11}(1). Since $M$ is $g$-divisible, $M^{**}$ is also $g$-divisible by Lemma~\ref{RSS 2.13}(3). Lemma~\ref{RSS 2.14}(2) then gives $\ovl{M}\ehd \ovl{M^{**}}$. Since $\ovl{R}$ is Cohen-Macaulay, $\GK_{\ovl{R}}(\ovl{M}^{**}/\ovl{M})\leq 0$ by Lemma~\ref{RSS 4.11}(1). We hence have that
\begin{equation}\label{ovl ** commute} \ovl{M^{**}}\ehd \ovl{M}\ehd(\ovl{M})^{**}.\end{equation}\par
Now let $J=M^{*}$. As $M\supseteq R$, $J\subseteq R$ is a reflexive left ideal of $R$. Therefore, as $R$ is Cohen-Macaulay, the $R$-module $R/J$ is $2$-pure by Lemma~\ref{RSS 4.11}(2). Part (1) then applies and shows $(\ovl{J})^*\ehd\ovl{J^*}$. Becasue $\ovl{R}$ is Cohen-Macaulay, we have that $(\ovl{J})^{**}\ehd \ovl{J}$ by Lemma~\ref{RSS 4.11}(1). These with Lemma~\ref{CM lemma}(3) and (\ref{ovl ** commute}) gives
\begin{equation*}
\ovl{M^*}=\ovl{J}\ehd(\ovl{J})^{**}=(\ovl{J^*})^*=(\ovl{M^{**}})^* = (\ovl{M})^*.\qedhere
\end{equation*}
\end{proof}

Finally we need \cite[Lemma~6.14]{RSS}. This result is already general enough for us.

\begin{lemma}\label{RSS 6.14}\cite[Lemma~6.14]{RSS}
Let $B=B(E,\NN,\sigma)$ where $E$ is a smooth elliptic curve, $\deg\NN\geq 1$ and $\sigma$ is of infinite order. Let $M$ be a finitely generated graded right $B$-submodule of $\Bbbk(E)[t,t^{-1};\sigma]=Q_\gr(B)$. Then $M\ehd\bigoplus_{n\geq 0}H^0(E,\OO(\mbf{q})\otimes \NN_n)$ for some divisor $\mbf{q}$. Set $M^*=\Hom_B(M,B)\subseteq\Bbbk(E)[t,t^{-1};\sigma]$. Then:
\begin{enumerate}[(1)]
\item $\End_B(M)\ehd B(E,\NN(\mbf{q}-\mbf{q}^\sigma),\sigma)$;
\item  $NN^*\ehd\End_B(M)$;
\item $M^*\ehd \bigoplus_{n\geq 0}H^0(E,\NN_n\otimes\OO(-\mbf{q}^{\sigma^n}))$.\qed
\end{enumerate}
\end{lemma}

We are now able to describe the images of our maximal orders in $\ovl{S}$. Proposition~\ref{RSS 6.7} is our result corresponding to Rogalski, Sierra and Stafford's Theorem~\ref{real RSS 6.7}. Again we are able to substitute in our previous results and leave the proof largely unchanged.

\begin{prop}\label{RSS 6.7}
Let $\bfd$ be an effective divisor with $\deg\bfd\leq 2$, and let $R=S(\bfd)$. Suppose $M$ is a finitely generated $g$-divisible right $R$-module with $R\subseteq M\subseteq S$. Set $U=\End_R(M)$ and $F=\End_R(M^{**})$. Then there is an effective divisor $\mbf{y}$ such that
\begin{equation}\label{RSS 6.7 eq}
 \ovl{F}\ehd\ovl{U}\ehd B(E,\LL(-\mbf{x}),\sigma) \; \text{\, where }\mbf{x} =\bfd-\mbf{y}+\mbf{y}^\sigma.
 \end{equation}
Moreover, if $V=F\cap S$, then $U\subseteq V\subseteq F$, and $(V,F)$ is a maximal order pair.
\end{prop}
\begin{proof}
We first check $\ovl{U}\ehd \ovl{F}$. By Lemma~\ref{RSS 2.12}(3), $U$ is $g$-divisible, and hence noetherian by Proposition~\ref{RSS 2.9}. By Lemma~\ref{RSS 6.4(3)}, there exists an ideal $K$ of $F$ such that $K\subseteq U$ and $\GK(F/K)\leq 1$. Given $x\in K$, clearly $xF\subseteq U$; it follows that $F_U$ is finitely generated. Thus as a right $U$-module, $\GK_U(F/K)\leq 1$ also. This then implies, $\GK_U(F/U)\leq 1$.
Set $N=F/U$; one can check $N/Ng\cong \ovl{F}/\ovl{U}$. Moreover, because $U$ is $g$-divisible, $N$ is $g$-torsionfree. Therefore $\GK(\ovl{F}/\ovl{U})\leq 1-1=0$ by \cite[Proposition~5.1(e)]{KL}; or in other words $\dim_\Bbbk \ovl{F}/\ovl{U}< \infty$. \par
Now  by Lemma~\ref{RSS 6.12}, $\ovl{M^*}\ehd (\ovl{M})^*$, and so $\ovl{M}(\ovl{M^*})\ehd \ovl{M}(\ovl{M})^*$ by Lemma~\ref{Ab ehd (Ageqn)B}. Moreover, by Lemma~\ref{RSS 6.14}(2), $\ovl{M}(\ovl{M})^*\ehd \End_{\ovl{R}}(\ovl{M})$. Because clearly $MM^*\subseteq U$, we have
$$\ovl{U}\supseteq \ovl{M}(\ovl{M^*})\ehd \ovl{M}(\ovl{M})^*\ehd \End_{\ovl{R}}(\ovl{M}).$$
On the other hand, by Lemma~\ref{RSS 2.12}(3), $\ovl{U}=\ovl{\End_R(M)}\subseteq \End_{\ovl{R}}(\ovl{M})$. This forces
\begin{equation}\label{RSS 6.7 ovlU ehd End(ovlM)}
\ovl{F}\ehd \ovl{U}\ehd \End_{\ovl{R}}(\ovl{M}).
\end{equation}
\par
By Corollary~\ref{nc serre2}, we can write $\ovl{M}\ehd\bigoplus_{i\geq 0}H^0(E,\OO(\mbf{y})\otimes \LL(-\bfd)_i )$ for some divisor $\mbf{y}$ on $E$. Because $\ovl{M}\supseteq\ovl{R}$, $\mbf{y}$ must be an effective divisor. Lemma~\ref{RSS 6.14}(1) then tells us that
$$\ovl{F}\ehd\ovl{U}\ehd \End_{\ovl{R}}(\ovl{M})\ehd B(E,\LL(-\mbf{x}),\sigma) \; \text{ where } \mbf{x}=\mbf{d}-\mbf{y}+\mbf{y^\sigma},$$
which proves (\ref{RSS 6.7 eq}). The last statement follows from Lemma~\ref{max order pair equiv def}.
\end{proof}

\subsection{Blowups and virtual blowups}

In Chapter~\ref{The rings S(d)} we defined the ring $S(\bfd)$ for an effective divisor with $\deg\bfd\leq 2$. As mentioned in Remark~\ref{why blowup}, $S(\bfd)$ should be thought of as ``the blowup of $S$ at $\bfd$". The algebra $F$ from Proposition~\ref{RSS 6.7} has many similar properties with the $S(\bfd)$. Most notable is that $\ovl{F}\cong F/gF$ is equal in high degrees to a twisted homogenous coordinate ring. We will be calling such $F$ \textit{virtual blowups} (Definition~\ref{virtual blowup}).\par
We now investigate precisely what divisors $\bfx$ arise in Proposition~\ref{RSS 6.7}. These will be the divisors at which we can ``blow up".

\begin{definition}\label{virtually effective}\cite[Definition~7.1]{RSS}.\index{virtually effective divisor}
Let $\rho:E\to E$ be an automorphism of infinite order. A divisor $\bfx$ on $E$ with $\deg\bfx\geq 0$ is called \textit{$\rho$-virtually effective} if there exists $n\geq 0$ such that $\bfx+\bfx^\rho+\dots+\bfx^{\rho^{n-1}}$ is an effective divisor on $E$.
\end{definition}

When the automorphism $\rho:E\to E$ is clear from context, we often drop the $\rho$ from $\rho$-virtually effective.

\begin{example}\label{veff div eg}
An effective divisor on $E$ is trivially virtually effective. In the other direction, the divisor $p-p^\sigma+p^{\sigma^2}$ gives an example of a noneffective $\sigma$-virtually effective divisor.
\end{example}

\begin{lemma}\label{RSS 7.3(1)}
The divisor $\mbf{x}$ in (\ref{RSS 6.7 eq}) is $\sigma$-virtually effective.
\end{lemma}
\begin{proof}
Let $\mbf{x}$, $U$, $F$ be as in Proposition~\ref{RSS 6.7}. Put $B(E,\NN,\sigma)$ where $\NN=\LL(-\bfx)$. By Proposition~\ref{RSS 6.7}, $\ovl{U}\ehd \ovl{F}\ehd B$. Let $n\geq 0$ be such that  $\ovl{U}_n=B_n=H^0(E,\NN_n)$. Then because $\ovl{U}\subseteq \ovl{S}=B(E,\LL,\sigma)$, we have that $H^0(E,\NN_n)\subseteq H^0(E,\LL_n)$. Enlarging $n$ if necessary, we can assume both $\deg\LL_n\geq 2$ and $\deg\NN_n>2$. Then \cite[Corollary~IV.3.2]{Ha} implies $\LL_n$ and $\NN_n$ are generated by their global sections. Therefore we have $\NN_n=\LL_n(-[\mbf{x}]_n)\subseteq \LL_n$, and hence $[\mbf{x}]_n$ is effective by Lemma~\ref{on divisors}.
\end{proof}

Our definition of a virtually effective divisor is different from Rogalski, Sierra and Stafford's \cite[Definition~7.1]{RSS}. These are equivalent notions by the next result. We prefer Definition~\ref{virtually effective} for it is simpler to state and better motivates the name ``virtually effective". We also include \cite[Proposition~7.3(2)]{RSS} into Lemma~\ref{RSS 7.3(2)} which gives a very useful characterisation of virtually effective divisors.

\begin{lemma}\label{RSS 7.3(2)} \cite[Proposition~7.3(2)]{RSS}.
Let $\rho:E\to E$ be an automorphism of infinite order and $\bfx$ be a divisor on $E$. The following are equivalent:
\begin{enumerate}[(1)]
\item $\mbf{x}$ is $\rho$-virtually effective;
\item For each $\rho$-orbit $\mathbb{O}$ in $E$, pick $p\in \mathbb{O}$ such that $\mbf{x}|_{\mathbb{O}}=\sum_{i=0}^n a_ip^{\rho^i}$. For each $k\in \Z$, $\mbf{x}$ satisfies
$$\sum_{i\leq k} a_i \geq 0 \; \;\text{ and } \; \; \sum_{i\geq k}a_i \geq 0;$$
\item $\mbf{x}$ can be written as $\mbf{x}=\mbf{u}-\mbf{v}+\mbf{v}^\rho$ where $\mbf{u}$ is an effective divisor supported on distinct $\rho$-orbits, and $\mbf{v}$ is an effective divisor such that $0\leq \mbf{v}\leq\mbf{u}+\mbf{u}^\rho+\dots+\mbf{u}^{\rho^{k-1}}$ for some $k\geq 1$.
\end{enumerate}
\end{lemma}
\begin{proof} For this proof we use Notation~\ref{[d]_n} with $\rho$ in place of $\sigma$. That is, we write $[\bfx]_n=\bfx+\bfx^\rho+\dots+\bfx^{\rho^{n-1}}$. We also introduce the temporary notation $p_i=p^{\rho^i}=\rho^{-i}(p)$ for $p\in E$ and $i\in\Z$. \par

$(1)\Rightarrow(2)$. Fix a $\rho$-orbit $\bbO$ of $E$ and choose $p\in \bbO$ such that $\mbf{x}|_{\mathbb{O}}=\sum_{i=0}^n a_ip_i$ for some $p=p_0\in E$ and $a_i\in\Z$. It clearly suffices to prove (2) on the $\sigma$-orbit $\bbO$. Hence without loss of generality we assume $\bfx=\bfx|_\bbO$. Given $m\in\N$, we calculate
$$[\mbf{x}]_m=\sum_{i=0}^n a_ip_i+\sum_{i=0}^n a_ip_{i+1}+\dots+\sum_{i=0}^n a_ip_{i+m-1}= \sum_{i=0}^n a_i(p_i+p_{i+1}+\dots+p_{i+m-1}). $$
Furthermore, if $m\geq n$ we get
\begin{multline}\label{(1)=>(2) eq}[\mbf{x}]_m=a_0p_0+(a_0+a_1)p_1+\dots +( \sum_{i\leq k} a_i)p_k+\dots +  (\sum_{i\leq n-1}a_i)p_{n-1} \\
\hspace{2cm}  +(\sum_{i\geq 0} a_i)(p_n+\dots+p_{m-1})+ (\sum_{i\geq 1} a_i)p_{m}+\dots +(\sum_{i\geq k} a_i)p_{m+k-1}+\dots +a_np_{n+m-1}.
\end{multline}
Now suppose that $[\mbf{x}]_m$ is effective for some $m\geq 0$. By enlarging $m$ if necessary we can assume that $m\geq n$. In which case the coefficient of each $p_i$ in equation (\ref{(1)=>(2) eq}) must be positive. In other words,
$$\sum_{i\leq k} a_i \geq 0 \; \;\text{ and } \; \; \sum_{i\geq k}a_i \geq 0\;\text{ for all }k\in\Z.$$
\par
$(2)\Rightarrow(3)$. This is one direction of \cite[Proposition~7.3(2)]{RSS}. \par

$(3)\Rightarrow (1)$. Suppose that $\mbf{x}=\mbf{u}-\mbf{v}+\mbf{v}^\rho$  as in (3), with say $\mbf{v}\leq[\mbf{u}]_k$. Then
\begin{equation*}
[\bfx]_{k}=[\mbf{u}]_k-[\mbf{v}]_k+[\mbf{v}^\rho]_k=[\mbf{u}]_k-\mbf{v}+\mbf{v}^{\rho^k},
\end{equation*}
which is clearly effective.
\end{proof}

To avoid confusion, we now reserve the notation $[\bfx]_n=\bfx+\bfx^\sigma+\dots+\bfx^{\sigma^{n-1}}$ from Notation~\ref{[d]_n} for the automorphism $\sigma$ from Hypothesis~\ref{standing assumption 2}. This is particularly relevant in the next lemma.

\begin{lemma}\label{[x]_3 is v effective}
Let $\bfx$ be a $\sigma$-virtually effective divisor. Then $[\bfx]_3=\bfx+\bfx^\sigma+\bfx^{\sigma^2}$ is $\tau$-virtually effective, where $\tau=\sigma^3$.
\end{lemma}
\begin{proof}
Suppose that $[\bfx]_n$ is effective. If necessary, enlarge $n$ so that $n$ is divisible by 3, say $n=3m$. We then have
\begin{multline*}
[\bfx]_3+[\bfx]_3^\tau+\dots+[\bfx]_3^{\tau^{m-1}}=(\mbf{x}+\mbf{x}^\sigma+ \mbf{x}^{\sigma^2})+(\mbf{x}+\mbf{x}^\sigma+\mbf{x}^{\sigma^2})^{\tau}+
\dots+(\mbf{x}+\mbf{x}^\sigma+\mbf{x}^{\sigma^2})^{\tau^{m-1}}\\
=\bfx+\bfx^\sigma+\dots+\bfx^{\sigma^{n-1}}=[\mbf{x}]_n.
\end{multline*}
This is effective by assumption; hence $[\bfx]_3$ is $\tau$-virtually effective.
\end{proof}

We now turn our attention away from divisors and back to subalgebras of $S$. We start with the definition of a virtual blowup. Below we also recall the definition of a virtual blowup of $T$.

\begin{definition}\label{virtual blowup}\index{virtual blowup of $S$}
\begin{enumerate}[(1)]
\item Let $\bfx$ be a $\sigma$-virtually effective divisor on $E$ with $\deg\bfx\leq 2$. We say that a cg subalgebra $F$ of $\Sg$ with $Q_\gr(F)=Q_\gr(S)$ is a \textit{virtual blowup of $S$ at $\bfx$} if:
\begin{enumerate}[(a)]
\item $F$ is a part of a maximal order pair $(F\cap S,F)$ of $S$.
\item $\ovl{F}\ehd B(E,\LL(-\bfx),\sigma)$.
\end{enumerate}
\item \cite[Definition~6.9 and Definition~7.1]{RSS}. Retain Notation~\ref{3 Veronese notation2} for $T$. Let $\bfx$ be a $\tau$-virtually effective divisor on $E$ such that $\deg\bfx\leq 8$. We say that a cg subalgebra $F$ of $T_{(g)}$ with $Q_\gr(F)=Q_\gr(T)$ is a \textit{virtual blowup of $T$ at $\bfx$} if:
\begin{enumerate}[(a)]
\item $F$ is a part of a maximal order pair $(F\cap S,F)$ of $T$.
\item $\ovl{F}\ehd B(E,\MM(-\bfx),\tau)$.
\end{enumerate}
\end{enumerate}
\end{definition}

In Definition~\ref{vblowup def bg} we gave a different definition of a virtual blowup of $T$. Lemma~\ref{max order pair equiv def} shows that it is an equivalent to Definition~\ref{virtual blowup}(2).

\begin{example}\label{effetive blowup is virtual blowup}
If $\bfd$ is an effective divisor on $E$ with $\deg\bfd\leq 2$, then by Theorem~\ref{S(d) thm}, $S(\bfd)$ indeed satisfies Definition~\ref{virtual blowup}.
\end{example}

For an explicit example of a blowup at the virtually effective divisor $p-p^\sigma+p^{\sigma^2}$ see Theorem~\ref{S(p-p1+p2)}.

\begin{remark}\label{why vblowup}
Like with the $S(\bfd)$, our main motivation for the name ``virtual blowup" is by analogy with \cite{RSS}. On top of this we will also prove in Theorem~\ref{3 Veronese of virtual blowup} that if $F$ is a virtual blowup of $S$, then $F^{(3)}$ is a virtual blowup of $T=S^{(3)}$. Justification for their case is given in \cite[Remark~7.5]{RSS}.
\end{remark}

With our new language we can give our main result of this chapter.

\begin{theorem}\label{RSS 7.4}
\begin{enumerate}[(1)]
\item Let $V\subseteq S$ be a $g$-divisible cg maximal $S$-order. Then:
\begin{enumerate}[(a)]
\item there is a maximal order $F\supseteq V$ such that $(V,F)$ is a maximal order pair;
\item $F$ is a virtual blowup of $S$ at a virtually effective divisor $\mbf{x}$ with $\deg\mbf{x}\leq 2$;
\item $\ovl{V}\ehd\ovl{F}\ehd B(E,\LL(-\mbf{x}),\sigma)$.
\end{enumerate}
\item If $U\subseteq S$ is any $g$-divisible cg subalgebra with $Q_\gr(U)=Q_\gr(S)$, then there exists a maximal order pair $(V,F)$ as in (1), such that $U$ is contained in, and equivalent, to $V$ and $F$.
\end{enumerate}
\end{theorem}

\begin{proof}
(1). By definition $Q_\gr(V)=Q_\gr(S)$. By Corollary~\ref{RSS 6.6}(2) and Lemma~\ref{max order pair equiv def}, $V$ is part of a maximal order pair $(V,F)$. This proves part (a). By Lemma~\ref{max order pair equiv def}, $F=\End_{S(\bfd)}(M^{**})$ for some effective divisor $\bfd$ with $\deg\bfd\leq 2$, and a finitely generated $g$-divisible right $S(\bfd)$-module $M$ satisfying $S(\bfd)\subseteq M\subseteq S$. By Proposition~\ref{RSS 6.7} and Lemma~\ref{RSS 7.3(1)}, $F$ is a virtual blowup of $S$ at a virtually effective divisor $\bfx$. Part ($c$) follows directly from part ($b$).\par

(2). By Theorem~\ref{RSS 5.25}, $U$ is contained in and equivalent to some $\End_{S(\bfd)}(M)$, where $\bfd$ is an effective divisor of $\deg\bfd\leq 2$ and $M=\widehat{US(\bfd)}$. Clearly $S(\bfd)\subseteq M\subseteq S$ and so we can apply Proposition~\ref{RSS 6.4}.
\end{proof}

In addition to a complete description of $g$-divisible maximal $S$-orders, we are able to obtain a description of any $g$-divisible subalgebra. We recall that the \textit{idealiser}\index{idealiser} of a right ideal $J$ of a ring $A$ is defined as $\I_{A}(J)=\{a\in A\,|\; aJ\subseteq J\}$.\index[n]{ia@$\I_A(J)$} In other words, it is the largest subalgebra of $A$ in which $I$ is a two-sided ideal. Similarly, we can define an idealiser for a left ideal of $A$.

\begin{cor}\label{RSS 7.6}
Let $U\subseteq S$ be a $g$-divisible subalgebra with $Q_\gr(U)=Q_\gr(S)$. Then $U$ is an iterated sub-idealiser inside a virtual blowup of $S$. More precisely the following holds.
\begin{enumerate}[(1)]
\item There exists a virtually effective divisor $\mbf{x}=\mbf{u}-\mbf{v}+\mbf{v^\sigma}$ with $\deg\mbf{x}\leq 2$, and a blowup $F$ of $S$ at $\mbf{x}$, such that $V=F\cap S$ contains, and is equivalent to, $U$. The pair $(V,F)$ is a maximal order pair.
\item There is a $g$-divisible algebra $W$ with $U\subseteq W\subseteq V$ and such that $U$ is a right sub-idealiser inside $W$, and $W$ is a left sub-idealiser inside $V$. In more detail:
    \begin{enumerate}[(a)]
    \item There exists a $g$-divisible left ideal $L$ of $V$ such that either $L=V$ or $V/L$ is $2$-pure. There exists a $g$-divisible ideal $K$ of $X=\I_V(L)$ such that $K\subseteq W\subseteq X$ and $\GK_X(X/K)\leq 1$.
    \item $V$ is a finitely generated left $W$-module, while $X/K$ is a finitely generated $\Bbbk[g]$-module. In particular $X$ is finitely generated over $W$ on both sides;
    \item The properties given for $W\subseteq V$ also hold true for the pair $U\subseteq W$ with left and right interchanged.
    \end{enumerate}
\end{enumerate}
\end{cor}
\begin{proof}
(1). This is just Theorem~\ref{RSS 7.4}(2). \par
(2). The proof of \cite[Corollary~7.6]{RSS} carries though here. One just needs to replace the references to \cite[Section~2]{RSS} with the appropriate results from our Chapter~\ref{prelims}.
\end{proof}

Theorem~\ref{RSS 7.4} shows any $g$-divisible maximal $S$-order is a virtual blowup at some virtually effective divisor. Conversely, we now show that given any virtually effective divisor $\bfx$ with $\deg\bfx\leq 2$, there exists a virtual blowup $F$ of $S$ at $\bfx$, where $V=F\cap S$ is necessarily a maximal $S$-order. This completes the classification of $g$-divisible maximal $S$-orders. \par

For the proof we need to recall the rings $T(\bfd)$ from Definition~\ref{T(d) def}. In Lemma~\ref{RSS2 5.10}, $T$ is graded via $T_n=T\cap S_n$.

\begin{lemma}\label{RSS2 5.10}
Let $\mbf{u}$ and $\mbf{v}$ be effective divisors on $E$ such that $\deg\mbf{u}\leq 2$ and $\mbf{v}\leq [\mbf{u}]_k$ for some $k\geq 1$. Then there exists a $g$-divisible right $S(\mbf{u})$-module $M$ with $S(\mbf{u})\subseteq M\subseteq S$ and such that
\begin{equation}\label{ehd 1 RSS2 5.10} \ovl{M}\ehd\bigoplus_{n\geq 0}H^0(E,\LL_n(-[\mbf{u}]_n+\mbf{v})).\end{equation}
\end{lemma}
\begin{proof}
By Theorem~\ref{S(d) thm}(1) with $\mbf{u}=\bfd$ we have $S(\mbf{u})^{(3)}=T([\mbf{u}]_3)$. By  \cite[Lemma~5.10]{RSS2} (the $T$-version of Lemma~\ref{RSS2 5.10}) there exists a $g$-divisible right $T([\mbf{u}]_3)$-module $N$ with $T([\mbf{u}]_3)\subseteq N\subseteq T$ and such that
\begin{equation}\label{RSS 7.4(3) eq1} \ovl{N}\ehd
\bigoplus_{m\geq0}H^0(E,\LL_{3m}(-[\mbf{u}]_{3m}+\mbf{v})).\end{equation}
Set $R=S(\mbf{u})$ and $M=\widehat{NR}$. Then $M$ is a $g$-divisible right $R$-module. As $N$ is a finitely generated right $R^{(3)}$-module, $NR$ is a finitely generated right $R$-module. By Lemma~\ref{RSS 2.13}(2), $M_R$ is then finitely generated also. Further, as $1\in N\subseteq NT=T$, we have $R\subseteq M\subseteq S$. This leaves (\ref{ehd 1 RSS2 5.10}) to be proven.\par
Now because $N$ is $g$-divisible and clearly $(NR)^{(3)}=N$, it is easy to see $M^{(3)}=N$. On the other hand, since $\ovl{R}=B(E,\LL(-\mbf{u}),\sigma)$, the Noncommutative Serre's Theorem (Corollary~\ref{nc serre2}) says that the $\ovl{R}$-module $\ovl{M}$ satisfies,
\begin{equation*}\label{RSS 7.4(3) eq2}
\ovl{M}\ehd \bigoplus_{n\in\N}H^0(E,\LL_{n}(-[\mbf{u}]_{n}+\mbf{w})) \text{ for some divisor }\mbf{w}.
\end{equation*}
Equation (\ref{RSS 7.4(3) eq1}) and $M^{(3)}=N$ then implies
\begin{equation}\label{gb secs RSS 5.10} H^0(E,\LL_{3m}(-[\mbf{u}]_{3m}+\mbf{v}))=H^0(E,\LL_{3m}(-[\mbf{u}]_{3m}+\mbf{w}))\;
\text{ for }m\gg0.\end{equation}
Enlarging $m$ if necessary, we can assume $\LL_{3m}(-[\mbf{u}]_{3m}+\mbf{v}))$ and $\LL_{3m}(-[\mbf{u}]_{3m}+\mbf{w}))$ are generated by their global sections. Thus (\ref{gb secs RSS 5.10}) implies that $\LL_{3m}(-[\mbf{u}]_{3m}+\mbf{v}))=\LL_{3m}(-[\mbf{u}]_{3m}+\mbf{w}))$. Hence we get $\mbf{v}=\mbf{w}$, proving (\ref{ehd 1 RSS2 5.10}).
\end{proof}

\begin{prop}\label{RSS 7.4(3)}
Let $\mbf{x}$ be a $\sigma$-virtually effective divisor with $\deg\mbf{x}\leq 2$. Then there exists a blowup $F$ of $S$ at $\mbf{x}$.
\end{prop}
\begin{proof}
Using Lemma~\ref{RSS 7.3(2)}, write $\mbf{x}=\mbf{u}-\mbf{v}+\mbf{v}^\sigma$ where $\mbf{u}$ and $\mbf{v}$ are effective and such that $0\leq\mbf{v}\leq [\mbf{u}]_k$ for some $k\geq 0$. By Lemma~\ref{RSS2 5.10} there exists a $g$-divisible right $S(\bfd)$-module $M$ with $S(\bfd)\subseteq M\subseteq S$ and such that
$$\ovl{M}\ehd\bigoplus_{n\geq 0}H^0(E,\LL_n(-[\mbf{u}]_n+\mbf{v})).$$
Put $U=\End_R(M)$ and $F=\End_R(M^{**})$. By equation (\ref{RSS 6.7 ovlU ehd End(ovlM)}) and Lemma~\ref{RSS 6.14},
$$\ovl{F}\ehd\ovl{U}\ehd B(E,\LL(-\bfx),\sigma).$$
By Lemma~\ref{max order pair equiv def}(1), $(F\cap S,F)$ is a maximal order pair.
\end{proof}

\subsection{Further study of virtual blowups}\label{further vblowups}

We end this Chapter with a further study of virtual blowups. The most important of the results here, and our first aim, is to show that the 3-Veroneses of our virtual blowups are virtual blowups of $S^{(3)}$.

\begin{lemma}\label{in B is Veroneses equal then ehd}
Let $B=B(E,\NN,\sigma)$ for some invertible sheaf $\NN$ on $E$ with $\deg\NN\geq 1$. Let $M,N$ be two finitely generated right $B$-submodules of $Q_\gr(B)=\Bbbk(E)[t,t^{-1};\sigma]$. If $M^{(d)}=N^{(d)}$ for some $d\geq 1$, then $M\ehd N$.
\end{lemma}
\begin{proof}
By Corollary~\ref{nc serre2}, there exists divisors $\bfx$ and $\bfy$ such that
\begin{equation}\label{in B is Veroneses equal then ehd eq}
 M\ehd \bigoplus_{n\geq 0} H^0(E,\OO(\bfx)\otimes\NN_n)\;\text{ and }\;N\ehd \bigoplus_{n\geq 0}H^0(E,\OO(\bfy)\otimes\NN_n).
 \end{equation}
Let $n\geq0$. By assumption $M_{dn}=N_{dn}$, thus $H^0(E,\OO(\bfx)\otimes\NN_{dn})=H^0(E,\OO(\bfy)\otimes\NN_{dn})$. By enlarging $n$ if necessary, we can assume that both $\deg(\OO(\bfx)\otimes\NN_{dn})\geq 2$ and $\deg(\OO(\bfy)\otimes\NN_{dn})\geq 2$. In which case both sheaves are generated by their global sections by \cite[Corollary IV.3.2]{Ha}. Hence $\OO(\bfx)\otimes\NN_{dn}=\OO(\bfy)\otimes\NN_{dn}$. It follows $\bfx=\bfy$  in (\ref{in B is Veroneses equal then ehd eq}), and then $M\ehd N$.
\end{proof}

\begin{lemma}\label{NR**=M**}
Let $R=S(\bfd)$ for some effective divisor of $\deg\bfd\leq 2$, and $M\subseteq\Sg$ be a finitely generated $g$-divisible right $R$-module. Let $d\geq 1$ and write $N=M^{(d)}$. Then
\begin{equation*}(NR)^{**}=M^{**}.\end{equation*}
\end{lemma}
\begin{proof} Now since $(NR)^{(d)}=N=M^{(d)}$, $(\ovl{NR})^{(d)}=\ovl{M}^{(d)}$. So by Lemma~\ref{in B is Veroneses equal then ehd}, $\ovl{NR}\ehd\ovl{M}$. Since $M$ is $g$-divisible and clearly $NR\subseteq M$, Lemma~\ref{RSS 2.14}(1) applies and shows $\GK(M/NR)\leq 1$. By Theorem~\ref{S(d) thm}(2), $R$ is Cohen-Macaulay. Therefore, as $\GK(M/NR)\leq 1 =\GK(R)-2$, we have $M\subseteq (NR)^{**}$ by Lemma~\ref{RSS 4.11}(1). It then follows $M^{**}\subseteq (NR)^{**}$. Since $NR\subseteq M$, we also have the reverse inclusion $(NR)^{**}\subseteq M^{**}$.
\end{proof}

\begin{lemma}\label{Veronese ** commute}
Let $\bfd$ be an effective divisor with $\deg\bfd\leq 2$ and let $R=S(\bfd)$. Suppose that $R\subseteq M\subseteq S$ is a finitely generated $g$-divisible right $R$-module. Then
$$(M^{**})^{(3)}=(M^{(3)})^{**}.$$
\end{lemma}
\begin{proof}
Let $R'=R^{(3)}$ and $N=M^{(3)}$. By Theorem~\ref{S(d) thm}(2), $R$ is Auslander-Gorenstein and Cohen-Macaulay. Hence by Lemma~\ref{RSS 4.11}(1), $\GK_{R}(M^{**}/M)\leq 1$. Lemma~\ref{noeth up n down}(2) implies that $M^{**}/M$ is finitely generated as a right $R'$-module. Hence by (\ref{GK(M)}), $\GK_{R'}(M^{**}/M)=\GK_{R}(M^{**}/M)$. As right $R'$-modules clearly
$$\frac{(M^{**})^{(3)}}{N}=\left(\frac{M^{**}}{M}\right)^{(3)}\subseteq \frac{M^{**}}{M},$$
and so we have that
\begin{equation}\label{GK(M**(3)/N)} \GK_{R'}((M^{**})^{(3)}/N)\leq \GK_{R'}(M^{**}/M)\leq 1.\end{equation}
Now $R'=T([\bfd]_3)$ by Theorem~\ref{S(d) thm}(1), so $R'$ is Auslander-Gorenstein and Cohen-Macaulay by Theorem~\ref{T(d) properties}(3). Hence by Lemma~\ref{RSS 4.11}(1) applied to $R'$, $N$, and by (\ref{GK(M**(3)/N)}), $(M^{**})^{(3)}\subseteq N^{**}$.\par

Conversely, consider $L=N^{**}R/NR$ as a right $R'$-module. Clearly $L_{R'}$ is finitely generated and is a homomorphic image of $(N^{**}/N)\otimes_{R'} R$. We therefore have
$$\GK_R(L)=\GK_{R'}(L)\leq \GK_{R'}(N^{**}/N\otimes_{R'} R)\leq \GK_{R'}(N^{**}/N)\leq 1,$$
where the second inequality follows from \cite[Proposition~5.6]{KL}, and the third from Lemma~\ref{RSS 4.11}(1). By Lemma~\ref{RSS 4.11}(1) applied to $L$ and $R$, and by Lemma~\ref{NR**=M**}
\begin{equation}\label{N** in M**}N^{**}R\subseteq (NR)^{**}=M^{**}.\end{equation}
Taking the 3rd Veronese in (\ref{N** in M**}) gives the reverse inclusion $N^{**}\subseteq (M^{**})^{(3)}$.
\end{proof}

For the next theorem we recall Notation~\ref{3 Veronese notation2} for $T=S^{(3)}$.

\begin{theorem}\label{3 Veronese of virtual blowup}
Let $F$ be a virtual blowup of $S$ at a $\sigma$-virtually effective divisor $\bfx$ and set $F'=F^{(3)}$. Then $F'$ is a virtual blowup of $T$ at the $\tau$-virtually effective divisor $\bfy=\bfx+\bfx^\sigma+\bfx^{\sigma^2}$.
\end{theorem}
\begin{proof}
First we note that $F$ indeed exists by Proposition~\ref{RSS 7.4(3)}, while $\bfy$ is $\tau$-virtually effective by Corollary~\ref{[x]_3 is v effective}. \par
On one hand $\ovl{F}\ehd B(E,\LL(-\bfx),\sigma)$ by definition; while on the other,
\begin{multline*} B(E,\LL(-\bfx),\sigma)_{3n}=H^0(E,\LL(-\bfx)_{3n})= H^0(E,\LL_{3n}(-\bfx-\bfx^\sigma-\dots-\bfx^{\sigma^{3n-1}}))\\
=H^0(E,\MM_{n}(-\bfy-\bfy^\tau-\dots-\bfy^{\tau^{n-1}}))= H^0(E,\MM(-\bfy)_n)=B(E,\MM(-\bfy),\tau)_n.
\end{multline*}
for all $n\geq 1$. Thus $\ovl{F}^{(3)}\ehd B(E,\MM(-\bfy),\tau)$. It is hence left to prove that $(F'\cap T,F')$ is a maximal order pair of $T$. \par
Apply Lemma~\ref{max order pair equiv def}(1) to the maximal order pair $(F\cap S,F)$ of $S$. We get an effective divisor $\bfd$ with $\deg\bfd\leq 2$, and a finitely generated $g$-divisible right $S(\bfd)$-module $M$ such that $S(\bfd)\subseteq M\subseteq S$ and $F=\End_{S(\bfd)}(M^{**})$. Let $R=S(\bfd)$, $R'=R^{(3)}$ and $N=M^{(3)}$. Since $M_R$ is $g$-divisible and finitely generated, $N_{R'}$ is also by Lemma~\ref{g div up} and Lemma~\ref{noeth up n down}(2). In addition, $R'\subseteq N\subseteq T$ easily follows from $R\subseteq M\subseteq S$. By Theorem~\ref{S(d) thm}(1) we know that $R'=T(\bfe)$ where $\bfe=[\bfd]_3$ has $\deg\bfe=3\deg\bfd\leq 6$. We claim $F'$, $R'$, $N$ satisfies Lemma~\ref{max order pair equiv def}(2b). For this it remains to prove $F'=\End_{R'}(N^{**})$.\par
Let $G=\End_{R'}(N^{**})$. Note that since $N^{**}=(M^{**})^{(3)}$ by Lemma~\ref{Veronese ** commute},
\begin{equation}\label{F' subset G}F'N^{**}\subseteq FM^{**}\cap T_{(g)}=M^{**}\cap T_{(g)}=N^{**}.
\end{equation}
Hence $F'\subseteq G$. Now consider $W=\End_R(N^{**}R)$ and $W'=W^{(3)}$. Similar to (\ref{F' subset G}) (with $N^{**}R$ in place of $M^{**}$) one can show $W'\subseteq G$. Clearly also $GN^{**}R\subseteq N^{**}R$. This shows $G\subseteq W$, and hence $G=W'$. Now by Remark~\ref{RSS 6.2}, $\End_R((N^{**}R)^{**})$ is the unique maximal order equivalent to and containing $W$. But by Lemma~\ref{Veronese ** commute} and Lemma~\ref{NR**=M**} (with $d=3$ and $M$ replaced by $M^{**}$),
$$(N^{**}R)^{**}=((M^{**})^{(3)}R)^{**}=(M^{**})^{**}=M^{**}.$$
So we have $W\subseteq F$, and thus $W'=G\subseteq F'$. So $F'=\End_{R'}(N^{**})$ and, by Lemma~\ref{max order pair equiv def}(2), $(F'\cap T,F')$ is a maximal order pair of $T$.
\end{proof}

Next we investigate when two virtual blowups are equivalent orders to each other. The following definition is from \cite[Section~5]{RSS}.

\begin{definition}
Two divisors $\bfx$ and $\bfy$ on $E$ are called \textit{$\sigma$-equivalent} if on every $\sigma$-orbit $\mathbb{O}$ of $E$, we have $\deg(\bfx|_{\mathbb{O}})=\deg(\bfy|_{\mathbb{O}})$.
\end{definition}

\begin{prop}\label{RSS 5.27}
Let $\bfx$ and $\bfx'$ be two virtually effective divisors on $E$, and let $F$ and $F'$ be two virtual blowups at $\bfx$ and $\bfx'$ respectively. Then $F$ and $F'$ are equivalent orders if and only if $\bfx$ and $\bfx'$ are $\sigma$-equivalent.
\end{prop}
\begin{proof}
Let $V=F\cap S$ and $V'=F'\cap S$ be the corresponding maximal $S$-orders. By definition $V$ and $V'$ are equivalent orders to $F$ and $F'$ respectively. Hence it is enough to prove the statement with $V$ and $V'$ in place of $F$ and $F'$. \par
($\Leftarrow$). The key to this proof is the construction of a normalised divisor in Proposition~\ref{RSS 5.7} and the observations of Remark~\ref{norm div not unique}. Fix a $\sigma$-orbit $\mathbb{O}$ and choose a $p\in \mathbb{O}$ such that both $\bfx|_{\bbO}$ and $\bfx'|_{\bbO}$ are supported on $\{ p^{\sigma^i}\;|\, i\geq 0\}$. Since $\bfx$ and $\bfx'$ are $\sigma$-equivalent, $\deg\bfx|_{\bbO}=\deg\bfx'|_{\bbO}=e_p$ say. Put $\bfd=\sum_{\bbO} e_p p$, where the sum ranges over every $\sigma$-orbit of $E$, and for each $\sigma$-orbit $\bbO$, $p$ and $e_p$ are chosen as above. Then $\bfd$ is a normalised divisor for both $V$ and $V'$. By Theorem~\ref{RSS 5.25}, $V$ and $V'$ are both equivalent orders to $S(\bfd)$, and therefore also to each other.\par

($\Rightarrow$). Now let $\bfd$ and $\bfd'$ be normalised divisors for $\bfx$ and $\bfx'$ respectively. By Theorem~\ref{RSS 5.25}, $V$ and $V'$ are equivalent orders to $R=S(\bfd)$ and $R'=S(\bfd')$ respectively. Hence, if $V$ and $V'$ are equivalent orders so are $R$ and $R'$. By Lemma~\ref{RSS 2.16=>} , $\ovl{R}$ and $\ovl{R'}$ are then equivalent orders. It then follows from \cite[Proposition~3.1.14]{MR} that there exists an $(\ovl{R},\ovl{R'})$-bimodule $M$. By the same result, the bimodule $M$ is finitely generated on both sides because both $\ovl{R}$ and $\ovl{R'}$ are noetherian. By \cite[Lemma~5.1]{RSS} $\bfd$ and $\bfd'$ are then $\sigma$-equivalent. By construction a divisor is $\sigma$-equivalent to any of its normalised divisors; in particular $\bfx$ and $\bfx'$ are $\sigma$-equivalent to $\bfd$ and $\bfd'$ respectively. Thus $\bfx$ and $\bfx'$ are $\sigma$-equivalent also.
\end{proof}

\begin{remark}\label{RSS 5.27 works in T}
The proof of Proposition~\ref{RSS 5.27} can be repeated to obtain the analogous statement for subalgebras of $T$. This would be a generalisation of \cite[Corollary~5.26]{RSS}.
\end{remark}

In Example~\ref{effetive blowup is virtual blowup} we noted that, given an effective divisor $\bfd$ with $\deg\bfd\leq 2$, the blowup $S(\bfd)$ is a virtual blowup at $\bfd$. Suppose $F$ is any other virtual blowup at $\bfd$. By Proposition~\ref{RSS 5.27} we know $F$ and $S(\bfd)$ are equivalent orders.

\begin{question}\label{unique}
 \begin{enumerate}[(1)]
\item Let $\bfd$ be an effective divisor with $\deg\bfd\leq 2$. Is $S(\bfd)$ the unique virtual blowup (in the sense of Definition~\ref{virtual blowup}) of $S$ at $\bfd$?
\item More generally, let $\bfx$ be a virtually effective divisor with $\deg\bfx\leq 2$. By  Proposition~\ref{RSS 7.4(3)} a virtual blowup $F$ of $\bfx$ at $S$ exists. Is $F$ the unique blowup of $S$ at $\bfx$?
\end{enumerate}
\end{question}

The author feels that the Cohen-Macaulay property, satisfied by $S(\bfd)$, should be enough to give a positive answer to the Question~\ref{unique}(1). On the other hand, in general, a virtual blowup will not be Cohen-Macaulay (see Corollary~\ref{S(p-p1+p2) bad homologically}). Question~\ref{unique}(2) then feels false since there should many different modules $M$ appearing in Lemma~\ref{RSS2 5.10}, and so many different virtual blowups $F=\End_R(M^{**})$ arising in Proposition~\ref{RSS 7.4(3)}. However, constructing such an example in $S$ proves difficult. This is something that was achieved for virtual blowups of $T$ as alluded to in \cite[Remark~10.7]{RSS}. An analogous question to Question~\ref{unique}(1) in $T$ was left unasked. The best we are able to do in the direction of Question~\ref{unique}(1) is the following proposition.

\begin{prop}\label{vblowup eff div}
Let $F$ be a virtual blowup of $S$ at an effective divisor $\bfd$ with $\deg\bfd\leq 2$. Suppose that the divisor $\bfd$ is supported on points with distinct $\sigma$-orbits. Then $F=S(\bfd)$. In other words $S(\bfd)$ is the unique blowup of $S$ at $\bfd$.
\end{prop}

\begin{proof}
Let $V=F\cap S$. Apply Theorem~\ref{RSS 5.25} to $V$. We find an effective divisor $\bfd'$ with $\deg\bfd'\leq 2$ that is supported at points with distinct $\sigma$-orbits, and such that $\widehat{VS(\bfd')}$ is finitely generated as a right $S(\bfd')$-module. The divisor $\bfd'$ is constructed in Proposition~\ref{RSS 5.7} and can be any normalised divisor for $\bfd$. If one looks into this construction, one sees that because $\bfd$ is supported on points with distinct $\sigma$-orbits, we can take $\bfd'=\bfd$. \par
Set $R=S(\bfd)$ and $M=\widehat{VR}$, and note that $M_R$ is finitely generated by Theorem~\ref{RSS 5.25}. Then $\ovl{M}_{\ovl{R}}$ (which is clearly nonzero) is also finitely generated. Hence by the Noncommutative Serre's Theorem (specifically Corollary~\ref{nc serre2})
\begin{equation}\label{vblowup eff div eq1}
\ovl{M}\ehd\bigoplus_{n\geq 1}H^0(E,\LL_n(-[\bfd]_n+\mbf{z}))\;\text{ for some divisor }\mbf{z}.
\end{equation}
By Theorem~\ref{RSS 5.25}, $V\subseteq W=\End_R(M)\subseteq S$, and, $V$ and $W$ are equivalent orders. Since $V$ is a maximal $S$-order, $V=W$. Hence by equation~(\ref{RSS 6.7 ovlU ehd End(ovlM)}) and Lemma~\ref{RSS 6.14}
\begin{equation}\label{vblowup eff div eq2}
\ovl{V}\ehd B(E,\LL(-\bfd+\mbf{z}-\mbf{z}^\sigma),\sigma)=
\bigoplus_{n\geq 0}H^0(E,\LL_n(-[\bfd]_n+\mbf{z}-\mbf{z}^\sigma)).
\end{equation}
But by definition of a virtual blowup at $\bfd$,
\begin{equation}\label{vblowup eff div eq3}
\ovl{V}\ehd B(E,\LL(-\bfd),\sigma)=\bigoplus_{n\geq 0}H^0(E,\LL_n(-[\bfd]_n)).
\end{equation}
Take $n\gg0$. Then both $\LL_n(-[\bfd]_n+\mbf{z}-\mbf{z}^\sigma)$ and  $\LL_n(-[\bfd]_n)$ are generated by their global sections \cite[Corollary~IV.3.2]{Ha}. Hence equations (\ref{vblowup eff div eq2}) and (\ref{vblowup eff div eq3}) imply that $\LL_n(-[\bfd]_n+\mbf{z}-\mbf{z}^\sigma)=\LL_n(-[\bfd]_n)$. So we have $[\bfd]_n+\mbf{z}-\mbf{z}^\sigma=[\bfd]_n$. Since $\sigma$ has no fixed points (see Remark~\ref{sigma is translation}) we must have $\mbf{z}=0$. But then (\ref{vblowup eff div eq1}) reads $\ovl{M}\ehd\ovl{R}$. Clearly $\ovl{R}\subseteq\ovl{M}$, and so $\GK(M/R)\leq 1$ by Lemma~\ref{RSS 2.14}(2). But $R$ is Cohen-Macaulay by Theorem~\ref{S(d) thm}(2), therefore Lemma~\ref{RSS 4.11}(1) implies $R=M$. It then follows $V\subseteq R$. \par
Now $V$ and $R$ are equivalent orders (by Proposition~\ref{RSS 5.27} for example). So as $V$ is a maximal $S$-order and $V\subseteq R\subseteq S$, $V=R$. Furthermore, $R=V$ is an equivalent order with $F$, hence $V=R=F$ because $R$ is a maximal order.
\end{proof}

\begin{remark}\label{vblowup eff div works in T}
As in Remark~\ref{RSS 5.27 works in T}, we note that the proof of Proposition~\ref{vblowup eff div} also works for an analogous statement for virtual blowups of $T$.
\end{remark}

Our final result of this chapter gives sufficient conditions for a virtual blowup to be contained in $S$. This is our analogue of \cite[Corollary~6.6(3)]{RSS}.

\begin{prop}\label{RSS 6.6(3)}
Let $F$ be a virtual blowup of $S$. Let $\bfd$ and $M$ be given by Lemma~\ref{max order pair equiv def}(1) so that $F=\End_{S(\bfd)}(M^{**})$. Suppose that every ideal $I$ of $S(\bfd)$ with $\GK(S(\bfd)/I)=1$ satisfies $\GK(S/IS)\leq 1$ (in particular, this holds if $S(\bfd)$ has no such ideals). Then $F$ is a maximal order contained in $S$.
\end{prop}

\begin{proof}
One may copy the proof given for \cite[Corollary~6.6(3)]{RSS} with the relevant notation interchanged. In doing so, one replaces the fact that $T(\bfd)$ is Cohen-Macaulay by Theorem~\ref{T(d) properties}, with the fact that $S(\bfd)$ is Cohen-Macaulay by Theorem~\ref{S(d) thm}.
\end{proof}


\begin{remark}\label{no sporadics remark}
Evidence from \cite[Section~9]{RSS2} suggests that generically the rings $T(\bfd)$ will have no ideals with co-GK-dimension 1. By Theorem~\ref{S(d) thm}(1) and Lemma~\ref{sporadics up n down}, this would pass to the $S(\bfd)$'s. Proposition~\ref{RSS 6.6(3)} then would tells us that generically, a virtual blowup is contained in $S$. In other words, generically a maximal $S$-order is a maximal order. In \cite[Proposition~10.3]{RSS} the authors construct a virtual blowup $F$ of $T$ such that $F\not\subseteq T$. Such an example in $S$ proves difficult to construct.
\end{remark}

\section{Classifying arbitrary maximal orders}\label{arbitrary max orders}

In this chapter we will present our main classification of maximal $S$-orders (Theorem~\ref{RSS 8.11}). In fact, it will turn out we already know of all the maximal $S$-orders $U$ (such that $\ovl{U}\neq \Bbbk$). We will show that they are all $g$-divisible, and hence fit into our classification of $g$-divisible maximal $S$-orders (Theorem~\ref{RSS 7.4}). The conclusions of Theorem~\ref{RSS 8.11} are therefore similar to Theorem~\ref{RSS 7.4}. We hence must weaken the $g$-divisibility assumption of Theorem~\ref{RSS 7.4} to the assumption that $\ovl{U}\neq \Bbbk$. \par
The assumption $\ovl{U}\neq \Bbbk$ is annoying yet necessary. We give a  ``silly" example of a maximal order $U$ such that $\ovl{U}=\Bbbk$ in Example~\ref{RSS 10.8}. The last part of this chapter studies a way of getting around the assumption $\ovl{U}=\Bbbk$. For this one must be willing to take higher Veronese subrings and semi-graded isomorphisms. \par

\subsection{The main classification}

To obtain our main classification of maximal $S$-orders, we must show that any cg subalgebra $U$ (satisfying $\ovl{U}\neq \Bbbk$) is contained in, and equivalent to, a $g$-divisible ring. Our strategy for this, first applied in \cite[Section~7]{Ro}, is via two steps. First we show a subalgebra $U$ of $S$ is equivalent to the ring obtained by adjoining $g$, namely $C=U\langle g\rangle$. Second we show $C$ is an equivalent order with its $g$-divisible hull. It is the latter which we tackle first. Before we come to this, we need a result from \cite{RSS}.

\begin{lemma}\label{RSS 8.6}\cite[Lemma~8.6]{RSS}
Let $B=B(E,\NN,\sigma)$ for some smooth elliptic curve $E$, invertible sheaf $\NN$ with $\deg\NN \geq 1$, and automorphism, $\sigma$ of $E$, with infinite order. Then for any nonzero $x\in B_k$, we have $B_nx+xB_n=B_{n+k}$ for all $n\gg 0$. In particular, if $A$ is a graded subalgebra of $B$ such that $A\neq \Bbbk$, then $B$ is noetherian $(A,A)$-bimodule. \qed
\end{lemma}

The proof of the next result is identical to that of \cite[Proposition~8.7]{RSS}. It is included for the convenience of the reader.

\begin{prop}\label{C to C hat}
Let $C$ be a cg subalgebra of $S$ satisfying $Q_\gr(C)=Q_\gr(S)$ and $\ovl{C}\neq\Bbbk$. Suppose that $g\in C$. Then for all $m\gg 0$, $C\cap Sg^m=g^m\widehat{C}$.\par
 If $C$ is noetherian, then $\widehat{C}$ is also finitely generated on both sides as a $C$-module.
\end{prop}
\begin{proof}
Let $U=\widehat{C}$. Certainly $U\subseteq S$ because $S$ is $g$-divisible. For $k\geq 0$ we define $M^k=\{a\in S|\, ag^k\in C\}$. Clearly eack $M^k$ is a $(C,C)$-bimodule. Moreover, both $M^k\subseteq M^{k+1}$ and $\bigcup_{i\geq 0} M^i = U$ hold. Inside $\overline{U}$ we then have the chain of $\overline{C}$-bimodules
$$\overline{C}= \overline{M^0}\subseteq \overline{M^1}\subseteq  \cdots \subseteq\bigcup_{k\geq 0}\overline{M^k}=\overline{U}.$$
By Lemma~\ref{RSS 8.6}, $\overline{S}$ is a noetherian $\overline{C}$-bimodule, and hence $\overline{M^n}=\overline{U}$ for some $n\geq 0$. \par
We claim $C\cap Sg^m=U\cap Sg^m=Ug^m$ for all $m\geq n$. Suppose for a contradiction that $s=ug^m\in g^mU\setminus C$ for some $m\geq n$. Choose $m$ and $s$ such that $\deg(u)$ is minimal. Now $\ovl{u}\in \ovl{U}=\ovl{M^n}$, and thus $u=a+bg$ for some $a\in M^n$ and $b\in S$. But then, $ug^n-ag^n=bg^{n+1}\in U\cap Sg^{n+1}=Ug^{n+1}$. Because $\deg(b)<\deg(u)$, the minimality of $\deg(u)$ implies $bg^{n+1}\in C$. But this forces $ug^{n}=bg^{n+1}+ag^n\in C$. Thus $s=ug^ng^{m-n}\in C$ which gives the desired contradiction. Hence indeed $C\cap Sg^m=g^mU$ for all $m\geq n$. \par
If $C$ is noetherian, then $U_C$ ($_CU$) is finitely generated because then $U$ is isomorphic to a right (left) ideal of $C$ via $c\mapsto g^nc$.
\end{proof}

As mentioned already, the assumption ``$\ovl{U}\neq \Bbbk$" appearing in Proposition~\ref{C to C hat} will be the assumption that replaces ``$g$-divisible" in our classification of maximal $S$-orders. \par

Clearly Proposition~\ref{C to C hat} proves that $C$ and $\widehat{C}$ are equivalent orders. We now aim to prove $U$ and $U\langle g\rangle$ are equivalent orders. Key to this are sporadic and minimal sporadic ideals.

\begin{definition}\label{min sporadic ideal}\index{sporadic ideal}\index{minimal sporadic ideal}
Let $R$ be a cg $\Bbbk$-subalgebra of $\Sg$.
\begin{enumerate}[(1)]
\item  A homogenous ideal $I$ of $R$ will be called \textit{sporadic} if $\GK(R/I)\leq 1$.
\item A sporadic ideal $K$ of $R$ is called \textit{minimal sporadic} if for all sporadic ideals $I$ of $R$, there exists an $n\geq 0$ such that $K_{\geq n}\subseteq I$.
\end{enumerate}
\end{definition}

We warn that Definition~\ref{min sporadic ideal}(1) does not completely agree with the definition given in \cite{RSS}. There, an ideal $I$ of $R$ was called sporadic when $\GK(R/I)=1$. It is Definition~\ref{min sporadic ideal}(1) that is most convenient for us.

\begin{remark}\label{RSS assumption 8.2}
In \cite[Section~8]{RSS} Rogalski, Sierra and Stafford work under increased assumptions designed ensure the existence of minimal sporadic ideals \cite[Assumption~8.2]{RSS}. It is not immediately obvious that $S^{(3)}$ satisfies these assumptions: this follows from \cite[Theorem~8.8 and Proposition~8.7]{RSS2}.
\end{remark}

The coming few results show that the algebras we consider do indeed have minimal sporadic ideals.

\begin{prop}\label{RSS 8.4}
Let $R$ be a cg $g$-divisible subalgebra of $S$ with $Q_\gr(R)=Q_\gr(S)$. Then $R$ has a minimal sporadic ideal.
\end{prop}
\begin{proof}
Set $R'=R^{(3)}$. By Lemma~\ref{g div up}, $R'$ is $g$-divisible subalgebra of $T=S^{(3)}$. By \cite[Proposition~8.4]{RSS}, $R'$ has a minimal sporadic ideal, $L$ say. Set $K=RLR$. Since $R$ is noetherian by  Proposition~\ref{RSS 2.9}, Lemma~\ref{sporadics up n down}(2) applies and shows that $\GK(R/K)\leq 1$. We claim that $K$ is in fact a minimal sporadic ideal of $R$.\par

Suppose that $I$ is any other sporadic ideal of $R$. Then by Lemma~\ref{sporadics up n down}(1), $J=I^{(3)}$ is a sporadic ideal of $R'$. Since $L$ is a minimal sporadic ideal of $R$, there exists $m\geq 0$ such that $L_{\geq m}\subseteq J$. We then have $RL_{\geq m}R\subseteq RJR\subseteq I$. By Lemma~\ref{Ab ehd (Ageqn)B}, $RL_{\geq m}R\ehd K$. Hence there exists $n\geq 0$ such that $K_{\geq n}\subseteq RL_{\geq m}R\subseteq I$ as required.
\end{proof}

\begin{cor}\label{S(d) and max order pair min sporadics} The following rings have a minimal sporadic ideal:
\begin{enumerate}[(1)]
\item $S(\bfd)$ for an effective divisor $\bfd$ on $E$ with $\deg\bfd\leq 2$;
\item Both $V$ and $F$ from a maximal order pair $(V,F)$. In particular, virtual blowups have minimal sporadic ideals.
\end{enumerate}
\end{cor}
\begin{proof}
(1). By Theorem~\ref{S(d) thm}(1) $S(\bfd)$ is $g$-divisible and so Proposition~\ref{RSS 8.4} applies. \par
(2). By definition of a maximal order pair (Definition~\ref{max order pair}) $V$ is a $g$-divisible subalgebra of $S$. It hence has a minimal sporadic ideal by Proposition~\ref{RSS 8.4}, call it $K$. By Definition~\ref{max order pair} again, $F$ has an ideal $J$ contained in $V$ and such that (as a $\Bbbk$-algebra) $\GK(F/J)\leq 1$. Set $L=JKJ$, then $L$ is an ideal of $F$ contained in $K$. By Lemma~\ref{RSS 2.15}(2)(4) it is sporadic. Let $I$ be any other sporadic ideal of $F$. Then $IJ$ is another sporadic ideal of $F$ by Lemma~\ref{RSS 2.15}(2)(4) again. Moreover $IJ\subseteq J\subseteq V$ holds; in particular it is an ideal of $V$ also. Since $F_V$ is finitely generated, $\GK(F/IJ)$ equals the GK-dimension of $F/IJ$ considered as a right $V$-module, $\GK_V(F/IJ)$ (see \cite[Corollary~5.4]{KL}). Thus we have
$$\GK(V/IJ)=\GK_V(V/IJ)\leq \GK_V(F/IJ)=\GK(F/IJ)\leq 1,$$
showing $IJ$ is a sporadic ideal of $V$. Since $K$ is minimal sporadic, there exists $n\geq 0$ such that $K_{\geq n}\subseteq IJ$. It follows that $L_{\geq n}\subseteq K_{\geq n} \subseteq IJ\subseteq I$.
\end{proof}

We would expect that algebras containing $g$ would always have minimal sporadic ideals, the best we can do is the following. The result is analogous to \cite[Corollary~8.8]{RSS}. With Proposition~\ref{RSS 8.4} in place we are able to use the same proof.

\begin{cor}\label{RSS 8.8} Let $C$ be a cg subalgebra of $S$ with $Q_\gr(C)=Q_\gr(S)$. Assume that $g\in C$ and that $\ovl{C}\neq\Bbbk$. By Proposition~\ref{RSS 8.4}, $\widehat{C}$ has a minimal sporadic ideal $J$. Then $K=C\cap \widehat{J}$ is sporadic ideal minimal among sporadic ideals $I$ such that $C/I$ is $g$-torsionfree.
\end{cor}
\begin{proof}
We first note that $\widehat{C}$ is noetherian by Proposition~\ref{RSS 2.9}, and it indeed has a minimal sporadic ideal $J$ by Proposition~\ref{RSS 8.4}. By Lemma~\ref{RSS 2.15}(2), $\widehat{J}$ is also a minimal sporadic ideal of $\widehat{C}$. By replacing $J$ with $\widehat{J}$ we can assume $J$ is $g$-divisible, or equivalently, that $\widehat{C}/J$ is $g$-torsionfree.\par

Let $I$ be any sporadic ideal $C$ such that $C/I$ is $g$-torsionfree. Set $K=J\cap C$. We first show $K\subseteq I$. By Proposition~\ref{C to C hat}, there exists an $r\geq0$ such that $H=g^r\widehat{C}\subseteq C$. It follows that $I\supseteq HIH=g^{2r}\widehat{C}I\widehat{C}$. Put $L=\widehat{C}I\widehat{C}$. By Lemma~\ref{RSS 2.15}(2)(4), $L$ is a sporadic ideal. As $J$ is a minimal sporadic ideal of $\widehat{C}$, $L\supseteq J_{\geq s}$ for some $s\geq 0$. We clearly also have $J_{\geq s}\supseteq g^sJ$.  Combining the above,
$$I\supseteq HIH=g^{2r}L\supseteq g^{2r+s}J\supseteq g^{2r+s}(J\cap C). $$
\par

Pick $m$ minimal such that $I\supseteq g^m(J\cap C)=g^mK$. If $m=0$ we're done, so suppose $m\neq 0$. In which case, by the minimality of $m$, the ideal $(I+g^{m-1}K)/I$ of $C/I$ is $g$-torsion. But by assumption, $C/I$ is $g$-torsionfree. This means that this ideal must be zero. Thus $I\supseteq K$ after all. \par

It remains to prove $\GK(C/K)\leq 1$. As $J$ is $g$-divisible with $\GK(\widehat{C}/J)\leq 1$, Lemma~\ref{RSS 2.14}(3) implies $\widehat{C}/J$ is a finitely generated $\Bbbk[g]$-module. In which case, the submodule $(C+J)/J\cong C/K$ is also a finitely generated $\Bbbk[g]$-module. Therefore, using \cite[Corollary~5.4]{KL}, $\GK(C/K)=\GK_{\Bbbk[g]}(C/K)\leq \GK(\Bbbk[g])\leq 1.$
\end{proof}

Again we want an $S$-version of a result in \cite{RSS} - in this case \cite[Lemma~8.9]{RSS}. Again the proof provided in \cite{RSS} does the job for us.

\begin{lemma}\label{RSS 8.9}
Let $C$ be a cg graded subalgebra of $S$ with $Q_\gr(C)=Q_\gr(S)$. Suppose that $g\in C$ and that $\ovl{C}\neq \Bbbk$. Then $C$ is finitely generated as a $\Bbbk$-algebra.
\end{lemma}
\begin{proof}
We prove the equivalent notion that $C_{\geq 1}$ is a finitely generated right ideal of $C$. We first claim that $C/(g^mS\cap C)$ is a finitely generated $\Bbbk$-algebra for all $m\geq 1$. The result holds for $m=1$ by Proposition~\ref{B just infinite}. Proceeding by induction, assume that we can find $a_1,\dots , a_k\in C_{\geq 1}$ whose images generate $C/(g^m S\cap C)$. Set $X=(g^mS\cap C)/(g^{m+1}S\cap C)$. Then, as $\ovl{C}$-bimodules (and up to shifts),
$$X\cong \frac{S\cap g^{-m}C}{gS\cap g^{-m}C}\cong\frac{S\cap g^{-m}C +gS}{gS} \subseteq S/gS\cong B(E,\LL,\sigma)$$
Hence, by Lemma~\ref{RSS 8.6}, $X$ is a finitely generated $\ovl{C}$-bimodule. Say that $X$ is generated by the images of $b_1,\dots,b_\ell\in C\cap g^mS$. It follows that $a_1,\dots,a_k,b_1,\dots, b_\ell$ generate $C/(g^{m+1}S\cap C)$. \par
Now by Proposition~\ref{C to C hat}, there exists $n\geq 0$ such that $g^m\widehat{C}=C\cap g^mS$ for all $m\geq n$. By the claim, $C/(g^{n+1}S\cap C)=C/g^{n+1}\widehat{C}$ is finitely generated as a $\Bbbk$-algebra. Choose $c_1,\dots, c_t\in C_{\geq 1}$ such that their images generate $C/g^{n+1}\widehat{C}$ as a $\Bbbk$-algebra. Then, for $f\in C_{\geq 1}$, there exists an $x\in\sum c_iC$ such that $f-x\in g^{n+1}\widehat{C}=g(g^n\widehat{C})\subseteq gC$. Thus $f\in gC+\sum c_iC$. That is, $C_{\geq 1}$ is generated by $g,c_1,\dots, c_t$.
\end{proof}

Finally we give the key result that shows that $U$ and $U\langle g\rangle$ are equivalent orders. The proof of Proposition~\ref{RSS 8.10} is an $S$-version of the proof of \cite[Proposition~8.10]{RSS}.

\begin{prop}\label{RSS 8.10}
Let $U$ be a cg subalgebra of $S$ with $\ovl{U}\neq \Bbbk$ and $D_\gr(U)=D_\gr(S)$. Then there exists a nonzero ideal of $C=U\langle g\rangle$ that is finitely generated both as a right and left $U$-module.
\end{prop}
\begin{proof}
By Lemma~\ref{Qgr(S) or Qgr(T)} either $Q_\gr(C)=Q_\gr(S)$ or $Q_\gr(C)=Q_\gr(S^{(3)})$. The case $Q_\gr(C)=Q_\gr(S^{(3)})$, then this is exactly \cite[Proposition~8.10]{RSS}. So we assume that $Q_\gr(C)=Q_\gr(S)$. By Lemma~\ref{RSS 8.9}, $C$ is finitely generated, say by $c_1,c_2,\dots,c_n$. For each $i$, we can write $c_i=\sum w_{ij}g^j$ for some $w_{ij}\in U$. Let $W$ be the subalgebra of $U$ generated by the $w_{ij}$. Then $W$ is a finitely generated subalgebra of $U$ such that $C=U\langle g\rangle = W\langle g\rangle$. \par
Given $n\geq 0$, we have $CW_{\geq n}=W_{\geq n}+W_{\geq n}g+W_{\geq n}g^2+\dots =W_{\geq n}C$ as $g$ is central. Therefore $CW_{\geq n}$ is an ideal of $C$. In addition, $C/CW_{\geq n}$ is the homomorphic image of the polynomial ring $W/W_{\geq n}[g]$. It follows that the Hilbert series of $C/CW_{\geq n}$ is at most as large as $\sum_i h(t)t^{3i}$, where $h(t)=h_{W/W_{\geq n}}(t)$ is a polynomial. In particular, $\GK(C/CW_{\geq n})\leq 1$. Moreover, we claim that the $g$-torsion submodule $K(n)=\mathrm{tors}_g(C/CW_{\geq n})$ is finite dimensional. To see this, note that since $W/W_{\geq n}[g]$ is clearly noetherian, $C/CW_{\geq n}$ is also. This shows $K(n)$ is finitely generated. Suppose that $K(n)=f_1(C/CW_{\geq n})+\dots +f_m(C/CW_{\geq n})$ with, say $f_ig^k=0$ for all $i$, Then
$$ K(n)\subseteq \sum_i f_i \left( \frac{W}{W_{\geq n}}+\frac{W}{W_{\geq n}}g+\dots +\frac{W}{W_{\geq n}}g^k \right).$$
The right hand side is clearly finite dimensional which proves the the claim. \par

For $n\geq 0$, we put $Z(n)=C\cap\widehat{CW_{\geq n}}$. Then for every $n$, $Z(n)/CW_{\geq n}=K(n)$ is finite dimensional, and so,
\begin{equation}\label{C/Z(n)}\dim_\Bbbk ([C/Z(n)]_m)=\dim_\Bbbk([C/CW_{\geq n}]_m)\;\text{ for } m\gg 0.\end{equation}
Furthermore, $\GK(C/Z(n))\leq 1$ and $C/Z(n)$ is $g$-torsionfree. Now let $J$ be a minimal sporadic ideal of $\widehat{C}$ given by Lemma~\ref{RSS 8.4}. By Lemma~\ref{RSS 2.15}(2), $\widehat{J}$ is also a minimal sporadic ideal of $\widehat{C}$. Thus by replacing $J$ with $\widehat{J}$, we may assume $\widehat{C}/J$ is $g$-torsionfree. In which case, $K=C\cap J$ is a sporadic ideal of $C$ minimal among co-$g$-torsionfree sporadic ideals by Corollary~\ref{RSS 8.8}. In particular, $K\subseteq Z(n)$ for all $n$. \par

Now, $C/K\cong (C+J)/J\hookrightarrow \widehat{C}/J$, and $\widehat{C}/J$ has an eventually periodic Hilbert series by Lemma~\ref{RSS 2.14}(1). It follows that there exists $r\geq 0$ such that $\dim_\Bbbk(\widehat{C}/J)_m\leq r$ for all $m\gg 0$. As $Z(n)\supseteq K$, $\dim_\Bbbk ([C/Z(n)]_m)\leq r$ for all $m\gg 0$. Fix $m\gg 0$ and let $d_n$ be the integer given in (\ref{C/Z(n)}). Then for all $n\geq m$ we have $d_n\leq r$, and also $d_n\leq d_{n+1}$ because $Z(n)\supseteq Z(n+1)$. It follows that $d_n=d_{n+1}$ for all $n\gg 0$, say for all $n\geq n_0$. Thus
\begin{equation}\label{CW ehd} CW_{\geq n}\ehd CW_{\geq n_0}\;\text{ for all } n\geq n_0.\end{equation} \par
Finally, if $W$ is generated as a $\Bbbk$-algebra by elements of degree at most $e$, then $CW_{\geq n_0} W_{\geq 1}\supseteq CW_{\geq n_0+e}$. By (\ref{CW ehd}), $\dim_\Bbbk (CW_{\geq n_0}/CW_{\geq n_0+e})<\infty$, and therefore $\dim_\Bbbk(CW_{\geq n_0}/CW_{\geq n_0}W_{\geq 1})<\infty$ as well. By the Graded Nakayama's Lemma, $CW_{\geq n_0}=W_{\geq n_0}C$ is then a finitely generated right $W$-module. Since $W\subseteq U$, it is also finitely generated as a right $U$-module. A symmetrical argument proves the statement on the left.
\end{proof}

Proposition~\ref{RSS 8.10} shows that $U$ and $U\langle g\rangle$ are equivalent orders; the details can be found in the proof of the next theorem. We are now ready to state and prove our main result. Theorem~\ref{RSS 8.11}(2) below is \cite[Theorem~8.11]{RSS} and uses Notation~\ref{3 Veronese notation2} for $T=S^{(3)}$. In the coming results, $T=S^{(3)}$ is regraded so that $T^{(d)}=S^{(3d)}$.

\begin{theorem}\label{RSS 8.11}
Let $d\geq 1$, and suppose that $U$ is a cg maximal $S^{(d)}$-order satisfying $\ovl{U}\neq \Bbbk$.
\begin{enumerate}[(1)]
\item If $d$ is coprime to 3, then there exists a $\sigma$-virtually effective divisor $\bfx$ satisfying $\deg\bfx\leq 2$, and a virtual blowup $F$ of $S$ at $\bfx$, such that $U=(F\cap S)^{(d)}$.
\item If $d$ is divisible by $3$, say $d=3e$, then exists a $\tau$-virtually effective divisor $\bfx$ with $\deg\bfx\leq 8$, and a virtual blowup $F$ of $T$ at $\bfx$, such that $U=(F\cap T)^{(e)}$.
\end{enumerate}
\end{theorem}
\begin{proof}
(1). By definition, $Q_\gr(U)=Q_\gr(S^{(d)})=Q_\gr(S)^{(d)}$, and so $D_\gr(U)=D_\gr(S)$. Hence Proposition~\ref{RSS 8.10} applies, and we get an ideal $I$ of $C=U\langle g\rangle$ that is finitely generated on both sides as a $U$-module. Now, as both a right and left $U$-module,
$$I\cong I^{(d)}\oplus I^{(1\;\mathrm{mod}\, d)}\oplus \dots\oplus I^{(d-1\;\mathrm{mod}\, d)},$$
where $I^{(i\;\mathrm{mod}\, d)}=\bigoplus_{j\in\N}I_{i+dj}$. Because $J=I^{(d)}$ is a direct summand of $I$, it is also finitely generated on both sides as a $U$-module. Clearly $J$ is also an ideal of $C'=C^{(d)}$.\par
Set $\widetilde{U}=U+J$, then $\widetilde{U}$ is finitely generated on both sides as a $U$-module. Hence, by Lemma~\ref{hence equiv orders}(1), $\widetilde{U}$ and $U$ share a common ideal, $K$ say. Clearly $\widetilde{U}$ and $C'$ have the common ideal $J$, so $U$ and $C'$ have the common ideal $JKJ$. In particular, $U$ and $C'$ are equivalent orders. \par
Now consider $C$. By Lemma~\ref{Qgr(S) or Qgr(T)} either $Q_\gr(C)=Q_\gr(S)$ or $Q_\gr(C)=Q_\gr(T)$. In fact $Q_\gr(C)=Q_\gr(T)$ is impossible because $Q_\gr(U)=Q_\gr(S)^{(d)}$ and $d$ is coprime to 3, thus $Q_\gr(C)=Q_\gr(S)$. By Proposition~\ref{C to C hat}, $C$ and $\widehat{C}$ have a common ideal. By Theorem~\ref{RSS 7.4}(2), $\widehat{C}$ is contained in, and an equivalent order to, some virtual blowup $F$ at a virtually effective divisor $\bfx$ with $\deg\bfx\leq 2$. Hence $V=F\cap S$ is a maximal $S$-order containing and equivalent to $C$. By Lemma~\ref{equiv orders go up}, $C'$ (and hence $U$) is contained in and equivalent to $V^{(d)}$ and $F^{(d)}$. As $U$ is a maximal $S^{(d)}$-order we have $U=C'=V^{(d)}$.\par
(2). This is \cite[Theorem~8.11]{RSS} with \cite[Theorem~8.8 and Proposition~8.7]{RSS2} proving that $T$ indeed satisfies the hypotheses of those results.
\end{proof}

In the other direction we obtain Theorem~\ref{main thm converse}. We note that Theorem~\ref{main thm converse}(2)(a-c) is an improvement on \cite{RSS}.

\begin{theorem}\label{main thm converse}
Let $d\geq 1$ and $\bfx$ a divisor on $E$.
\begin{enumerate}[(1)]
\item If $d$ is coprime to 3 and $\bfx$ is $\sigma$-virtually effective with $\deg\bfx\leq 2$, then there exists a blowup $F$ of $S$ at $\mbf{x}$. Moreover:
 \begin{enumerate}[(a)]
 \item $F'=F^{(d)}$ is a maximal order with $Q_\gr(F)=Q_\gr(S^{(d)})$;
 \item $V'=F'\cap S$ is a maximal $S^{(d)}$-order;
 \item there exists an ideal $L$ of $F'$, contained in $V'$, such that $\GK(F'/L)\leq 1$.
 \end{enumerate}
\item If $\bfx$ is $\tau$-virtually effective with $\deg\bfx\leq 8$, then there exists a blowup $F$ of $T$ at $\mbf{x}$. Moreover:
\begin{enumerate}[(a)]
\item $F'=F^{(d)}$ is a maximal order with $Q_\gr(F')=Q_\gr(T^{(d)})$;
\item $V'=F'\cap T$ is a maximal $T^{(d)}$-order;
\item there exists an ideal $L$ of $F'$, contained in $V'$, such that $\GK(F'/L)\leq 1$.
\end{enumerate}
\end{enumerate}
\end{theorem}
\begin{proof}
(1). By Proposition~\ref{RSS 7.4(3)}, a blowup $F$ of $S$ at $\mbf{x}$ exists. In particular $(F,V)$, where $V=F\cap S$, is a maximal order pair of $S$. Since $F$ must be $g$-divisible, Proposition~\ref{g-div max orders up n down}(1a) implies $F'$ is a maximal order in $Q_\gr(S^{(d)})$. Similarly $V'$ is a maximal $S^{(d)}$-order by Proposition~\ref{g-div max orders up n down}(1b). By definition of a maximal order pair (Definition~\ref{max order pair}), there exists an ideal $K$ of $F$ contained in $V$, and such that $\GK(F/K)\leq 1$. By Lemma~\ref{sporadics up n down}(1), $L=K^{(d)}$ is an ideal of $F'$ with $\GK(F'/L)\leq 1$. Clearly $L\subseteq V'$ also holds.\par
(2). By \cite[Theorem~7.4(3)]{RSS} a blowup $F$ of $T$ at $\mbf{x}$ exists. The rest of the proof is the same with the relevant definitions in $S$ replaced by those in $T$, and Proposition~\ref{g-div max orders up n down}(2) used instead of Proposition~\ref{g-div max orders up n down}(1).
\end{proof}

Theorem~\ref{RSS 8.11} and Theorem~\ref{main thm converse} answers \cite[Question 9.4]{RSS} with the additional assumption of $\ovl{U}\neq \Bbbk$. It is worth emphasising that we allow $n$ to be divisible by 3 in Corollary~\ref{g-div max orders up n down generalised}. Again we remark that $S$ and $T$ are graded differently below with $S^{(3d)}=T^{(d)}$ holding.

\begin{cor}\label{g-div max orders up n down generalised}
Let $U$ be a cg graded subalgebra of $S$ with $Q_\gr(U)=Q_\gr(S)^{(d)}$ for some $d\geq 1$ and such that $\ovl{U}\neq \Bbbk$. Let $n\geq 1$ and set $U'=U^{(n)}$.
\begin{enumerate}[(1)]
\item If $U$ is a maximal $S^{(d)}$-order then $U'$ is a maximal $S^{(nd)}$-order.
\item If $U$ is a maximal order then $U'$ is a maximal order.
\end{enumerate}
\end{cor}
\begin{proof}
(1). Suppose first that $d$ is divisible by 3, say $d=3e$. Then by Theorem~\ref{RSS 8.11}(2), $U=V^{(e)}$ where $V=F\cap T$ for some virtual blowup $F$ of $T$. By Theorem~\ref{main thm converse}(2b), $U'=V^{(ne)}$ is a maximal $T^{(ne)}=S^{(nd)}$-order. \par
Now suppose that $d$ is coprime to 3. By Theorem~\ref{RSS 8.11}(1), $U=V^{(d)}$ where $V=F\cap S$ for some virtual blowup $F$ of $S$. If $n$ is coprime to 3, then $U'=V^{(nd)}$ is a maximal $S^{(nd)}$-order by Theorem~\ref{main thm converse}(1b). Finally assume that $n=3m$ for some $m\geq 1$. By Theorem~\ref{3 Veronese of virtual blowup}, $G=F^{(3)}$ is a virtual blowup of $T$. Set $W=G\cap T$. By Theorem~\ref{main thm converse}(2) again, $W^{(md)}=V^{(nd)}=U'$ is a maximal $T^{(md)}=S^{(nd)}$-order.\par

(2). Since $U\subseteq S^{(d)}$, if it is a maximal order it is certainly a maximal $S^{(d)}$-order.
Hence by Theorem~\ref{RSS 8.11}, $U=V^{(e)}$ where $V=F\cap S$ for some virtual blowup $F$ (of either $S$ or $T$) where either $e=d$ or $e=\frac{d}{3}$. In either case $V$ and $F$ are equivalent orders by definition, and hence by Lemma~\ref{equiv orders go up} $U=V^{(e)}$ and $F^{(e)}$ are equivalent orders. But $U$ is a maximal order, so $U=F^{(e)}$. The proof now follows that of part (1) with $F$ in place of $V$, and with Theorem~\ref{main thm converse}(1a) and Theorem~\ref{main thm converse}(2a) in place of Theorem~\ref{main thm converse}(1b) and Theorem~\ref{main thm converse}(2b).
\end{proof}

As immediate corollaries to this classification, we are able to satisfy the intuition that maximal orders are ``nice" rings. A particularly notably example of this is that maximal orders are automatically
noetherian.

\begin{cor}\label{RSS 8.11'}
Let $d\geq 1$, and let $U$ be a cg maximal $S^{(d)}$-order such that $\ovl{U}\neq \Bbbk$.  Equivalently, let $F$ be a virtual blowup of $S$ at $\bfx$, $V=F\cap S$ and $U=V^{(d)}$. Then $U$, $V$ and $F$ are all strongly noetherian; in particular noetherian and finitely generated as a $\Bbbk$-algebra.
\end{cor}

\begin{proof}
Since $F$ and $V$ are $g$-divisible cg subalgebras of $S$, they are strongly noetherian by Proposition~\ref{RSS 2.9}. By \cite[Proposition~4.9(2)(3)]{ASZ}, $U$ is then strongly noetherian
\end{proof}

Next we complete the proof of Theorem~\ref{blowup properties}. Comments on the terms below can be found in Remark~\ref{desirable properties}.

\begin{cor}\label{RSS 8.12}\cite[Corollary~8.12]{RSS}.
Let $d\geq 1$ be coprime to 3, and let $U$ be a cg maximal $S^{(d)}$-order satisfying $\ovl{U}\neq\Bbbk$. Equivalently, let $F$ be a virtual blowup of $S$ at $\bfx$, $V=F\cap S$ and $U=V^{(d)}$. Then $U,V, F$ have cohomological dimension at most 2, they have balanced dualizing complexes and all satisfy the Artin-Zhang $\chi$-conditions.
\end{cor}
\begin{proof}
By Theorem~\ref{RSS 6.7}, $\ovl{V}$ and $\ovl{F}$ are equal in high degrees to a twisted homogenous coordinate ring. By \cite[Lemma~2.2]{Ro} and \cite[Lemma~8.2(5)]{AZ.ncps}, $\qgr(\ovl{V})$ and $\qgr(\ovl{F})$ has cohomological dimension 1 and, $\ovl{V}$ and $\ovl{F}$ satisfy $\chi$. The fact that $V$ and $F$ satisfy $\chi$ and have cohomological dimension at most 2 follows from \cite[Theorem~8.8]{AZ.ncps}. Since $V$ and $U=V^{(d)}$ are noetherian by Corollary~\ref{RSS 8.11'}, $V$ is a noetherian $U$-module. Then $U$ then also has these properties by \cite[Proposition~8.7(2)]{AZ.ncps}. Finally, $V,F$ and $U$ have balanced dualizing complexes by the above and \cite[Theorem~6.3]{VdB.dualizing}.
\end{proof}

Missing from the above corollaries are the properties Auslander-Gorenstein and Cohen-Macaulay. Corollary~\ref{S(p-p1+p2) bad homologically} shows that these conditions do not hold in general. \par
Like \cite{RSS} with maximal $T$-orders, we are able to determine the simple objects of $\qgr(U)$ for a maximal $S$-order $U$. The ``$g$-torsion" half of the proof of Corollary~\ref{simples of qgr} follows the proof of \cite[Corollary~8.13]{RSS}.

\begin{cor}\label{simples of qgr}
Let $U$ be a cg maximal $S$-order with $\ovl{U}\neq \Bbbk$. The isomorphism classes of the simple objects of $\qgr(U)$ are in 1-1 correspondence with the closed points of the elliptic curve $E$ together with a (possibly empty) finite set.
\end{cor}
\begin{proof}
By Theorem~\ref{RSS 8.11}, $U=F\cap S$ for a virtual blowup $F$.  Write $\pi: \gr(U)\to \qgr(U)$ for the natural functor defining $\qgr(U)$ and let $\pi M$ be a simple object of $\qgr(U)$. Then $M$ is a finitely generated right $U$-module such that $\dim_\Bbbk M/N<\infty$ for all nonzero submodules $N$ of $M$. If $\mathrm{tors}_g(M)\neq 0$ then $\mathrm{tors}_g(M)\ehd M$ by assumption. It would then follow that $\mathrm{tors}_g(M)=M$. Thus either $M$ is $g$-torsion or $M$ is $g$-torsionfree.\par

Assume first that $M$ is $g$-torsion. Since $M$ is finitely generated, there exists an $n\geq 1$ such that $g^nM=0$. By \cite[Lemma~3.8]{RSS}, $\rann(M)$ is a prime ideal of $U$, and hence $gM=0$. Now let $B=B(E,\NN,\sigma)$ be such that $U/gU\ehd B$, as in Definition~\ref{virtual blowup}. Under the equivalence $\qgr(B)\sim\coh(E)$ from Theorem~\ref{nc serre}, the simple objects of $\qgr(B)$ are in 1-1 correspondence with closed points of $E$. Thus the isomorphism classes of simples $\pi M$ of $\qgr(U)$ such that $M$ is $g$-torsion are in 1-1 correspondence with $E$.\par
Now suppose $M$ is $g$-torsionfree. Write $U'=U^{(3)}$. As a $U'$-module we have
\begin{equation*}\label{qgr simples eq}
 M=M^{(3)}\oplus M^{(1\;\mathrm{mod}\, 3)}\oplus M^{(2\;\mathrm{mod}\, 3)}
\end{equation*}
where $M^{(k\;\mathrm{mod}\, 3)}=\bigoplus_{i}M_{3i+k}$ are finitely generated $g$-torsionfree right $U'$-module. Fix $k=0,1,2$ and let $N=M^{(3)}$. Because $\pi(M)$ is simple, $NU\ehd M$ holds. Similarly, if $L$ is a nonzero $U'$-submodule of $N$, then $LU\ehd M$. It follows that $L=(LU)^{(3)}\ehd N$, and so $\pi N$ is simple in $\qgr(U')$ (we are abusing notation and using $\pi$ for $\qgr(U')$ also). Thus we have $M\ehd NU$ for some $g$-torsionfree right $U'$-module $N$ such that $\pi(N)$ is simple in $\qgr(U')$.
\par
On the other hand, $U=F\cap S$ for some virtual blowup $F$ of $S$ by Theorem~\ref{RSS 8.11}. By Theorem~\ref{3 Veronese of virtual blowup}, $F^{(3)}$ is a virtual blowup of $T$, and so $U'=F\cap T$ is a maximal $T$-order. By \cite[Corollary~8.3]{RSS} and its proof, there are finitely many (possible zero) isomorphism classes of the simple objects $\qgr(U')$ of the form $\pi N$ where $N$ is a finitely generated $g$-torsionfree right $U'$-modules. From the previous paragraph, $M\ehd NU$ for such a $U'$-module $N$. Thus (up to isomorphism) there can only be finitely many possibilities for $\pi(M)\in \qgr(U)$.
\end{proof}

\subsection{Rings contained in $\Bbbk+gS$}\label{Rings contained in gS}

The assumption $\ovl{U}\neq\Bbbk $ of Theorem~\ref{RSS 8.11} is annoying yet, as shown in Example~\ref{RSS 10.8}, is necessary. In \cite[Section~9]{RSS}, Rogalski, Sierra and Stafford present a trick that allowed them to bypass this assumption, at least up to semi-graded isomorphism (defined below) and taking a Veroneses subring. Here we present the results we are able to obtain using similar techniques.

\begin{definition}\label{semi-graded hom}
Let $A$ and $B$ be $\N$-graded algebras. A homomorphism $\theta:A\to B$ is called \textit{graded of degree $\alpha\geq 0$}\index{graded in degree $\alpha$} if $\theta(A_m)\subseteq B_{\alpha m}$ for all $m\geq 0$. The map $\theta$ is called \textit{semi-graded}\index{semi-graded homomorphism} if it is graded in degree $\alpha$ for some $\alpha\geq 0$.
\end{definition}

\begin{prop}\label{RSS 9.3}
Let $d\geq 1$, and let $U$ be a cg maximal $S^{(d)}$-order generated in the single degree $d$. Then there exists $e\geq 1$ and a virtual blowup $F$ of $S$ or $T$ such that up to semi-graded isomorphism $U\cong (F\cap S)^{(e)}$.
\end{prop}
\begin{proof}
This proof is based on that of \cite[Corollary~9.3]{RSS}. Veronese rings are given the grading induced from $S$. That is, for a cg subalgebra $A$ of $S$, $A_n=A\cap S_n$. \par
Pick $m\geq 0$ maximal such that $U\subseteq \Bbbk+ g^m S$. Since $U_d\neq 0$ and $g^m S\subseteq S_{\geq 3m}$, we necessarily have that $m \leq \frac{d}{3}$. Define $\theta:U\to S$ via $u\mapsto g^{-mn}u$ for $u\in U_{nd}$. Clearly $\theta$ is an injective ring homomorphism. Set $e=d-3m$. Then for $n\geq 0$,
\begin{equation}\label{graded in degree alpha eq}
  \theta(U_{d n})=\theta(U_d^n)=\theta(U_d)^n=(g^{-m}U_d)^n=g^{-m n}U_{d n}\subseteq S_{n(d-3m)}=S_{ne},
\end{equation}
which shows $\theta$ is graded of degree $\frac{e}{d}$ (if one grades $U$ via $U_n=U\cap S_{dn}$, then $\theta$ is graded in degree $e$. However the current grading is more convenient for the proof).\par

Set $U'=\theta(U)$, then $U'$ is a cg subalgebra of $S$ generated by $U'_e=\theta(U_d)=g^{-m}U_d$. By choice of $m$, $U'_e\not\subseteq gS$. It is left to prove $U'$ remains a maximal $S^{(e)}$-order. Certainly $D_\gr(U')=D_\gr(U)=D_\gr(S)$ holds, and hence $Q_\gr(U')=Q_\gr(S^{(e)})$. Suppose that $U$ is an equivalent order with some $W$ such that $U'\subseteq W\subseteq S^{(e)}$; say with $xWy\subseteq U'$ and $x,y\in U'$. Define $\psi: S^{(e)}\to S^{(d)}$, via $s\mapsto g^{nm} s$ for $s\in S_{ne}$. A like argument to that of (\ref{graded in degree alpha eq}) shows that $\psi$ is also an injective semi-graded homomorphism. Moreover, it is clear that $\psi|_{U'}=\theta^{-1}$. We then have
$$\psi(x)\psi(W)\psi(y)=\psi(xWy)\subseteq \psi(U')=U\subseteq \psi(W)\subseteq S^{(d)}.$$
Since $U$ is a maximal $S^{(d)}$-order, it follows $U=\psi(W)$ and $\theta(U)=U'=W$. This proves that $U'$ is a maximal $S^{(e)}$-order. Since $\ovl{U'}\neq \Bbbk$ we can now apply Theorem~\ref{RSS 8.11}.
\end{proof}

To understand a cg algebra $U\subseteq \Bbbk+gS$ that is not generated in a single degree, we must also take a Veronese subring on top of the semi-graded homomorphism appearing in Proposition~\ref{RSS 9.3}. Since we are allowing the taking of higher Veronese subrings, for the next two results we can essentially just take the 3-Veronese of our rings and apply the results from \cite[Section~9]{RSS}. We recall that we are writing $T=S^{(3)}$.

\begin{lemma}\label{RSS 9.1}
Let $U$ be a noetherian cg subalgebra of $S$ such that $U\not\subseteq \Bbbk[g]$. Then there exists $\alpha,\beta\geq 0$, and an injective homomorphism $\theta:U^{(\alpha)}\to S$ with $\theta(U_{\alpha n})\subseteq S_{\beta n}$ for all $n\geq 0$. Moreover, $U'=\theta(U^{(\alpha)})$ satisfies $U'\not\subseteq\Bbbk+gS$ and $D_\gr(U')=D_\gr(U)$.
\end{lemma}
\begin{proof}
Set $V=U^{(3)}$. Suppose that $U_n\not\subseteq\Bbbk[g]$, say $xg^k\in U_n\setminus\Bbbk[g]$ with $k\geq 0$ minimal. Then $x\notin gS$, and hence, $x^3\notin gS$ as $gS$ is completely prime. It follows that $x^3g^{3k}\in U_{3n}\setminus \Bbbk[g]$, which shows $V\not\subseteq\Bbbk[g]$. By Lemma~\ref{noeth up n down}(1), $V$ is noetherian. Applying \cite[Proposition~9.1]{RSS} to $V$ we can find $\alpha'$, $\beta'$ and $\theta$ such that $\theta(V_{\alpha' n})\subseteq T_{\beta' n}$ for all $n\geq 0$, and such that $\theta(V^{(\alpha')})\not\subseteq \Bbbk+gT$. Set $\alpha=3\alpha'$ and $\beta=3\beta'$, then $\theta(U_{\alpha n})\subseteq S_{\beta n}$ for all $n\geq 0$, and $\theta(U^{(\alpha)})\not\subseteq\Bbbk+gS$. From \cite[Proposition~9.1]{RSS}  again,
\begin{equation*}
D_\gr(U)=D_\gr(V)=D_\gr(\theta(V^{(\alpha')}))=D_\gr(\theta(U^{(\alpha)}))\subseteq D_\gr(T)=D_\gr(S). \qedhere \end{equation*}
\end{proof}

\begin{cor}\label{RSS 9.5}
Let $U\subseteq S$ be a noetherian cg subalgebra of $S$ with $D_\gr(U)=D_\gr(S)$.
\begin{enumerate}[(1)]
\item There is an $\alpha\geq 1$ and a semi-graded isomorphism $U^{(\alpha)}\cong U'$ such that $U'\subseteq S$ is a noetherian algebra satisfying $D_\gr(U')=D_\gr(S)$ and $\ovl{U'}\neq \Bbbk$.
\item If $C=U'\langle g\rangle$ and $D=\widehat{C}$, then $D$ is finitely generated on both sides as a $U'$-module. The algebra $D$ is described by Corollary~\ref{RSS 7.6}.
\end{enumerate}
\end{cor}
\begin{proof}
Part (1) easily follows from Lemma~\ref{RSS 9.1}. \par

(2). By the proof of Proposition~\ref{RSS 9.1}, $U'\subseteq T$ so we could just apply \cite[Corollary~8.14]{RSS}. However a direct proof is short and so we included it. \par

By Proposition~\ref{RSS 8.10}, there is an ideal $I$ of $C$ that is finitely generated on both sides as a $U$-module. By Lemma~\ref{Ore argument}(2), there exists a $u\in U$ such that $uI\subseteq U$. Thus for any $x\in I$, $uxC\subseteq U$. Because $U$ is noetherian, this shows $C_U$ is finitely generated. By Proposition~\ref{C to C hat}, $D_C$ is finitely generated. Hence it follows that $D_U$ is finitely generated. Symmetrically $_UD$ is also finitely generated. Since $D$ is $g$-divisible it satisfies the hypothesis of Corollary~\ref{RSS 7.6}.
\end{proof}

\section{Examples}

In our final chapter we compute some examples demonstrating some subtleties of the theory. The standout example is Theorem~\ref{S(p-p1+p2)} where we construct an explicit example of a virtual blowup. Also notable are Example~\ref{original S(p-p1+p2)} and Example~\ref{S(p+p1)} where we explore the delicacy required when defining an algebra by its generators. Example~\ref{RSS 10.8} gives an example of a maximal order $U$ satisfying $\ovl{U}=\Bbbk$.

\subsection{A virtual blowup example}

Here we study in detail a particular example of what will be prove to be a virtual blowup. We are able to give algebra generators of a virtual blowup (notoriously a hard problem in noncommutative algebra) as well as realising it as an endomorphism ring.

\begin{theorem}[Proposition~\ref{S(p-p1+p2) main prop}, Lemma~\ref{F is blowup} and Corollary~\ref{S(p-p1+p2) bad homologically}] \label{S(p-p1+p2)}
Let $p\in E$ and $\mbf{x}=p-p^\sigma+p^{\sigma^2}$. We set
\begin{itemize}
\item $X_1=S(p+p^{\sigma^2})_1=\{ x\in S_1\,|\; \ovl{x}\in H^0(E,\LL(-\bfx))\};$
\item $X_2=S(p)_1S(p^{\sigma^2})_1\subseteq \{ x\in S_2\,|\; \ovl{x}\in H^0(E,\LL_2(-[\bfx]_2))\};$
\item $X_3=\{ x\in S_3\,|\; \ovl{x}\in H^0(E,\LL_3(-[\bfx]_3))\}.$
\end{itemize}
Put $U=\Bbbk\langle X_1, X_2, X_3\rangle$. Then $U$ is a virtual blowup of $S$ at the virtually effective divisor $\bfx$. In particular
\begin{enumerate}[(1)]
\item $U$ is maximal order contained in $S$.
\item $U$ is noetherian and Corollaries~\ref{RSS 8.11'} and \ref{RSS 8.12} hold.
\item $U$ is neither Auslander-Gorenstein nor AS-Gorenstein nor Cohen-Macaulay. Moreover, $U$ has infinite injective dimension.
\item Set $R=S(p)$ and $M=R+S(p)_1S_1R$ considered as a right $R$-module. Then $U=\End_{R}(M)$.
\end{enumerate}
\end{theorem}

A priori, we do not know that $X_2$ and $X_3$ are respectively the homogenous elements of $U$ of degree 2 and 3 - this is proven in (\ref{X_i=U_i}).\par
What's curious about $U$ from Theorem~\ref{S(p-p1+p2)} is the summand $X_2$. It would be more natural to take $U=\Bbbk\langle X_1, X'_2, X_3\rangle$ where $X'_2=\{ u\in S_2\,|\; \ovl{u}\in H^0(E,\LL_2(-[\bfx]_2))\}$, and hence $\ovl{X_2}=B(E,\LL(-\bfx),\sigma)_2$ - a definition more in keeping with the definition of $S(\bfd)$ (Definition~\ref{S(d) def}). Example~\ref{original S(p-p1+p2)} studies this second ring and shows it is far from being a maximal order. This suggest that finding algebra generators for an arbitrary virtual blowup would be more difficult than Theorem~\ref{S(p-p1+p2)} above.

Before proceeding with the proof of Theorem~\ref{S(p-p1+p2)}, we require a few computational lemmas to help us. The first of these is \cite[Lemma~3.1]{Ro}.

\begin{lemma}\label{Ro3.1}\cite[Lemma~3.1]{Ro}
Let $\MM$ and $\NN$ be invertible sheaves on $E$ with both $\deg \NN\geq 2$ and $\deg \MM \geq 2$. Consider the natural map $$\mu : H^0(E,\MM)\otimes H^0(E,\NN)\to H^0(E,\NN\otimes\MM).$$
The map is surjective unless $\deg\MM=\deg\NN=2$ and $\MM\cong\NN$, in which case the image is a 3 dimensional vector space. \qed
\end{lemma}

For the next lemma we recall from Definition~\ref{S(p)} that for $q\in E$,
$$S(q)_1=\{x\in S_1\,|\; \ovl{x}\in  H^0(E,\LL(-q))\}\;\text{ and }\;S(q)=\Bbbk\langle S(q)_1\rangle.$$
 In particular we will be writing
$$\ovl{S(q)_1}=H^0(E,\LL(-q)).$$
We will also be using Notation~\ref{p^sigma^j} and Notation~\ref{[d]_n} frequently for the rest of this chapter.

\begin{lemma}\label{Ro3.1 applied} Let $B=\ovl{S}=B(E,\LL,\sigma)$. Then
\begin{enumerate}[(1)]
\item For $q,r\in E$, if $r\neq q^{\sigma^2}$ then $\ovl{S(q)_1}\,\ovl{S(r)_1}=H^0(E,\LL_2(-q-r^\sigma))\subseteq B_2$.
\item Let $n\geq 1$ and $p{(0)},p{(1)},\dots, p{(n)}\in E$. Suppose that $p(1)\neq p(0)^{\sigma^2}$. Then
$$ \prod_{i=0}^{n}\ovl{S(p{(i)})_1}= H^0(E,\LL_{n+1}(-p{(0)}-p{(1)}^\sigma-p{(2)}^{\sigma^2}-\dots-p{(n)}^{\sigma^{n}}))\subseteq B_{n+1}.$$
\end{enumerate}
\end{lemma}
\begin{proof}
(1). This is almost what is in \cite[Lemma~4.1(2)]{Ro}. Since our statement is different we include a proof but note that it is essentially Rogalski's proof.\par
By definition $\ovl{S(q)_1}=H^0(E,\LL(-q))$ and $\ovl{S(r)_1}=H^0(E,\LL(-r))$. Their product is the image inside $B_2$ of the natural map
\begin{multline}\label{mu}
H^0(E,\LL(-q))\otimes H^0(E,\LL(-r))\overset{1\otimes \sigma}\longrightarrow  H^0(E,\LL(-q))\otimes H^0(E,\LL(-r)^\sigma)\\
 \overset{\mu}\longrightarrow H^0(E,\LL(-q)\otimes\LL(-r)^\sigma)=H^0(E,\LL_2(-q-r^{\sigma})).
 \end{multline}
Since $\deg\LL(-q)=\deg\LL(-r)=2$, the map $\mu$ in (\ref{mu}) is surjective as long as $\LL(-q)\not\cong\LL(r)^\sigma$ by Lemma~\ref{Ro3.1}. Write
$$\LL(-q)=\LL\otimes\OO_E(-q)\text{ and }\LL(-r)^\sigma=\LL^\sigma\otimes\OO_E(-r)^\sigma.$$
By Abel's Theorem \cite[Corollary~4, page~5]{Mum}, $\LL^\sigma\otimes\OO_E(-q)^\sigma\cong\LL\otimes \OO_E(-q)^{\sigma^2}$. Thus $\LL(-q)\cong\LL(r)^\sigma$ if and only if $\LL\otimes\OO_E(-q)\cong \LL\otimes \OO_E(-r)^{\sigma^2}$. Since $\LL$ is invertible and $\OO_E(-r)^{\sigma^2}=\OO_E(-r^{\sigma^2})$, this happens if and only if $\OO_E(-q)\cong\OO_E(-r^{\sigma^2})$. Therefore by \cite[Proposition~II6.13]{Ha}, $\LL(-q)\cong\LL(-r)^\sigma$ if and only if $\bfx-q\sim \bfx-r^{\sigma^2}$ in the sense of \cite[Section~II6]{Ha}. Since $\sigma$ is a translation by a point of infinite order, say $\tilde{p}\mapsto \tilde{p}+s$, this is equivalent to $q=r+2s$. That is $q=\sigma^2(r)$, or equivalently $r=\sigma^{-2}(q)=q^{\sigma^2}$.\par

(2). The $n=1$ case is part (1). We do the case $n=2$, with the general case following by induction. From part (1) we have $\ovl{S(p{(0)})_1}\,\ovl{S(p{(1)})_1}=H^0(E,\LL_2(-p(0)-p(1)^\sigma))$. Since $\deg\LL_2(-p(0)-p(1)^\sigma)\neq \deg\LL(-p(2))$, Lemma~\ref{Ro3.1} implies the multiplication map
\begin{equation*}
H^0(E,\LL_2(-p(0)-p(1)^\sigma))\otimes H^0(E,\LL(-p(2)))\longrightarrow H^0(E,\LL_3(-p(0)-p(1)^\sigma-p(2)^{\sigma^2}))
 \end{equation*}
is surjective. But the image is exactly $\ovl{S(p{(0)})_1}\,\ovl{S(p{(1)})_1}\,\ovl{S(p{(2)})_1}$.
\end{proof}


Another lemma which we will be using regularly in the following calculations is \cite[Lemma~4.1]{Ro}. We state and prove an extension of it here which is implicit in \cite{Ro}.

\begin{lemma}\label{Ro 4.1}
Let $q,r\in E$. Then
\begin{enumerate}[(1)]
\item $S_1S(q)_1=S(q^\sigma)_1S_1$ with $\ovl{S_1S(q)_1}=\ovl{S(q^\sigma)_1S_1}=H^0(E,\LL_2(-q^\sigma))$;
\item $\dim_\Bbbk S(q)_1S(r)_1=4$ if and only if $r\neq q^{\sigma^2}$, while $\dim_\Bbbk S(q)_1S(q^{\sigma^2})_1=3$;
\item $S(q)_1S(r)_1=S(r^{\sigma})_1S(q^{\sigma^{-1}})_1$ if and only if $r\neq q^{\sigma^2},q^{\sigma^{-4}}$; whilst,
$$S(q)_1S(q^{\sigma^2})_1\subsetneq S(q^{\sigma^3})_1S(q^{\sigma^{-1}})_1.$$
\end{enumerate}
\end{lemma}
\begin{proof}
(1) and (2) make up \cite[Lemma~4.1]{Ro} with the exception of showing that $\ovl{S_1S(q)_1}= H^0(E,\LL_2(-q^\sigma))$. By definition $\ovl{S_1S(q)_1}$ is the image of the multiplication map
\begin{multline}\label{Ro 4.1 eq 1} H^0(E,\LL)\otimes H^0(E,\LL(-q))\longrightarrow H^0(E,\LL)\otimes H^0(E,\LL(-q)^\sigma)\\
\longrightarrow H^0(E,\LL\otimes \LL(-q)^\sigma)=H^0(E,\LL_2(-q^\sigma)).\end{multline}
Since $\deg\LL=3>2=\deg\LL(-q^\sigma)$, $\LL\not\cong\LL(-q^\sigma)$. By Lemma~\ref{Ro3.1}, the map in (\ref{Ro 4.1 eq 1}) is then surjective.\par

(3). Since $gS\subseteq S_{\geq 3}$, we can identify $S_{\leq 2}$ and $\ovl{S}_{\leq 2}$ and prove the statements in $\ovl{S}$ instead. Let $q,r\in E$. As in (\ref{Ro 4.1 eq 1}), $\ovl{S(q)_1}\,\ovl{S(r)_1}$ is the image of the natural map
$$H^0(E,\LL(-q))\otimes H^0(E,\LL(-r))\longrightarrow  H^0(E,\LL(-q)\otimes \LL(-r)^\sigma).$$
Now on the right hand side $\LL(-q)\otimes \LL(-r)^\sigma=\LL_2(-q-r^\sigma)$. Moveover,  by the Riemann-Roch theorem (Corollary~\ref{RR to E}), $\dim_\Bbbk H^0(E,\LL_2(-q-r^{\sigma}))=4$. By (2), $\dim_\Bbbk \ovl{S(q)_1}\ovl{S(r)_1}=4$ if and only if $r\neq q^{\sigma^2}$. That is,  $\ovl{S(q)_1}\ovl{S(r)_1}= H^0(E,\LL_2(-q-r^{\sigma}))$ if and only if $r\neq q^{\sigma^2}$.
Similarly $\ovl{S(r^{\sigma})_1}\,\ovl{S(q^{{\sigma^{-1}}})_1}$ is the image of the natural map
\begin{multline*}
H^0(E,\LL(-r^\sigma))\otimes H^0(E,\LL(-q^{\sigma^{-1}}))\longrightarrow \\
H^0(E,\LL(-r^\sigma)\otimes \LL(-q^{\sigma^{-1}})^\sigma)=H^0(E,\LL_2(-q-r^{\sigma})).
\end{multline*}
This time $\ovl{S(r^{\sigma})_1}\ovl{S(q^{{\sigma^{-1}}})_1}=H^0(E,\LL_2(-q-r^{\sigma}))$ if and only if $q^{\sigma^{-1}}\neq (r^{\sigma})^{\sigma^{2}}$ (or equivalently $r\neq q^{\sigma^{-4}}$) again by (2). So if $r\neq q^{\sigma^2},q^{\sigma^{-4}}$, then
$$\ovl{S(q)_1}\ovl{S(r)_1}=H^0(E,\LL_2(-q-r^{\sigma}))=\ovl{S(r^{\sigma})_1}\ovl{S(q^{\sigma^{-1}})_1}.$$
If  $r=q^{\sigma^2}$, then $\dim_\Bbbk \ovl{S(q)_1}\ovl{S(q^{\sigma^2})_1}=3$ and $\dim_\Bbbk \ovl{S(q^{\sigma^3})_1}\ovl{S(q^{\sigma^{-1}})_1}=4$. Hence
\begin{equation}\label{S(p)S(p2)}\ovl{S(q)_1}\ovl{S(q^{\sigma^2})_1}\subsetneq H^0(E,\LL_2(-q-q^{\sigma^3}))=\ovl{S(q^{\sigma^3})_1}\ovl{S(q^{\sigma^{-1}})_1}.\end{equation}
Making the substitution $q\mapsto q^{\sigma^{-3}}$ in (\ref{S(p)S(p2)}) does the $r\neq q^{\sigma^{-4}}$ case.
\end{proof}

Our first lemma on the ring $U$ from Theorem~\ref{S(p-p1+p2)}, specifically the divisor $\bfx$, is a routine computation.

\begin{lemma}\label{bfx v eff}
Let $\bfx=p-p^\sigma+p^{\sigma^2}$, as in Theorem~\ref{S(p-p1+p2)}. Then $\bfx$ is $\sigma$-virtually effective. Moreover, for all $n\geq 2$, $[\bfx]_n$ is effective with
$$[\bfx]_n=p+p^{\sigma^2}+\dots+p^{\sigma^{n-1}}+p^{\sigma^{n+1}}.$$
In particular $[\bfx]_2=p+p^{\sigma^3}$ and $[\bfx]_3=p+p^{\sigma^2}+p^{\sigma^4}$ are effective. \qed
\end{lemma}

\begin{notation}\label{p^sigma=p_1} On top of the current notation, we will be using the more compact $p_i=p^{\sigma^i}=\sigma^{-i}(p)$ for $i\in \Z$, for the coming calculations.
\end{notation}

In the next lemma we will be regularly applying Lemma~\ref{Ro 4.1}(3) to two points in the same $\sigma$-orbit. In our new notation this reads
$$S(p_i)_1S(p_j)_1=S(p_{j+1})_1S(p_{i-1})_1\; \text{ if }j\neq i+2,i-4.$$

\begin{lemma}\label{ovlU}
Let $U$ and $\bfx$ be as in Theorem~\ref{S(p-p1+p2)}. Set $B=B(E,\LL(-\bfx),\sigma)$. Then
$$\ovl{U}=\Bbbk\oplus\ovl{X_1}\oplus \ovl{X_2}\oplus B_{\geq 3}.$$
\end{lemma}
\begin{proof}
Write $V=\ovl{U}$. By definition $V=\Bbbk\langle \ovl{X_1},\ovl{X_2},\ovl{X_3}\rangle$, and clearly $V_{\leq 1}=\Bbbk\oplus \ovl{X_1}$. Since $p\neq p_2$, it is easy to prove $\ovl{X_1}=\ovl{S(p)_1}\cap \ovl{S(p_2)_1}$. Then
\begin{equation}\label{ovlU1^2} \ovl{X_1}^2\subseteq \ovl{S(p)_1}\ovl{S(p_2)_1}=\ovl{X_2},\end{equation}
which shows $V_2=\ovl{X_2}$. By Lemma~\ref{bfx v eff}, $[\bfx]_3$ is effective, and so $H^0(E,\LL_3(-[\bfx]_3))\subseteq \ovl{S}_3$. It is then obvious from the definition that $\ovl{X_3}=B_3$. Since $\ovl{X_1},\ovl{X_2}\subseteq B$ we have
 \begin{equation}\label{ovlU1U2+ovlU2U1} \ovl{X_1X_2}+\ovl{X_2X_1}\subseteq B_3=\ovl{X_3}.\end{equation}
It follows $V_3=\ovl{X_3}.$ Since $V$ is generated in degrees less than or equal to 3, we have also proved $V\subseteq B$.\par

Now using Lemma~\ref{Ro 4.1}(3) for equality 1 below, Lemma~\ref{Ro3.1 applied}(2) for 2, and Lemma~\ref{bfx v eff} for 3, we have
\begin{multline}\label{V4}
V_4\supseteq \ovl{X_2}^2=\ovl{S(p)_1}(\ovl{S(p_2)_1}\ovl{S(p)_1})\ovl{S(p_2)_1}\,\overset{1}=\,
\ovl{S(p)_1}(\ovl{S(p_1)_1}\ovl{S(p_1)_1})\ovl{S(p_2)_1}\\ \overset{2}=H^0(E,\LL(-p-p_2-p_3-p_5))\,\overset{3}=\,H^0(E,\LL(-[\bfx]_4))=B_4.
\end{multline}
Also, using Lemma~\ref{Ro3.1 applied}(2), we can write
\begin{equation}\label{V3} \ovl{X_3}=\ovl{S(p)_1}\ovl{S(p_1)_1}\ovl{S(p_2)_1}.\end{equation}
So a similar calculation using Lemma~\ref{Ro 4.1}(3), Lemma~\ref{Ro3.1 applied}(2) and Lemma~\ref{bfx v eff} obtains
\begin{multline}\label{V5} V_5\supseteq \ovl{X_2}\ovl{X_3}= (\ovl{S(p)_1}\ovl{S(p_2)_1})(\ovl{S(p)_1}\ovl{S(p_1)_1}\ovl{S(p_2)_1}) =\ovl{S(p)_1}(\ovl{S(p_1)_1}\ovl{S(p_1)_1})\ovl{S(p_1)_1}\ovl{S(p_2)_1}\\
=H^0(E,\LL_5(-p_0-p_2-p_3-p_4-p_6))=H^0(E,\LL_5(-[\bfx]_5))=B_5\end{multline}
Now for $n\geq 6$, it is easily seen that we can write $n=3a+4b+5c$ with $a,b,c$ nonnegative integers. In which case
$$V_n\supseteq V_3^{a}V_4^{b}V_5^{c}=B_3^aB_4^bB_5^c=B_{3a+4b+5c}=B_n$$
by equations (\ref{V4}), (\ref{V3}) and (\ref{V5}), Lemma~\ref{Ro3.1 applied}(2) and Lemma~\ref{bfx v eff}.
\end{proof}

At this point we note that (\ref{ovlU1^2}) can be lifted to $S$ since $gS\subseteq S_{\geq 3}$. Similarly, as $X_3\cap gS=g\Bbbk= S_3\cap gS$, (\ref{ovlU1U2+ovlU2U1}) can also be lifted to $S$. This proves that
\begin{equation}\label{X_i=U_i}
X_1,\, X_2,\text{ and }X_3\text{ are indeed the homogenous elements of }U\text{ of degree 1, 2, and 3},
\end{equation}
as was alluded to after Theorem~\ref{S(p-p1+p2)}.\par

Next we look to realise $U$ as an endomorphism ring.

\begin{notation}\label{virtual blowup notation}
Let $R=S(p)$, $M=R+S(p)_1S_1R$ and $H=\End_{R}(M)$. We also put $F=\End_R(M^{**})$ and $V=F\cap S$.
\end{notation}

\begin{lemma}\label{U subset H}
Retain the notation of Theorem~\ref{S(p-p1+p2)} and Notation~\ref{virtual blowup notation}. Then
$U\subseteq H.$
\end{lemma}
\begin{proof}
Here we again use Notation~\ref{p^sigma=p_1}. We warn that because $gS\subseteq S_{\geq 3}$, we can and will regularly identify $S_{\leq 2}=\ovl{S_{\leq 2}}$. Since $U$ is generated by $X_1,X_2,X_3$ we need to show $X_1,X_2,X_3\subseteq H$.  As $H=\{x\in Q_\gr(S)\,|\; xM\subseteq M\}$, we must prove $X_iM\subseteq M$ for each $i$. \par

\textit{Proof of} $X_1M\subseteq M$. Since $X_1\subseteq S(p)_1=R_1$, we clearly have $X_1R\subseteq R\subseteq M$. So it is
\begin{equation}\label{U1M subset M}   X_1S(p)_1S_1R\subseteq M\end{equation}
that we need to prove. First we show that
$$\ovl{X_1}\,\ovl{S(p)_1S_1}\subseteq \ovl{R_3}=H^0(E,\LL_3(-p-p_1-p_2)).$$
For this, we have $\ovl{X_1}=H^0(E,\LL(-p-p_2))$, and $\ovl{S(p)_1S_1}=H^0(E,\LL_2(-p))$ by Lemma~\ref{Ro 4.1}(1). Thus $\ovl{X_1}\,\ovl{S(p)_1S_1}$ is the image of the natural map
$$H^0(E,\LL(-p-p_2))\otimes H^0(E,\LL_2(-p)) \longrightarrow H^0(E,\LL(-p-p_2)\otimes \LL_2(-p)^\sigma).$$
Further, $\LL(-p-p_2)\otimes \LL_2(-p)^\sigma=\LL_3(-p-p_1-p_2)$. Hence indeed $\ovl{X_1}\,\ovl{S(p)_1S_1}\subseteq \ovl{R_3}$. This implies $X_1S(p)_1S_1\subseteq R_3+(S_3\cap Sg)$. But $\deg(g)=3$, and so
$$S_3\cap Sg=S_0g=\Bbbk g=R_0g=R_3\cap Sg.$$
Hence $X_1S(p)_1S_1\subseteq R_3$ and (\ref{U1M subset M}) follows. \par

\textit{Proof of} $X_2M\subseteq M$. First, clearly
\begin{equation}\label{U2M subset H 1} X_2R=S(p)_1S(p_2)_1R\subseteq S(p)_1S_1R\subseteq M.\end{equation}
Next, using Lemma~\ref{Ro 4.1}(1) for equality 1 and Lemma~\ref{Ro 4.1}(3) for equality 2, we also have
\begin{multline}\label{U2M subset H 2} X_2S(p)_1S_1R=(S(p)_1S(p_2)_1)(S(p)_1S_1R)\overset{1}=(S(p)_1S_1)(S(p_1)S(p_{-1}))R\\
\overset{2}=(S(p)_1S_1)S(p)_1^2R=(S(p)_1S_1)R_1^2R\subseteq S(p)_1S_1R\subseteq M. \end{multline}
Equations (\ref{U2M subset H 1}) and (\ref{U2M subset H 2}) together show $X_2M=X_2(R+S(p)_1S_1R)\subseteq M$.\par

\textit{Proof of} $X_3M\subseteq M$. From (\ref{V3}) we have $\ovl{X_3}=\ovl{S(p)_1}\ovl{S(p_1)_1}\ovl{S(p_2)_1}$. Since $\deg(g)=3$ and $gS\cap S_3=\Bbbk g$, this implies
$$X_3=S(p)_1S(p_1)_1S(p_2)_1+\Bbbk g$$
(in fact one can prove $g\in S(p)_1S(p_1)_1S(p_2)_1$ but this is not necessary for us). As $g\in R$ and is central, we clearly we have $gM=Mg\subseteq M$. Therefore what we need to prove is
\begin{equation}\label{U3 in H} S(p)_1S(p_1)_1S(p_2)_1M=S(p)_1S(p_1)_1S(p_2)_1(R+S(p)_1S_1R)\subseteq M.\end{equation}
First, using Lemma~\ref{Ro 4.1}(1) for equality 1 below, we have
\begin{multline}\label{U3M subset H 1}
(S(p)_1S(p_1)_1S(p_2)_1)R\subseteq (S(p)_1S(p_1)_1S_1)R\overset{1}=(S(p)_1S_1S(p)_1)R\\
=(S(p)_1S_1)R_1R \subseteq S(p)_1S_1R\subseteq M.
\end{multline}
Secondly, with multiple uses of  Lemma~\ref{Ro 4.1},
\begin{multline}\label{U3M subset H 2}(S(p)_1S(p_1)_1S(p_2)_1)(S(p)_1S_1)R=(S(p)_1S_1)(S(p)_1S(p_1)_1S(p_{-1}))R \\
=(S(p)_1S_1)S(p)_1^3R=(S(p)_1S_1)R_3R\subseteq(S(p)_1S_1)R\subseteq M.
\end{multline}
Equations (\ref{U3M subset H 1}) and (\ref{U3M subset H 2}) together give (\ref{U3 in H}).
\end{proof}

\begin{lemma}\label{M H with bars}
Retain the notation of Theorem~\ref{S(p-p1+p2)} and Notation~\ref{virtual blowup notation}. Then
\begin{enumerate}[(1)]
\item $\ovl{M}\ehd \bigoplus_{n\geq 0} H^0(E,\OO_E(p^\sigma)\otimes \LL(-p)_n)=\bigoplus_{n\geq 0} H^0(E,\LL_n(p-p^{\sigma^2}-\dots-p^{\sigma^{n-1}}));$
\item $\ovl{H}\ehd B(E,\LL(-\bfx),\sigma)\ehd\ovl{U}$.
\end{enumerate}
\end{lemma}
\begin{proof}
Let $n\gg 0$. As $R$ is generated in degree 1, $R_n\subseteq R_1S_1R_{n-2}=S(p)_1S_1R_{n-2}$, and so $\ovl{M}_n=\ovl{S(p)_1}\ovl{S_1}\ovl{S(p)_{n-2}}$. By Lemma~\ref{Ro 4.1}(1), $\ovl{S(p)_1}\ovl{S_1}=H^0(E,\LL_2(-p))$. Hence $\ovl{M}_n$ is the image of the map
$$H^0(E,\LL_2(-p))\otimes H^0(E,\LL(-p)_{n-2})\overset{\mu}\longrightarrow H^0(E,\LL_2(-p)\otimes\LL(-p)_{n-2}^{\sigma^2}).$$
On the right hand side $$\LL_2(-p)\otimes\LL(-p)_{n-2}^{\sigma^2}=\LL_n(-p-p^{\sigma^2}-\dots-p^{\sigma^{n-1}})= \OO_E(p^\sigma)\otimes \LL(-p)_n=\NN.$$
So provided $\deg\NN=2n+1>2$, the map $\mu$ is surjective by Lemma~\ref{Ro3.1}.\par
(2). By Lemma~\ref{U subset H} and Lemma~\ref{RSS 2.12}(3) $\ovl{U}\subseteq \ovl{H}\subseteq \End_{\ovl{R}}(\ovl{M})$. On the other hand $\End_{\ovl{R}}(\ovl{M})\ehd B(E,\LL(-\bfx),\sigma)$ by part (1) and Lemma~\ref{RSS 6.14}(1). Then by Lemma~\ref{ovlU} $\ovl{U}\ehd B(E,\LL(-\bfx),\sigma)\ehd \ovl{H}$.
\end{proof}

Now we study the 3-Veronese of $U$. For the next proposition we must recall Notation~\ref{3 Veronese notation bg} and Definition~\ref{T(d) def} for the blowup subalgebras $T(\bfd)$ of $T$.

\begin{prop}\label{3 Veronese of H}
Retain the notation of Theorem~\ref{S(p-p1+p2)} and Notation~\ref{virtual blowup notation}. Then
$$U^{(3)}=H^{(3)}=T(p+p^{\sigma^2}+p^{\sigma^4}),$$
where $T(p+p^{\sigma^2}+p^{\sigma^4})$ is the blowup of $T$ from Definition~\ref{T(d) def}. In particular, $U^{(3)}$ is a maximal order.
\end{prop}
\begin{proof}
Let $\bfy=p+p^{\sigma^2}+p^{\sigma^4}$. By Lemma~\ref{bfx v eff}, $\bfy=[\bfx]_3$. So by definition, $T(\bfy)=\Bbbk\langle X_3\rangle$. In particular, with Lemma~\ref{U subset H}, we already have
\begin{equation}\label{3 Veronese of H eq} T(\bfy)\subseteq U^{(3)}\subseteq H^{(3)}.\end{equation}
It is hence enough to prove $H^{(3)}= T(\bfy)$. We do this by showing they are equivalent orders. \par

Since $M_R$ is finitely generated, Lemma~\ref{Ore argument}(2) finds a nonzero $x\in R$ such that $xM\subseteq R$. Choose any nonzero $y\in M$, then clearly $xHy\subseteq xM\subseteq R$. Now notice $(Mx)M\subseteq MR\subseteq M$, meaning $Mx\subseteq H$ also. As $yR\subseteq M$, we then have $yRx\subseteq H$. So $H$ and $R$ are equivalent orders. By Lemma~\ref{equiv orders go up} $H^{(3)}$ and $R^{(3)}$ are equivalent orders. Now $R^{(3)}=T(p+p^{\sigma}+p^{\sigma^2})$ by Theorem~\ref{S(d) thm}(1) and is an equivalent order with $T(p+p^{\sigma^2}+p^{\sigma^4})=T(\bfy)$ by \cite[Corollary~5.27]{RSS}. Thus $H^{(3)}$ and $T(\bfy)$ are equivalent orders. However $T(\bfy)$ is a maximal order by \cite[Theorem~1.1]{Ro}, so (\ref{3 Veronese of H eq}) forces $H^{(3)}=T(\bfy)$.
\end{proof}

\begin{lemma}\label{F is blowup}
Let $F$, $V$ and $\bfx$ be as in Notation~\ref{virtual blowup notation} and Theorem~\ref{S(p-p1+p2)}. Then $F$ is a blowup of $S$ at $\bfx$. In other words,
\begin{enumerate}[(1)]
\item $(F,F\cap S)$ is a maximal order pair of $S$ in the sense of Definition~\ref{max order pair}.
\item $\ovl{F}\ehd B(E,\LL(-\bfx),\sigma)$.
\end{enumerate}
Moreover,
\begin{equation}\label{U in H in F} U\subseteq H\subseteq \widehat{H}\subseteq F.\end{equation}
\end{lemma}
\begin{proof}
By Lemma~\ref{hat end commute}, $\widehat{H}=\End_R(\widehat{M})$, while $M^{**}=(\widehat{M})^{**}$ by Lemma~\ref{RSS 2.13}(3). By Proposition~\ref{RSS 6.4}, $(F,F\cap S)$ is then the unique maximal order pair containing and equivalent to $\widehat{H}$. This and Lemma~\ref{U subset H} show (\ref{U in H in F}). By Proposition~\ref{RSS 6.7}, $\ovl{F}\ehd \ovl{\widehat{H}}\ehd B(E,\LL(-\bfx'),\sigma)$ for some $\sigma$-virtually effective $\bfx'$. We must prove $\bfx=\bfx'$.\par
Now by Proposition~\ref{C to C hat}, $H$ is an  equivalent order with $\widehat{H}$, and hence $H$ is contained in and equivalent to $F$ as well. By Lemma~\ref{equiv orders go up}, $H^{(3)}$ is then contained in and equivalent to $F^{(3)}$. But $H^{(3)}=T(p+p^{\sigma^2}+p^{\sigma^4})$ by Proposition~\ref{3 Veronese of H}; and furthermore $p+p^{\sigma^2}+p^{\sigma^4}=[\bfx]_3$ by Lemma~\ref{bfx v eff}. Now $T([\bfx]_3)$ is a maximal order by \cite[Theorem~1.1]{Ro}, and thus $F^{(3)}=T([\bfx]_3)$. In particular $\ovl{F}^{(3)}=\ovl{T([\bfx]_3)}=B(E,\LL_3(-[\bfx]_3),\sigma)$. On the other hand $\ovl{F}^{(3)}\ehd B(E,\LL(-\bfx'),\sigma)^{(3)}$; therefore for $n\gg 0$ we have
$$
H^0(E,\LL_{3n}(-[\bfx]_{3n}))=B(E,\LL(-\bfx),\sigma)_{3n}= B(E,\LL(-\bfx'),\sigma)_{3n}=H^0(E,\LL_{3n}(-[\bfx']_{3n})).
$$
Since for $n\gg0$, $\LL_{3n}(-[\bfx]_{3n})$ and $\LL_{3n}(-[\bfx']_{3n})$ are generated by their global sections \cite[Corollary~IV.3.2]{Ha}, it follows that $[\bfx]_{3n}=[\bfx']_{3n}$. It should be obvious that this implies $\mbf{x}=\mbf{x'}$ but we spell out the details anyway. \par
Take $n\gg0$, then $[\bfx^{\sigma^{3n}}]_3=[\bfx]_{3(n+1)}-[\bfx]_{3n}=[\bfx']_{3(n+1)}-[\bfx']_{3n}=[(\bfx')^{\sigma^{3n}}]_3$, and so $[\bfx]_3=[\bfx']_3$. If we denote $\mathbb{O}=\{ p^{\sigma^j}\,|\, j\in\Z\}$, then clearly we must have $\bfx'\cap\mathbb{O}=\bfx'$. Thus we can write $\bfx'=\sum_{i\in\Z}a_ip^{\sigma^i}$ for some $a_i\in\Z$ with all but finitely many zero. From  $[\bfx]_3=[\bfx']_3$, we have
$$\sum_{i\in\Z}(a_i+a_{i-1}+a_{i-2})p^{\sigma^{i}}=p+p^{\sigma^2}+p^{\sigma^4}.$$
Looking at coefficients of each $p^{\sigma^i}$, this forces
$$ a_i+a_{i-1}+a_{i-2}=\begin{cases}
1 \;\text{ if }i=0,2,4;\\
0 \;\text{ otherwise.}
\end{cases}
$$
If $a_{j}\neq 0$ for some $j=i-2\geq 3$, then as $a_{j}+a_{j+1}+a_{j+2}=0$ at least one of $a_{j+1}\neq 0$ or $a_{j+2}\neq 0$. Inductively we get an infinite number of $a_j\neq 0$. Thus $a_j=0$ for all $j\geq 3$. Similarly $a_i=0$ for all $i\leq -1$. We are reduced to the following system of equations:
$$
\begin{array}{llllll}
   a_0 &\; & \; & \; & \; &  =1\\
   a_0 & + & a_1 &\; & \; &  =0\\
   a_0 & + & a_1 & + & a_2 & =1\\
   \; & + & a_1 & + & a_2 &  =0\\
   \; & \; & \; & \; & a_2 & =1
\end{array}
$$
from which it is obvious $a_0=1$, $a_1=-1$ and $a_2=1$. That is $\bfx=\bfx'$.
\end{proof}

We now look to improve inequalities to equalities in (\ref{U in H in F}). This is first achieved in $\ovl{S}=B(E,\LL,\sigma)$.

\begin{lemma}\label{S(p-p1+p2) with bars}
Retain the notation  from Theorem~\ref{S(p-p1+p2)} and Notation~\ref{virtual blowup notation}. Then
$$\ovl{U}=\ovl{H}=\ovl{\widehat{H}}=\ovl{F}.$$
\end{lemma}

\begin{proof}
Set $B=B(E,\LL(-\bfx),\sigma)$ considered as a subalgebra of $Q_\gr(\ovl{S})=\Bbbk(E)[t,t^{-1};\sigma]$. By Lemma~\ref{F is blowup}, $\ovl{F}\ehd B$. Since $B$ is Auslander-Gorenstein and Cohen-Macaulay by Theorem~\ref{B properties}(3) we have $\ovl{F}\subseteq B$ by Lemma~\ref{CM lemma}(2). Thus, with Lemma~\ref{F is blowup},
\begin{equation}\label{S(p-p1+p2) with bars eq 1} \ovl{U}\subseteq \ovl{H}\subseteq \ovl{\widehat{H}}\subseteq \ovl{F}\subseteq B.\end{equation}
So proving $\ovl{U}=\ovl{F}$ is enough. By (\ref{S(p-p1+p2) with bars eq 1}) and Lemma~\ref{ovlU}, $\ovl{U}_{\geq 3}=\ovl{F}_{\geq 3}$. \par
To prove $\ovl{U}_2=\ovl{F}_2$ assume otherwise that $\ovl{U}_2\subsetneq \ovl{F}_2$. Because $gS\subseteq S_{\geq 3}$, we can and will identify $S_{\leq 2}=\ovl{S_{\leq 2}}$. Also we note that since $[\bfx]_2=p+p^{\sigma^3}$ is effective, $B_2\subseteq \ovl{S}_2$. Now, by the Riemann-Roch Theorem (specifically Corollary~\ref{RR to E}) $\dim_\Bbbk B_2=4$, whilst $\dim_\Bbbk \ovl{U}_2=3$ by Lemma~\ref{Ro 4.1} and (\ref{X_i=U_i}). Thus by (\ref{S(p-p1+p2) with bars eq 1}), it must be the case that $\ovl{F_2}=B_2$. In other words, $F_2=\{ x\in S_2\,|\; \ovl{x}\in B_2\}$. In which case $F$ contains the ring $U'=\Bbbk\langle X_1, F_2, X_3\rangle$. But this is exactly $U'$ from Example~\ref{original S(p-p1+p2)}. Since $F$ is $g$-divisible we would then have $F\supseteq \widehat{U'}$. However, in Example~\ref{original S(p-p1+p2)}(2) we will show that $\widehat{U'}=S$. This would mean $F\supseteq S$, which certainly contradicts Lemma~\ref{F is blowup}(2). Thus we must have $F_2=U_2$, or equivalently $\ovl{F}_2=\ovl{U}_2$. \par
Finally we show $\ovl{U}_1=\ovl{F}_1$. Again, since $\dim_\Bbbk B_1=2$ and $\dim_\Bbbk\ovl{U}_1=1$, if $\ovl{U}_1\neq \ovl{F}_1$, then $\ovl{F}_1=B_1$. But in this case, $\ovl{F}_2\supseteq B_1^2=B_2$ by Theorem~\ref{B properties}(4). This would contradict $\ovl{U}_2=\ovl{F}_2\subsetneq B_2$, hence $\ovl{U}_1=\ovl{F}_1$.
\end{proof}

Finally we can conclude $U$ is a virtual blowup.

\begin{prop}\label{S(p-p1+p2) main prop}
Let $U$ and $\bfx$ be as in Theorem~\ref{S(p-p1+p2)}. Then $U$ is a virtual blowup of $S$ at $\bfx$. More specifically, retaining notation from Notation \ref{virtual blowup notation},
$$U=H=\widehat{H}=F.$$
\end{prop}
\begin{proof}
By Lemma~\ref{F is blowup}, $U\subseteq H\subseteq\widehat{H}\subseteq F$. So we need to show $U=F$. We first prove $U$ is $g$-divisible. Set $W=\widehat{U}$. Since $U\subseteq F$ and $F$ is $g$-divisible by Lemma~\ref{F is blowup}, $U\subseteq W\subseteq F$. Therefore $\ovl{U}=\ovl{W}=\ovl{F}$ by Lemma~\ref{S(p-p1+p2) with bars}. As $gS\subseteq S_{\geq 3}$, this shows $U_{\leq 2}=W_{\leq 2}$. Proceeding by induction, let $n\geq3$ and assume $W_{<n}=U_{<n}$. Let $w\in W_n$, say with $wg^k\in U$. Since $\ovl{U}=\ovl{W}$, certainly $W\subseteq U+gS$, and so there exists $u\in U_n$ and $s\in S_{n-3}$ such that $w=u+sg$. We then have $sg^{k+1}+ug^k=wg^k\in U$, and so $sg^{k+1}=wg^k-ug^k \in U$ also. Hence $s\in W$. Moreover, since $\deg(w)=\deg(sg)$, we have $\deg(s)<\deg(w)=n$. By induction, $s\in U$. Therefore $w=u+gs\in U$ also, which proves $U=W$. \par

Now we have two $g$-divisible rings $U\subseteq F$, with $\ovl{U}=\ovl{F}$. Comparing Hilbert series we have
$$h_U(t)=h_{\ovl{U}}(t)/(1-t^3)=h_{\ovl{F}}(t)/(1-t^3)=h_F(t).$$
Since $U\subseteq F$, this forces $U=F$.
\end{proof}

As mentioned after Theorem~\ref{S(p-p1+p2)}, the definition of $X_2$ is surprising. Indeed the author had initially thought that we should have $U=\Bbbk\langle X_1,X_2',X_3\rangle$, where this time $X_2'=\{x\in S_2\,|\; \ovl{x}\in H^0(E,\LL_2(-[\bfx]_2))\}$. We now study this second ring. We also note that Example~\ref{original S(p-p1+p2)} completes the missing step in Lemma~\ref{S(p-p1+p2) with bars}.

\begin{example}\label{original S(p-p1+p2)} Let $U$ and $\bfx$ be as in Theorem~\ref{S(p-p1+p2)}.
Set $U'=\Bbbk\langle X'_1,X'_2,X'_3\rangle$, where $X'_i=\{x\in S_i\,|\;\ovl{x}\in H^0(E,\LL_i(-[\bfx]_i))\}$ for $i=1,2,3$. Then
\begin{enumerate}[(1)]
\item $U\subseteq U'$;
\item $\widehat{U'}=S$;
\item $U'$ is neither left nor right noetherian.
\end{enumerate}
\end{example}
\begin{proof}
(1). This is obvious from the definitions of $U$ and $U'$.\par
(2). Here we use Notation~\ref{p^sigma=p_1}. Identifying $\ovl{S_{\leq 2}}$ and $S_{\leq 2}$; $X'_2=S(p_3)_1S(p_{-1})_1$ by Lemma~\ref{Ro3.1 applied}(1), and is a 4-dimensional  $\Bbbk$-vector space by the Riemann-Roch Theorem (Corollary~\ref{RR to E}). We first prove that $g\in S(p_3)_1S(p_{-1})_1S(p_3)_1$, and so that
\begin{equation}\label{original S(p-p1+p2) eq}(X'_2)^2=S(p_3)_1S(p_{-1})_1S(p_3)_1S(p_{-1})_1\supseteq gS(p_{-1})_1.\end{equation}
Let $Y=S(p_{-1})_1S(p_3)_1$. Write $S(p_3)_1=v_1\Bbbk+v_2\Bbbk$. Then $S(p_3)_1Y=v_1Y+v_2Y$, and therefore
\begin{equation}\label{Ro 4.6 trick}
 \dim_\Bbbk S(p_3)_1Y=2\dim_\Bbbk Y-\dim_\Bbbk(v_1Y\cap v_2Y).
\end{equation}
Now, applying \cite[Lemma~4.2]{Ro} with $q=p_3=\sigma^{-3}(p)$, we can choose a basis of $S(p_5)_1=w_1\Bbbk+w_2\Bbbk$ such that $v_1w_1+v_2w_2=0$ and $v_1S\cap v_2S=v_1w_1S=v_2w_2S$. In which case
\begin{equation}\label{Ro 4.6 Z}
v_1Y\cap v_2Y=\{v_1w_1s\,|\; w_1s,w_2s\in Y\}=v_1w_1Z,\; \text{ where } Z=\{s\in S_1\;|\,S(p_5)_1s\subseteq Y\}.
\end{equation}
In particular $\dim_\Bbbk(v_1Y\cap v_2Y)=\dim_\Bbbk Z$. We compute $\ovl{Z}$ which, as $gS\subseteq S_{\geq 3}$, can be identified with $Z$. Take a section $\ovl{s}\in\ovl{S}_1=H^0(E,\LL)$, say $\ovl{s}$ vanishes at the effective divisor $\bfd$. Then $\ovl{S(p_5)}_1\ovl{s}$ will consist of global sections of $\LL_2$ vanishing at $p_5+\bfd^\sigma$. On the other hand $\ovl{Y}$ consists of global section of $\LL_2$ vanishing at $p_{-1}+p_4$ by Lemma~\ref{Ro3.1 applied}(1). In which case $\ovl{s}\in Z$ if and only if $\ovl{S(p_5)_1}\ovl{s}$ consists of global sections of $\LL_2$ vanishing at $p_{-1}+p_4$; that is, if and only if $p_5+\bfd^\sigma\geq p_{-1}+p_4$; or equivalently $p_{-2}+p_3\leq \bfd$. Thus $\ovl{Z}=H^0(E,\LL(-p_{-2}-p_3))$. In particular, $\dim_\Bbbk Z=\dim_\Bbbk\ovl{Z}=1$ by Riemann-Roch (Corollary~\ref{RR to E}). This means $\dim_\Bbbk S(p_3)_1Y=7$ by (\ref{Ro 4.6 trick}) and Lemma~\ref{Ro 4.1}. On the other hand $\ovl{S(p_3)_1Y}=H^0(E,\LL_3(-p-p_3-p_5))$, which is 6 dimensional by the Riemann-Roch Theorem. We hence must have $g\in S(p_3)Y$. This shows (\ref{original S(p-p1+p2) eq}). \par

We similarly show $g\in S(p_{-1})_1S(p_3)_1S(p_{-1})_1$. The same argument used to show (\ref{Ro 4.6 trick}) and (\ref{Ro 4.6 Z}) gives
\begin{equation}\label{Ro 4.6 trick2} \dim_\Bbbk S(p_{-1})_1S(p_3)_1S(p_{-1})_1=2\dim_\Bbbk Y-\dim_\Bbbk Z,\end{equation}
where this time $Y=S(p_3)_1S(p_{-1})_1$ and $Z=\{s\in S_1\;|\,S(p_1)_1s\subseteq Y\}$. As before we compute $\ovl{Z}$. By Lemma~\ref{Ro3.1 applied}(1), $\ovl{Y}=H^0(E,\LL_2(-p_3-p))$, whereas from the definition $\ovl{S(p_1)_1}=H^0(E,\LL(-p_1))$. Thus $\ovl{S(p_1)}\ovl{s}\subseteq Y$ if and only if $\ovl{s}$ vanishes at $p_2+p_{-1}$; that is, if and only if $\ovl{s}\in H^0(E,\LL(-p_2-p_{-1}))$. Therefore $\ovl{Z}=H^0(E,\LL(-p_2-p_{-1}))$; this is 1-dimensional by the Riemann-Roch Theorem. On the other hand $\dim_\Bbbk Y=4$, by Lemma~\ref{Ro 4.1}(2), and hence $\dim_\Bbbk S(p_{-1})_1S(p_3)_1S(p_{-1})_1=7$ from (\ref{Ro 4.6 trick2}). Since $\dim_\Bbbk\ovl{S(p_{-1})_1S(p_3)_1S(p_{-1})}=6$ by Lemma~\ref{Ro3.1 applied}(2) and the Riemann-Roch, we are forced to have $g\in S(p_{-1})_1S(p_3)_1S(p_{-1})$. In particular $gS(p_3)_1\subseteq (X'_2)^2$. Hence, with (\ref{original S(p-p1+p2) eq}),
$$ (X'_2)^2\supseteq g(S(p_{-1})_1+S(p_3))_1=gS_1.$$
It follows that $S_1\subseteq \widehat{U'}$. But $S=\Bbbk\langle S_1\rangle$; hence $\widehat{U'}=S$. \par

(3). By (2) and Proposition~\ref{C to C hat}, there exists an $k\geq 0$ such that $g^kS\subseteq U'$. Hence, if $U'$ is right noetherian, then $S_{U'}$ would be finitely generated. In which case, $\ovl{S}_{\ovl{U'}}$ would also be finitely generated. Put $\NN=\LL(-\bfx)$, by part (1) and Lemma~\ref{ovlU} we have $\ovl{U}\subseteq\ovl{U'}\subseteq B(E,\NN,\sigma)\ehd \ovl{U}.$ Hence also $\ovl{U'}\ehd B(E,\NN,\sigma)$. Thus by Corollary~\ref{nc serre2} if $\ovl{S}_{\ovl{U'}}$, then $\ovl{S}_n=H^0(E,\OO_E(\mbf{u})\otimes \NN_n)$ for some divisor $\mbf{u}$ and for all $n\gg 0$. It would then follow from the Riemann-Roch Theorem (Corollary~\ref{RR to E}) that $\dim_\Bbbk \ovl{S}_n=n+\deg(\mbf{u})$. However $\ovl{S}_n=H^0(E,\LL_n)$, and hence $\dim_\Bbbk\ovl{S}_n=3n$ for all $n$. This would give a contradiction, and therefore $U'$ cannot be right noetherian. A symmetric argument shows $U'$ is not left noetherian.
\end{proof}

Next, we show that we cannot expect nice homological properties to hold for virtual blowups. The following lemma is proved within the proof of \cite[Example~10.4]{RSS}.

\begin{lemma}\label{RSS 10.4}
Let $B$ be a cg domain with $\GKdim (B)=2$. Suppose that $B$ is Auslander-Gorenstein and Cohen-Macaulay.   Let $V$ is be a graded subalgebra of $B$ such that $0<\dim_\Bbbk(B/V)<\infty$. Then
\begin{enumerate}[(1)]
\item $\Ext_V^2(\Bbbk,V)\neq 0.$
\item $V$ is neither Auslander-Gorenstein nor AS-Gorenstein nor Cohen-Macaulay.\qed
\end{enumerate}
\end{lemma}

The proof of Corollary~\ref{S(p-p1+p2) bad homologically}(1) also essentially comes from the proof of \cite[Example~10.4]{RSS}. This time we include the proof.

\begin{cor}\label{S(p-p1+p2) bad homologically}
 Let $U$ and $\bfx$ be as in Theorem~\ref{S(p-p1+p2)}.
\begin{enumerate}[(1)]
\item $U$ is neither Auslander-Gorenstein nor AS Gorenstein nor Cohen-Macaulay;
\item $U$ has infinite injective dimension.
\end{enumerate}
\end{cor}
\begin{proof}
(1). Let $V=\ovl{U}$ and $B=B(E,\LL(-\bfx),\sigma)$. Note that by Theorem~\ref{B properties}(3), $B$ is Auslander-Gorenstein and Cohen-Macaulay. By Theorem~\ref{S(p-p1+p2)}, $U$ is a virtual blowup at $\bfx$. Hence $V\ehd B$; while $V\subseteq B$ by Lemma~\ref{CM lemma}(2). We claim $V\neq B$. Since $V\subseteq \ovl{S}$, it is enough to prove $B\not\subseteq \ovl{S}$. Now because $\deg\LL=3$ and  $\deg\LL(-\bfx)=2$, both $\LL$ and $\LL(-\bfx)$ are very ample by \cite[Corollary~IV.3.2]{Ha}. Therefore  both $\LL$ and $\LL(-\bfx)$ are generated by their respective global sections. Therefore, if $B_1\subseteq \ovl{S}_1$ (i.e. $H^0(E,\LL(-\bfx))\subseteq H^0(E,\LL)$) were true, then $\LL(-\bfx)\subseteq \LL$. Because $\bfx$ is not effective this is impossible. Thus we have
\begin{equation}\label{V subsetehd B} V\subsetneq B\;\text{ and } V\ehd B.\end{equation}\par

Now $U$ is $g$-divisible by Proposition~\ref{S(p-p1+p2) main prop}. Therefore by \cite[Theorem~5.10]{Lev}, $U$ is Auslander-Gorenstein if and only if $V\cong U/gU$ is Auslander-Gorenstein; with the same statement holding with either AS-Gorenstein or Cohen-Macaulay replacing Auslander-Gorenstein. Since $B$ is Auslander-Gorenstein and Cohen-Macaulay, Lemma~\ref{RSS 10.4}(2) applies. This shows $V$, and thus $U$, is neither Auslander-Gorenstein nor AS Gorenstein nor Cohen-Macaulay.\par

(2). It is enough to prove $\Ext_U^n(\Bbbk,U)\neq 0$ for infinitely many $n$. Since $U$ is $g$-divisible by Proposition~\ref{S(p-p1+p2) main prop}, the Rees-Lemma (for a precise statement see \cite[Proposition~3.4(b)]{Lev}) says $\Ext_V^{n}(\Bbbk,V)=\Ext_{U}^{n+1}(\Bbbk, U)$ where $V=\ovl{U}\cong U/gU$. It hence is suffices to prove $\Ext_V^n(\Bbbk,V)\neq 0$ for infinitely many $n$. \par
Set $B(E,\LL(-\bfx),\sigma)$. From (\ref{V subsetehd B}) we have $V\subsetneq B$ and $V\ehd B$. We first claim $\Ext_V^n(B,V)=0$ for all $n\geq1$. Applying \cite[Corollary~10.65]{Rot}, with their $(A,B,C,R,S)$ replaced with our $(B, B, V, V, B)$, we get
\begin{equation}\label{Rot 10.65 applied}
\Ext_V^n(B, V)=\Ext_V^n(B\otimes_B B, V)=\Ext_B^n(B, I) \, \text{ for } I=\Hom_V(B,V).
\end{equation}
As $B_B$ is obviously projective, $\Ext_B^n(B, I)=0$ for all $n\geq 1$. Thus (\ref{Rot 10.65 applied}) proves the claim.  \par

Set $L=B/V$, then $L$ is nonzero and finite dimensional. Consider the following short exact sequence of right $V$-modules, $0\to V\to B\to L\to 0$. Apply $\Hom_V(-,V)$ and look at the long exact sequence. Using the claim, one sees that
\begin{equation}\label{Ext(L,V)}\Ext_V^{n}(L,V)=0 \;\text{ for all }n\geq 2.\end{equation}\par

By Lemma~\ref{S(p-p1+p2) with bars} and its proof, $\dim_\Bbbk L=2$. Hence, if $x\in L\setminus\{0\}$ is homogenous of maximal degree, then $\dim_\Bbbk(xV)=\dim_\Bbbk (L/xV)=1$. We therefore have an exact sequence $0\to \Bbbk\to L \to \Bbbk\to 0$. Applying $\Hom_V(-,V)$ again and looking at the long exact sequence we get
\begin{equation}\label{S(p-p1+p2) long exact} \cdots \longrightarrow \Ext_V^2(L,V)\longrightarrow \Ext_V^2(\Bbbk,V)\longrightarrow\Ext_V^3(\Bbbk,V)\longrightarrow \Ext_V^3(L,V)\longrightarrow \cdots. \end{equation}
Thus $\Ext_V^2(\Bbbk,V)\cong\Ext_V^3(\Bbbk,V)$. Now from Lemma~\ref{RSS 10.4}(1) we already have that $\Ext_V^2(\Bbbk,V)\neq 0$, hence $\Ext_V^3(\Bbbk,V)\neq 0$ also. Looking further along the long exact sequence (\ref{S(p-p1+p2) long exact}) and applying induction, we obtain $\Ext_V^n(\Bbbk,V)\neq0$ for all $n\geq 2$.
\end{proof}

\subsection{Other examples}

Let $p,q\in E$. Recall from Definition~\ref{S(p+q)} that we define $S(p+q)$ to be generated by elements from degrees 1, 2 and 3. After Definition~\ref{S(p+q)} it is remarked  that we do not know whether the elements from degree 3 are necessary. That is, can $S(p+q)$ ever be generated in degrees 1 and 2? The answer to this is currently unknown. The following example shows that for at least some choices of $p$ and $q$, degrees 1, 2 and 3 are all necessary.

\begin{example}\label{S(p+p1)}
Let $p\in E$ and consider $R=S(p+p^{\sigma})$. Write
\begin{equation*}\label{S(p+p1) gen in 3} R=\Bbbk\langle V_1, V_2, V_3\rangle\;
\text{ for }V_i=\{x\in S_i\,|\; \ovl{x}\in H^0(E,\LL_i(-[p+p^{\sigma}]_i))\}\end{equation*}
as in Definition~\ref{S(p+q)}. Then $R$ is generated in degrees 1, 2 and 3 and no fewer.
\end{example}

\begin{proof}
For this proof Notation~\ref{p^sigma=p_1} is in force and we identify $S_{\leq 2}=\ovl{S}_{\leq 2}$. Also we note that from definition we have
\begin{equation}\label{ovlVi of S(p+p1)}
\begin{array}{l}
\ovl{V_1}=H^0(E,\LL(-p-p_1)),\\ \ovl{V_2}=H^0(E,\LL_2(-p-2p_1-p_2)),\\ \ovl{V_3}=H^0(E,\LL_3(-p-2p_1-2p_2-p_3)).
\end{array}
\end{equation}

Finally note that $R_i=V_i$ for $i=1,2,3$ by Lemma~\ref{S(p+q)_i=V_i}. \par

It is obvious that $R$ can be generated in degrees 1, 2 and 3. Since $R_1=V_1$ and $\dim_\Bbbk V_1=\dim_\Bbbk \ovl{V_1}=1$ by the Riemann-Roch (Corollary~\ref{RR to E}), $\Bbbk\langle R_1\rangle\cong \Bbbk[t]$. So certainly $R$ cannot be generated in degree 1. To show $R$ is not generated in degrees 1 and 2  we must show
\begin{equation}\label{k<V1V2>_3}\Bbbk\langle V_1,V_2\rangle_3=V_1^3+V_1V_2+V_2V_1\neq V_3.\end{equation}
What we first show is $V_1V_2,V_2V_1\subseteq X=S(p)_1S(p_2)_1S_1$. Now by (\ref{ovlVi of S(p+p1)}) and Lemma~\ref{Ro 4.1}(1), $\ovl{V_2}\subseteq H^0(E,\LL_2(-p_2))=\ovl{S(p_2)_1S_1}$.
Hence $V_2\subseteq S(p_2)_1S_1$. Since clearly $V_1=S(p+p_1)_1\subseteq S(p)_1$, it follows that $V_1V_2\subseteq X$. By Lemma~\ref{Ro 4.1}(1), $X=S(p)_1S_1S(p_1)_1$ also. We can similarly show $V_2\subseteq S(p)_1S_1$ and $V_1\subseteq S(p_1)_1$. This gives $V_2V_1\subseteq X$, and then  $V_1V_2+V_2V_1\subseteq X$ follows. By \cite[Lemma~4.6(1)]{Ro}, $g\notin X$ and thus (since $g\in V_3$ by definition) $V_1V_2+V_2V_1\neq V_3$. Finally, it is clear that
$$\ovl{V_1}^2\subseteq H^0(E,\LL_2(-p-2p_1-p_2))=\ovl{V_2}.$$
Hence $V_1^2\subseteq V_2$, and $V_1^3\subseteq V_1V_2+V_2V_1$. Thus (\ref{k<V1V2>_3}) holds.
\end{proof}

Our final example shows that the assumption $\ovl{U}\neq\Bbbk$ in our classification of maximal $S$-orders is necessary. In \cite[Example~10.8]{RSS}, Rogalski, Sierra and Stafford show that when $Q_\gr(U)=Q_\gr(S)^{(3)}$ silly examples can occur. Their example can be modified to work for orders in $S$.\par
For this example, we recall the notion of a semi-graded homomorphism from Definition~\ref{semi-graded hom}.

\begin{example}\label{RSS 10.8}
Choose $a,b,c\in S_1$ that generate $S$ as an $\Bbbk$-algebra. Define the subalgebra $S^g=\Bbbk\langle ga,gb,gc\rangle$ of $S^{(4)}$ and set $U=S^g\langle g\rangle$. Then:
\begin{enumerate}[(1)]
\item $\ovl{U}=\Bbbk$;
\item There are semi-graded isomorphisms $S^g\cong S$ and $S[x]/(x^4-g)\cong U$. Also $U^{(4)}=S^g$ holds;
\item $U$ is noetherian and $gU$ is a completely prime ideal of $U$ such that, up to semi-graded isomorphism, $U/gU\cong S/gS$;
\item $\widehat{U}=S$, but $S$ is not finitely generated on either side as a $U$-module;
\item $U$ is a maximal order with $Q_\gr(U)=Q_\gr(S)$.
\end{enumerate}
\end{example}
\begin{proof}
(1). This is obvious. \par
(2). The semi-graded homomorphism  $\varphi:S\to S$, $s\mapsto g^is$ for $s\in S_i$ clearly has image $S^g$. It is injective as $S$ is a domain. Define $\psi:S[x]\to S$ defined by $\psi(s)=g^is$ for $s\in S_i$ and $\psi(x)=g$. Then $\psi$ has $\mathrm{im}(\psi)=U$ and $\ker(\psi)=(x^4-g)$. One may view $\psi$ as graded of degree $4$ if one is willing to accept $S[x]$ as $(\N+\frac{3}{4}\N)$-graded and $\deg x=\frac{3}{4}$. To see $S^g=U^{(4)}$, note that clearly $g^4\in S^g$. Moreover, since $\deg(g)=3$ and $4$ are coprime we have
$U^{(4)}=(S^g\langle g\rangle)^{(4)}=(S^g\langle g^4\rangle)^{(4)}=S^g$. \par

(3). The fact the $U$ is noetherian follows directly from  $U\cong S[x]/(x^4-g)$. In this isomorphism $x$ gets identified with $g$, we therefore have
$$U/gU\cong S[x]/(x,x^4-g)\cong S/gS.$$\par

(4). If $s\in S_n$, then $g^ns\in S^g\subseteq U$, and so $s\in \widehat{U}$. Hence $\widehat{U}=S$. If $S_U$ (or $_US$) was finitely generated then $\ovl{S}_{\ovl{U}}$ (respectively $_{\ovl{U}}\ovl{S}$)  would be finitely generated. This is impossible because $\ovl{U}=\Bbbk$.  \par

(5). Since $g\in U$, we have that $g^{-1}ga=a\in Q_\gr(U)$; similarly the other two generators of $S$, $b,c\in Q_\gr(U)$. Thus indeed $Q_\gr(U)=Q_\gr(S)$.\par
Identify $U= S[x]/(x^4-g)$. We shift the grading given in part (2) so that $\deg x=3$ and $\deg (S_1)=4$. With this new grading we have that $U^{(4)}=S$ is a maximal order. Suppose $V$ is an equivalent order containing $U$. By Lemma~\ref{equiv orders go up}, $U^{(4)}$ and $V^{(4)}$ are equivalent orders. By (2), $U^{(4)}=S$ is a maximal order, and hence $U^{(4)}=V^{(4)}$. Take a homogenous element $v\in V$, say $\deg(v)=d$. If 4 divides $d$, then we already know $v\in U$, so suppose $4$ does not divide $d$.  As 3 and 4 are coprime, we can choose $n$ with $1\leq n\leq 3$ and such that $d+3n$ is divisible by 4. In which case, $vx^n\in V^{(4)}=U^{(4)}=S$. Hence we can write $v=sx^{-n}$ for some $s\in S$. Also we have $S=U^{(4)}=V^{(4)}\ni v^4=(sx^{-n})^4=s^4g^{-n}$. This implies $s^4\in gS$. But $gS$ is a completely prime ideal, hence $s\in gS$; say $s=gt$ for some $t\in S=U^{(4)}$. Putting these together, $v=sx^{-n}=gtx^{-n}=tx^{4-n}$. But since $n\leq 3$, we have $v=tx^{4-n}\in U$. Thus $V=U$, and thus $U$ is a maximal order.
\end{proof}

\section{Proof of Lemma~\ref{Grassmannian}}

In this Appendix we present the proof of Lemma~\ref{Grassmannian} which was omitted. We certainly do not claim originality of the statement and proof, however we were unable to find an appropriate reference.\par
Recall that we are writing $\Gr(m,V)$\index[n]{grmv@$\Gr(m,V)$}\index{Grassmannian} for the grassmannian of $m$-dimensional subspaces of a fixed $n$-dimensional $\Bbbk$-vector space $V$. It is well know that $\Gr(m,V)$ is a projective variety and has a standard open cover of copies of $\mathbb{A}^{(n-m)m}$ (see for example \cite[Chapters 11.3, 11.5]{Hassett}). The standard affine cover is given as follows. First fix a basis of $V$ and let $W$ be an $m$ dimensional subspace. Then form an $m\times n$ matrix $A$ which has as rows the coordinates of a basis $W$. Next one row reduces $A$ into a matrix which has the $m$ columns $(1,0,\dots,0)^T$, \dots, $(0,\dots,0,1)^T$ in positions $1\leq s_1<\dots<s_m\leq n$ respectively. Finally one identifies the remaining matrix entries with $\mathbb{A}^{(n-m)m}$.

\begin{lemma}[Lemma~\ref{Grassmannian}]\label{grass appendix}
Let $V_2$ and $V_3$ be a $4$-dimensional and $7$-dimensional $\Bbbk$-vector space respectively, and set
$$\Omega_2= \{(W_1,W_2)\in \Gr(3,V_2)^2\,|\; \dim_\Bbbk (W_1\cap W_2)=2\}$$
and
\begin{equation*}\Omega_3=\{(W_1,W_2,W_3)\in \Gr(6,V_3)^3\,| \dim_\Bbbk (W_1\cap W_2\cap W_3)=4\}.\end{equation*}
Then $\Omega_2$ and $\Omega_3$ are Zariski open subsets of $ \Gr(3,V_2)^2$ and $\Gr(6,V_3)^3$ respectively; and the maps $\psi_2:\Omega_2\to \Gr(2,V_2)$ and  $\psi_3:\Omega_3\to \Gr(4,V_3)$, given by
\begin{equation*}\psi_2:(W_1,W_2)\longmapsto W_1\cap W_2  \;\text{ and }\; \psi_3:(W_1,W_2,W_3)\longmapsto W_1\cap W_2\cap W_3\end{equation*}
are morphisms of varieties.
\end{lemma}
\begin{proof}
Since $\Omega_2=\Gr(3,V_2)^2\setminus \Delta$, where $\Delta$ is the diagonal of $\Gr(3,V_2)^2$, $\Omega_2$ is open. We will not prove $\psi_2$ is a morphism since the proof is both similar to, and easier than, that of the corresponding proof for $\psi_3$. Thus from here we concentrate on $V=V_3$, $\Omega=\Omega_3$ and $\psi=\psi_3$.\par
Fix a basis of $V$ and consider the image $\Omega^*$ of $\Omega$ under the isomorphism taking $\Gr(6,V)^3$ to the dual $(\Gr(6,V)^*)^3=\Gr(1,V)^3$ (see \cite[Proposition~11.18]{Hassett}). Under this identification $$\Omega^*=\{(w_1\Bbbk,w_2\Bbbk,w_3\Bbbk)\in\Gr(1,V)^3\,|\; w_1,w_2,w_3 \text{ are linearly independent}\}.$$
To see $\Omega^*$ is open we identify $\Gr(1,V)=\bbP(V)=\bbP^6$. Write $w_i\Bbbk=[a_{i0}:\dots :a_{i6}]$. Then $(w_1\Bbbk,w_2\Bbbk,w_3\Bbbk)\notin \Omega^*$ (i.e. $w_1,w_2,w_3$ are linearly dependent) if and only if the determinants of all $3\times 3$ minors of the matrix $(a_{ij})$ vanish. In other words, an element is not in $\Omega^*$ if and only it is a zero of a finite set of homogenous equations. So $\Omega^*$ is open and therefore $\Omega$ is also open. \par
To prove $\psi$ is a morphism we first observe $\psi$ commutes with taking the dual, that is, the following diagram commutes.

\begin{equation*}\label{omega grass duals}
\xymatrix{
\Omega \ar[r]^{\psi} \ar[d]_{\cong} & \Gr(4,V)  \\
\Omega^* \ar[r]^{\psi^*} & \Gr(3,V) \ar[u]^{\cong} \\
}
\end{equation*}
where the isomorphisms in  the dual maps $\Omega\overset{\simeq}\longrightarrow\Omega^*$ and $\Gr(3,V)\overset{\simeq}\longrightarrow\Gr(4,V)$, and
$$\psi^*:(w_1\Bbbk,w_2\Bbbk,w_3\Bbbk) \mapsto \mathrm{span}(w_1,w_2,w_3).$$
Hence to prove $\psi$ is a morphism, it is enough to prove $\psi^*$ is a morphism. Identify $\Gr(1,V)=\bbP(V)=\bbP^6$ and let
$$\omega=([a_0:\dots:a_6],[b_0:\dots :b_6],[c_0:\dots:c_6])\in \Omega^*\subseteq\bbP^6\times\bbP^6\times\bbP^6 .$$
Then
$$\psi^*(\omega)=\mathrm{rowspace}\left( \begin{array}{ccc} a_0 &  \dots & a_6\\ b_0 &  \dots & b_6 \\ c_0 & \dots & c_6 \end{array} \right).$$
As $\omega\in\Omega^*$, the rows are linearly independent, therefore we can choose linearly independent columns: $\left(\begin{array}{ccc}a_i\\b_i\\c_i\end{array}\right)$,$\left(\begin{array}{ccc}a_j\\b_j\\c_j \end{array}\right)$ and $\left(\begin{array}{ccc}a_k\\b_k\\c_k\end{array}\right)$, with $1\leq i<j<k\leq 6$. Put $A=\left(\begin{array}{ccc}a_i&a_j&a_k\\b_i&b_j&b_k\\c_i&c_j&c_k\end{array}\right),$
then also,
$$\psi^*(\omega)=\mathrm{rowspace}\left(A^{-1}\left(\begin{array}{cccc} a_0 &  \dots & a_6\\ b_0 & \dots & b_6 \\ c_0 & \dots & c_6 \end{array} \right)\right)$$
which realises $\psi^*(\omega)$ inside one of the standard $(\mathbb{A}^4)^3$ covering $\Gr(3,V)$. Explicitly, for a fixed $i<j<k$ as above, and $\{i',j',k',\ell'\}=\{1,\dots,6\}\setminus \{i,j,k\}$ with $i'<j'<k'<\ell'$, the rational map $\psi^*:\Omega^*\dashrightarrow (\mathbb{A}^4)^3$ is given by the rows of the matrix
$$A^{-1}\left(\begin{array}{cccc}a_{i'}&a_{j'}&a_{k'}&a_{\ell'}\\b_{i'}&b_{j'}&b_{k'}&b_{\ell'}\\c_{i'}&c_{j'}&c_{k'}&c_{\ell'}\end{array}\right).$$
Each element of this matrix is of the form $\frac{h}{f}$, where $h$ and $f$ are homogenous polynomials in the $\{a_r,b_s,c_t\}$ of degree $3$. Hence this does define a rational map on $\Omega^*$. The domain of definition of this rational map is the open set $\Theta_{ijk}$ of $\Omega^*$ consisting of the points where $\left(\begin{array}{ccc}a_i\\b_i\\c_i\end{array}\right)$,$\left(\begin{array}{ccc}a_j\\b_j\\c_j \end{array}\right)$ and $\left(\begin{array}{ccc}a_k\\b_k\\c_k\end{array}\right)$ are linearly independent. Varying $1\leq i<j<k\leq 6$, these open sets cover $\Omega^*$ and the maps $\psi^*|_{\Theta_{ijk}}$ all glue because they are globally defined. Thus we get $\psi^*:\Omega^*\to \Gr(3,V)$ as a well defined morphism of varieties. As noted earlier this implies $\psi$ is a morphism of varieties.
\end{proof}

The proof of Lemma~\ref{Grassmannian} certainly works for the general statement with arbitrary dimensions of $V$ and $\Omega$. For sake of notation we only proved the statement we required.

\newpage

\printindex[n]
\printindex
\newpage

\bibliographystyle{plain}
\bibliography{Bib}

\end{document}